\documentclass[10pt,a4paper]{article}
\usepackage[latin1]{inputenc}
\usepackage[margin=3.5cm]{geometry}
\usepackage{amsmath}
\usepackage{amsfonts}
\usepackage{stmaryrd}
\usepackage{amssymb}
\usepackage{multirow}
\usepackage{amsthm}
\usepackage{enumerate}
\usepackage[table]{xcolor}
\usepackage{hhline}
\usepackage{makecell}
\usepackage{tikz}
\usepackage{ifthen}
\usepackage{subfiles}
\usepackage{accents}
\usepackage{mathrsfs}
\newenvironment{tproof}{\begin{proof}[\textbf{Proof}]}{\end{proof}}
\newlength{\dhatheight}

\newcommand{\aaa}{\alpha}
\newcommand{\aA}[4]{(#1,#2,#3)A^{#4}}
\newcommand{\bB}[4]{(#1,#2,#3)B^{#4}}
\newcommand{\AB}[4]{(#1,#2,#3)(AB)^{\lfloor{\frac{#4}{2}}\rfloor-1}A^{#1-2\lfloor{\frac{#4}{2}}\rfloor}}
\newcommand{\ABB}[4]{(#1,#2,#3)AB^{#4}}
\newcommand{\BAB}[4]{(#1,#2,#3)B(AB)^{\lfloor{\frac{#4}{2}}\rfloor-1}A^{#4-2\lfloor{\frac{#4}{2}}\rfloor}}
\newcommand{\aAbig}[3]{(#1,#2)A^{#3}}
\newcommand{\bBbig}[3]{(#1,#2)B^{#3}}

\newcommand{\jj}{\mathrm{j}}
\newcommand{\ja}{\mathrm{j}}
\newcommand{\jm}{\mathrm{j}^{-1}}

\newcommand{\tj}{\tilde{j}}
\newcommand{\tjp}{\tilde{j}^\prime}
\newcommand{\tjA}{\tilde{j}_{A}}
\newcommand{\tjB}{\tilde{j}_{B}}
\newcommand{\Wrow}[1]{\mathrm{W}_n|_{\mathrm{row}(#1)}}
\newcommand{\Krow}[1]{\mathrm{K}|_{\mathrm{row}f(#1)}}
\newcommand{\hsp}{\hspace{1mm}}
\newcommand{\hin}{\hspace{2mm}}
\newcommand{\rossa}[1]{{#1}^{\mathrm{root}}}
\newcommand{\Pac}[1]{P_{#1}^{\mathrm{ac}}}
\newcommand{\BR}[1]{\mathcal{B}\mathrm{row}_{#1}}
\newcommand{\zAF}[1][]{\mathrm{0}_{\AAnk#1}}
\newcommand{\quasi}[1]{\mathrm{quasi}(#1)}

\newcommand{\rmo}[1]{\mathrm{rm}(#1)}
\newcommand{\SSS}{\mathcal{S}}
\newcommand{\XXX}{\mathcal{X}}
\newcommand{\ih}{\mathrm{ih}}

\newcommand{\Aut}[1]{\mathrm{Aut}_{#1}}
\newcommand{\AnkT}[1]{\AAnk[][T_{#1}]}
\newcommand{\AnkB}[1]{\AAnk[][B_{#1}]}

\newcommand{\Lie}[1]{\mathrm{Lie}(#1)}
\newcommand{\ossa}{2\mathrm{a}1\mathrm{d}}

\newcommand{\tlarrow}{\leftarrowtriangle}
\newcommand{\mirsl}[1]{P_{#1}^0}
\newcommand{\pse}[1][17]{\mkern-#1mu \searrow}
\newcommand{\mult}[2]{\mathrm{mult}(#1,#2)}
\newcommand{\Ind}{\mathrm{Ind}}
\newcommand{\trivr}[1]{\mathbf{1}_{#1}}

\newcommand{\exchange}[2]{\mathrm{ue}\left(#1,#2\right)}

\newcommand{\TT}[2]{\mathrm{Stab}_{T_{#1}}(#2)}
\newcommand{\Stab}[2]{\mathrm{Stab}_{#1}(#2)}
\newcommand{\DDDD}[1][]{\mathcal{D}_{#1}} 
\newcommand{\PPPP}{P_\pi}

\newcommand{\ffff}{\tilde{f}}

\newcommand{\SRV}[1]{\mathrm{srg}(#1)}
\newcommand{\rcirc}{\circ_{GL}}
\newcommand{\KKK}{\mathfrak{K}}
\newcommand{\kkk}{\mathfrak{k}}

\newcommand{\ooo}{\mathfrak{o}}

\newcommand{\db}[1]{#1^{\ssquare}{}_\square}
\newcommand{\Pridbnk}[1][]{(\mathcal{C}^{\ssquare}{}_\square)_{#1}}

\newcommand{\Priubnk}[1][]{(\mathcal{C}^{\square}{}_\ssquare)_{#1}}
\newcommand{\tPridb}[1][]{(\hat{\mathcal{C}}^{\ssquare}{}_\square)_{\kkk#1}}
\newcommand{\tPridbnk}[1][]{(\hat{\mathcal{C}}^{\ssquare}{}_\square)_{#1}}
\newcommand{\tPriub}[1][]{(\hat{\mathcal{C}}^{\square}{}_\ssquare)_{\kkk#1}}
\newcommand{\tPriubnk}[1][]{(\hat{\mathcal{C}}^{\square}{}_\ssquare)_{#1}}

\newcommand{\AAA}[1][]{\mathcal{A}_{\kkk#1}}
\newcommand{\AAnk}[1][]{\mathcal{A}_{#1}}
\newcommand{\Prink}[1][]{\mathcal{C}_{#1}}
\newcommand{\tPrink}[1][]{\hat{\mathcal{C}}_{#1}}

\newcommand{\BBnk}[1][]{\mathcal{B}_{#1}}
\newcommand{\VVV}{\mathcal{V}}
\newcommand{\EEE}{\mathcal{E}}
\newcommand{\UUU}[1]{\mathcal{N}_{#1}}

\newcommand{\Proj}[1]{\mathrm{Pr}_{#1}}
\newcommand{\s}{\setminus}

\newcommand{\rii}[2]{\overset{#1}{\underset{#2}{\rightarrow}}}
\newcommand{\ssim}[2]{\overset{#1}{\underset{#2}{\sim}}}
\newcommand{\sett}[1]{\mathbf{#1}}

\newcommand{\Pri}[1][]{\mathcal{C}_{\kkk#1}}
\newcommand{\tPri}[1][]{\hat{\mathcal{C}}_{\kkk#1}}

\newcommand{\Set}[1]{\mathrm{Set}(#1)}
\newcommand{\itr}[2]{i_{\text{tr}}(#1,#2)}

\newcommand{\Skal}[2]{\mathcal{S}(#1,#2)}
\newcommand{\Skac}[2]{\mathcal{S}_{\mathrm{comp}}(#1,#2)}

\newcommand{\K}{\mathrm{K}}

\newcommand{\QQ}{\mathrm{Q}}

\newcommand{\XX}{\mathrm{X}}

\newcommand{\FF}{\tilde{\mathrm{F}}}

\newcommand{\F}{\mathrm{F}}
\newcommand{\aF}{\mathrm{F}}
\newcommand{\YY}{\mathrm{Y}}

\newcommand{\WW}{\mathrm{W} }

\newcommand{\Z}{\mathrm{Z}}
\newcommand{\JJ}{\mathrm{J}}
\newcommand{\Jdeg}[1][]{\mathrm{J}^{#1}_{\mathrm{deg}}}
\newcommand{\JJJJ}[2]{\mathrm{J}^{#1}_{#2}}

\newcommand{\I}[1][]{\mathrm{I}^{#1}}

\newcommand{\sI}[1][]{\mathrm{I}^{\mathrm{st}#1}}
\newcommand{\sXi}{\mathrm{\Xi}^{\mathrm{st}}}
\newcommand{\siota}[1][]{\iota^{\mathrm{st}#1}}

\newcommand{\Om}{\mathcal{O}}

\newcommand{\Omm}[1]{\mathcal{O}_{\mathbf{f}}(#1)}

\newcommand{\row}[1]{\text{row}(#1)}
\newcommand{\rowf}[1]{\text{row}f(#1)}

\newcommand{\bb}{\bullet}

\newcommand{\x}{\bb}

\newcommand{\eu}{\text{ue}}
\newcommand{\A}{\mathbb{A}_{\kkk}}
\newcommand{\Ai}{\mathbb{A}}

\newcommand{\co}{\text{c}}
\newcommand{\e}{\text{e}}
\newcommand{\gray}{gray!30}
\newcommand{\grayi}{gray!60}
\newcommand{\grid}[1]{\draw (0,0)[step=1,\gray]grid(#1,#1);}

\newcommand{\di}{6}
\newcommand{\mm}[2]{(#2,\di-#1)}
\newcommand{\rec}[5][]{\draw(#3,\di-#2)rectangle(#5,\di-#4)node[black,midway]{#1};}
\newcommand{\dia}[2]{\draw(#1-0.5,\di+0.5-#1)node{#2};}
\newcommand{\bu}[1]{\draw[fill=black]#1 circle (0.2);}
\newcommand{\ci}[4][\di]{\ifthenelse{\equal{#4}{a}}{\draw (#2-0.5,#1+0.5-#3)[gray!70,fill=gray!70]circle(0.17);}{\ifthenelse{\equal{#4}{b}}{\fill (#2-0.5,#1+0.5-#3)circle(0.2);}{\ifthenelse{\equal{#4}{c}}{\draw (#2-0.5,#1+0.5-#3)[line width=1pt,gray!70]circle(0.2);}{}} }\ifthenelse{\equal{#4}{d}}{\draw (#2-0.5,#1+0.5-#3)[line width=1pt]circle (0.2);}{\ifthenelse{\equal{#4}{e}}{\draw[line width=1pt,gray!70](#2-0.7,#1+0.3-#3)--(#2-0.3,#1+0.3-#3)--(#2-0.5,#1+0.8-#3)--(#2-0.7,#1+0.3-#3);}{\ifthenelse{\equal{#4}{f}}{\draw[line width=1pt](#2-0.7,#1+0.3-#3)--(#2-0.3,#1+0.3-#3)--(#2-0.5,#1+0.8-#3)--(#2-0.7,#1+0.3-#3); }}}} 

\usetikzlibrary{patterns}

\newcommand{\rr}[4][]{\fill[white](-1+#3,\di-#2+1)rectangle +(1,-#4);\draw (-1+#3,\di-#2+1)[\grayi,pattern color= orange!85, pattern=north east lines]rectangle +(1,-#4);\draw(#3-0.2,\di-#2)node{#1};}

\newcommand{\rn}[4][]{\draw(-1+#3,\di-#2)[line width=3pt,blue!34]rectangle+(#4,1);\draw (#3-1,\di-#2+0.2)node{#1}; }
\newcommand{\rbr}[4][]{\draw(-1+#3,\di-#2)[line width=2pt,dashed,blue!29]rectangle+(#4,1);\draw (#3,\di-#2+0.2)node{#1}; }

\newcommand{\rl}[4][]{\draw(-1+#3,\di-#2)[\grayi,fill=cyan!27]rectangle+(#4,1);\draw (#3,\di-#2+0.4)node{#1}; }
\newcommand{\lex}[4][]{\draw[orange](#3-1,\di-#2+0.5)--(#3-1+#4,\di-#2+0.5);\draw (#3,\di-#2+0.4)node{#1}; }

\newcommand{\fir}[2][1]{\draw(0,#1)node[anchor=west,fill=white,inner sep=2pt]{$#2$};}
\newcommand{\bfir}[1]{\draw(0,1)node[anchor=west,black,fill=white,inner sep=2pt]{$#1$};}

\newcounter{row}
\newenvironment{ti}[1][0.6]{\begin{tikzpicture}[scale=#1]\setcounter{row}{1}}{\end{tikzpicture}}
\newcommand{\f}[2]{\draw (\value{row}-0.5,\di+0.5-\value{row})node{#1};\foreach \x/\y in {#2} {\ci{\x}{\value{row}}{\y}}\stepcounter{row}}

\newcommand{\rif}{\overset{\mathrm{f}}{\rightarrow}}

\newcommand{\chh}[1][]{}

\newcommand{\Di}[1]{\mathrm{dim}\left(#1\right)} 

\newcommand{\X}[2][]{\mathbf{X}_{#1}(#2)}

\newcommand{\vsp}{\vspace{1,5mm}}
\newtheorem{theorem}{Theorem}[section]
\newtheorem{maincor}[theorem]{Main corollary}
\newtheorem{stheorem}{Theorem}[subsection] 

\newtheorem{prop}[theorem]{Proposition}
\newtheorem{sprop}[stheorem]{Proposition}
\newtheorem{lem}[theorem]{Lemma}
\newtheorem{slem}[stheorem]{Lemma}
\newtheorem{cor}[theorem]{Corollary }
\newtheorem{scor}[stheorem]{Corollary }
\newtheorem{ob}[theorem]{Observation} 
\newtheorem{Ecor}[stheorem]{Exchange corollary}
\newtheorem{sob}[stheorem]{Observation}
\newtheorem{NPQ}[stheorem]{Non-precise question}

\newtheorem{fact}[theorem]{Fact}
\newtheorem{sfact}[stheorem]{Fact}

\theoremstyle{definition}
\newtheorem{Strong}[stheorem]{Strong question}

\newtheorem{sexample}[stheorem]{Example}
\newtheorem{axbx}[theorem]{Definition}
\newenvironment{defi}
  {\pushQED{\qed}\axbx}
  {\popQED\endaxbx}
 \newtheorem{axbxv}[theorem]{Warning}
 \newenvironment{warning}
   {\pushQED{\qed}\axbxv}
   {\popQED\endaxbxv} 
 \newtheorem{saxbx}[stheorem]{Definition}
  \newenvironment{sdefi}
   {\pushQED{\qed}\saxbx}
   {\popQED\endsaxbx} 
\newtheorem{axbxx}[stheorem]{Convention}  
  \newenvironment{convention}
    {\pushQED{\qed}\axbxx}
    {\popQED\endaxbxx}
    \newtheorem{axbxxj}[theorem]{Convention}  
      \newenvironment{aconvention}
        {\pushQED{\qed}\axbxxj}
        {\popQED\endaxbxxj}
 \newtheorem{axbxxx}[theorem]{Remark}  
   \newenvironment{remark}
     {\pushQED{\qed}\axbxxx}
     {\popQED\endaxbxxx}   
      \newtheorem{saxbxxx}[stheorem]{Remark}  
        \newenvironment{sremark}
          {\pushQED{\qed}\saxbxxx}
          {\popQED\endsaxbxxx}

\theoremstyle{remark}

\DeclareFontFamily{U}{mathb}{\hyphenchar\font45}
\DeclareFontShape{U}{mathb}{m}{n}{
      <5> <6> <7> <8> <9> <10> gen * mathb
      <10.95> mathb10 <12> <14.4> <17.28> <20.74> <24.88> mathb12
      }{}
\DeclareSymbolFont{mathb}{U}{mathb}{m}{n}

\DeclareMathSymbol{\act}{3}{mathb}{'374}
\DeclareMathSymbol{\ssquare}{3}{mathb}{"05}
\newtheoremstyle{subdefi}{1pt}{1pt}{}{}{\bfseries}{.}{.5em}{}
\theoremstyle{subdefi}
\newtheorem*{axbxxxt}{Subdefinition}  
   \newenvironment{subdef*}
     {\pushQED{\qed}\axbxxxt}
     {\popQED\endaxbxxxt}

\begin{document}
\title{On Fourier Coefficients of $GL(n)$-Automorphic Functions over Number Fields}
\date{}
\author{Eleftherios Tsiokos}
\maketitle
\begin{abstract}
We study Fourier coefficients of $GL_n(\A)$-automorphic functions $\phi$, for $\A$ being the adele ring of a number field $\kkk$. Let FC be an abbreviation for such a Fourier coefficient (and FCs for plural). Roughly speaking, in the present paper we process FCs by iteratively using the operations: Fourier expansions, certain exchanges of Fourier expansions, and conjugations. In Theorem \ref{general} we express any FC  in terms of---degenerate in many cases---Whittaker FCs. For FCs obtained from the trivial FC by choosing a certain ``generic" term in each Fourier expansion involved, we establish a shortcut (Main corollary \ref{maincor}) for studying their expressions of the form in Theorem \ref{general}. The shortcut gives considerably less information, but it remains useful on finding automorphic representations so that for appropriate choices of  $\phi$ in them, the FC is factorizable and nonzero. Then in Theorems \ref{thD1}, \ref{thD2}, and \ref{thD3}, we study examples of  FCs on which this shortcut applies, with many of them turning out to ``correspond" to more than one unipotent orbit in $GL_n.$ 

For most of the paper, no knowledge on automorphic forms is necessary.

\end{abstract}
\tableofcontents 
\section{Introduction} We study Fourier coefficients of $GL_n(\A)$-automorphic functions, for $\A$ being the adele ring of a number field $\kkk$.

We refer to the results mentioned in the abstract as the three main results of the paper, which are: Theorem \ref{general}; Main corollary \ref{maincor}; and (the similar to each other) Theorems \ref{thD1}, \ref{thD2}, and \ref{thD3}. Other than some modest uses of algebraic geometry, the proofs and statements of these results seem to be very elementary. In particular they do not require\footnote{ The three main results are about pairs $(N,H)$ for: $N$ being a unipotent algebraic subgroup of $H$, and $H$ being a general linear group or a  mirabolic group. Also, these results are restricted to characteristic zero but are not restricted to number fields.

The paper is less readable for one not acquainted with automorphic representations, but it should still be possible to reach the three main results by skipping every statement along the way that involves automorphic functions; this skipping may require some effort in the introduction.} knowledge in: automorphic forms, representation theory, and number theory.

The results of the present paper are introduced in Subsection \ref{threemain}. Before it, we motivate the results by discussing Fourier coefficients of automorphic forms, and explaining some notation. Some of these notations lack precision, but they are clarified in Section \ref{allnot}; in particular, the notations in Subsection \ref{treesroughly} are addressed in Section \ref{allnot} (mainly in Subsection \ref{treessubs}) in a clarified, refined, and independent way. 

Plenty of new notation is used, most of it---$\AAnk$-operations and $\AAnk$-trees (introduced in Subsection \ref{treesroughly})---is about the standard types of operations we mentioned in the abstract.  

I can be reached at etsiokos@gmail.com
\subsection{Questions on Fourier coefficients of automorphic forms}\label{FCAut}
 
Let $\KKK$ be an algebraically closed field of characteristic zero over which we define all varieties, $\kkk$ be a number field contained in $\KKK$, and $n$ a positive integer. By an additive function we mean an algebraic homomorphism from an algebraic group to  the additive group $\KKK$. We write ``AF" in place of ``additive function" (and ``AFs" for plural). Let $\AAnk[n]$  be the set of AFs with domain  a unipotent algebraic subgroup of $GL_n,$ and $\AAnk:=\cup_{k>0}\AAnk[k].$ For $\F$ being an AF, we denote its domain by $D_\F.$ For $N$ being a unipotent algebraic group we define
$\AAnk(N)$ to be the set of AFs with domain equal to $N$. We denote by $\F_{\emptyset,n}$ the AF with domain the trivial subgroup of $GL_n$.

Let $\ast$ be the text used to denote a set of AFs. Consider any text $\ast^\prime$ appearing in the paper which is obtained from $\ast$ by attaching as a subscript anywhere within it the text ``$\kkk,$". Then $\ast^\prime$ is used to denote the set consisting of the AFs of $\ast$ which are defined over $\kkk.$ For example for $\ast=\AAnk[n]$ we have $\ast^\prime=\AAA[,n] $. A minor exception: we write $\AAA $ instead of $\AAA[,]$.

For $\kkk$ we fix a choice of a  unitary character $\psi_\kkk$ on $\kkk\s \A$. Let $\F\in\AAA[,n]$. We define $\psi_\F$ by: $\psi_\F(x)=\psi_\kkk(\F(x))$ for every $x\in D_\F(\A)$. We attach to $\F$ the Fourier coefficient of $GL_n(\A)$-automorphic functions, which maps each such automorphic function $\phi$ to the function $\F(\phi)$ on $GL_n(\A)$ given by 
\begin{equation}\label{zintegral}\F(\phi)(g):=\int_{D_\F(\kkk)\s D_\F(\A)}\phi(ng)\psi_\F^{-1}(n)dn\qquad\forall g\in GL_n(\A).\end{equation}The  abbreviation ``FC"---which was used in the abstract---is not used again.

We denote by $\Aut{\kkk,n}$ the set of  $GL_n(\A)$-automorphic representations. We restrict our attention to the subset $\Aut{\kkk,n,>}$  of $\Aut{\kkk,n},$ which consists of the representations the elements of which admit an absolutely convergent Eisenstein series expansion over discrete data. Let $\F\in\AAA[,n] $ and $\pi\in\Aut{\kkk,n,>}$. We say that $\F(\pi)$ is factorizable (resp. nonzero), if there is\footnote{Since automorphic representations are irreducible, if the function $\F(\varphi)$ is not identically zero for one choice of $\varphi$, it is not identically zero for any not identically zero choice of $\varphi$.}  an automorphic form $\varphi$ in $\pi$ for which the function $\F(\varphi)$ is factorizable (resp. not  identically zero). If $\F(\pi)$ is nonzero we write $\F(\pi)\neq 0$, and otherwise we write $\F(\pi)=0$.

For an AF $\F\in\AAA[,n] $ we list below four questions. Weak question \ref{zStrongc} is the one the paper mostly addresses, followed to a considerably lesser extent by Weak question \ref{zStrongcc}, and to an even lesser extent by Strong questions \ref{zStrongc} and \ref{zStrongcc}.  For each among Weak questions \ref{zStrongc} and \ref{zStrongcc}, we give two versions which turn out to be equivalent. Section \ref{forms} includes the proof of these equivalences and  that the $\pi$ positively answering the second version of Weak question \ref{zStrongc} (resp. \ref{zStrongcc}) are also positively answering Strong question \ref{zStrongc} (resp. \ref{zStrongcc}). Section \ref{forms} is introduced in the last two paragraphs of the next subsection. Weak questions \ref{zStrongc} and \ref{zStrongcc} include the concepts $\Om(\F),$ $\mult{a}{\F}$, and  $\Om^\prime(\pi),$ which are defined later in Sections \ref{forms} and \ref{variety} ($\Om(\F)$ and $\mult{a}{\F}$ are studied more in Section \ref{variety}, and independently from Section \ref{forms} and automorphic forms). Until defining and studying these concepts, the reader can keep in mind that: $\Om(\F)$ consists of the ``minimal nilpotent orbits attached" to $\F;$ $\mult{a}{\F}$ is a form of multiplicity of $a$ in $\F$; and that for $\pi\in\Aut{\kkk,n,>}$, the orbit $\Om(\pi)^\prime$ turns out\footnote{proving this is completed in the present paper in Corollary \ref{attachingfc} which is a corollary of Theorem \ref{folk}.} to be equal to the maximal nilpotent orbit attached to $\pi$. Also, the precise meaning of ``equivalent" (referring to the two versions in each question) is: $\Om^\prime(\pi)$ belongs to the set in Version 1$\iff$ $\pi$ positively answers Version 2.

\begin{Strong}\label{zStrongc}\hspace{-5pt}\textbf{($\F$).}\textit{ For which $\pi\in\Aut{\kkk,n,>}$,   $\F(\pi) $ is factorizable and nonzero?}

\end{Strong}

\noindent\textbf{Weak question \ref{zStrongc}.($\F$). }

\vspace{1mm}\noindent{\textbf{Version 1.}}\textit{ Describe the set \begin{equation}\label{zmultone}\{a\in\Om(\F):\mult{a}{\
\F}=1\}.\end{equation}}\textbf{Version 2. }\textit{For which $\pi\in\Aut{\kkk,n,>}$,  $\F(\pi) $ is factorizable and  $\Om^\prime(\pi)\in\Om(\F)$?  }
\begin{sremark}For the choices of $\F$ for which Weak question \ref{zStrongc}.($\F$) is addressed in the present paper, the set in (\ref{zmultone}) is the same as the set $\Omm{\F}:=\{a\in\Om(\F):\mult{a}{\
\F}<\infty\}$, and we use the later set to describe our results. For more on this topic see Remark \ref{multremark}\end{sremark}
\begin{Strong}\label{zStrongcc}\hspace{-5pt}\textbf{($\F$).}\textit{ For which $\pi\in\Aut{\kkk,n,>}$, $\F(\pi)$ is nonzero?}
\end{Strong}
\noindent\textbf{Weak question \ref{zStrongcc}.($\F$).}

\vspace{1mm}
\noindent\textbf{Version 1. }\textit{Describe the set $\Om(\F)$.}

\vspace{1mm} 
 \noindent\textbf{Version 2. }\textit{For which  $\pi\in\Aut{\kkk,n,>}$ the orbit $\Om^\prime(\pi)$ is minimal in the set $$\{\Om^\prime(\rho):\rho\in\Aut{\kkk,n,>}\hspace{1mm}\text{ and }\hspace{1mm}\F(\rho)\text{ is nonzero}\}. $$   }

\subsection{ $\AAnk$-trees and basic relations to Fourier coefficients of automorphic forms.}\label{treesroughly}
Let $\F\in \AAA[,n],$ and $\phi $ be a $GL_n(\A)$-automorphic function. The  main operations we will be applying to $\F(\phi)$ are the following:\begin{itemize}
\item[(a)] \textbf{Fourier expansions.}  Consider  a unipotent algebraic  subgroup $V$ of $GL_n$ defined over $\kkk,$  such that:  $D_\F\supseteq [V,V]$; $v\F=\F$ for all $v\in V$ (we define $v\F$ in (c) below); and $\F(V\cap D_\F)=0$. Then, the function $\F(\phi)$ is invariant on $V(\kkk)(D_\F\cap V)(\A)$, and we can apply to $\F(\phi)$ the Fourier expansion along $V(\kkk)(V\cap D_\F)(\A)\s V(\A).$
\item[(b)] \textbf{Certain exchanges of Fourier expansions.} In the unfolding of Rankin-Selberg integrals, Fourier expansions as in (a) have been used considerably. We apply to $\F(\phi)$ a special case of such unfoldings, in which the unipotent group occurring in the integration is exchanged by an other unipotent group. For the needs of the present paper, this unfolding has been explained in sufficient generality  in $\cite{GRS2}$ in Lemma 7.1. This lemma appears in the present paper as Part 2 of Fact \ref{autsteps}. 

 \item[(c)] \textbf{Conjugations.} For $\gamma\in GL_n(\kkk)$ we define $\gamma\F$ to be the AF with domain $\gamma D_\F\gamma^{-1}$, given by  \begin{equation}\label{zconjeasy}\gamma\F(n)=\F(\gamma^{-1}n\gamma)\qquad\text{for }n\in \gamma D_\F\gamma^{-1}. \end{equation} We then have $$\F(\phi)(\gamma^{-1} g)=(\gamma\F)(\phi)\text{ for }g\in GL_n(\A).  $$
\end{itemize}
One directly notices that these three operations on $\F(\phi)$ can be expressed  as operations on $\F$ without referring to $\phi$ (for example, in formula (\ref{zconjeasy}) we see this for  conjugations). This is how we proceed, and we will be calling any such operation on $\F$ an $(\AAnk,\kkk)$-operation. We iteratively apply $(\AAnk,\kkk)$-operations in a way that gives certain directed trees with vertices labeled with AFs, which we call $(\AAnk,\kkk)$-trees, or if $\kkk$ is not specified, $\AAnk$-trees (similarly $\kkk$ is removed or added in other notations). 

An $\AAnk$-step, say $\xi$, is an  $\AAnk$-tree formed by applying only one $\AAnk$-operation. If this $\AAnk$-operation corresponds to (a) (resp. (b), (c)) above, we say that $\xi$ is an $\e$-step (resp. $\eu$-step, $\co$-step). 

We choose the direction of any $(\AAnk,\kkk)$-tree, say $\Xi$, so that the edges point away from a vertex which we call the input vertex of $\Xi$; the other external vertices are called output vertices of $\Xi$  (if $\Xi$ is trivial, its  input vertex is also called its output vertex). The label of the input (resp. an output) vertex of $\Xi$ is called the input (resp. an output) AF of $\Xi$.

Let $\F\in\AAA$, $\SSS\subset\AAA$, and $H$ be an algebraic group defined over $\kkk$. By an $(\F\rightarrow S, H,\kkk)$-tree we mean an $(\AAnk,\kkk)$-tree with: input AF equal to $\F$; all its output AFs belonging in $\SSS$;  all its data defined inside $H$, and each conjugation being by an element in $H(\kkk)$. If we only want to specify a proper subset of $\{\F,\SSS,H,\kkk\} $ we adjust the notation as needed, for example we write ``$(\F,H,\kkk)$-tree". 

Throughout the current paragraph, Version 2 for both Weak questions \ref{zStrongc} and \ref{zStrongcc} is assumed to be chosen. The way each $(\AAnk,\kkk)$-step corresponds to an identity in automorphic forms is given in Fact \ref{autsteps} (the word ``Fact" is reserved for known results). Then Fact \ref{nonzeroeulerian} follows which states that for an $(\F,GL_n,\kkk)$-tree $\Xi$ the following two hold:\begin{enumerate}\item The function $\F(\phi)$ is nonzero if and only if there is an output AF $\Z$ in $\Xi$ for which $\Z(\phi)$ is nonzero.\item The function $\F(\phi)$ is factorizable if and only if the next two conditions are true: (i) there is at most one output vertex of $\Xi$, so that for its  label $\Z$ we have $\Z(\phi)$ is nonzero, and (ii) if such a vertex exists, $\Z(\phi)$ is factorizable.\end{enumerate}Due to this fact, for any choice of an AF $\F$ in $\AAA[,n]$ and a subset $\SSS$ of $\AAA[,n]$, Weak (resp. Strong) questions \ref{zStrongc}.$(\F)$ and \ref{zStrongcc}.$(\F)$ are reduced to:
\begin{enumerate}
\item[(i)] Some information on: which AFs in $S$ are output AFs of an $(\F\rightarrow \SSS,GL_n,\kkk)$-tree, and  which of these AFs occur as the label of only one output vertex of an $(\F\rightarrow \SSS,GL_n,\kkk)$-tree.  
\item[(ii)] Weak (resp. Strong) questions \ref{zStrongc}.$(\Z)$ and \ref{zStrongcc}.$(\Z)$ for $\Z\in \SSS$. 
\end{enumerate}
We denote by $U_n$ the set of unipotent upper trianglular $n\times n$ matrices. We first consider the choice $\SSS=\AAA(U_n). $ Weak questions \ref{zStrongc}.$(\Z)$ and \ref{zStrongcc}.$(\Z)$ for $\Z\in\AAA(U_n)$ are addressed in Theorem \ref{Eukerianf} which is an easy and expectable refinement of known results.  In some cases it is preferable to slightly refine Theorem \ref{Eukerianf} by replacing $\AAA{(U_n)}$ with a bigger set. The choice I make for such a bigger set is $\Pri[,n]$; where $\Prink[n]$  is the subset of $\AAnk[n]$, the domains of the AFs of which, are unipotent radicals of parabolic subgroups of $GL_n$ (recall we obtain from the previous subsection (first paragraph) that $\Prink[n]=\Pri[,n]\cap\AAA$). Then Weak questions \ref{zStrongc}.$(\Z)$ and \ref{zStrongcc}.$(\Z)$ for $\Z\in\Pri[,n]$, are addressed by this replacement which is stated as Theorem \ref{folk} and is   directly reduced to Theorem \ref{Eukerianf} by using a lemma about  a $(\Z\rightarrow\AAnk(U_n),GL_n,\kkk)$-tree (Lemma \ref{radexpress}). 

From Theorems \ref{general} and  \ref{folk} we obtain Corollary \ref{easyrefcor}, which in turn gives the results of Section \ref{forms} which we mentioned in the last paragraph of Subsection \ref{FCAut}.  

\subsection{The results of the paper}\label{threemain}
We denote by $P_n$ the standard parabolic subgroup of $GL_n,$ with Levi isomorphic to $GL_1\times GL_{n-1},$ so that $GL_1$ appears in the upper left corner. The first main result is Theorem $\ref{general},$ in the proof of which, for every AF $\F\in\AAA[,n], $ an $(\F\rightarrow\AAA(U_n),P_n,\kkk)$-tree which we call $\sXi(\F)$ or $\sXi$, is inductively defined. This result is based on simple refinements and modifications of unfolding processes that have been used in the study of the  $GL_n\times GL_m$ Rankin-Selberg integrals for $n\neq m$. For example, by appropriately removing from $\sXi(\F_{\emptyset,n})$ one copy of each among $\sXi(\F_{\emptyset,k})$ for $k<n,$  the tree obtained roughly corresponds to calculating the Fourier expansion over $U_n(\kkk)\s U_n(\A)$ of a $GL_n(\A)$ cusp form.

We discussed a motivation for constructing an $(\F\rightarrow\AAA(U_n),GL_n,\kkk)$-tree  in the last two paragraphs of the previous subsection. As in \cite{GK}, in the Fourier expansion above, and other later unfoldings, the reasons for using $P_n$ instead of  $GL_n$ can be seen  in the inductive definition of $\sXi(\F)$ and in many other constructions involving $\sXi(\F).$ 

In most constructions of $\AAnk$-trees in the present paper, the definition of $\sXi(\F)$ is  followed considerably (but usually not followed completely for too long within a big $\AAnk$-tree).   

In my experience, in many of the AFs $\F\in\AAA[,n]$ I am interested in, understanding (i) in the previous subsection for an $(\F\rightarrow\Pri[,n])$-tree by directly constructing it, takes some time (whether or not this $\AAnk$-tree follows closely the definition of $\sXi(\F)$).  The second main result of the present paper (Main corollary \ref{maincor}) gives a shortcut to these calculations if both the following hold: \begin{itemize}
\item Among the questions we stated for $\F$,   we restrict to addressing  Weak question \ref{zStrongc}.$(\F)$.
\item $\F$ is obtained as an output-AF of an $(\F_{\emptyset,n},GL_n)$-tree, so that each AF of an $\e$-step of this tree which is encountered in the path (inside this tree) from $\F_{\emptyset,n}$ to $\F$, is chosen in a ``certain generic way".  By ``certain generic way" we  mean that: it belongs to an open subset of the variety consisting of the output AFs of the $\e$-step, on which (referring to the open subset) an algebraic subgroup of $GL_n$ acts by conjugation freely and transitively (and hence by Zariski's main theorem this group is isomorphic to this open subset). Such AFs $\F$ form a set we call $\BBnk[n]$. 

An example of an AF in $\BBnk[n] $, is any AF, say $\WW_n$, in $\AAnk(U_n)$ which is nontrivial on every positive simple root group (of course the next paragraph is trivial for $\F$ being equal to $\WW_n$). The proof that $\WW_n\in\BBnk[n] $, is  directly discerned in the proof of the Fourier expansion over $U_n(\kkk)\s U_n(\AAnk)$ of $GL_n(\A)$-cusp forms. 

\end{itemize}

 We roughly use Main corollary \ref{maincor} as follows: once in the computation of an $\F$-tree (with $\F\in\BBnk[n] $) a nilpotent orbit $a$ of dimension $2\Di{D_\F}$ ``occurs", we obtain from Main corollary \ref{maincor} that this orbit is minimal among the nilpotent orbits ``occurring" (in fact of minimal dimension), and that it does not ``occur" again; also the orbits in the set $\Omm{\F}:=\{a\in\Om(\F):\mult{a}{\F}<\infty\}$ are exactly the ones in the previous sentence (and hence this set is equal to $\{a\in\Om(\F):\mult{a}{\F}=1\}$).

Some thoughts of D. Ginzburg, on a dimension equation he has formulated which is satisfied by many among the familiar Rankin-Selberg integrals, are related to Main corollary \ref{maincor}. A short discussion on such relations appears later (from Section \ref{variety}) in the paper, in Subsection \ref{zpre3}, after Example \ref{zktwo}.

The third main result (Theorems \ref{thD1}, \ref{thD2}, and \ref{thD3}) answers Weak question \ref{zStrongc}.$(\F)$ (through Version 1), for some choices of $\F\in\BBnk[N],$ for a positive integer $N$. For many of these choices of AFs $\F,$ there are more than one orbits positively answering Weak question \ref{zStrongc}.$(\F).$ For a quick rough idea of the AFs addressed in the third main result see the pictures in Subsection \ref{zpre}. The proofs of the theorems comprising the third main result are closely related to each other, and we proceed in them by: \begin{itemize}
\item Using Main corollary \ref{maincor} to considerably reduce the  $\AAnk$-tree computations one directly encounters. Each use of Main corollary \ref{maincor} is either direct, or---even better some times--- through Exchange corollary \ref{exchange} (which is a corollary of Main corollary \ref{maincor}).
\item Considerably (but not completely) following the  definition of $\sXi$ to compute the remaining $\AAnk$-trees.
\end{itemize}

We close with a family of Rankin-Selberg integrals which unfold  to a form involving an AF $\F$ which in each case (of integral) is either adressed by Theorem \ref{thD1} or by Theorem \ref{thD2}. As a result for each automorphic representation positively answering Weak question \ref{zStrongc}.$(\F)$ we obtain a factorizable (and conjecturally nonzero) integral. Beyond that, I am at an early stage of studying these integrals.

Other results of the paper are the ones contained in Section \ref{optionalsection}, and also Theorem \ref{folk} (which we already mentioned in the previous subsection). 

In Section \ref{optionalsection} we encounter AFs as suggested in its title (the notation ``$J_\F$" is given in \ref{uniqJ}). Only its first result (proposition \ref{embedding}) is needed for other sections of the paper (Section \ref{compositions}).  In the next two propositions  of this section we obtain trees to address a few cases of Strong questions \ref{zStrongc} and \ref{zStrongcc}, and then in Theorem \ref{th3}  we construct a certain $(\db{\F}\rightarrow\F )$-path where: $\db{\F}\in\Prink[n]\cap\BBnk[n]$; $\F\in\Prink[n]$; $\db{\F}$ and $\F$ correspond to the same orbit, more precisely  $\Om(\db{\F})$ is equal to $\Om(\F)$ (and is a singleton due to previous sections); the precise meaning of ``...-path" is given in Definition \ref{quasidef}. 

I will be grateful to receive comments, that will help me improve the references to known results.

\subsection{Differences with the previous version and relations with \cite{Tsiokos}}The present paper has few differences with its previous version (on the arXiv), mainly: minor corrections\footnote{Mainly correcting the repeated mistake of mentioning a free action with an open orbit in cases where I should have instead said: an action with an open orbit which when restricts to this orbit is free. These actions are the ones appearing in Definition \ref{defBBB} (in the present version only the restriction to the open orbit is mentioned). Some of the other corrections appear: in relation to Theorem \ref{thD3} (in its proof and in Definition \ref{ladefi}) in Definitions \ref{ossagroup} and \ref{sIk}, in the second picture in the proof of Observation \ref{unipc}.}, some additional references, and removing ``global" from ``global unfolding processes" in Remark \ref{sXilit} and (for the same reason) in the first paragraph of the Subsection \ref{threemain} (note that I have not given any precise meaning to the word ``unfolding").

 A small part of the content of the present paper appeared in my older paper \cite{Tsiokos}. In order to understand the content of the  present paper, checking \cite{Tsiokos} may be confusing and definitely not necessary. However, there are some calculations in \cite{Tsiokos} that: I find interesting, and do not appear in the present paper. I may transfer such calculations (with updated notations) in a future work. Problems on reading \cite{Tsiokos} include: the two mistakes mentioned below, I was at an early stage of obtaining the results of the present paper, notations that I introduced there and I am no longer using, worse presentation. 

\vsp
\noindent{\textbf{Mistake 1.}} I will completely remove Subsection 4.2 in \cite{Tsiokos}, and my current replacement of it is Section \ref{forms} of the present paper up to (and including) Remark \ref{anco}. The problems with Subsection 4.2 include: (i) very unfair referencing; (ii) we need the condition $\pi\in\Aut{\kkk,n,>}$, and then the proofs are completed by Theorem \ref{folk} in the present paper.\hfill$\triangle$ 

\vsp 
\noindent\textbf{Mistake 2.} Conjecture C.2 in \cite{Tsiokos} (Subsection 3.1) is wrong. It admits trivial counterexamples, for example by choosing the domain of $\mathcal{F}$ to be the unipotent radical of the $GL_5$-parabolic subgroup with Levi isomorphic to $GL_2\times GL_1\times GL_2$ so that $GL_1$ is in the middle. This conjecture, in all its uses that I made (within proofs), can be replaced by  Main corollary \ref{maincor} which is stated and proved in the present paper\footnote{I also withdraw from the introduction in \cite{Tsiokos} in ``Subsection 2" (which is named ``The FCs that are studied in this paper, and the proofs that are only up to a conjecture")  the last two paragraphs except for the overlapping with the sentence in the present paper on which the current footnote is attached.}.\hfill$\triangle$

\vsp 

The operation on AFs in the present paper denoted by the symbol ``$\circ$", corresponds by (\ref{zintegral}) to usual function composition of Fourier coefficients (when they are viewed as operators). In \cite{Tsiokos} this function composition was: used, denoted again by the symbol ``$\circ$", and called  ``convolution". In the present paper I only used the symbol ``$\circ$" (and only for AFs); even if I had used a word, I would have avoided the name ``convolution".
\subsection{Acknowledgements}
I thank Evangelos Routis for helping me understand some algebraic geometry relevant to the present paper. I did this work while I was a postdoc in Tel Aviv University.
 \section{Notations and prerequisities}\label{allnot}
I have added some hints within the current section on deferring reading some of its parts (also note there is an index of notations at the end of the paper). The first such hint is: Subsections \ref{allnot1} and \ref{treessubs} are almost sufficient\footnote{To remove ``almost" we must add Definitions \ref{embjf} and \ref{tlarrown}. Of course, for relating this Theorem to the literature in automorphic forms (Remark \ref{sXilit}) one also needs Fact \ref{autsteps}.

One can also skip more parts of the current section (and be able to read Theorem \ref{general} and its proof), including: in Subsection \ref{allnot1} everything after \ref{scent}; in Subsection \ref{treessubs} everything after \ref{veenot1}, and also \ref{quasidef}.} for reading Theorem \ref{general}.  

Throughout the current section $n$ is assumed to be any positive integer.
\newcounter{de} 
\subsection{Notations for Groups, Varieties, and Morphisms }\label{allnot1}
\begin{sdefi}[Varieties, algebraic groups]\label{vara}All algebraic varieties are defined over an algebraicaly closed field of \textbf{characteristic zero}, which we denote by $\KKK$. For varieties and algebraic groups we adopt the definitions in \cite{Springer}. Notice therefore that varieties are \textbf{not} assumed to be irreducible. Let $G$ be an algebraic group and $H$ be a closed subgroup of $G$; then we say that $H$ is an algebraic subgroup of $G$.  

In the statements involving automorphic forms we further assume that $\KKK$ contains the complex numbers.\end{sdefi}

\begin{sdefi}[Parabolic subgroups, root data]\label{parroot} A parabolic subgroup of $GL_n$, most frequently will be called a $GL_n$-parabolic subgroup. $GL_n$-parabolic subgroups, their  Levi subgroups, and all root data of $GL_n,$ is assumed to be chosen in the standard way, unless otherwise specified.

The Levi and the unipotent radical of a $GL_n$-parabolic $P$ are denoted by $M_P$ and $U_P$ respectively.

We denote by $U_n$ the unipotent radical of the Borel of $GL_n$. We denote by $P_n$ the $GL_n$-parabolic subgroup with Levi isomorphic to $GL_1\times GL_{n-1},$ so that $GL_1$ is embedded in the upper left corner. For $j$ being the lower right corner embedding\footnote{That is $g\rightarrow\begin{pmatrix}1&\\&g\end{pmatrix} $.} of $GL_{n-1}$ in $GL_n$ we define $$\mirsl{n}:=j(SL_{n-1})U_{P_n}.$$ 

A matrix entry is always called an entry. Consider any nondiagonal entry, say the $(i,j)$ entry. The root group of $GL_n$ which is nontrivial on the $(i,j)$ entry is denoted by $U_{n,(i,j)}$ or  by $U_{(i,j)}$, and is called a $GL_n$-root group or root group (see Convention \ref{omitn} below).

\noindent{If} $\alpha$ is a root of $GL_n$, the root group corresponding to it is denoted by $U_\alpha$ or (to give some emphasis on $n$) by $U_{n,\alpha}$. 

The standard maximal torus of $GL_n$ is denoted by $T_n.$ Let $H$ be either a Levi of $GL_n$-parabolic subgroup or---by using Definition \ref{sesc} below---a standard copy of a general linear group in $GL_n$. We denote by $W_H$ the Weyl group of $H$ with respect to $T_n\cap H$. Each element of  $W_H$ is identified with the permutation matrix that represents it. We write $W_n$ instead of $W_{GL_n}$.
\end{sdefi}
\begin{convention}\label{omitn}Fix any statement of the paper, and let $x$ be the biggest number for which  $GL_x$ is ``considered". Exactly in these cases we write: ``$U_{(i,j)}$" in place of $U_{x,(i,j)}$, ``root group" in place of ``$GL_x$-root group", ``root" in place of ``root of $GL_x$", and ``$\JJ_r$" in place of the later-defined (Definition \ref{zjr}) ``$\JJ_r^x$".

Explicitly, in the cases that $x$ is removed we have: $x=n$ in Sections \ref{allnot}, \ref{Theorem A}, \ref{pri}, \ref{forms}, \ref{variety}, in Section \ref{optionalsection} in Proposition \ref{bigger} and  Theorem \ref{th3},and in Section \ref{appendixs}; $x=k+n$ in the rest of Section \ref{optionalsection}; and $x=N$ in Section \ref{compositions}.

We also proceed similarly with the later defined concept of $\ossa$-groups (Definition \ref{ossagroup}). \end{convention}
\begin{sdefi}[Trace, transpose, Lie algebras]\label{trtr}For a matrix $X,$ we denote its trace by $\mathrm{tr}(X), $ and its transpose by $X^t.$ The Lie algebra of an algebraic group $G$ is denoted by $\Lie{G}. $\end{sdefi}
\begin{sdefi}[Additions on the notations of Subsection \ref{FCAut}]\label{clar}For AFs and Fourier coefficients attached to them, we adopt all the definitions of Subsection \ref{FCAut}, by only adding the following clarifications and refinements:\begin{enumerate}
\item For $GL_n(\A)$-automorphic representations we choose a standard definition, which we recall in the appendix in \ref{preaut}. There, we also recall some basic facts on automorphic forms and fix some notations at the same time.
\item A function $f$ on $GL_n(\A)$ is said to be factorizable if for each place $\upsilon$ there is a function $f_\upsilon$ on $GL_n(\kkk_\upsilon)$, such that:
$f(g)=\prod_{\upsilon}f_\upsilon(g_\upsilon) $ for all $g\in GL_n(\A)$, where the product is over all places.
\item\label{zcharzero} The field $\kkk$ is always a subfield of $\KKK$ (but not\footnote{Even though my aim was to obtain the present work for $\kkk$ being a number field, allowing it to be any field of characteristic zero rarely required any change in the thougts not involving automorphic forms, and hence this bigger generality is adopted.} necessarily a number field). In every case that we require $\kkk$ to be a number field, automorphic forms are involved. 
\item Algebraic groups over $\kkk$ are often called $\kkk$-groups. We do the same for the (variety) morphisms we consider, for example AFs and isomorphisms.
\item If a $\kkk$-group $H$ consists of $n\times n$ matrices it is assumed that its $\kkk$-structure is such that $H(\kkk)$ is the subset of $H$ consisting of matrices with entries over $\kkk$. 
\end{enumerate}\end{sdefi}
\begin{sremark}[$\kkk$]

Almost always, the field $\kkk$ ``over which we are working" remains the same during a proof or a statement. However, I have chosen to adopt only to a small extent conventions\footnote{The main such convention is Convention \ref{sXrt}. Any other such convention is stated and adopted within a proof or a definition.}  that remove $\kkk$ from the notations.
\end{sremark}
\begin{sdefi}[$|...|$]\label{sa1} The number of elements of a finite set $S$ is denoted by $|S|.$\end{sdefi}
\begin{sdefi}[\text{$[H_1,H_2]$}]\label{sa2} For any two groups $H_1,H_2$ defined as subgroups of an other group (usually a general linear group), we denote by $[H_1,H_2]$ their commutator group.\end{sdefi}
\begin{sdefi}[$\gamma\F$]\label{gla} Let $\F\in\AAnk[n] $ and $\gamma\in GL_n(\kkk)$. Recall that by $\gamma\F$ we denote the AF in $\AAnk[n] $ with domain $D_{\gamma\F}=\gamma D_\F\gamma^{-1}$ that satisfies: $\gamma\F(n)=\F(\gamma^{-1}n\gamma)$ for every $n\in D_{\gamma\F}. $\end{sdefi} 
\begin{sdefi}[$\mathrm{Stab}_{H}(\F)$, centralizer]\label{scent} Let $\F\in \AAnk[n],$ and $H$ be a subgroup of $GL_n.$ We define 
$$\mathrm{Stab}_{H}(\F):=\{h\in H:h\F=\F\}.$$Let $A,B$ be two sets consisting of $n\times n$ matrices. Then the centralizer of $A$ in $B$ is the set $\{b\in B:ba=ab\text{ for all }a\in A\}$.
\end{sdefi} 
\begin{sdefi}[$\F|_H$]\label{nrest} Let $\F\in\AAnk[n]$ and $H$ be an algebraic subgroup of $GL_n.$ We denote by $\F|_H $ the restriction of $\F$ on $D_\F\cap H.$\end{sdefi}
\begin{sdefi}[$\JJ_r,\JJ_r^n$]\label{zjr} Let $1\leq r\leq n.$ By $\JJ_r^n$  or $\JJ_r$ we denote an AF with domain $D_{\JJ_r}=\prod_{1\leq i\leq r-1,i<j\leq n}U_{(i,j)}$  which is nontrivial on all simple root groups contained in $D_{\JJ_r}$. Of course $\JJ_r$ is determined only up to a standard torus conjugation, and this choice is made each time without mention. To be clear, $\JJ_1:=\F_{\emptyset,n}$.\end{sdefi}

\begin{sdefi}[$\X{\F}$]\label{vardef}For every AF $\F\in\AAnk[n] $ we define the variety 
\begin{equation*}\X{\F}:=\{J\in\Lie{GL_n}:\mathrm{tr}(Ju)=\F(\exp{u})\quad\forall u\in \Lie{D_\F} \}.\qedhere \end{equation*}\end{sdefi} \begin{sfact}[Known]\label{Liehom}
Let $N$ be a unipotent $\kkk$-subgroup of $GL_n.$ There is a bijection between $\AAA(N) $ and Lie algebra $\kkk$-homomorphisms $\Lie{N}\rightarrow\KKK$ given by mapping each $\F\in\AAA(N)$ to the function $$x\rightarrow \F(\exp(x))\qquad\text{for all }x\in\Lie{N}. $$
\end{sfact}
\begin{sdefi}[$\AnkT{n}$, $\hat{...}$,  $..._\nless$]\label{hatnless}\begin{multline}\AnkT{n}:=\{\F\in\AAnk[n]: D_\F\text{ is generated by root groups}\}=\\\{\F\in\AAnk[n]:tD_Ft^{-1}=D_\F\quad\forall t\in T_n\}.\end{multline} The second equality is known, and we do not use it.

Let $\XXX$ be any subset of $\AnkT{n}.$ We define $\hat{\XXX}$ to be the set consisting of the AFs $\F\in\XXX$ such that: if for numbers $i,i^\prime,j,j^\prime$, both  $\F(U_{(i,j)})$ and $\F(U_{(i^\prime,j^\prime)})$ are  nontrivial, then $i\neq  i^\prime$ and $j\neq j^\prime$. We also define $\XXX_{\nless}$ to be the set consisting of the AFs in $\XXX$ satisfying: for any two different roots $\alpha,\beta$ such that $U_\alpha,U_\beta\subset D_\F$ and both $\F(U_\alpha)$ and $\F(U_\beta)$ are nontrivial, we have $\alpha\not >\beta$ (that is, $\alpha-\beta$ is not a sum of positive roots).

For some choices of $\XXX$ the hat does not appear in the middle, for example we write $\hat{\AAnk[]}[T_{\text{$n$}}]$ (only due to a difficulty in compiling the text).\end{sdefi}

\begin{sdefi}[$J_\F$]\label{uniqJ} Let $\F\in\AnkT{n}$. The unique element in $D^t_\F\cap\X{\F}$ is denoted by $J_\F$.
We mainly use the concept ``$J_\F$"  for $\F\in\Prink[n]; $ partly because for  $\F\in\Prink[n] $ we have $\Om(\F)=\{\Om(J_\F)\}, $ where $\Om(\F) $ is later defined in Definitions \ref{OFX} and \ref{AOF}, (see also Corollary \ref{easyrefcor}).\end{sdefi}
\begin{sdefi}[$\times,\prod,\times^\searrow,\prod^\searrow$]\label{tpts} As usual the symbols ``$\times$" and ``$\prod$"are used for denoting direct products of groups. We also use the notations $\times^\searrow$ and $\prod^\searrow$ to express a Levi of a $GL_n$-parabolic as the direct product of general linear groups so that each one appears below the previous one; for more information see Definition \ref{searrow}. 

The direct product of subgroups of a (same for all of them) general linear group is assumed to be the internal direct product (the internal is also assumed in some other cases).
\end{sdefi}
\begin{sdefi}[Nontrivial on an entry, row, or column] Let $w\in W_n$, and  $V$ be an algebraic subgroup of $wU_nw^{-1}$. Given an entry, we say that $V$ is nontrivial on it if and only if: this entry is nondiagonal and a matrix in $V$ is nonzero on it. Given a row (resp. column), we say that $V$ is nontrivial on it if and only if: there is an entry on this row (resp. column) on which $V$ is nontrivial. The previous two sentences are also adopted when: we replace all occurrences of ``nontrivial" with ``trivial", and ``nonzero" with ``zero".
\end{sdefi}
\subsection{Embeddings.}\label{embsubsub} The current subsection is mostly used in Sections \ref{pri}, \ref{optionalsection}, and \ref{compositions}.

\begin{sdefi}[Embedding, $j(\F)$]\label{embjf}Let $H^\prime$ and $G$ be two algebraic groups, and $H$ be an algebraic subgroup of $H^\prime$. By an embedding $j$ of  $H$ in $G$ we mean an (algebraic) isomorphism of $H$ onto an algebraic subgroup of $G.$ For $\F$ being an AF with domain $D_\F\subseteq H^\prime$ we define $j(\F):=\F\cdot j^{-1}$ where ``$\cdot$" is usual function composition after restricting the domain as needed, that is:
\begin{equation*}j(\F)(u)=\F(j^{-1}(u))\qquad\qquad\quad\forall u\in D_{j(\F)}:=j(D_\F\cap H).\qedhere \end{equation*}\end{sdefi}

\begin{sdefi}[Standard, lower right corner, upper left corner (all referring to embeddings and copies)]\label{sesc}Consider a positive integer $m<n$, and let $\F\in\AAnk[m].$ Let $A(x):=\{1,...,x\}$ for every positive integer $x.$ An embedding $j:GL_m\rightarrow GL_n,$ is called a standard embedding if there is an injective function $f:A(m)\rightarrow A(n)$ such that:\begin{itemize}
\item for $(i,j)\in f(A(m))\times f(A(m)),$ the $(i,j)$ entry of $j(g)$ is the same as the $(f^{-1}(i),f^{-1}(j))$ entry of $g$;
\item for $(i,j)\in A(n)\times A(n)-f(A(m))\times f(A(m)),$ we have: $(i,j)\text{ entry of }j(g)=\left\{\begin{array}{cc}0&\text{if }i\neq j\\1&\text{if }i=j
\end{array}\right.. $
\end{itemize} 
The image $j(GL_m)$ (resp. the AF $j(\F)$) is called a standard copy of $GL_m$ (resp. a standard copy of $\F$).

If $A(m)=\{n-m+1,...,n\}$, we call $j$ (resp. $j(GL_m)$, $j(\F)$) the lower right corner embedding (resp. lower right corner copy, lower right corner copy) of $GL_m$ in $GL_n$ (resp. of $GL_m$ in $GL_n$, of $\F$ in $\AAnk[n]$). If instead $A(m)=\{1,...,m\}$, we replace ``lower right" with ``upper left".\end{sdefi} 

\begin{sdefi}[$\mathcal{D}_n$, $\Set{T}$, $\SRV{T}$]\label{zdnf}Let $Z_m$ be the center of $GL_m$. We denote by $\DDDD[n] $ the set of one dimensional tori of the form $j(Z_m),$ for $m,j,f$ being data as in the previous definition. We also define \begin{equation*}\Set{j(Z_m)}:=\{f(1),...,f(m)\}\quad\text{ and }\quad \SRV{j(Z_m)}:=\{U_{n,(f(i),f(i+1))}: 1\leq i\leq m-1\}.\qedhere\end{equation*}\end{sdefi}
\begin{sremark}The notation ``$\Z_m$" appears rarely, and usual ways of defining $T$ are trough $\Set{...}$  or by using the definition below to say that $T$ is a $\DDDD[n] $-component of $\TT{n}{\F}$ for an AF $\F$.\end{sremark}    \begin{sdefi}[$\mathcal{D}_n$-components of $\TT{n}{\F}$, $\Set{\TT{n}{\F}}$]\label{dnsetdn} Let $\F\in\hat{\AAnk[]}[T_{\text{$n$}}]$. Among the ways $\TT{n}{\F}$ is written as a direct product of tori in $\DDDD[n],$  we consider one (unique up to permutation of terms), say $\TT{n}{\F}=T_1...T_l, $ such that any two different components, say $T_i$ and $T_j$, have no nontrivial entries in common (that is, any two matrices, one in $T_i$ and one in $T_j$, are equal only in entries with value in $\{0,1\}$). We call $T_1,...,T_l$ the $\DDDD[n] $-components of $\TT{n}{\F}.$ We also define $$\Set{\TT{n}{\F}}:=\{\Set{T}:T\text{ is a }\DDDD[n]\text{-component of }\TT{n}{\F}\}. $$ For example for $\F$ being the  AF with domain the unipotent radical of the $GL_4$-parabolic subgroup with Levi isomorphic to $GL_2\times GL_2$ given by 
       $$\F\begin{pmatrix}1&&x&y\\&1&z&w\\&&1&\\&&&1 \end{pmatrix}=x+w\qquad\forall x,y,z,w\in\KKK, $$ the set $\TT{n}{\F} $ has two $\DDDD[n] $-components and they are given by
       $$\left\{\begin{pmatrix}t&&&\\&1&&\\&&t&\\&&&1\end{pmatrix}:t\in\KKK \right\},\qquad \left\{\begin{pmatrix}1&&&\\&t&&\\&&1&\\&&&t \end{pmatrix}:t\in\KKK \right\},$$ and $\Set{\TT{n}{\F}}=\{\{1,3\}\{2,4\}\}.$\end{sdefi} 
      
\begin{sdefi}[$\jj H$, $\ja\F$, $\jm\F$, $GL_n^T$, $GL_n^S$, $\F^T$, $\F^S$]\label{jjj}Let $j:GL_m\rightarrow GL_n$ be a standard embedding, $T:=j(Z_m)$ (recall $Z_m$ is the center of $GL_n$), and   $H$ be a subgroup  of $GL_m.$ Sometimes, instead of using notations such as $j(H),j(\F),...$ we use the following notations:\begin{enumerate}
\item We define $\jj H:=j(H).$ For $\F\in\AAnk[m] $ (resp. $\F\in\AAnk[n] $) we define $\ja \F:=j(\F)$ (resp. $\jm\F:=j^{-1}(\F)$). The symbol $\jj$  (in contrast to $j$) has no standalone meaning.  Of course in every use of this notation we explain which is the choice of $j$; however, even if a reader misses such an explanation, I hope making the right guess will be easy from the context.
\item We define $GL_n^T:=GL_n^{\Set{T}}:=j^\prime(GL_{n-m}) $ where $j^\prime$ is the standard embedding of $GL_{n-m}$ in $GL_n$ for which $j(GL_m)\cap j^\prime(GL_{n-m})$ is trivial. For $\F\in\AAnk[n] $ we define $\F^T:=\F^{\Set{T}}:=\F|_{GL_n^T}$ and $\jm\F^T:=\jm\F^{\Set{T}}:={j^\prime}^{-1}(\F)={j^\prime}^{-1}(\F|_{j^\prime(GL_{n-m})})$; we also define $\F^\emptyset:=\F$.
\item Finally, when automorphic forms are involved, we are frequently identifying groups with standard copies they admit.
\end{enumerate}
\end{sdefi}

\begin{sdefi}[Blocks of a Levi, $\times^\searrow,$ $\prod^\searrow$]\label{searrow}Let $M$ be the Levi of a $GL_n$-parabolic subgroup. We write 
$$M={\prod_{1\leq i\leq k}}^{\pse}G_i\quad\text{or}\quad M=G_1\times^\searrow...\times^\searrow G_k$$ for $G_1,...,G_k$ being the subgroups of $M$ which are standard copies of general linear groups, so that for $1\leq i_1<i_2\leq k$, $G_{i_2}$ is below $G_{i_1} $ (that is,  for $(x_1,y_1)$ and $(x_2,y_2)$ being the positions of nondiagonal entries on which $G_{i_1}$ and $G_{i_2}$ are respectively nonzero, we have $x_1<x_2$ (and hence $y_1<y_2$)). Each such  $G_i$ is called the $i$-th block of M. Therefore, in cases that we write $$M={\prod_{1\leq i\leq k}}^{\pse}\jj G_{n_i}\qquad\left(\text{resp. }M={\prod_{1\leq i\leq k}}^{\pse}G_{n_i}\right) $$it is implied that  the $i$-th block of $M$ is a copy of $GL_{n_i}$ which we denote by $\jj GL_{n_i}$ (resp. implied that the $i$-th block of $M$ is identified with $GL_{n_i}$).\end{sdefi}

\subsection{Nilpotent orbits. }\label{inil}Nilpotent and unipotent orbits are in bijective correspondence, and I have chosen to use the former. We therefore define $\UUU{n}$ to be the set of nilpotent orbits of the action by conjugation of $GL_n$ on $\Lie{GL_n}.$ 

Given an orbit $a\in\UUU{n},$ for $a_1\times a_1$,...,$a_m\times a_m$ being the dimensions of the Jordan blocks of a Jordan matrix belonging to $a,$ we obtain a partition $n=a_1+...+a_m$ which we frequently identify with $a,$ and write $$a=[a_1,...,a_m].$$ We do not require that $a_1,...,a_m$ are in decreasing (or any other) order, and they are allowed to  equal $0.$ The set $\UUU{n}$ admits the partial order $``>"$ described as follows. Let $a=[a_1,...,a_m]$ and $b=[b_1,..,b_m]$ be two different elements of $\UUU{n}, $ such that $a_1\geq...\geq a_m $ and $b_1\geq...\geq b_m$. We write $a>b$ if and only if  $$\sum_{1\leq i\leq r}a_i\geq \sum_{1\leq i\leq r}b_i $$ for every $1\leq r\leq m.$ The words ``bigger" and ``smaller" always refer to this order. For two  orbits $a,b\in\UUU{n}$ that  none is bigger from the other, we say they are unrelated.

If $X$ is a nilpotent matrix in $\Lie{GL_n}$, we denote by $\Om(X)$ its nilpotent orbit. 

In the appendix (Subsection \ref{nila}) we give some information on $\UUU{n}$, mainly basic information on the Richardson orbit which we start to encounter in the present paper in Section \ref{forms}.

\subsection{Mathematical statements (theorems, propositions,...), variables and constants}\label{ssl}For naming mathematical statements I have mostly used the words: Theorem, Proposition, Lemma, Corollary, Claim, Fact, and Observation.  The way the first four are used must be quite usual. A claim is used as a lemma but it is stated and proved within the proof of an other statement. An observation is a statement similar to a lemma but very easy (or even trivial) to prove.

Some statements have the label ``Known" next to their name. This label  means that\begin{itemize}\item the statement is known;\item even if the statement is followed by any part of a proof it admits, this part of proof is also known.\end{itemize}I preferred to avoid the choice between ``Theorem" and ``Proposition" for the statements with the label ``Known", and hence I use one word for all such cases which is ``Fact". However, the names ``Corollary", ``Lemma", and ``Claim" are  used with the label ``Known". Most of the observations can be discerned in the literature, but I do not attach to them the label ``Known", due to their low level of difficulty.

Of course there are situations the above formalizations would need changes, for example in certain statements in which most of the effort is in the formulation and not in the proof, but at least for the present paper I found them adequate.

 A square is used as the end of proof sign. If it refers to a proof  occurring within an other proof, it is followed by a reminder of what we just proved. For example: $\square$Claim,  $\square$Step i.   For some other environments including definitions and remarks we similarly use a triangle.

\begin{sdefi}[\text{$[...\tlarrow...]$}]\label{tlarrown}
Let $x,y,...$ be variables appearing in a precisely stated environment $\ast$ (e.g: $\ast$ can be the statement or the proof of a theorem, proposition,...). Let  $x_1,y_1,...$ be values the previous variables respectively admit.  With phrases such as  ``$\ast$ for $[x\tlarrow x_1,y\tlarrow y_1,...]$" we refer to the statement obtained from $\ast$ when  $x,y,...$ are given these values.\end{sdefi}

\begin{sremark}[Common names for variables and constants]In the list below we give examples of names we choose for variables. We include concepts which we define later. We do not include notations in which the variable lies only in a proper ``subnotation" (for example $\X{\F} $ and $\Om(\F)$). It is by no means a complete list, for example the notations involving subscripts and superscripts are almost completely skipped.

\vsp
\noindent{Integers}\hfill$n$, $N$, $k$, $l$, $m,$ $r,$ $s$, $t$, $a$, $b$\\orbits\hfill $a$, $b$, $c$\\
AFs\hfill$\F$, $\F^\ih$, $\Z$, $\QQ,$ $\XX$, $\YY$, $\zAF$\\
$\AAnk$-trees, $\AAnk$-paths\hfill$\Xi$, $\xi$, $\I$\\Unipotent algebraic groups\hfill$V$, $N$, $L$, $X$, $Y$, $C$\\ $GL_n$-parabolic subgroup (resp. Levi of a $GL_n$-parabolic subgroup)\hfill $P$ (resp. $M$)\\algebraic groups (in general)\hfill$H$\\roots \hfill$\alpha$, $\beta$\\Sets of AFs\hfill $\SSS$, $\XXX$ 

 The symbol $\F^\ih$ is used for AFs which after a few calculations are processed by an inductive hypothesis. The symbol $\zAF$ is used for trivial AFs (with nontrivial domain). For $\AAnk$-trees that are possibly ``big", capital letters are preferred. For $\AAnk$-trees that are known to be $(\eu,\co)$-paths we usually use ``$\I$" (even if they are ``small"). The symbols $X$, $Y$, and $C$, are mostly used as in Definition \ref{euoperation}. The symbol $\XXX$ is also used as a constant (Subsection \ref{zpre3}).
 
 The choices of fonts above are also adopted frequently by constants. For example, some sets of AFs, trees, and groups, all with a constant meaning up to $n,N$, and $\F$, are respectively: $\AAnk [n], $ $\BBnk[n] $, $\Prink[n] $, $\SSS_{n,N}$ (this consists of pairs of AFs); $\sXi(\F)$ and $\sI(\F)$; $GL_n,$  $P_n$.\end{sremark}
\subsection{$\AAnk$-trees, and the operation $\circ$}\label{treessubs}
In addition to the assumption that $n$ is chosen to be any positive integer, we also assume throughout the current subsection that $\kkk$ is chosen to be any subfield of $\KKK$. It may worth clarifying that: this choice for $n$ and $\kkk$ must be made within each statement\footnote{By a statement in the current subsection we consider everything admitting the counter consisting of three numbers (e.g. definitions).} involving them (and only once within each such statement).
\begin{sdefi}[$\circ$]\label{circd}
Consider a unipotent algebraic group $N$, which is written as a product of unipotent algebraic subgroups $N=N_2N_1,$ with $N_1\triangleleft N$. Consider AFs $\F_1\in\AAnk(N_1), $ and $\F_2\in\AAnk(N_2).  $ Assume that these two functions have a common extension to an AF in $\AAnk(N).$ We see that this extension has to be unique, and we denote it by $\F_2\circ\F_1.$ We have 
$$\F_2\circ\F_1(n_2n_1)=\F_2(n_2)+\F_1(n_1), $$for all $n_1,n_2$ respectively in $N_1,N_2.$ 

I emphasize: the meaning of  saying that $\F_2\circ\F_1$ is defined, includes that $\F_2\circ\F_1$  is an AF.
\end{sdefi}

\begin{sob}Let $N$, $N_1$, $N_2$, $\F_1$, and $\F_2$,  be defined as in the first two sentences of the definition above. The AF $\F_2\circ\F_1$ is  defined if and only if 
\begin{equation}\label{ztrombas}\F_1|_{N_1\cap N_2}=\F_2|_{N_1\cap N_2},\qquad\text{and}\qquad n_2\F_1=\F_1\quad\text{ for all }n_2\in N_2. \end{equation}
\end{sob}
This observation is used without mention. In many among the cases we consider, the intersection $N_1\cap N_2$ is trivial.

\begin{sdefi}[$\rcirc$, used in Sections \ref{optionalsection} and \ref{compositions}]\label{rcircd}Consider an integer $1\leq r< n$. Let $\F\in\AAnk[n]$ be such that $D_\F$ is contained in the lower right corner copy of $GL_{n-r}$ in $GL_n.$ Then the AF $\F\circ\JJ_r$ is also denoted by $\F\rcirc\JJ_r.$\end{sdefi}

 \begin{sfact}[Known]\label{areconnected}
Let $N$ be a unipotent $\kkk$-group. Then the map $\exp: \Lie{N}\rightarrow N$ is a $\kkk$-isomorphism of varieties. In particular, $N$ is connected. 
\end{sfact}\begin{proof}See for example 15.31 in \cite{Milne}.\end{proof}
\begin{scor}[Known, mostly optional]\label{semprod}
Let $N$ be  unipotent $\kkk$-group, and  $N_1$
 be a normal  $\kkk$-subgroup of $N,$ such that $\Di{N}=\Di{N_1}+1.$ There is a $\kkk$-subgroup $N_2$ of $N$ such that N is the semidirect product of $N_1$ and $N_2$, that is: $N=N_2N_1$ and $N_1\cap N_2=1.$ 
\end{scor}\begin{proof}
Let $v\in\Lie{N}(\kkk)-\Lie{N_1}(\kkk).$ By Fact \ref{areconnected} we obtain: (i) a unipotent $\kkk$-subgroup $N_2$ of $N$ with its Lie algebra being the span of $v,$ (ii) then (since  $\Lie{N}=\Lie{N_2N_1}$) that $N=N_2N_1$, (iii) and finally (since $\Di{N_2\cap N_1}=0$) that $N_2\cap N_1$ is trivial.
\end{proof}
\begin{scor}[Known]\label{partlysplit}
Let $N$ be a unipotent $\kkk$-group, and  $N_1$
 be a normal  $\kkk$-subgroup of $N.$ Then there is a  $\kkk$-subgroup $N_2$ of $N$, such that:
 $$a)\hspace{2mm}\F(N_2\cap N_1)=0\text{ for all }\F\in\AAnk(N)\qquad\text{and}\qquad b)\hspace{2mm} N=N_2N_1. $$ \end{scor}
 \begin{proof}To obtain $a)$ it is sufficient to prove 
\begin{equation}\label{zcom}N_2\cap N_1\subseteq[N,N]. \end{equation}Let $v_1,...,v_k$ be a basis of a vector space complement of $\Lie{N_1[N,N]}(\kkk) $ inside $\Lie{N}(\kkk).$ By including several uses of Fact \ref{areconnected}, the proof proceeds by: obtaining for each $i$ a (unique) subgroup $V_i$ of $N$ which is $\kkk$-isomorphic to $\KKK$ and its Lie algebra is the span of $v_i;$ then choosing $N_2:=(\prod_{1\leq i\leq k}V_i)[N,N]$ and noticing that $\Lie{N_2}$ is equal to the vector space direct sum of  $\Lie{[N,N]}$ and the span of $v_1,...v_k$; and finally obtaining (\ref{zcom})  and $b)$. \end{proof}

\begin{sdefi}[$\AAnk$-operation, complete after Definitions \ref{eoper}, \ref{coper}, and \ref{euoperation}]\label{Aoper} An $\AAnk$-operation is any operation which is an $\e$-operation or a $\eu$-operation or a conjugation, and these in turn are defined below.\end{sdefi}
\begin{sremark}[Definitions that are ``missing"]
To avoid defining concepts which we do not use, we end up not defining some concepts which are basic and closely related to the defined ones. For example, we do not add in each of the three definitions below a notation that specifies the choice of ``$\kkk$"; by preferring in most of the paper to formulate the reasoning in terms of  $\AAnk$-trees instead of $\AAnk$-operations, it turns out to be sufficient to make such additions only to the former (e.g. with the later defined concepts ``$(\AAnk,\kkk )$-trees" and ``$(\AAnk,\kkk )$-steps").\end{sremark}
\begin{sdefi}[$\e$-operations]\label{eoper}
  Let $V$ be a unipotent $\kkk$-subgroup of $GL_n$. Consider the operation $\e(V) $ given by:\begin{equation}\label{zeset}\e(V)(\F):=\{\Z\in\AAA[,n] :\Z|_{D_\F}=\F\text{ and }D_\Z=VD_\F\} \end{equation}and having as domain the AFs $\F\in\AAA[,n] $ for which: $$ D_\F\supseteq [V,V],\qquad v\F=\F\text{ }\quad\forall v\in V,\qquad\F(V\cap D_\F)=\{0\}.$$ 
  
  Every such operation is called an $\e$-operation. We say that $\e(V) $ is the $\e$-operation over $V$. Here ``$\e$" stands for ``Fourier expansion".\end{sdefi} 

\begin{sdefi}[Conjugations]\label{coper}
Let $\gamma\in GL_n.$ We call a conjugation the operation mapping each $\F\in\AAnk[n] $ to $\gamma\F$.\end{sdefi}

\begin{sdefi}[$\eu$-operations.]\label{euoperation} Let $X,Y$ be two unipotent algebraic subgroups of $GL_n$. We consider an operation $\exchange{X}{Y} $ the domain of which consists of the AFs $\F\in\AAnk[n]$ such that for an algebraic subgroup $C$ of $D_\F$ we have:\begin{enumerate}
 \item $C\supseteq[X,X]\cup[Y,Y]\cup[X,Y];$
 \item $D_{\F}=YC;$
  \item For $\gamma\in X\cup Y$ we have $\gamma(\F|_C)=\F|_C$ (notice this implies that $X$ and $Y$ normalize $C$);
  \item $\F(X\cap C)=\F(Y\cap C)=\{0\}$; 
  \item The action by conjugation of  $Y\cap C\s Y $ on the set \begin{equation}\label{zomz}\{\QQ\in\AAnk(XC):\QQ|_C=\F|_C \} \end{equation}is free and transitive\footnote{After reading up to (and including) Definition \ref{actionestep} one sees that this property is restated by saying: the action by conjugation of  $Y\cap C\s Y $, on the $(\F|_C,\e)$-step obtained from $\e(X)$, is free and transitive.}.

 \end{enumerate}
One can easily\footnote{Assume on the contrary that there are two subgroups $C_1$ and $C_2$ of $D_\F$ so that: conditions 1,...,5 are true for $C$ replaced by $C_1$ or $C_2$; and $C_1\not\subseteq C_2.$ We have $C_1\subseteq YC_1 =YC_2.$ Hence we obtain $c_1\in C_1,y\in Y-C_2,c_2\in C_2$ such that $c_1=yc_2.$ For $x\in X$ we have $xc_1x^{-1}=xyx^{-1}xc_2x^{-1}$. From this equality and conditions 3 for $C=C_1,C_2$ and from Lemma \ref{euequivalent} below (that (i)$\iff$(ii)), for $C=C_1$ we obtain that all the AFs in the set in (\ref{zomz}) have the same value at $y$ which is impossible due to Fact \ref{areconnected}} check that $C$ is the unique subgroup of $D_\F$  satisfying these conditions. 
 
 For each $\F$ in the domain of $\eu(X,Y),$ $\exchange{X}{Y}(\F) $ is the  AF satisfying: \begin{equation*}D_{\exchange{X}{Y}(\F)}=XC,\qquad\eu(X,Y)(\F)|_C=\F|_C,\qquad\text{and}\qquad \eu(X,Y)(\F)(X)=\{0\}.\end{equation*}

 We call $\exchange{X}{Y}$ a $\eu$-operation.\end{sdefi}
 
Due to Fact \ref{areconnected}, the map from $\AAnk(N)$ to the vector space dual of $\Lie{[N,N]\s N}$ is a bijection. The set $\AAnk(N)$ is made into a variety by requiring that this map is an isomorphism (of varieties).

\begin{slem}[Known]\label{euequivalent}Let $\F,X,Y,C$ satisfy everything in the last definition,  up to and including the fourth condition. Then the following conditions are equivalent:

\vsp 
\noindent\textbf{(i)} condition 5 is true;\hspace{3mm}\textbf{(ii)} condition 5 is true after interchanging $X$ and $Y$; 

\vsp
\noindent\textbf{(iii)} the map from $\Lie{X\cap C}\s\Lie{ X}\times\Lie{Y\cap C}\s\Lie{ Y}$ to $\KKK$ defined by $$(x,y)\rightarrow \F([\exp(x),\exp(y)])$$ is a nondegenerate bilinear form. 

\vsp
\noindent\textbf{(iv)} The map from $X\cap C\s X$ to  $\AAnk(Y\cap C\s Y) $ given by$$x\rightarrow (y\rightarrow \F([x,y]) )$$ is an  isomorphism (of groups and of varieties); \hspace{3mm}\textbf{(v)} condition (iv) holds after interchanging $\F$ and $\Z$.
\end{slem}\begin{proof} By using Fact \ref{areconnected} for $[N\tlarrow X,Y]$, Fact \ref{Liehom}, and noticing that $\F([x,y])=(x^{-1}\F)(y)-\F(y)$, we obtain $(i)\iff(iii)$. Since $(iii)$ remains the same after interchanging X and Y, we obtain  $(i)\iff(ii).$ Finally by using again Fact \ref{areconnected} for $[N\tlarrow X,Y]$ we obtain $(iv)\iff(iii)\iff(v)$.
 \end{proof}

\begin{scor}[Known]\label{sym}Let $\F,X,Y$ be as in the definition above. The operation $\exchange{Y}{X}$ is defined, it contains in its domain the AF $\exchange{X}{Y}(\F)$, and if we  further assume that $\F(Y)=\{0\},$ we have \begin{equation*}\exchange{Y}{X}(\exchange{X}{Y}(\F))=\F.\qedhere \end{equation*}\end{scor}
\begin{proof}The lemma above implies that $\exchange{Y}{X}$ is defined, and the rest follows.\end{proof}In ``concrete" cases, we use the corollary above without mention\footnote{For example we do not mention it in every case that the domains of the AFs involved are generated by $\ossa$-groups (which we define later in \ref{ossagroup}).}.

We roughly explained the concept of an $\AAnk$-tree in Subsection \ref{treesroughly}. In the rest of the current subsection we precisely explain $\AAnk$-trees and other  closely related notations. To explain them, we start by defining a more general type of tree---which is never mentioned after the current section--- which we call ``$\AAnk$-labeled out-tree".  

\begin{sdefi}[$\AAnk$-labeled out-tree, input vertex]\label{Alab1}
By an $\AAnk$-labeled out-tree we mean a directed tree in which \begin{enumerate}
\item one vertex is called an input vertex, and all the edges point away from the input vertex;
\item Each vertex is labeled with an AF in $\AAnk[n]$.  
\end{enumerate} 
\end{sdefi} We describe the definition again in a more formal way (which we use only in  a few occasions).  
\begin{sdefi}[$\AAnk$-labeled out-tree]\label{Alab2}
 Consider a triplet $\Xi=(\VVV,\EEE,f)$, where: $\VVV$ is a set;  $\EEE$ is a subset of $\{(v,u):v,u\in \VVV\text{ and }v\neq u\}$; and $f$ is a function from $\VVV$ to $\AAnk[n]. $ We call the elements in $\VVV$ (resp. $\EEE$) the vertices (resp. edges) of $\Xi.$ For each $v\in \VVV$ we call $f(v)$ the label of $v$.  
 
 We say that $\Xi$ is an $\AAnk$-labeled out-tree if in addition: the graph with vertices (resp. edges) the elements of $\VVV$ (resp. the elements of $\{\{u,v\}:(u,v)\in\EEE\}$) is a tree; there is a $v_0\in\VVV$ such that for all $(v,u)\in\EEE$, the number of edges connecting $u$ with $v_0$ is bigger (by one) from the number of edges connecting $v$ with $v_0$ (we say that $v_0$ is connected to itself by zero vertices). 
\end{sdefi}
\begin{sdefi}[Input,output,...]\label{outAl}Let $\Xi=(\VVV,\EEE,f)$ be an $\AAnk$-labeled out-tree. The input vertex of $\Xi$ is already defined. An output vertex  $u$ of $\Xi$ is  an element in $\VVV$ such that for every  $u^\prime\in\VVV$ we have $(u,u^\prime)\not\in\EEE.$ The label of the input vertex (resp. any vertex, an output vertex) of $\Xi$ is called the input AF (resp an AF, an output AF) of $\Xi$.
\end{sdefi}
\begin{convention}\label{convention} Unless otherwise specified, we identify any two $\AAnk$-labeled out-trees which are ``isomorphic as labeled directed trees".

We describe the previous paragraph precisely. Let $(\VVV,\EEE,f)$ be an $\AAnk$-labeled out-tree. Let $s: \VVV^\prime\rightarrow \VVV$
be a bijection of a set $\VVV^\prime$ onto $\VVV$, and $\EEE^\prime=\{(s^{-1}(x),s^{-1}y)|(x,y)\in E\}$. Then 
$(\VVV^\prime,\EEE^\prime,f\cdot s)$ is an $\AAnk$-labeled out-tree which we identify with $(\VVV,\EEE,f)$.
\end{convention}Since $\AAnk$-labeled out-trees are the most general trees we consider, we define subtrees so that they preserve labels and direction. More presisely:
\begin{sdefi}[Subtree]\label{subtree}  Let $\Xi=(\VVV,\EEE,f)$ be an $\AAnk$-labeled out-tree.  Consider a choice of sets  $\VVV^\prime\subseteq \VVV$ and $\EEE^\prime\subseteq\EEE$, for which the triplet $\Xi^\prime:=(\VVV^\prime,\EEE^\prime,f|_{\VVV^\prime})$ is also an $\AAnk$-labeled out-tree. Then we say that $\Xi^\prime$ is a subtree of $\Xi.$

More generally assume that for two texts $\ast$ and $\ast^\prime$ we encounter the text $\ast$-sub$\ast^\prime$. Then by an $\ast$-sub$\ast^\prime$ of $\Xi$ we mean a subtree of $\Xi$ which is an $\ast$-$\ast^\prime$. For example, for $\F\in\AAnk$, by using the later-defined concept of $\AAnk$-quasipaths, an $(\F,\e)$-subquasipath of $\Xi$ is a subtree of $\Xi$ which is an $(\F,\e)$-quasipath.
\end{sdefi}
\begin{sdefi}[Initial subtree]
Let $\Xi$ be an $\AAnk$-labeled out-tree. If for any text $\ast$ we encounter the text ``an initial $\ast$ of $\Xi$", then by the later text we mean: an $\ast$ of $\Xi$ which has the same input vertex as $\Xi$. For example $\ast $ can by any of the later-defined texts: $\AAnk$-subtree, $\AAnk$-subquasipath, $\AAnk$-subpath.\end{sdefi}

\begin{sdefi}[Depth]\label{outinput}Let $\Xi$ be an $\AAnk$-labeled out-tree. The depth of a vertex $v$ of $\Xi$ is the number of edges of the path connecting $v$ to the input vertex of $\Xi$. The depth of $\Xi$ is the biggest number among the depths of vertices of this tree. \end{sdefi}

Before defining $\AAnk$-trees in general, we define the ones which are obtained (as in Corollary \ref{everye} below) from a single $\AAnk$-operation, and we call them  $\AAnk$-steps.   

\begin{sdefi}[$\AAnk$-steps, compete after Definitions \ref{estepdef}, \ref{eusteps}, \ref{costepsdef}, \ref{trees}]\label{Astepdef} An $(\AAnk,\kkk)$-step is an $\AAnk$-labeled out-tree which is  an $(\e,\kkk)$-step or a $(\eu,\kkk)$-step or a $(\co,\kkk)$-step (and each of them in turn is defined below). We also call it $\AAnk$-step.\end{sdefi}
\begin{sdefi}[$\e$-steps]\label{estepdef} Consider an $\AAnk$-labeled out-tree $\xi$ of depth one with input AF $\F$ and with all its AFs in $\AAA$, and a $\kkk$-group N such that: $N\supseteq D_\F\supseteq [N,N]$; any two output vertices of $\xi$ are labeled with different AFs, and  the set of output AFs of $\xi$ is equal to $\{\Z\in\AAA(N): \Z|_{D_\F}=\F\} $ (hence this set is nonempty). Then $\xi$ is called an $(\e,\kkk)$-step or an $\e$-step. 

Except for the current subsection, it is always assumed that $N\neq D_\F$ (more on this in the convention below). Then, by Fact \ref{areconnected} we have $\Di{N}>\Di{D_\F}$ (some times we skip to mention this use of Fact \ref{areconnected}). 

Due to the relation of $\e$-steps with Fourier expansions, an alternative name we adopt in many cases for an output AF of an $\e$-step is: a term of this $\e$-step. More on this  in Definition \ref{termse}.\end{sdefi} \begin{convention}[Excluding trivial $\e$-steps and $\eu$-steps] Except for the current subsection, $\e$-steps, $\e$-operations, $\eu$-steps (defined below), and $\eu$-operations are always assumed to be nontrivial. That is: $\e$-operations and $\eu$-operations are not the identity operations in their domain (by identifying a singleton with its element); for $\e$-steps it is explained above; for $\F,\Z,C$ as in Definition \ref{eusteps}  below we have $D_\F\neq C$ (and hence $D_\Z\neq C$).\end{convention}
\begin{sdefi}[Action on an $\e$-step]\label{actionestep}
Consider an $\e$-step $\xi$ with its AFs belonging in $\AAnk[n] $ (based on later notation this is an $(\e, GL_n)$-step), and let $H$ be a subgroup of $GL_n$ which acts by conjugation on the set of output AFs of this $\e$-step. We then say that $H$ acts on $\xi$ by conjugation. 
\end{sdefi}
\begin{sdefi}[$\eu$-steps]\label{eusteps}By a $\eu$-step, say $\I$, we denote an $\AAnk$-labeled out-tree with only two vertices, so that for $\F,\Z$ respectively being its input AF and its output AF, and for $C:=D_\F\cap D_\Z$  we have:\begin{enumerate}
\item $C\supseteq [D_\F,D_\F]\cup[D_\Z,D_\Z]\cup[D_\F,D_\Z]$;  
\item $\F|_C=\Z|_C;$
\item\label{eustepslast} The action by conjugation of $C\s D_\F$, on the $(\F|_C,\e,\KKK)$-step for which $\Z$ is one of its output AFs, is free and transitive.  
 
\end{enumerate}If all the data\footnote{More precisely, the AFs $\F,\Z$ are defined over $\kkk$,  and the action by conjugation of $C(\kkk)\s D_\F(\kkk)$, on the $(\F|_C,\e,\kkk)$-step for which $\Z$ is one of its output AFs, is free and transitive.} is defined over $\kkk$, we say that $\I$ is a $(\eu,\kkk)$-step.\end{sdefi}
\begin{slem}[Known]\label{eustepslem}Let $\F,$ $\Z$, and $C$ be as in the definition above except that the third condition is not required. The following conditions are equivalent:

\vsp 
\noindent\textbf{(i)}condition \ref{eustepslast} holds\hspace{3mm}\textbf{(ii)} condition \ref{eustepslast} holds after interchanging $\F$  and $\Z$; 

\vsp
\noindent\textbf{(iii)} the map from $\Lie{C}\s\Lie{ D_\Z}\times\Lie{ C}\s\Lie{ D_\F}$ to $\KKK$ defined by $$(x,y)\rightarrow \F([\exp(x),\exp(y)])$$ is a nondegenerate bilinear form; 

\vsp
\noindent\textbf{(iv)} The map from $ C\s D_\F$ to  $\AAnk( C\s D_\Z) $ given by$$x\rightarrow (y\rightarrow \F([x,y]) )$$ is an  isomorphism (of groups and of varieties); \hspace{3mm}\textbf{(iv)} condition (iv) holds after interchanging $\F$ and $\Z.$
\end{slem}\begin{proof}It directly follows from Lemma \ref{euequivalent}, Corollary \ref{partlysplit}, and Fact \ref{areconnected}. Alternatively one can avoid Corollary \ref{partlysplit} by repeating the proof of Lemma \ref{euequivalent}.
\end{proof}
\begin{sdefi}[$\co$-steps]\label{costepsdef}Let $\F\in\AAA[,n]$ and $\gamma\in GL_n(\kkk).$ Then the $\AAnk$-labeled out-tree with only two vertices, having $\F$ and $\gamma\F$ as input and output AF respectively is called a $(\co,\kkk)$-step or a $\co$-step, and we say it is obtained from $\gamma$.\end{sdefi}
\begin{sdefi}\label{obtain1}Given an $\AAnk$-operation, say $f,$ we say that an $\AAnk$-step $\xi$ is \textbf{obtained from} $f$, if the input AF $\F$ of this $\AAnk$-step belongs in the domain of $f$ and $f(\F)$ is equal to the set of output AFs of $\xi.$ Of course in the previous sentence, in case $\xi$ is not an $\e$-step, we identify the singleton set $f(\F)$ with its element. An $\e$-step obtained from an $\e$-operation, say $\e(V)$, is said to be \textbf{over} $V$.
\end{sdefi}
\begin{scor}[Known]\label{everye}Every $e$-step, (resp. $\eu$-step, $\co$-step) is obtained from an $\e$-operation, (resp. $\eu$-operation, $\co$-operation). Let $\F,\Z $ respectively be the input and output AFs of a $\eu$-step. Then this $\eu$-step is obtained from an operation $\eu(X,Y) $ so that for $C$ being as in Definition \ref{euoperation} we have $C:=D_\F\cap D_\Z$.

Finally, for every $\e$-operation (resp. $\eu$-operation, $\co$-operation), and every AF $\F$ in its domain, there is a unique $e$-step, (resp. $\eu$-step, $\co$-step) with input AF equal to $\F$ which is obtained from this operation.\end{scor}
\begin{proof}
The first sentence for $\co$-steps and the second paragraph are trivial. The rest  directly follows from Corollary \ref{partlysplit}.\end{proof}

\begin{sdefi}[$\AAnk$-trees]\label{trees}
An $\AAnk$-labeled out-tree $\Xi$ is called an $(\AAnk,\kkk)$-tree if and only if  the following two properties hold:\begin{itemize}\item $\Xi$ has finite depth.
 \item Consider any vertex $v$ of $\Xi$ which is not an output vertex. Then the subtree of $\Xi$ of depth one that contains as many vertices as possible and has $v$ as input vertex, is an $(\AAnk,\kkk)$-step. To avoid confusion: ``depth one" refers only to the subtree, and not in any way to $\Xi$; recall  the meaning of  ``subtree"  includes preservation of  direction and labels.   
 \end{itemize}  An $(\AAnk,\kkk)$-tree is also called an $\AAnk$-tree. For an $\AAnk$-step $\xi$ which is a subtree of an $\AAnk$-labeled out-tree (usually an $\AAnk$-tree) $\Xi$, we say that ``$\xi$ is an $\AAnk$-step of $\Xi$". Of course $\AAnk$-steps are $\AAnk$-trees.\end{sdefi}

\begin{sdefi}[More on $\AAnk$-trees]\label{moretrees}Let $\SSS$ be a subset of $\AAnk$, $\ast$ be the text obtained by removing the brackets of a subset of $\{\e,\eu,\co\}$ (we have not given any stand alone meaning to the symbols $\e,\co,\eu$), and $H$ be an algebraic group. By an \begin{equation}
\label{ztreesnotation}(\F\rightarrow \SSS,\ast,H,\kkk)\text{-tree}\end{equation}  we mean  an $\AAnk$-tree, say $\Xi,$ such that: (i) $\F$ is the input AF of $\Xi$; (ii) $\SSS$ contains the  output AFs of $\Xi$; (iii) any $\AAnk$-step of $\Xi$ is a $\diamond$-step for $\diamond $ belonging in the set obtained from the text $\ast$ (by putting back the brackets); (iv) all groups and matrices defining $\Xi$ are inside\footnote{That is, the domain of every AF of $\Xi$ is contained in $H$, and every $\co$-step of $\Xi$ is obtained by conjugating with a matrix in $H$.} $H$; (v) $\Xi$ is an $(\AAnk,\kkk)$-tree, and every $\co$-step of $\Xi$ is obtained by conjugating with an element in $H(\kkk)$. 

If any subcollection from the requirements (i),...,(v), is removed, we respectively remove from the notation in (\ref{ztreesnotation}) the text ``$\F$", ``$\rightarrow \SSS$", ``$,\ast$", ``$,H$", ``$,\kkk$". For example an $(\F,H)$-tree is an $\AAnk$-tree satisfying (i) and (iv). Of course, if (iv) is removed but (v) remains, we can remove from (v) the part after the comma. Two minor clarifications: if all but (i) is removed, we write ``$\F$-tree" instead of ``$(\F)$-tree"; and if all except (v) is removed we write ``$(\AAnk,\kkk)$-tree" instead of ``$(\kkk)$-tree".

If $\SSS$ consists of a single element it is identified with it. 

We also adopt the current definition by replacing in its text every occurrence of ``tree" with ``step".  \end{sdefi}\begin{sdefi}[Trivial $\AAnk$-tree] An $\AAnk$-tree with only one vertex is called trivial. For $\F$ being its AF, we usually refer to it as the trivial $\F$-tree.\end{sdefi}
\begin{remark}
	In case in the definiton above $H=GL_n$, even though the meaning of ``$(\F\rightarrow S,\ast,\kkk)$-tree" is the same as the meaning of ``$(\F\rightarrow S,\ast,GL_n,\kkk)$-tree", I frequently adopt the later in the next sections.
\end{remark}

\begin{sdefi}[$\AAnk$-paths, $\AAnk$-quasipaths, $\quasi{\Xi}$]\label{quasidef}

An $(\AAnk,\kkk)$-labeled out-tree $\Xi$, in which every vertex is followed\footnote{That is, for $\Xi=(\VVV,\EEE,f)$, for every $u\in\VVV$ there is at most one $v\in\VVV$ such that $(u,v)\in\EEE$.} by at most one vertex, and which is a subtree of an $(\AAnk,\kkk)$-tree, say $\Xi^\prime,$ is called any of the names (depending on what we need to specify): $\AAnk$-path, $(\AAnk,\kkk)$-path, $\AAnk$-subpath of $\Xi^\prime,$ $(\AAnk,\kkk)$-subpath of $\Xi^\prime.$ Note that the last two names are implied from the first two by Definition \ref{subtree}.

Assume (for the rest of the current definition) there is no $\AAnk$-tree $\Xi^{\prime\prime}$ satisfying: $\Xi^{\prime\prime}$ is a proper subtree of $\Xi^\prime$, and $\Xi$ is a proper subtree of $\Xi^{\prime\prime}$ (by ``proper" we mean not equal). Then we say that $\Xi^\prime$ is an $(\AAnk,\kkk)$-quasipath or an $\AAnk$-quasipath, and we write $\Xi^\prime=\quasi{\Xi}$.  Of course, a more explicit description of $\quasi{\Xi}$,  is to obtain it from $\Xi$, by adding all the vertices and edges of each $(\e,\kkk)$-step which: has as its input AF and one of its output AFs, labels of adjacent vertices in $\Xi$.

We adopt the definitions obtained by changing\footnote{Also, if we choose ``path" we replace in (iii) $\Xi$ with $\quasi{\Xi}$.} in the text of the Definition \ref{moretrees} all occurrences of ``trees"  either with ``path",  or with ``quasipath".

In contrast to the notations in the previous paragraph, the notation $\quasi{\Xi}$ is used only in a few cases (mainly Section \ref{variety}), and in these few cases we assume $\kkk=\KKK$.\end{sdefi}   
\begin{sdefi}[Applying embeddings and $\circ$ to an $\AAnk$-tree]\label{ctree}In addition to $n$, let $n^\prime$ also be a positive integer, and let $H$ (resp. $H^\prime$) be a $\kkk$-subgroup of $GL_n$ (resp. $GL_{n^\prime}$). Let $j$  be a $\kkk$-embedding from $H$ in $H^\prime$. Let $\Xi$ be an $(\AAnk[n] ,H,\kkk)$-tree. 

We define $j(\Xi) $ to be the $(\AAnk[n^\prime] ,H^\prime,\kkk)$-tree which is obtained from $\Xi$, by replacing each AF $\XX$ of $\Xi$ with the AF $j(\XX).$     

Let $\F$ be an AF in $\AAA[,n] $ such that: $H\subseteq\Stab{GL_n}{\F}$; and for every AF $\XX$ of $\Xi$, the AF $\XX\circ\F$ is defined (which is the case if for example $H\cap D_\F$ is trivial). We define  $\Xi\circ\F$ to be the $(\AAnk,\Stab{GL_n}{\F},\kkk)$-tree obtained from $\Xi$, by replacing each AF $\XX$ of $\Xi$ with the AF $\XX\circ\F. $ \end{sdefi}

\begin{sdefi}[The operation $\vee$]\label{veenot1} Let $\Xi=(\VVV,\EEE,f)$ be an $(\AAnk{,\kkk})$-tree, let $\SSS$ be a subset of the  output AFs of $\Xi$, and for each $\XX\in\SSS$ let $\Xi(\XX)=(\VVV_\XX,\EEE_\XX,f_\XX)$ be an $(\AAnk{,\kkk})$-tree with input AF equal to $\XX$. It is assumed that any two output vertices of $\Xi$ have different labels. We use Convention \ref{convention} exactly once so that: (i) for each $\XX\in \SSS$, the input vertex of $\Xi(\XX)$ is the same as the output vertex of $\Xi$ labeled with $\XX$; and (ii) there are no other equalitites of vertices in any two among the trees $\Xi,\Xi(\XX)$ for $\XX\in \SSS $. We define the $(\AAnk,\kkk)$-tree:\begin{equation}\label{zveenotz}\Xi(\XX)\vee_{\XX}\Xi:=\left(\VVV\cup(\cup_{\XX\in \SSS} \VVV_\XX),\EEE\cup(\cup_{\XX\in \SSS} \EEE_\XX),f^\prime\right) \end{equation} where $f^\prime$ is the function from $\VVV\cup(\cup_{\XX\in \SSS} \VVV_\XX)$ to $\AAA[,n]$ given by: $f^\prime|_{\VVV}=f$, $f^\prime|_{\VVV_{\XX}}=f_{\XX}.$ 
 
 Since we write ``$\vee_\XX$" instead of say ``$\vee_{\XX\in \SSS}$", we always specify after each use of this notation the set $\SSS$ over which $\XX$ varies. In some cases that $\SSS$ is a singleton we write $\vee$ instead of $\vee_\XX$.\end{sdefi}\begin{sdefi}[$\vee$ over vertices, mostly optional]\label{veenot2} It is esentially the same as the previous definition. In detail: drop the second sentence and (i); and replace each occurrence of ``AF" and ``AFs" respectively with ``vertex" and ``vertices" (in particular, this makes $\SSS$ a set of vertices).\end{sdefi}
\begin{sdefi}[$\mathrm{I}^{-1}$]\label{inversep} Let $\I$ be a $(\eu,\co)$-path. We denote by $\I[-1]$ the $(\eu,\co)$-path obtained by reversing the direction in $\I$ (here we used Corollary \ref{sym}).
\end{sdefi}

The rest of the current subsection is mostly needed in: Sections \ref{pri}, \ref{optionalsection}, and \ref{compositions}.
\begin{sdefi}[Terms of an $\e$-step, constant and nonconstant terms]\label{termse} Recall the output AFs of an $\e$-step $\xi$ are also called the terms of $\xi$. Let $\F$ be the input AF of $\xi$, and $\e(V)$ be an $\e$-operation from which $\xi$ is obtained. Also let $0_V$ be the trivial AF with domain $V.$ Then we see that $0_V\circ\F $ is among the terms of $\xi$, and we call it the constant term of $\xi.$ 

Of course the constant term depends on the choice of $\e$-operation among the ones from which $\xi$ is obtained. Unless otherwise specified it is assumed that: if $V$ can be chosen to be a root group, and in a unique way, then the constant term is with respect to this choice; if an $\e$-operation from which $\xi$ is obtained is mentioned (for any reason), then again, the constant term is with respect to this operation.

The terms of $\xi$ other than the constant term are called the nonconstant terms of $\xi$.
\end{sdefi}\begin{sdefi}[Used in Sections \ref{pri}, \ref{optionalsection} and \ref{compositions}]\label{alongconstant}
We say that an $\e$-quasipath $\xi$ is \textbf{along constant terms} if: the input AF of each $\e$-step of $\xi$ other than the initial $\e$-step of $\xi$, is the constant term of an $\e$-step in $\xi$. The constant term of the last $\e$-step of $\xi$ is called the \textbf{output constant term} of $\xi$.
\end{sdefi}\begin{sdefi}\label{ovtg}
Only for the needs of the current definition we define:

\begin{subdef*} Let $V$ be an algebraic subgroup of $U_n$ . Let $r$ be the smallest row on which $V$ is nontrivial. We define $\rmo{V}$ to be the biggest number for which $V$ is nontrivial on the  $(r,\rmo{V})$ entry.\end{subdef*}

Let $S$ be a finite set of algebraic subgroups of $U_n$ (resp. a set of pairs of the form $(X,Y)$ where: $X$ and $Y$ are unipotent algebraic subgroups of $GL_n$).  We say that an $\e$-quasipath $\xi$ is \textbf{over} the groups $V$ for $V\in S$ (resp. a $\eu$-path $\xi$ is obtained by applying $\eu(X,Y)$ for $(X,Y)\in S$) if: for each $V\in S $ (resp. $(X,Y)\in S$) there is exactly one $\AAnk$-step in $\xi$ which is obtained from the operation $\e(V)$ (resp. $\eu(X,Y)$). We also use small modifications of these phrases, for example by adding for emphasis the word ``successively". In case $\xi$ is an $\e$-quasipath for which the ``over" is used as previously\footnote{to be more precise, a text of the form ``over $\ast$groups"  is used, where $\ast$ can be any text (for example ``over $\ast$groups" can be ``over the groups" or ``over the subgroups"). Such texts appear only within the usual proof environment (never in a statement of a theorem, proposition, lemma,...).}, we assume unless otherwise specified that the \textbf{order} the $\e$-steps of $\xi$ occur within $\xi$ starts with $\rmo{V}$ being as big as possible and in each step it becomes smaller.\end{sdefi}
 
\begin{sdefi}[More on $\AAnk$-paths, only used in Sections \ref{variety} and \ref{optionalsection}] Let $\Xi$ be an $\AAnk$-path. We say that $\Xi$ is along constant terms if and only if $\quasi{\Xi}$ is along constant terms. 

Let $\Xi^\prime$ be a second $\AAnk$-path, with its input AF being equal to the output AF of $\Xi.$ We define $\Xi^\prime\vee\Xi$ to be the $\AAnk$-path with same output AF as $\Xi^\prime$ satisfying: \begin{equation*}\label{zpaparia}\quasi{\Xi^\prime\vee\Xi}=\quasi{\Xi^\prime}\vee\quasi{\Xi}.\qedhere\end{equation*}\end{sdefi}

A question occurring naturally is:
\begin{NPQ}Let $\kkk\subset\kkk^\prime$ be two fields. Which $(\AAnk,\kkk)$-trees can be mapped to $(\AAnk,\kkk^\prime)$-trees in a way ``compatible" with mapping the set of $\kkk$-rational points of any variety to the set of $\kkk^\prime$-rational points of the same variety?\end{NPQ}

\noindent{I} have not addressed this or any similar question. Of course, the answer is positive for  $\AAnk$-paths and  $\AAnk$-quasipaths.

\begin{sdefi}[Conjugate of an $\AAnk$-tree, mostly optional\footnote{This definition is too simple for its size. The extent to which one will guess the meaning (in any of the uses that we make) without reading the definition, may be sufficient.}]\label{conjt} Within the current definition: all $\AAnk$-trees are $(\AAnk,\kkk)$-trees and all matrices have entries in $\kkk$; Convention \ref{convention} is used \textbf{only} when it is mentioned. Also, the two subdefinitions below are only used within the current definition. 

\vspace{1mm}
\noindent\textbf{Subdefinition 1}. Let $\Xi$  be an $\AAnk[n]$-tree and $\gamma\in GL_n$. Then $\gamma\Xi$ is defined to be the $\AAnk[n]$-tree obtained from $\Xi$ by replacing each AF $\F$ of $\Xi$ with the AF $\gamma\F.$\hfill$\triangle$Subdefinition 1 

\vspace{1mm}
\noindent\textbf{Subdefinition 2}. Let $\Xi$ be an $\AAnk[n]$-tree, and $\SSS$ be a set consisting of $\AAnk[n]$-subtrees of $\Xi$ such that: each $\AAnk$-step in $\Xi$ is in exactly one among the $\AAnk$-trees in $S.$ Let $s$ be a function with domain $\SSS$ taking as values $\AAnk[n]$-trees, such that for each $\xi\in \SSS$ one of the following two happens: (i) the trees $\xi$ and $s(\xi)$ have the same input and output vertices and in each such vertex they have the same label; (ii) $\xi$ has a unique output vertex, the input AF and output AF of $\xi$ are the same, say equal to $\XX$, and $s(\xi)$ is the trivial $\XX$-tree. Then we can start extending the domain of $s$ by defining $s(\Xi_u^{\prime}\vee_u\Xi^\prime):=s(\Xi_u^{\prime})\vee_us(\Xi^\prime)$  where: $u$ is over all output vertices of $\Xi^\prime$; the domain of $s$ is already extended enough to contain $\Xi^\prime$ and $\Xi_u^\prime$ for each such $u$; and Convention \ref{convention} is used for vertices coming from (ii). At the end of this process we obtain $s(\Xi)$.\hfill$\triangle$Subdefinition 2

Let $\Xi$ be an $(\AAnk[n],\kkk)$-tree. 
 
 Let $s_1$ be a function with domain the set of $\AAnk$-steps of $\Xi$, and values being $\AAnk$-steps, such that\begin{enumerate}
 \item $s_1(\xi)=\xi$ for any $\co$-step $\xi$ of $\Xi$
 \item Let $\xi$ be an $\AAnk$-step of $\Xi$ satisfying $s_1(\xi)\neq \xi.$ For an element $\gamma\in GL_n(\kkk)$ we have $s_1(\xi)=\co_\XX(\gamma^{-1})\vee_{\XX}\gamma\xi\vee\co(\gamma) $ where: $\co(\gamma)$ is the $\co$-step obtained from $\gamma$ with the same input AF as $\xi$, $\XX$ varies over all output AFs of $\xi,$ and $\co_\XX(\gamma^{-1})$ is the $\co$-step obtained from $\gamma^{-1}$ with input AF equal to $\XX.$
  \end{enumerate}By using Convention \ref{convention} we apply Subdefinition 2 for $[s\tlarrow s_1]$, and obtain the definition of $s_1(\Xi)$. Let $\co_1$ and $\co_u$ for each output vertex $u$ in $s_1(\Xi)$, be $(\co, GL_n)$-steps, with their vertices chosen so that we can define $(s_1(\Xi))^\prime:=\co_1\vee s_1(\Xi)\vee_u \co_u$.

  Finally consider a function $s_2$ such that: Subdefinition 2 is satisfied for $[\Xi\tlarrow(s_1(\Xi))^\prime, s\tlarrow s_2]$; and if $s_2(\xi)\neq \xi,$ then $\xi$ is a $\co$-path and $s_2(\xi)$ is trivial. 
  
  We say that $s_2((s_1(\Xi))^\prime)$ is a conjugate of $\Xi.$    \end{sdefi}
\subsection{Applying $\AAnk$-trees to automorphic forms}\label{ssaa} In the current subsection we assume that $\kkk$ is  a number field, $\F\in\AAA[,n],$ and $\phi$ is a $GL_n(\A)$-automorphic function.
\begin{sfact}[Known]\label{autsteps}

\noindent{Part 1. }For each $(\F,\e,GL_n,\kkk)$-step, we have the absolute and uniform on the compact subsets of $GL_n(\A)$ convergence: \begin{align}\label{zfournz}&&&&&&&\F(\phi)(g)=\sum \Z(\phi)(g)&\forall g\in GL_n(\A), \end{align}where the sum is over the output AFs of the $\e$-step. Also, for $N$ being the domain of the output AFs of this $\e$-step, by choosing a group $N_2$ as in Lemma \ref{partlysplit} for $[N_1=D_\F]$, the expansion in (\ref{zfournz}) is the Fourier expansion of $\F(\phi) $ over $N_2(\kkk)(N_1\cap N_2)(\A)\s N_2(\A).$

\vsp 
\noindent{Part 2.} Consider an operation $\exchange{X}{Y}$, with $X,Y$ being $\kkk$-subgroups of $GL_n$ (of course unipotent), which contains $\F$ in its domain. Then  \begin{align}\label{stel}&&&&&&\F(\phi)(g)=\int_{(Y\cap C)(\A)\s Y(\A)}\exchange{X}{Y}(\F)(\phi)(yg)dy&&\forall g\in GL_n(\A), \end{align}where $C$ is as in Definition \ref{euoperation}.

\noindent{Part 3. }For  $\gamma\in GL_n(\kkk),$ we have \begin{align*}&&&&&&&&(\gamma\F)(\phi)(g)=\F(\phi)(\gamma^{-1}g)&&\forall g\in GL_n(\A). \end{align*}

\end{sfact} 
\begin{proof}
Part 1 is obtained by the analytic properties of $\phi,$  Part 2 is given in Lemma 7.1 in  \cite{GRS2}, and Part 3 is trivial.  \end{proof}
\begin{sfact}[Known]\label{nonzeroeulerian}
Let $\Xi$ be an $(\F,GL_n,\kkk)$-tree. Then

\vsp 
\noindent{Part 1.} The function $\F(\phi)$ is nonzero if and only if there is an output AF $\Z$ in $\Xi$ for which $\Z(\phi)$ is nonzero.

\vsp
\noindent{Part 2. }The function $\F(\phi)$ is factorizable if and only if the next two conditions are true: (i) there is at most one output vertex of $\Xi$, so that for its  label $\Z$ we have $\Z(\phi)$ is nonzero, and (ii) if such a vertex exists, $\Z(\phi)$ is factorizable.

\end{sfact}
\begin{proof}
We are directly reduced to the case in which $\Xi$ is an $\AAnk$-step. In case $\Xi$ is a $\co$-step both parts are trivial (Part 3 of Fact \ref{autsteps}). In case $\Xi$ is a $\eu$-step, both parts follow from Part 2 of Fact \ref{autsteps} and  Corollaries \ref{sym} and \ref{everye} (the roles of $\F$ and $\exchange{X}{Y}(\F)$ must be inverted for the  if" in Part 1 and the ``only if " in Part 2) If $\Xi$ is an $\e$-step, Part 1 and the ``if" in Part 2 follow from Part 1 of Fact \ref{autsteps}.

We are therefore left with proving the ``only if'' in Part 2, and only for $\Xi$ being an $\e$-step. Let $N$ be the domain of any term of $\Xi.$ By Fact \ref{areconnected} (including Corollary $\ref{semprod}$) we are reduced to the case $N=N_2\rtimes D_\F$  for a one dimensional $\kkk$-subgroup $N_2$ of $N$. We consider a factorization 
$$\F(\varphi)(g)=\prod_\upsilon \F(\varphi)_\upsilon(g_u)\qquad\forall g\in GL_n(\A) $$ satisfying $\F(\varphi)_\upsilon(1)=1$ for almost all places $\upsilon.$ For any output AF $\Z$ of $\Xi$ we have \begin{equation}\label{zzpsi}\Z(\varphi)(n)=c_\Z\psi_\Z(n)\qquad\forall n\in N(\A)\end{equation} for a complex number $c_\Z.$

Let $M$ be the subset of the output AFs $\Z$ of $\Xi$ for which the restrictiong of $\Z(\varphi)$ on $N_2(\A)$ is nontrivial. We assume that $M$ contains at least two elements, and eventually reach a contradiction. 

We choose a place  $\upsilon_1$, an element $y\in\kkk_{\upsilon_1}$, and an AF $\Z_1\in M $ such that:\begin{enumerate}
\item $0\neq \F(\varphi)_{\upsilon_1}(y)\neq \psi_{\Z_1,\upsilon_1}(y)\qquad\text{and}\qquad\F(\phi)_{\upsilon_1}(1)=1;$
\item the automorphic form $\varphi$ (and hence every character $\psi_\Z$ for  $\Z\in M$) is unramified at $\upsilon_1$.
\end{enumerate} Let $X:=\left(\prod_{\upsilon\neq \upsilon_1} N_2(\kkk_\upsilon)\right)\times N_2(\ooo_{\upsilon_1})$, where $\ooo_{\upsilon_1}$ is the ring of integers of $\kkk_{\upsilon_1}$. For $n\in X$ we have 

  $$\F(\varphi)(n)-\frac{1}{\F(\varphi)_{\upsilon_1}(y)}\F(\varphi)(yn)=0. $$By applying to  $\F(\varphi)$
  the  Fourier expansion corresponding to the $\e$-step $\Xi,$ evaluated at  $n$ and at $yn$; and then applying (\ref{zzpsi}) for each $\Z\in M$, we obtain 
 \begin{equation}\label{zzdedekind}\sum_{\Z\in M} c_\Z\left(1-\frac{\psi_{\Z,\upsilon_1}(y)}{\F(\varphi)_{\upsilon_1}(y)}\right)\psi_\Z(n) =0\qquad\forall n\in X.\end{equation}Since $N_2(\A)=N_2(\kkk)X,$ equality (\ref{zzdedekind}) is true for all $n\in N_2(\A),$ and therefore becomes  a nontrivial linear combination of characters of $N_2(\kkk)\s N_2(\A)$, which is impossible. 
\end{proof}
\section{The $\AAnk$-tree $\sXi$}\label{Theorem A}
In Theorem \ref{general} below, for every AF $\F$ with  $D_\F\subseteq U_n$, an $\F$-tree $\sXi(\F)$ is inductively defined, which will be used considerably throughout the paper. 

Before reading the proof, it may be helpful (but optional) to check the two remarks after the proof. 

 \begin{theorem}\label{general}Let $n$ be any positive integer and $\kkk$ be any subfield of $\KKK$. Let $\F$ be an AF in $\AAA[,n] $ satisfying $D_\F\subseteq U_n$.  Then  there is an $$\left(\F\rightarrow\AAA(U_n),\mirsl{n},\kkk\right)\text{-tree}.$$

 \end{theorem}

\begin{tproof}We inductively define an $\left(\F\rightarrow\AAA(U_n),\mirsl{n},\kkk\right)$-tree which we call $\sXi(\F).$ The induction is on $n$ and downwards on $\Di{D_\F}.$ More precisely: if $D_\F=U_n$ we define $\sXi(\F)$ to be the trivial $\F$-tree (that is, $\sXi(\F)$ has only one vertex); and otherwise we define $\sXi(\F)$, assuming that the $\left(\F^\prime\rightarrow\AAA(U_{n^\prime}),\mirsl{n^\prime},\kkk\right)$-tree $\sXi(\F^\prime)$ is defined for all $n^\prime $ and $\F^\prime\in\AAA[,n^\prime] $ satisfying one among 1) and 2) below $$1)\hspace{2mm}n^\prime<n\quad\quad\quad 2)\hspace{2mm}n^\prime=n\text{ and }\Di{D_{\F^\prime}}>\Di{D_\F}. $$  

Let $j$ be the lower right corner embedding of $GL_{n-1}$ in $GL_n.$ Assume that an AF $\F^\ih\in \AAA[,n]$ satisfies \begin{equation}\label{zsmir}\F^\ih= j(\F^{\prime})\circ\YY\end{equation}
 where\begin{itemize}
 \item $\YY$ is an AF with domain $D_{\YY}=\prod_{1<j\leq n}U_{(1,j)},$ which is trivial on all root groups contained in $D_\YY$ except possibly $U_{(1,2)}$ (therefore $j(\mirsl{n-1}(\kkk))\in\Stab{\mirsl{n}}{\YY}$);
 \item $\F^\prime$ is an AF satisfying $D_{\F^{\prime}}\subseteq U_{n-1}.$\end{itemize} Then we define $\sXi(\F^\ih):=j(\sXi(\F^{\prime}))\circ\YY.$

  We will define an $(\F,\mirsl{n},\kkk)$-tree $\xi$, each output AF ${\F^\ih}^\prime$ of which, has domain contained in $U_n$, and either satisfies $\Di{D_{{\F^\ih}^\prime}}>\Di{D_\F}$ or is chosen as in (\ref{zsmir}) in the place of $\F^\ih$.  After that, the proof ends by defining $$\sXi(\F):=\sXi({\F^\ih}^\prime)\vee_{{\F^\ih}^\prime}\xi$$ where ${\F^\ih}^\prime $ varies over all output AFs of $\xi$. 
  
  In the cases that  $\xi$ will turn out to contain an  $\e$-step, say $\xi_\e$, this is its last $\AAnk$-step,  because the dimension of the domain of the terms of $\xi_\e$ (that is, the output AFs of $\xi_\e$) is bigger from $\Di{D_\F}$.

   Consider the smallest number $m\geq 2$ for which the set $A_m:=\prod_{m\leq i\leq n}U_{(1,i)}$ is contained in $D_\F$. To be clear, by $\prod_{n+1\leq i\leq n}U_{(1,i)}$ we mean the trivial subgroup of $GL_n.$

\vspace{1mm}
\noindent\textit{Case $1$: $\F(A_m)=\{0\}$ and $2<m$.} In this case $\xi$ is chosen to be the $(\F,\e)$-step over $U_{(1,m-1)}$.\hspace*{50mm}\hfill$\square$Case 1 

\vspace{1mm}
\noindent\textit{Case $2$: $m=2$ and $\F(A_2)=\{0\}$.}  In this case (\ref{zsmir}) is true for $[\F^\ih\leftarrowtriangle\F],$ and hence we choose $\xi$ to be the trivial $\F$-tree.\hfill$\square$Case 2

\vspace{1mm}
\noindent\textit{Case $3$: $\F(A_m)\neq\{0\}$.} We start with some definitions and observations, and return to defining $\xi$ in the next paragraph. Consider integers $1<l\leq n$ and $1\leq k<l$, and an algebraic subgroup $D$ of $U_n$. We define $L_{(k,l)}$ to be the subgroup of $U_{n}$ which is generated by all the positive root groups except the $U_{(i,l)}$ with $k<i< l$. For example $L_{(l-1,l)}=U_{n}$. For $k>1$, the group $L_{(k-1,l)}$ is a normal subgroup of $L_{(k,l)}$, and hence $D\cap L_{(k-1,l)}$ is a normal subgroup of $D\cap L_{{k,l}}$. Since unipotent algebraic groups are connected,  the identity embedding 
\begin{equation}\label{embedding1}j_{D,k,l}:D\cap L_{(k-1,l)}\s D\cap L_{(k,l)}\rightarrow L_{(k-1,l)}\s L_{(k,l)},\end{equation} is either trivial or an isomorphism. 

Back to defining $\xi,$ let $l_1$ be the biggest number such that $\F(U_{(1,l_1)})\neq\{0\}$. The $\AAnk$-tree $\xi$ will either finish with an  $\e$-step  in at  most $l_1-2$ $\AAnk$-steps or it will consist of $l_1-1$ $\AAnk$-steps each one being a $\eu$-step or a $\co$-step.

 Let $\F^0:=\F. $ Assume that for a number $i $ satisfying $1\leq i\leq l_1-2$, the first $i-1$ $\AAnk$-steps of $\xi$ have been defined, that none of them was an $\e$-step, and that the output AF of the $i-1$-th one has been given the name $\F^{i-1}$. To avoid confusion, if $i=1$, this means that no $\AAnk$-steps have been defined. The $i$-th $\AAnk$-step in $\xi$---and in case it is not an $\e$-step, its output AF $\F^i$---are defined as follows:\begin{enumerate}
\item If $j_{D_{\F^{i-1}},l_1-i,l_1}$ is trivial and $U_{(1,l_1-i)}\not\subseteq D_{\F^{i-1}} $, the $i$-th $\AAnk$-step in $\xi$ is the $(\F^{i-1},\e)$-step  over $U_{(1,l_1-i)}$. Hence this is the last $\AAnk$-step in $\xi.$
\item If $j_{D_{\F^{i-1}},l_1-i,l_1}$ is trivial and $U_{(1,l_1-i)}\subseteq D_{\F^{i-1}} $, there is an element $u\in U_{(l_1-i,l_1)}(\kkk)$, such that $u\F^{i-1}(U_{(1,l_1-i)})=\{0\}$. We choose the $i$-th $\AAnk$-step in $\xi$ to be the $(\F^{i-1}\rightarrow u\F^{i-1},\co)$-step and we define $\F^i:=u\F^{i-1}$.
\item Finally assume that $j_{D_{\F^{i-1}},l_1-i,l_1}$ is an isomorphism. Let $\F^i$ be the AF defined by \begin{multline}D_{\F^i}=U_{(1,l_1-i)}(D_{\F^{i-1}}\cap L_{(l_1-i-1,l_1)}),\qquad\F^i(U_{(1,l_1-i)})=\{0\},\quad\text{and}\\\F^i(u)=\F^{i-1}(u)\quad\forall u\in D_{\F^{i-1}}\cap L_{(l_1-i-1,l_1)}.\end{multline} The $i$-th $\AAnk$-step in $\xi$ is the $(\F^{i-1}\rightarrow\F^{i},\eu)$-step.  
\end{enumerate}   

 Assuming  that no $\e$-step was encountered, we have   
 \begin{equation}\label{downreplaced}\prod_{1<j\leq n}U_{(1,j)}\subseteq D_{\F^{l_1-2}}\subseteq L_{(1,l_1)}, \end{equation}
and the last $\AAnk$-step of $\xi$ is the $\co$-step obtained from the minimal length element  $w\in W_{n}$ such that (\ref{zsmir}) is correct for $[\F^\ih\tlarrow w\F^{l_1-2}].$\hfill$\square$Case 3

\end{tproof} 
\begin{defi}[$\sI(\F)$, $\siota(\F)$]\label{sIdef}For $\F\in\AAnk$, we define $\sI(\F)$ to be the biggest initial $(\eu,\co)$-subpath of $\sXi(\F)$. We denote by $\siota(\F)$ the output AF of $\sI(\F). $\end{defi}
\begin{defi}\label{sXom}
When we write ``$\sXi$" (resp. ``$\sI$") we mean $\sXi(\XX) $ (resp. $\sI(\XX)$) for some AF $\XX\in\AAnk$ (and some field $\kkk$ over which $\XX$ is defined).  \end{defi}
\begin{remark}[\textbf{Relations of $\sXi$ to the literature}]\label{sXilit} The definition of $\sXi$ is based on simple refinements and modifications of unfolding processes that have been used in the study of the $GL_n\times GL_m$ Rankin-Selberg integrals for $n\neq m.$ In particular:

\vspace{1mm} 
\noindent{(i).} The way the action of the mirabolic group on the AFs on its unipotent radical gives an inductive argument, is found in \cite{GK}, and in the Fourier expansion over $U_n(\kkk)\s U_n(\A)$ of a $GL_n(\A)$-cusp form. Note that by appropriately removing from $\sXi(\F_{\emptyset,n})$ one copy of each among $\sXi(\F_{\emptyset,k})$ for $k<n,$  the tree obtained roughly corresponds to calculating this Fourier expansion.

\vspace{1mm}
\noindent{(ii).} For $n\neq m,m-1,$ consider the form of the $GL_n\times GL_{m}$ Rankin-Selberg integral which is obtained from its initial form by applying $g\rightarrow {g^t}^{-1}$, where $g$ is the variable in the integration over $GL_m(\kkk)\s GL_m(\A).$ Then, the unfolding of this form contains $\eu$-steps which is worth comparing to Case 3 in the proof  Theorem $\ref{general}.$

Of course the relations (i) and (ii) above are also present in many works appearing later, and some additional relations can be found from  some of the $\AAnk$-trees discerned for example in \cite{GRS1} and in \cite{GRS2}. 

I hope in a later version of the present paper---after I learn more on the relevant literature--- I will improve the current remark. Comments are welcome.\end{remark}

\begin{remark}[Examples]
 Three examples one can check early in the reading of the present paper are given next (they include pictures):  \begin{itemize}\item In Example \ref{example} we fully follow $\sXi(\F)$, for an $\F\in\AAnk[6]  $ so that---for the later defined concepts ``$\Om(\F)$" and ``$\Omm{\F}$"---it turns out $|\Om(\F)|=|\Omm{\F}|=2.$  \item In Observation \ref{unipc} we present $\sI$ in cases that only trivially differ from $\AAnk$-trees appearing in the literature.\item A conjugate of $\sI(\F)$ for  certain higher dimensional analogues of the choice of $\F$ in Example \ref{example} is presented exactly after this example in the pairs of successive pictures with the first picture being the one with label $\F=\F_{4,4,3}$, $\F_{7,3,2}$, $\F_{(4,4),2}.$ Also in the second last picture of the same subsection, we discern a conjugate of $\sI(\F_{(3,2),3}) $ from the rectangles labeled with 1,2 and 3. \end{itemize} \end{remark}
 \begin{aconvention}[Dependence on $\kkk$]\label{sXrt}A tree $\sXi$ depends on the choice of $\kkk.$ Roughly speaking $\kkk$ is assumed to be chosen each time so that all nearby $\AAnk$-trees are $(\AAnk,\kkk)$-trees.
 
 Since many mentions of $\sXi$ are inside environments which are either proofs or definitions, I add some precision. Consider a tree  $\sXi$ mentioned within such an environment; then the definition of $\kkk$ given in this environment, is the same as the one defining this tree. \end{aconvention}
 In the ``most concrete"\footnote{In particular, we do not consider Theorem \ref{general} and Main corollary \ref{maincor} to be among the ``most concrete" parts} parts of the present paper, the domains of the AFs we consider, are generated by groups we call $\ossa$-groups, which are defined as follows:
 \begin{defi}[$\ossa$-groups]\label{ossagroup}Let $n$ be a positive integer. Let $S$ be a finite set of root groups (in $GL_n$) such that for any roots $\alpha,\beta$ for which $U_\alpha$ and $U_\beta$ belong in $S$ we have: neither $\alpha+\beta$ nor $\alpha-\beta$ is a root, and $\alpha\neq -\beta$. Any at most one dimensional algebraic subgroup of  $\prod_{V\in S} V$ is called a $\ossa$-group.
 	
 	Of course by a use of Fact \ref{areconnected} (not mentioned each time), $\ossa$ groups admit a more explicit description and the only zero dimensional is the trivial one.
 	
 	Recall the standard choice is assumed for root data.
 \end{defi} 
 \begin{remark}
 In this notation, ``a" stands for ``additive", then ``2" appears because $x^2=0$ for every $x$ in its Lie algebra, and ``1d" stands for ``at most one dimensional".
 \end{remark}

 Among the ``most concrete" parts (again) of the present paper, choices of $\F$ for which $D_\F$ is not generated by root groups, are mainly considered\footnote{Even in them, one can apply several $\eu$-steps and obtain an AF with domain generated by root groups. In the present paper I have not followed such a path; but I did followed such a path in \cite{Tsiokos}, in the special cases which also appear there.}  in Section \ref{compositions}. 
\section{The set of AFs $\Prink[n]$}\label{pri} In the current section $n$ (resp. $\kkk$) is any positive integer (resp. any subfield of $\KKK$). Recall  $\Prink[n]$ is the set of AFs  with domain a 
  unipotent radical of a $GL_n$-parabolic subgroup. The main result of the current section is Lemma \ref{radexpress}, in which, for each $\F\in\Prink[n] $ we obtain information on an $(\F\rightarrow \AAnk(U_n) )$-tree  that is later needed in the paper.
  
  Recall we defined $J_\F$ (for $\F\in\AnkT{n} $) in Definition \ref{uniqJ}, and $\hat{...}$ and $..._{\nless}$ in Definition \ref{hatnless}.

  Most statements of the paper in which $\Prink[n] $ is mentioned, are reduced  to replacing $\Prink[n]$ with  $\tPrink[n]$ or with $\Prink[n,\nless]$, by using the conjugation in Lemma \ref{conjugatehat} below.
  \begin{defi}[\text{$...[a]$}]\label{afise}
  Let $a\in\UUU{n}$, and $\XXX$ be a subset of $\Prink[n]. $ We define $$\XXX[a]:=\{\F\in\XXX:\Om(J_\F)=a\}.$$Also, for $a=[a_1,...,a_m]$, we write $\XXX[a_1,...,a_m]$ instead of $\XXX[[a_1,...,a_m]]$.\end{defi}

In the next observation and next lemma we give conjugations that we frequently use.

\begin{ob}\label{diagonal}
 Let $\F\in\tPri[,n]$. Let $M$ be the Levi of the $GL_n$-parabolic subgroup with unipotent radical $D_\F$.  For any $\DDDD[n]$-component $T$ of $\TT{n}{\F}$, let $\mathrm{up}(T),\mathrm{down}(T),b_T(\mathrm{up}(T))<b_T(\mathrm{up}(T)+1)<...<b_T(\mathrm{down}(T)),$  be the numbers such that  
 $$\Set{T}=\{b_T(i):\mathrm{up}(T)\leq i\leq\mathrm{down}(T)\}.$$ 
 
  Let $T^1,T^2$ be two $\DDDD[n]$-components  of $\TT{n}{\F}$ such that
 
 $$\mathrm{up}(T^1)\leq \mathrm{up}(T^2)\leq\mathrm{down}(T^1)\leq \mathrm{down}(T^2).$$

  Then $\Stab{GL_n}{\F}$ has as a subset a  $\ossa$-group defined over $\kkk$ with nonzero entries: 
  $$\{(b_{T^1}(i),b_{T^2}(i)):\mathrm{up}(T^2)\leq i\leq \mathrm{down}(T^1)  \}.$$   \end{ob}

 \begin{lem}\label{conjugatehat} Let $MN$ be the Levi decomposition (with $M$ being the Levi) of a $GL_n$-parabolic subgroup. Let $k$ be the number of blocks of $M.$ For $1\leq i\leq k$, let $n_i$  be such that the $i$-th  block of $M$ is isomorphic to $GL_{n_i}$, and let $\jj SL_{n_i}$ and $\jj U_{n_i}$ respectively be the standard copies of $SL_{n_i}$ and $U_{n_i}$ contained in this block. Let $\F\in\AAA(N)$. There is a $\gamma\in(\jj U_{n_1}\times{\prod_{2\leq i\leq k}}\jj SL_{n_i})(\kkk)$ such that
 $$\gamma\F\in \AAA(N)_{\nless}.  $$
 
 \end{lem}
 \begin{tproof}We proceed  inductively on $k$.

 \vsp
 \noindent\textit{Proof for $k=2$.} We do an induction on $n$. Let $R_1:=\prod_{2\leq j\leq n}U_{(1,j)}$.  
 
 \vspace{1mm}
 \noindent\textit{Case 1: $\F(R_1\cap D_{\F})=\{0\}$}. This case follows from the inductive hypothesis for $[\F\tlarrow\jj^{-1}\F^{\{1\}}]$.\hfill$\square$Case 1
 
 \vspace{1mm}
 \noindent\textit{Case 2: $\F(R_1\cap D_\F)\neq \{0\}$}. After conjugating with an element in  $W_{\jj SL_{n_2}}$ we are reduced to the subcase $\F(U_{(1,n_1+1)})\neq \{0\}$. As in the picture below let $$X=\prod_{2\leq j\leq n_1}U_{(1,j)}\prod_{n_1+2\leq i\leq n}U_{(i,n_1+1)}\qquad\text{and}\qquad Y:=\prod_{2\leq i\leq n_1}U_{(i,n_1+1)}\prod_{n_1+2\leq j\leq n}U_{(1,j)}. $$ In the picture $D_\F$ is the group whithin the black wide lines. By conjugating with an appropriate element in   $X(\kkk)$ we are further reduced to assuming $\F(Y)= \{0\}$. With these reductions Case 2 is obtained from the inductive hypothesis for $[\F\tlarrow\jj^{-1}\F^{\{1,n_1\}}]$.\renewcommand{\di}{7}
    
 \
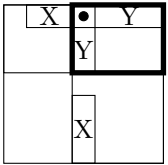
 \begin{tikzpicture}[scale=0.3]
 \bu{(3.5,6.5)}
 \rec{0}{0}{7}{7}
 \rec{0}{0}{3}{3}
 \rec{3}{3}{7}{7}
 \draw[black,line width=2pt]\mm{0}{3}--\mm{0}{7}--\mm{3}{7}--\mm{3}{3}--cycle;
 \draw(3,0)rectangle(4,3)node[midway]{X};
 \draw(1,6)rectangle(3,7)node[midway]{X};
 \draw(4,6)rectangle(7,7)node[midway]{Y};
 \draw(4,6)rectangle(3,4)node[midway]{Y};
 \end{tikzpicture}
 \hfill$\square$Case 2$\square$Proof for $k=2$

 We fix a value of $k$ bigger than two, and assume that for smaller values of it the lemma is settled. For $1\leq i\leq k,$ we define $p_i$ to be the group homomorphism on $M$ which restricts to the identity homomorphism on the $i$-th block of $M$ and is trivial on all the other blocks of $M$. Let $A=\{(\sum_{1\leq i\leq k-1}n_i)+1,...,\sum_{1\leq i\leq n}n_i\}$. The inductive hypothesis for $k-1$ reduces the problem to the case in which $\jj^{-1}\F^{A}\in \AnkT{n-n_k}_{\nless}$. Next we prove the following claim.
 
 \vspace{1mm}
 
 \noindent\textbf{Claim.}\textit{  There is an element  $w$ in  $W_{\jj SL_{n_{k-1}}}$ such that for the group $$M^{\prime}:=\left(\jj U_{n_1}\times\prod_{2\leq i\leq k-1} \jj SL_{n_i}\right)\cap \mathrm{Stab}_{SL_n}(w\F^{A})$$ we have $p_{k-1}(M^{\prime})\supseteq \jj U_{n_{k-1}}$.}

 \vsp
 \noindent\textit{Proof of Claim.} Among the $\DDDD[n]$-components of $\TT{n}{w\F^{A}}$, let   $T_w^1,T_w^2,...,T_w^{n_{k-1}}$  be the ones nontrivially intersecting the $k-1$-th block of $M$.   By Observation \ref{diagonal} we see that  one choice of $w$ satisfying the claim is  the minimal length element in $W_{\jj SL_{n_{k-1}}}$  satisfying: $$\left(\max(\Set{T_w^i})-\max(\Set{T_w^j})\right)\left(|\Set{T_w^i}|-|\Set{T_w^j}|\right)\leq 0.$$\hfill$\square$Claim
 
 Let $B:=\{1,...,n\}-A=\{1,2,...\sum_{1\leq i\leq k-1}n_i\}.$   Then from the case $k=2$, we obtain an element $\delta\in (\jj U_{n_{k-1}}\times \jj SL_{n_k})(\kkk)$, for which $\jj^{-1}(\delta w\F)^{B}\in\hat{\AAA}[T_{n_{k-1}+n_k}]$.  By using the claim, consider an element $\gamma_1\in M^{\prime}(\kkk)$ satisfying $p_{k-1}(\gamma_1)=p_{k-1}(\delta)$. We  see that for the element $\gamma_2\in (\jj U_{n_1}\times{\prod_{2\leq i\leq k}}\jj SL_{n_i})(\kkk)$ defined by
 $$ p_i(\gamma_2)=\left\{\begin{array}{cc}p_i(\delta)&\text{if }i\in\{k-1,k\}\\p_i(\gamma_1)&\text{if }1\leq i<k\end{array}\right., $$we have $\gamma_2 w\F\in\hat{\AAA}(N)$ and then we easily find a $w^\prime\in W_{\prod_{2\leq i\leq k}\jj SL_{n_i}}$ such that $w^\prime\gamma_2 w\F\in\AAA(N)_{\nless}$. \hfill$\square\hspace{1mm}k>2$
 
 \end{tproof}
 \begin{remark}Note that the lemma above can directly be expressed without mentioning $GL_n.$ \end{remark}
\begin{defi}[$\mathrm{I}^{\mathrm{st},k}(\F)$,  $\iota^{\mathrm{st},k}(\F)$]\label{sIk}Let\chh[withoutcommand]  $\F\in\AAnk[n] $ and $k>1$ We define $\sI[,k](\F)$ to be the initial $\AAnk$-subpath of $\sI(\F),$  which has as last $\AAnk$-step the  $k-1$-th $\co$-step among: the $\co$-steps in $\sI(\F),$ which are as in the last sentence of the proof of Theorem \ref{general} (and hence the are obtained by conjugating by an element $w\in W_n$). We also define $\sI[,1](\F)$ to be the trivial $\F$-tree.

For $k\geq 1$ we denote by $\siota[,k](\F)$ the output AF of $\sI[,k](\F)$.  \end{defi}
\begin{remark}For the choices of $\F$ for which $\sI[,k](\F)$ will be considered, an equivalent way to define $\sI[,k](\F)$ is to say: $\sI[,k](\F)$  is the initial $\AAnk$-subpath of $\sI(\F)$ which has as last $\AAnk$-step the $k-1$-th $\co$-step of $\sI(\F)$. 

For $\F\in\tPrink[n]$ and $\prod^{\searrow}_{1\leq i\leq x}\jj GL_{n_i}$ \chh[des to jj]be the Levi of the $GL_n$-parabolic with unipotent radical $D_\F,$ one more equivalent definition of $\sI[,k](\F) $ is to say: it is the biggest among the $(\F,GL_{\sum_{1\leq i\leq k}n_i} )$-subpaths of $\sI(\F)$ that have as last $\AAnk$-step a $\co$-step.
\end{remark}

We use frequently the observation below. Note that a trivial modification\footnote{More precisely, by conjugating $\sI[,k]$ with an element in $W_n$ and composing (usual function composition) in many cases several successive $\eu$-operations, we obtain a $(\eu,\co)$-path found in this reference.} of $\sI[,k](\F)$ for $\F\in\tPrink[n]$ appears in \cite{JiangLiu} (in the proof of Theorem 5.4).

I use the word ``Observation" for the statement below, as a way of ranking (low) the  difficulty of only the thoughts in its proof which are not contained in the proof of Theorem \ref{general}.  

\begin{ob}\label{unipc}
 Let $\F\in\Pri[,n]$ and $\prod_{1\leq i\leq x}^{\searrow}GL_{n_i}$ be the Levi of the $GL_n$-parabolic with unipotent radical $D_\F$. Let $T$ be a torus in $\DDDD[n]$ such that $\min(\Set{T})=1$. It is assumed that $\F$ is nontrivial on every element in $\SRV{T}$, and that it is trivial on all the other root groups that are in the same row as an element in $\SRV{T}$, except possibly if this is the  $\max(\Set{T})$-th row. Let $\K$ be a $\kkk$-AF over an upper triangular unipotent $\kkk$-subgroup of $\left(\prod_{1\leq i\leq x}^{\searrow}GL_{n_i}\right)\cap GL_n^T $ such that $\K\circ\F$ is well defined. Let $k=|\Set{T}|$, and $w_k$ be the minimal length element of $W_n$ such that $\Set{w_kTw_k^{-1}}=\{1,...,k\}$. Finally let $j$ be the upper left corner embedding of $GL_{n_1+...+n_k}$ to $GL_n.$ Then $\sI[,k](\K\circ\F)$ is an $(j(\mirsl{n_1+...+n_k})D_\F,\kkk)$-path such that:
$$\siota[,k](\K\circ\F)=w_k(\K\circ\F^{T}\circ\F|_{R_{\max(\Set{T})}})\circ\JJ_k $$where $R_{\max(\Set{T})}:=\prod_{\max(\Set{T})<j\leq n}U_{(\max(\Set{T}),j)}$.
\end{ob}
\begin{proof} The proof is easily obtained by uncovering the definition of $\sI[,k]$, as discerned in the proof of Theorem \ref{general}, but we proceed in a self contained way by explicitly describing the $\AAnk$-steps of  $\sI[,k]$. First we roughly present a special case by giving the pictures below, and then (without any use of pictures) we present the proof with words as usual. Let $d_1<...<d_k$ be the elements of $\Set{T}$ (of course $d_1=1$).

The colors describe the $\eu$-steps by the rules in Subsection \ref{zpre}. The domain of $\F$ (resp. $\siota[,2](\F)$, $\siota[,3](\F)$) is generated by the root groups within the black thick lines in the first (resp. second, third) picture. The domain of $\K$ (resp. $w_2\K$, $w_3\K$) is within the (diagonal) blocks appearing in the first (resp. second, third) picture that do not contain any number $d_i$ (in order to avoid creating a contrast with how diagonal entries are labeled troughout the paper we do not mention $w_2\K$ and $w_3\K$ in the pictures).   
 
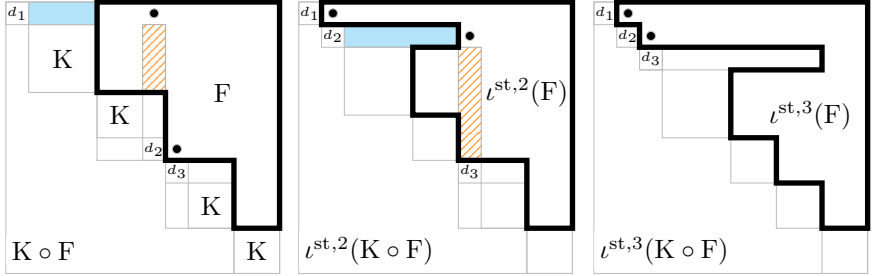
\renewcommand{\di}{12}
\begin{tikzpicture}[scale=0.3,gray!50]
\bu{(6.5,11.5)}
\bu{(7.5,5.5)}
\rl{1}{2}{3}
\rr{2}{7}{3}
\rec{0}{0}{12}{12}
\rec{0}{0}{1}{1}
\rec[$\K$]{1}{1}{4}{4}
\rec[$\K$]{4}{4}{6}{6}
\rec[$\K$]{8}{8}{10}{10}
\rec[$\K$]{10}{10}{12}{12}
\rec{6}{6}{7}{7}
\rec{7}{7}{8}{8}
\rec{4}{4}{7}{7}
\rec{7}{7}{10}{10}
\rec{10}{10}{12}{12}
\draw[black, line width=2pt]\mm{0}{4}--\mm{4}{4}--\mm{4}{7}--\mm{7}{7}--\mm{7}{10}--\mm{10}{10}--\mm{10}{12}--\mm{0}{12}--cycle;
\draw(9.5,8)node[black]{$\F$};
\draw(0.5,11.5)node[black]{{\tiny$d_1$}};
\draw(6.5,5.5)node[black]{{\tiny$d_2$}};
\draw(7.5,4.5)node[black]{{\tiny$d_3$}};
\bfir{\K\circ\F}
\end{tikzpicture}
\begin{tikzpicture}[scale=0.3,gray!50]
\bu{(1.5,11.5)}
\bu{(7.5,10.5)}
\rl{2}{3}{5}
\rr{3}{8}{5}
\rec{0}{0}{12}{12}
\rec[{\tiny$d_1$}]{0}{0}{1}{1}
\rec[{\tiny$d_2$}]{1}{1}{2}{2}
\rec[{\tiny$d_3$}]{7}{7}{8}{8}
\rec[]{2}{2}{5}{5}
\rec[]{5}{5}{7}{7}
\rec[]{8}{8}{10}{10}
\rec[]{10}{10}{12}{12}
\rec{7}{7}{10}{10}
\rec{10}{10}{12}{12}
\draw(10,8)node[black]{$\siota[,2](\F)$};
\draw[black, line width=2pt]\mm{0}{1}--\mm{1}{1}--\mm{1}{7}--\mm{2}{7}--\mm{2}{5}--\mm{5}{5}--\mm{5}{7}--  \mm{7}{7}--\mm{7}{10}--
\mm{10}{10}--\mm{10}{12}--\mm{0}{12}--cycle;
\bfir{\siota[,2](\K\circ\F)}
\end{tikzpicture}
\begin{tikzpicture}[scale=0.3,gray!50]
\bu{(1.5,11.5)}
\bu{(2.5,10.5)}
\rec{0}{0}{12}{12}
\rec[{\tiny$d_1$}]{0}{0}{1}{1}
\rec[{\tiny$d_2$}]{1}{1}{2}{2}
\rec[{\tiny$d_3$}]{2}{2}{3}{3}
\rec{2}{2}{3}{3}
\rec[]{3}{3}{6}{6}
\rec[]{6}{6}{8}{8}
\rec[]{8}{8}{10}{10}
\rec[]{10}{10}{12}{12}
\draw(9.5,7)node[black]{$\siota[,3](\F)$};
\draw[black, line width=2pt]\mm{0}{1}--\mm{1}{1}--\mm{1}{2}--\mm{2}{2}--\mm{2}{10}--\mm{3}{10}--\mm{3}{6}--\mm{6}{6}--\mm{6}{8}--\mm{8}{8}--\mm{8}{10}--\mm{10}{10}--\mm{10}{12}--\mm{0}{12}--cycle;
\bfir{\siota[,3](\K\circ\F)}
\end{tikzpicture}

The proof in words follows. We proceed inductively on $k$. For $k=2,$ the path $\sI[,2](\K\circ\F)$ starts by applying $\exchange{U_{(1,i)}}{U_{(i,d_2)}}$ for $i=n_1,n_1-1,...,2$ (in this order), and ends by conjugating with $w_2$. We therefore assume that $k>2$.

 We have\begin{equation}\label{zKF}\siota[,2](\K\circ\F)=\K^\prime\circ\F^\prime\circ\JJ_{2}, \end{equation}where
 $$\F^\prime:=w_2\F|_{\prod_{1<i<j>n_1+n_2} U_{(i,j)}}\qquad\text{ and }\qquad\K^\prime=w_2(\K\circ\F^{\{1,d_2\}\cup\{i:n_1+n_2<i\leq n\}} ).$$Let $j^\prime$ be the lower right corner embedding of $GL_{n-1}$ to $GL_n$. From (\ref{zKF}) and the inductive hypothesis for $[\F\tlarrow{j^\prime}^{-1}(\F^\prime),\K\tlarrow{j^\prime}^{-1}(\K^\prime),k\tlarrow k-1]$ we obtain the observation for the current value of $k$.
\end{proof}
\begin{remark}[Proceeding one column (instead of one root) at a time] By the observation above and Fact \ref{autsteps} we obtain the modification of equation (\ref{stel}) for   $[\F\tlarrow\K\circ\F]$ where ``$\exchange{X}{Y}$" is replaced by ``$\siota[,2](\K\circ\F) $". To better explain how this equation relates to (ii) in Remark \ref{sXilit}, we obtain it by ``simultaneously exchanging" all the roots of the $d_2$-th column in the sense that follows.

 The domain of $\exchange{\prod_{2\leq i\leq n_1}U_{(1,i)}}{\prod_{2\leq i\leq n_1}{U_{(i,d_2)}}}$ does not contain $\K\circ\F$ for many choices of $\K$, but it does contain $\F$ and it maps it to $w_2^{-1}\siota[,2] (\F)$. Hence by Fact \ref{autsteps} we obtain equation (\ref{stel}) for $[X\tlarrow \prod_{2\leq i\leq n_1}U_{(1,i)},Y\tlarrow\prod_{2\leq i\leq n_1}{U_{(i,d_2)}}]$ (and the current choice of $\F$). Since $\K\circ\F$  is defined we can apply to this equation the Fourier coefficient corresponding to $\K$, and by also using that $\K\circ\sI[,2](\F) $ is defined and $D_\K$ normalizes $\prod_{2\leq i\leq n_1}{U_{(i,d_2)}}$, and also by conjugating with $w_2$, we obtain the equation in the previous paragraph.\end{remark} 

\begin{lem}\label{radexpress}Let $\F\in\Pri[,n]$.  There is an $(\F,\mirsl{n},\kkk)$-tree $\Xi$ with the next two properties:\begin{enumerate}
\item There is an output vertex of $\Xi$, labeled with an AF $\Z$ in $\AAA(U_n),$ such that $J_\Z$ and $J_\F$ are $\mirsl{n}(\kkk)$-conjugate.
\item Every other output vertex of  $\Xi$ is  labeled with an AF in $\cup_{b>\Om(J_\F)}\Pri[,n][b]$.
\end{enumerate}

\end{lem}
\begin{tproof}By Lemma \ref{conjugatehat} we are reduced to the case $\F\in\tPri[,n]. $ Let $R_1:=\prod_{1<j\leq n}U_{(1,j)}$. Assume  that $R_1\subseteq D_\F.$ Then there is a $w\in\mirsl{n}(\kkk)$ which normalizes $D_\F$, and  $\mathrm{Stab}_{\mirsl{n}}(w\F|_{R_1})$ contains the lower right corner copy of $\mirsl{n-1}$ inside $GL_n.$ Hence it is enough to prove the lemma for $[\F\tlarrow\jm\F^{\{1\}}]$. Hence (by induction on $n$) we are further reduced to assuming that $R_1\not\subseteq D_\F.$ We also do a downward induction on $\Di{D_\F}.$ To sum up this paragraph and be more precise, we are left with: proving the lemma for  a choice of $\F\in \tPri$ satisfying $R_1\not\subseteq D_\F$, assuming the lemma is correct for $[\F\tlarrow\F^\prime]$ for any $\F^\prime\in \Pri[,n]$ satisfying  $\Di{D_{\F^\prime}}>\Di{D_\F}.$

Let $S$ be the set in $\Set{\TT{n}{\F}}$ with $1\in\Set{T}.$  If $|S|=1,$ by applying to $\F$ the $\e$-operation over the product of the root groups in the first row and not contained in $D_\F$, we are done by the inductive hypothesis. We are left with the case $k>1.$

To construct $\Xi$, we start with $\sI(\F),$ the output AF $\siota(\F)$ of which, is given in Observation \ref{unipc}.

 We continue the construction of $\Xi$ with the $(\siota(\F),\e)$-quasipath $\xi$  along constant terms over the root generated subgroups $N_1,...,N_{k-1}$ of $\prod_{k<j<n}U_{(k,j)} $ which are inductively defined by requiring that: for $1\leq r\leq k-1,$ the centralizer\footnote{Recall the meaning of the centralizer of a group $H$ in a group $G$ is the usual one except possibly that $H\subseteq G$ is not necessary.} of  $\left(\prod_{1\leq y\leq r-1}N_y\right)D_{\siota(\F)}$ in $\prod_{k<j<n}U_{(k,j)}$ is equal to $\prod_{1\leq y\leq r}N_r$. Note that $\prod_{1\leq y\leq 0}N_y=1$.  Also recall that the symbol $\prod$ denotes the direct product (and hence trivial intersections of the factors are implied). Of course  each $N_r$ easily admits an explicit description as a product $\prod_l U_{(k,l)}$ for $l$ taking the several biggest values not already taken in the same product for a smaller choice of $r$. Finally notice that $\left(\prod_{1\leq y\leq k-1}N_y\right)D_{\siota(\F)}$ is a unipotent radical of a $GL_{n}$-parabolic.

Let $\F^\ih_0$ be the output constant term of $\xi$. We see that $\F^\ih_{0}\in \Pri[,n][\Om(J_\F)]$. Since $\Di{D_{\F^\ih_{0}}}>\Di{D_\F}$, the  inductive hypothesis contains the lemma for $[\F\tlarrow\F^\ih_{0}]$. Hence we obtain an $(\F^\ih_{0},\mirsl{n},\kkk)$-tree $\Xi(\F^\ih_{0} )$ for which:
\begin{enumerate}
\item There is an output vertex of $\Xi(\F^\ih_0)$, labeled with an AF $\Z$ in $\AAA(U_n),$ such that $J_\Z$ and $J_{\F^\ih_0}$ are $\mirsl{n}(\kkk)$-conjugate. Hence $J_\Z$ and $J_\F$ are $\mirsl{n}(\kkk)$-conjugate.
\item Every other output vertex of  $\Xi(\F^\ih_0)$ is labeled with an AF in $\cup_{b>\Om(J_{\F^\ih_0})}\Pri[,n][b]$ (which is equal to $\cup_{b>\Om(J_{\F})}\Pri[,n][b]$).
\end{enumerate}

We are left with showing that for any  other output AF $\F^\ih_1$ of $\xi$, there is an  $$(\F^\ih_1\rightarrow\cup_{b>\Om(J_{\F})},\Pri[,n][b],\mirsl{n},\kkk)\text{-tree},$$say $\Xi(\F^\ih_1),$  because then the proposition is obtained for $\Xi:=\Xi(\F^\ih)\vee_{\F^\ih}\xi\vee\sI(\F)$, where $\F^\ih$ varies over all output AFs of $\xi$. 

By using  Observation \ref{unipc}, we obtain an AF $\XX\in\Pri[,n]$, such that $\siota[,k](\XX)=\F^\ih_1.$ In $\XX$ the inductive hypothesis is applied, and hence we are left with proving the following claim.

\vsp 
\noindent\textbf{Claim}\textit{ $\Om(J_{\XX})>\Om(J_{\F})$.} 

\vsp
\noindent\textit{Proof of Claim. }Let $d_k:=\max(\Set{T})$ for $T$ being as in Observation \ref{unipc} for $[\F\tlarrow\XX,\K\tlarrow\F_{\emptyset,n}] $ (k there is chosen as it is here). Therefore the only row in which there are possibly more than one root groups on which $\XX$  doesn't vanish is the $d_k$-th row. Let 
$$A:=\{l|U_{(d_k,l)}\in D_\XX\text{, and } \XX(U_{(d_k,l)})\neq \{0\}. $$ Let $\QQ$ be the restriction of $\XX$ at the columns that are to the right  of $A$ (that is the $x$-th columns for $x>\max(A)$). Among the elements in $\{S\in\Set{\TT{n}{\QQ}}:S\cap A\neq \emptyset\}$ choose one, say $S_A$,  with the biggest number of elements,  and let $l_A$ be the element in the intersection $S_A\cap A.$ Let $M$ be the Levi of the $GL_n$-parabolic subgroup with unipotent radical $D_\XX$.

By using Observation \ref{diagonal} we obtain a unipotent element $u\in\Stab{M}{\QQ}$  such that: \begin{enumerate}
\item The root group $U_{(d_k,l_A)}$ is the unique one in the $d_k$-th row on which $u\XX$ is nontrivial. 
\item Other than the $d_k$-th row there is at most one more row on which $\XX$ and $u\XX$ differ. This row exists if and only if  there is a number $d\neq d_k$ such that $\XX(U_{(d,l_A)})\neq\{0\},$ and in this case it is the $d$-th row. 
\end{enumerate}
In case the number $d$ was not defined; it means $|S_A|$ is a number in the partition of $\Om(J_\F)$; and that $\Om(J_{u\XX})$ is obtained (as a partition) from $\Om(J_\F)$ by removing one occurrence of each of the numbers  $|S_A|,k$ and adding  one occurrence of  $k+|S_A|.$ Hence $\Om(J_{u\XX})>\Om(J_\F)$, and since $u$ normalizes $D_\F^t$, we have $J_{u\XX}=uJ_\XX u^{-1}$, and the claim is obtained\footnote{Even if we don't use the equality $J_{u\XX}=uJ_\XX u^{-1}$,  the  weaker claim obtained still implies the lemma directly. The same is true for the equality $J_{wvu\XX}=wvuJ_\XX(wvu)^{-1}$ which appears later in the proof}.

We are left with the case in which $d$ was defined. The reader will need here to check Definition \ref{boreln} and Lemma \ref{nless} in the appendix (and hence recall the definition of $\X{...}$ in Definition \ref{vardef}). By using again Observation \ref{diagonal} we obtain a unipotent element $v$ fixing $D_{u\XX}$ such that $vu\XX$ is trivial on $U_{(d,l_A)}$ and it doesn't differ from $u\XX$ in any other root group. Let $d^\prime$ be the biggest number such that the $k$-th block of $M$ is nontrivial on the $d^\prime$-row (the same block is nontrivial on the $d,d_k$-rows). Let $N$ be the group generated by all the root groups of $D_\XX$, except for the ones lying: in the $k$-th block of $M$ and in the $d^\prime$-th row. Let $w$ be an element in $W_M$ which permutes the rows $d$ and $d^\prime$ and also satisfies  $(wvu\XX)|_{N}\in \AnkB{n}_{\nless}$.  Since $wuv$ normalizes $D_\XX^t$, we obtain $J_{wvu\XX}={wvu}J_\XX(wvu)^{-1}$, which in turn implies the first equality in (\ref{zeqineq}) below. The other parts of (\ref{zeqineq}) are obtained from Lemma \ref{nless} and $J_{wvu\XX}\in \X{wvu\XX|_{N}}$.
\begin{equation}\label{zeqineq}\Om(J_{\XX})=\Om(J_{wvu\XX})\geq\text{ the unique orbit in }\Om((wvu\XX)|_{N})=\Om(J_{(wvu\XX)|_{N}}). \end{equation} It is therefore enough to prove that the last among these orbits is bigger from $\Om(J_\F).$ We directly find a number $m$ in the partition $\Om(J_\F)$ which is smaller from $k+|S_A| $, such that $\Om(J_{(wvu\XX)|_{N}})$  is obtained from $\Om(J_\F)$ by replacing $m$ and $k$ with $m-|S_A|$ and $k+|S_A|$ respectively.\hfill$\square $Claim$\qedhere$\end{tproof} 
 
\section{Applying AFs in $\Prink[n]$ to automorphic forms}\label{forms}In the current section---which is not needed for any of the three main results of the paper----we state Theorem \ref{folk}, and  reduce it to its special case (Theorem \ref{Eukerianf}) in which $\F\in\AAA(U_n),$ where $n$ (resp. $\kkk$) is a positive integer (resp. number field) that we fix throughout the section. Subsection \ref{Eisexp} in the appendix is a continuation of the current section, namely by proving there Theorem \ref{Eukerianf}.  Theorem \ref{folk} (including the proof I give) is an easy refinement of known results (see Remark \ref{local}).

To state Theorem $\ref{folk}$ we need to define $\Om^\prime(\pi).$ Among the equivalent ways of defining it we use one that is very directly related to the proof.

\begin{defi}[$\Om^\prime(\pi)$]\label{Ppi}
Let $\pi\in\Aut{\kkk,n},$ and $\prod_{1\leq i\leq m}GL_{n_i}$ be the Levi up to isomorphism in the discrete inducing  data of $\pi$.  Let $a_i$ be the divisor of $n_i$, for which the cuspidal representation defining the $GL_{n_i}(\A)$-Speh representation in the inducing data, is defined over $GL_{a_i}.$ Let $b_i:=\frac{n_i}{a_i}.$ Finally let $\PPPP$ be a $GL_n$-parabolic subgroup with Levi $\prod_{1\leq i\leq m} GL_{b_i}^{a_i}.$ We define $\Om^\prime(\pi)$ to be the Richardson  orbit openly intersecting $\Lie{U_{\PPPP}}$. This definition is the same for all previous choices of $\PPPP$ due to Fact \ref{ptelos}.
\end{defi}\begin{remark} We do not use the concept of inducing orbits, but the reader acquainted with it, can discern it both in the definition above and in the proof of Theorem \ref{Eukerianf}.  A concept frequently considered in relation to the topic of the current section is $\Om(\pi). $ We neither review it nor use it except: in Corollary \ref{attachingfc} where we see how the proof of  $\Om(\pi)=\{\Om^\prime(\pi)\}$ for every $\pi\in\Aut{\kkk,n,>},$ is completed; and towards the end of Subsection \ref{zpre3} where a dimension equation of D. Ginzburg is discussed.\end{remark}

\begin{theorem}\label{folk} Let $\F\in\Pri[,n],$   and $\pi\in\Aut{\kkk,n,>}$ . 
\begin{description}
\item[Part A]Assume that $\Om(J_\F)=\Om^\prime(\pi)$. Then $\F(\pi)$ is  factorizable and nonzero.
\item[Part B] Assume the orbit $\Om(J_\F)$ is bigger or unrelated to $\Om^\prime(\pi)$. Then $\F(\pi)=0$.
\end{description}
\end{theorem} 
\begin{tproof}(mostly deferred for the appendix). By using Lemma \ref{radexpress} and Fact \ref{nonzeroeulerian}, the theorem is directly reduced to the special case in which $\Pri[,n] $ is replaced with $\AAA(U_n).$ This case is proved in the appendix (Theorem \ref{Eukerianf}).
\end{tproof}
\begin{remark}\label{local}Specifically on $\AAnk$-trees one discerns in the proof of Theorem \ref{folk}, since they are closely related to $\sXi,$ their relations to the literature are mostly within the topic of Remark \ref{sXilit}.

Among results in which either $D_\F\neq U_n$ or $\F$ is trivial on at least one root group contained in $D_\F$, the first ones analogous to the Theorem above that I am aware of appear in the local setting in the works \cite{Zel} (Section 8) and \cite{MW1}. Since then many related works have appeared both in the local and in the global setting and for groups with root systems other than $A_n$ (works addressing both the global setting and the group $GL_n$ include: \cite{Ginzbconj}, \cite{Ginzeul}, \cite{JiangLiu}, \cite{Cai}).
\end{remark}
\begin{remark}\label{anco}By the classification of automorphic forms by R.P.Langlands and the classification of the discrete spectrum (\cite{MW2}) (and uniqueness of analytic continuation), we directly obtain the extension of Part B in which  $\Aut{\kkk,n,>}$ is replaced with $\Aut{\kkk,n}$. Every other vanishing of an integral of automorphic forms in the present paper can be similarly extended (either by repeating this argument or using the extended Part B). 
  \end{remark}
\begin{cor}\label{easyrefcor}

\vspace{1mm}
\noindent\textit{Part 1.} Let $\F\in\AAA[,n] $, and  $\Xi$ be an $(\F\rightarrow\Pri[,n],GL_n,\kkk)$-tree (known to exist by Theorem \ref{general}). We define $\Om(\F)$ to be the set consisting of the minimal orbits of 
$$\{a\in\UUU{n}|\text{An output AF in }\Xi\text{  belongs to }\Pri[,n][a]\}. $$ The set $\Om(\F)$ is independent of the choice  of $\Xi$.

\vspace{1mm}
\noindent\textit{Part 2.} Let $\pi\in\Aut{\kkk,n,>}. $ 
\begin{enumerate}
\item[(i)] Assume that there is an $(\F,GL_n,\kkk)$-tree, one vertex of which is labeled with an AF in $\Pri[,n][\Om^\prime(\pi)].$ Then   $\F(\pi)$ is  nonzero.
\item[(ii)] Assume that $\Om^\prime(\pi)\in\Om(\F)$, and fix an $(\F\rightarrow\Pri[,n] ,GL_n,\kkk)$-tree. Then this tree has at most one output vertex with  label in $\Pri[,n][\Om^\prime(\pi)] $ if and only  if $\F(\pi)$ is factorizable.
\item[(iii)] Assume that for every $a\in\Om(\F)$,  the orbit $\Om^\prime(\pi)$ is smaller or unrelated to $a$. Then $\F(\pi)=0$.\chh[compare with your pr]
\end{enumerate}
  
\end{cor}
\begin{proof}We denote by $\Om(\F,\Xi)$ the definition we gave for $\Om(\F).$

From Part A (resp. both parts, Part B) of Theorem \ref{folk}  and Fact \ref{nonzeroeulerian} we obtain Part 2.(i) (resp. Part 2.(ii) for $[\Om(\F)\tlarrow\Om(\F,\Xi) ]$ where  $\Xi$ is the fixed tree in the statement of (ii), Part 2.(iii) for $[\Om(\F)\tlarrow\Om(\F,\Xi) ]$ where $\Xi$ is any $(\F\rightarrow\Pri[,n],GL_n,\kkk)$-tree). 

We are left with Part 1, that is, we prove for any two $(\F\rightarrow\Pri[,n],GL_n,\kkk )$-trees $\Xi_1,\Xi_2,$ we have $\Om(\F,\Xi_1)=\Om(\F,\Xi_2).$

Let $a$ be an orbit in $\Om(\F,\Xi_1)$, and $\rho$ be a representation in $\Aut{\kkk,n,>}$ such that $\Om^\prime(\rho)=a.$ Part 2.(i) implies that $\F(\rho)\neq 0.$ Therefore  $a$ must be bigger or equal to an orbit in $\Om(\F,\Xi_2),$ because otherwise Part 2.(iii) for $[\Om(\F)\tlarrow\Om(\F,\Xi_2)]$ would imply $\F(\rho)=0.$  By letting $a$ vary in $\Om(\F,\Xi_1)$, and then interchanging the roles of $\Xi_1$ and $\Xi_2$ we are done.\end{proof}
\begin{remark}[Weak and Strong questions \ref{zStrongc} and \ref{zStrongcc}] Let again $\F\in\AAA[,n] $  and  $\Xi$ be an $(\F,GL_n,\kkk)$-tree in Part 1 of the corollary above, and also let $a\in\Om(\F)$. If there are finitely many output  vertices of $\Xi$ with label in $\Prink[n][a]$, we denote by $\mult{a}{\F}$ the number of such vertices.  From the corollary above we obtain that: $\mult{a}{\F}= 1$ for one choice of $\Xi$ if and only if $\mult{a}{\F}= 1$ for every choice of $\Xi$; for each among Weak questions \ref{zStrongc}.$(\F)$ and \ref{zStrongcc}.$(\F)$, the two versions are equivalent; the $\pi$ positively answering (the second version of) Weak question \ref{zStrongc}.$(\F)$ (resp. \ref{zStrongcc}.$(\F)$) are also positively answering Strong question \ref{zStrongc}.$(\F)$ (resp. \ref{zStrongcc}.$(\F)$).\end{remark}

 In the next section, the concepts $\Om(\F)$ and $\mult{a}{\F}$ are studied more and independently from the current section (and automorphic forms).
\section{Attaching a variety $\X{\F}$ to an AF $\F$}\label{variety}In the current section we prove the second main result of the paper which is Main corollary \ref{maincor}. The reason for its name is because we can view it as a corollary of Theorem \ref{general} and the study of the variety $\X{\F}$.  A reading of Sections \ref{optionalsection} and \ref{compositions} witch involves encountering statements without knowing their proofs (but no other gaps) is possible by remembering from the current section the following: Definitions \ref{defBBB}, \ref{brow}, \ref{sbsquares}; and everything starting with \ref{AOF}, except that from Proposition \ref{dimofoccurence} one can only remember that Definition \ref{AOF} is independent of the choice of $\Xi.$ In particular, the variety in the title of the section, although central in the proof of Main corollary \ref{maincor} and in the independence of Definition \ref{AOF} on the choice of $\Xi$, it appears neither on the statement of Main corollary \ref{maincor} nor in this Definition. 

Throughout the current section $n$ is a positive integer. Recall that for every $\F\in\AAnk[n]$ we have defined
$$\X{\F}:=\{J\in \Lie{GL_n}: \F(\exp{x})=tr(xJ)\text{ for every }x\in \Lie{D_\F}\}. $$\begin{defi}[$\cong $]If two varieties $X$ and $Y$ are isomorphic we write $X\cong Y.$\end{defi} 
Recall that throughout the paper all varieties are defined over an algebraically closed field $\KKK$ of characteristic 0.
\begin{fact}[Known\footnote{A weaker version of Fact \ref{zariski} which appears in more texts, is by further assuming for $f$ to be birational. Fact \ref{zariski} is directly obtained from this weaker version and by the following known statement. \textit{Let $\phi:X\rightarrow Y$ be a dominant morphism, between two irreducible varieties $X$ and $Y$. Assume that there is an $x_0\in X$ with $\phi^{-1}(\phi(x_0))$ finite. Then $\Di{X}=\Di{Y}$, and 
 there is an open subset $U$ of $X$ such that $\phi^{-1}(\phi(x))$ is a set with $|K(X):K(Y)|$ elements for every $x\in U$.}}, corollary of Zariski's main theorem]\label{zariski}Let $f:X\rightarrow Y$ be a bijective morphism, between two irreducible varieties $X,Y$. For $Y,$ we further assume that it is normal. Then $f$ is an isomorphism.

\end{fact}\begin{remark}
We always use Fact \ref{zariski} through the Lemma \ref{dominant} below (which is obtained directly from this fact). In the proof of Lemma \ref{properties} below, we use Fact \ref{zariski} only in  a special case which is part of Lemma \ref{eustepslem}. If we were removing the other uses of Fact \ref{zariski}, we would obtain a weaker more technical version of Main corollary \ref{maincor}, which however is still sufficient for the third main result of the paper; more generally it is sufficient for everything in the next sections except possibly for Corollary \ref{laststrong2}. \end{remark}
\begin{lem}[Known]\label{dominant}Let $X,M$ be two varieties, each one admitting an algebraic  action by an algebraic group $H.$ Let $f:X\rightarrow M$ be an $H$-equivariant  algebraic morphism, and $y$ be a point in $M$. For the action  of  $H$ in $M$ assume that it is transitive and free. Then $X$ is isomorphic to $f^{-1}(y)\times H$. 
\end{lem}
\begin{proof} The morphism from $H$ to $M$ given\chh[normal] by $h\rightarrow hy, $ is an isomorphism due to Fact \ref{zariski}.  Hence we  assume that $M=H$, and that $y$ is the identity element. Consider the morphism $H\times f^{-1}(y)\rightarrow X$, given by $(h,x)\rightarrow hx. $ We see that it admits the inverse $$x\rightarrow (f(x),f(x)^{-1}x), $$ and thus it is an isomorphism.
\end{proof}\noindent{In} all the uses of this lemma, $f$ will be equal to a function as in the next definition.
\begin{defi}\label{zarlem} Let $\sett{X}$ be a subset of $\Lie{GL_n}$, and $N$ be a unipotent algebraic subgroup of $GL_n$. We define by $ \itr{\sett{X}}{N}$, to be the function with domain $\sett{X}$ and values the algebraic morphisms on $N$ given by
$$\itr{\sett{X}}{N}(J)(u)=\mathrm{exp}(\text{tr}(\mathrm{log}(u)J))\quad\text{ for all }J\in\sett{X}\text{ and }u\in N. $$We choose the codomain of $ \itr{\sett{X}}{N}$ so that it is surjective. We  use this definition only in cases that the values of $ \itr{\sett{X}}{N}$ are in $\AAnk(N). $\end{defi}
We start the study of $\X{\F},$ by addressing the effect of $\AAnk$-operations, in the lemma below.
\begin{lem}\label{properties}Let $\F\in\AAnk[n]$.

\begin{enumerate}
\item Let $\kkk$ be any  subfield of $\KKK$. For \chh[Ffurther] an $(\F,\e,GL_n,\kkk)$-step, we have   $$\X[\kkk]{\F}={\bigcup}_\Z\X[\kkk]{\Z}, $$where the union is over the output AFs $\Z$ of this $\e$-step. Also for any such $\Z,$ the variety $\X{\Z}$ is a closed subvariety of $\X{\F}. $  
\item\label{preu} Let $\Z$ be the output AF of an $(\F,\eu,GL_n)$-step and $a\in\UUU{n}$. There is an isomorphism of $\X{\Z}$ onto $\X{\F}$ which restricts to an isomorphism of $\X{\Z}\cap a$ onto $\X{\F}\cap a$. 
\item\label{prconj} Let $\gamma\in GL_n. $ Then $\X{\gamma\F}=\gamma\X{\F}\gamma^{-1}. $

\end{enumerate} \end{lem}
\begin{tproof}Part 3 is trivial. We continue with Part 1. For $\Z$ being any output AF of the $\e$-step, we have $\F([D_\Z,D_\Z])=\{0\}.$ This together with $\Lie{[D_\Z,D_\Z]}=[\Lie{D_Z},\Lie{D_\Z}],$ imply that every  extension of the composition of $\F$ with $\exp$ from $\Lie{D_\F}$ to a $\kkk$-linear function on $\Lie{D_\Z}$ is also a Lie algebra homomorphism, which in turn (by Fact \ref{Liehom}) is lifted to an AF with domain $D_\Z$, and hence we obtain Part 1. 

We are left with Part 2.  Consider any $\eu$-step and let $\QQ,\QQ^\prime$ respectively be its input and output AF. We define the dimension of this $\eu$-step to be the number $\Di{D_{\QQ}\cap D_{\QQ^\prime}\s  D_{\QQ}}$.

\vsp 
\noindent\textbf{Claim.}\textit{ Let $\I$ be a $(\eu,GL_n)$-step. Then there is a $(\eu, GL_n)$-path  which consists of one dimensional $\eu$-steps and has the same input and output AFs as $\I$.}

\vsp
\noindent\textit{Proof of Claim.} We proceed inductively on the dimension of $\I$, and hence assume it is bigger from one. We freely use Lemma \ref{eustepslem} (that $(i)\iff(iii)$) and Fact \ref{areconnected}. We denote the input AF and output AF of $\I$ with $\QQ$ and $\QQ^\prime$ respectively.

Let $\tilde{Y}_1$ be a codimension-1 algebraic subgroup of $D_{\QQ}$ which contains $D_\QQ\cap D_{\QQ^\prime}.$ We obtain that the set $$\tilde{X^1}:=\{x\in D_\Z:\F([x,y])=0\text{ }\forall y\in\tilde{Y}_1 \} $$ is a subgroup of $D_\Z$ containing $D_\QQ\cap D_{\QQ^\prime}$ and having dimension $\Di{D_\QQ\cap D_{\QQ^\prime}}+1.$ Let  $\QQ^{\prime\prime}\in\AAnk(\tilde{Y}_1\tilde{X}^1)$  be defined by $\QQ^{\prime\prime}(xy)=\Z(x)+\F(y)$ for all $x\in \tilde{X}^1,y\in\tilde{Y}_1$.

We see that: the $(\QQ\rightarrow\QQ^{\prime\prime},\eu,GL_n)$-step and the $(\QQ^{\prime\prime}\rightarrow\QQ^{\prime},\eu,GL_n)$-step are defined, the first has dimension one, and the second has dimension  one less from the dimension of $\I$. Hence we are done by the inductive hypothesis.\hfill$\square$Claim 

By this claim  we are reduced to the case in which the $\eu$-step is one dimensional, and hence (by Corollary \ref{semprod})  is obtained from an operation $\exchange{X}{Y}$ such that for $C$ being as in Definition \ref{euoperation} we have \begin{equation}\label{zzoned}X\cap C=1\qquad\text{ and }\qquad Y\cap C=1.\end{equation} Other than obtaining (\ref{zzoned}), we do not use that the $\eu$-step is of dimension one. By conjugating $\F$ with an appropriate element, and using the third part of the lemma, we are further reduced to assuming that $\F(Y)=\{0\}$, which in turn implies:
\begin{equation}\label{zinversegood}\eu(Y,X)(\eu(X,Y)(\F))=\F
\end{equation}(recall $\exchange{Y}{X}$ is defined due to Lemma \ref{eustepslem}). By (\ref{zzoned}) we obtain that the action by conjugation of $Y$ on \begin{equation}\label{zsetzz}
\{\Z\in\AAnk(XC):\Z|_{C}=\F|_{C}\}
\end{equation}is transitive and free. Hence (by also using Part 1) we can apply Lemma \ref{dominant} for the following choice of data:  $f=\itr{\X{\F}}{ XC}$, $M$ (therefore) is the set in (\ref{zsetzz}),  $H=Y$, and $y=\exchange{X}{Y}\F$.  We obtain:
\begin{equation}\label{zczc5}\X{\F}\overset{\text{Lemma \ref{dominant}}}{\cong} Y\times f^{-1}(y)\cong Y\times (\X{\F}\cap\X{\eu(X,Y)\F}). \end{equation}
 We apply the same reasoning with: $\F$ replaced by $\eu(X,Y)(\F)$, and $\eu(X,Y)$ replaced by $\eu(Y,X)$. Then (by also using (\ref{zinversegood})),   in the place of formula (\ref{zczc5}) (with the middle term removed)  we obtain an isomorphism:
\begin{equation}\label{zczc6} \X{\eu(X,Y)\F}\cong X\times (\X{\eu(X,Y)\F}\cap\X{\F}). \end{equation}
Hence (\ref{zczc5}) and (\ref{zczc6}) gives an isomorphism:
$$\X{\F}\cong\X{\eu(X,Y)\F}. $$By checking the construction of this isomorphism we see that it preserves the orbits of $\UUU{n}$; and hence:
\begin{equation*}\X{\F}\cap a\cong\X{\eu(X,Y)\F}\cap a. \qedhere\end{equation*}\end{tproof}

\begin{remark}
An alternative argument for Part 2 of the lemma above is by observing that 
\begin{equation}\label{zirdim}(\X{\F}\cap a)\times \KKK\cong (\X{\eu(X,Y){\F}}\cap a)\times\KKK,  \end{equation}

which follows because (by similarly using Lemma \ref{dominant}) both varieties are isomorphic to $\X{\F|_C}\cap a$ (again one can remove the intersections with $a$). Formula (\ref{zirdim}) alone does not imply that $\X{\F}\cap a$ and $\X{\exchange{X}{Y}\F}\cap a$ are isomorphic, but it does imply that: (i) they have the same dimension, and (ii) one is irreducible if and only if the other is. These two conditions are sufficient for establishing Main corollary \ref{maincor} the way we do. Also notice that replacing  $\eu$-operations with more general operations is sufficient for (\ref{zirdim}). 

It may worth observing the relations of the proof of the lemma above and the alternative argument in the first sentence of the current remark with the proof of Part 2 in Fact \ref{autsteps}.
\end{remark}

\begin{defi}[$\mathcal{B}_n$, $\mathcal{B}_n$-paths] \label{defBBB}
\chh[epikefalida] Consider an  $(\F_{\emptyset,n},GL_n)$-path called $\Xi$. For every  $\e$-step $\xi$ of $\quasi{\Xi}$ we assume that:\begin{enumerate}
\item An algebraic subgroup of $GL_n$ acts freely and transitively by conjugation on an open subset of the variety consisting of the terms of $\xi$;
\item the vertex of $\Xi$ which is also an output vertex of  $\xi$, is labeled with an AF in this open subset (here we do not identify any different vertices). \end{enumerate}
 We denote by $\BBnk[n]$ the set of AFs which can be realized as the output AFs of such  paths. 

Any path which is as $\Xi$ except that the input AF (instead of being $\F_{\emptyset,n} $) can be any AF of $\BBnk[n]$, say $\F$, is called an $(\F,\BBnk[n])$-path, and if $\F$ needs no mention (resp. $\Z$ is its output AF) is called a $\BBnk[n]$-path (resp. $(\F\rightarrow\Z,\BBnk[n])$-path). \end{defi}

The next proposition is behind\footnote{Frequently in an indirect way. Namely, once Main corollary \ref{maincor} is established, Proposition \ref{basic} is not mentioned again} every use we make of information of the form: an AF F belongs to $\BBnk[n]$.

\begin{prop}\label{basic}Let $\F\in\BBnk[n]$ and $a\in\UUU{n}$. If $\X{\F}\cap a$ is nonempty, it is an irreducible variety of dimension $\Di{a}-\Di{D_\F}$.
\end{prop}
\begin{tproof}By proceeding inductively on $\Di{D_\F}$, the proposition is  obtained directly by: the claim below, and (to start the induction) that $a$ is an  irreducible variety.

\vspace{1mm}
\noindent\textit{\textbf{Claim}: Consider a $\BBnk[n]$-path $\xi$ of depth one (therefore $\quasi{\xi}$ is an $\AAnk$-step) and let $\Z_1$ and $\Z_2$ respectively be its input and output AFs. Assume that $\X{\Z_1}\cap a$ is an irreducible variety, and that $\X{\Z_2}\cap a\neq \emptyset$. Then $\X{\Z_2}\cap a$ is also an irreducible variety, and 
\begin{equation}\label{zbzb3}\Di{\X{\Z_2}\cap a}=\Di{D_{\Z_1}}-\Di{D_{\Z_2}}+\Di{\X{\Z_1}\cap a}. \end{equation}}\textit{Proof of claim.} In case $\xi$ is a $\eu$-step or a $\co$-step the proof directly follows from Lemma \ref{properties}, and hence we assume that $\quasi{\xi}$ is an $\e$-step. Let $H$ be an algebraic subgroup of $GL_n$ acting freely and transitively on  an open subset $\SSS$ of the variety of output AFs of $\quasi{\xi}$. Let $f:=\itr{\X{\Z_1}}{D_{\Z_2}}$.  By Part 1 in Lemma \ref{properties} we see that the image of $f$ consists of the output AFs of $\xi.$ We have $$\cup_{h\in H} \X{h\Z_2}\cap a=f^{-1}(\SSS)\cap a,$$ and hence
\begin{equation}\label{zopenset}\cup_{h\in H} \X{h\Z_2}\cap a \text{ is a nonempty open subset of }\X{\Z_1}\cap a. \end{equation}
Next (by also using Part 3 in Lemma \ref{properties}) we apply Lemma \ref{dominant} to the restriction of $f$ at $\cup_{h\in H} \X{h\Z_2}\cap a $.
 We obtain 
\begin{equation}\label{zbzb1}\cup_{h\in H} \X{h\Z_2}\cap a\cong H\times (\X{\Z_2}\cap a). \end{equation}
From the irreducibility of $\X{\Z_1}\cap a$, (\ref{zopenset}), and (\ref{zbzb1}), we obtain both the irreducibility of $\X{\Z_2}\cap a$  and the dimension equation (\ref{zbzb3}).\hfill\hspace{160pt}$\square$Claim\end{tproof}

\begin{defi}[$\BR{n}$]\label{brow}  Let $\K\in\hat{\AAnk}[T_n]$. Let $i_1<....<i_t$ be the rows on which $D_{\K} $ is nontrivial. Let $\K_0:=\F_{\emptyset,n}$, and $x$ vary over $1,...,t$. Let $\K_x$  be the restriction of $\K$ on the first $i_x$-th rows, and  $\xi_x$ be the $(\K_{x-1}\rightarrow \K_x)$-path with: $\quasi{\xi_x}$ being an $\e$-step. Consider the $\AAnk$-path\begin{equation}\label{zBrs}\xi_t\vee....\vee\xi_2\vee\xi_1.\end{equation} We assume there is a number $f(x)$ satisfying $\K(U_{(i_x,f(x))})\neq \{0\}$ (and hence it is the unique such number), and let $T\{f(x)\}$ be the torus in $\DDDD[n]$ with $\Set{T\{f(x)\}}=\{f(x)\}$. We assume that an algebraic subgroup of  $T\{f(x)\}\prod_{i_x< i\leq n}U_{(i,f(x))}$ freely and transitively acts (by conjugation) in an open subset of the terms of $\quasi{\xi_x}$, and $\K_x$ is in this open subset. 

Notice then that the path in (\ref{zBrs}) is an $(\F_{\emptyset,n},\BBnk[n])$-path. In particular $\F\in\BBnk[n].$

We define $\BR{n}$ to be the set consisting of all such AFs $\K$.\end{defi}

\begin{defi}[$(\mathcal{C}^{\ssquare}{}_\square)_{n},$ $(\mathcal{C}^{\square}{}_\ssquare)_{n}$]\label{sbsquares}We define $\Pridbnk[n]$ (resp. $\Priubnk[n]$) to be the set consisting of the AFs $\F$ in $\Prink[n] $ such that: for $\prod_{1\leq i\leq k}^{\searrow}GL_{n_i}$ being the Levi of the $GL_n$-parabolic subgroup with unipotent radical $D_\F$, we have $n_1\leq...\leq n_k$ (resp. $n_1\geq...\geq n_k$).

We use $\hat{...}$ and the subscript $\nless$ (defined in \ref{hatnless}) together with the current notations, and hence for example we have  $\tPridbnk[n]=\Pridbnk[n]\cap\tPrink[n] $ and $\tPriubnk[n]=\Priubnk[n]\cap\tPrink[n] $.\end{defi}

\begin{prop}\label{dimunip}Let $a\in\UUU{n}$ and $\F\in\Prink[n][a]$.\begin{enumerate}[A.]
\item\label{zttt} The isomorphism class of $\X{\F}\cap a$, is the same for all such $\F$.
\item\label{zt} $\X{\F}\cap a$ is irreducible, and has dimension $\frac{\Di{a}}{2}$.
\item\label{ztt} $\X{\F}\cap b=\emptyset$  for any $b\in\UUU{n}$ which is smaller or unrelated to $a$.
\end{enumerate} 
\end{prop}
\begin{proof}
One easily finds an AF $\F^\prime$  in $\tPridbnk[n]\cap\BR{n} $.  Therefore\footnote{ $\F^\prime$ is nontrivial in every row. By Fact \ref{ptelos} we obtain that $J_{\F^\prime}$ is the Richardson orbit openly intersecting $\Lie{D_{\F^\prime}}$, and then by (ii) in Fact \ref{rich3} we obtain $\Di{D_{\F^\prime}}=\frac{\Di{a} }{2}.$} $\Di{D_{\F^\prime}}=\frac{\Di{a} }{2}$ and $\F^\prime\in\BBnk[n], $ which in turn give together with Proposition \ref{basic}, the proof of part B for $[\F\tlarrow\F^\prime]$. Part \ref{ztt} is reduced by \ref{conjugatehat} to assuming that $\F\in\AnkT{n}_{\nless},$ and then it is a special case of Lemma \ref{nless}. After we prove part \ref{zttt}, from it and from the special case we proved for part B, we  obtain part B in general. 

Hence we are left with proving Part $\ref{zttt}$. If one uses the (stated and proved later) Theorem \ref{th3} the proof is quicker\footnote{That is, Part A directly follows from Theorem \ref{th3} and Lemmas \ref{radexpress},  \ref{properties}, and \ref{conjugatehat}}; we choose instead to not depend on \ref{th3}. Let $\Z\in\AAnk(U_n)$ be as in the statement of Lemma \ref{radexpress}. Then this lemma together with part \ref{ztt}, and  Lemma \ref{properties}, gives the isomorphism $$\X{\F}\cap a\cong\X{\Z}\cap a.$$ Hence we only need to prove part \ref{zttt} for  $\F$ varying  in $\AAnk(U_n)$. In turn this is implied by proving that: \begin{equation}\label{zmaximz}\X{\Z_1}\cap a\cong\X{\Z_2}\cap a, \end{equation}for all $\Z_1,\Z_2\in\AAnk(U_n)[a], $ for which $J_{\Z_2} $ is obtained from  $J_{\Z_1}$ by permuting two adjacent Jordan blocks of different sizes. Let $m$ be the smallest number such that for a standard embedding $j$ of $GL_m$ in $GL_n,$ the group $j(GL_m)$ contains the two previous adjacent Jordan blocks.  

Notice that \begin{equation}\label{zconvx}\Z_i=\Z_i|_{j(GL_m)}\circ\K\qquad\text{and}\qquad j(GL_m)\subseteq\Stab{GL_n}{D_\K}\end{equation}
for some AF $\K.$ For $i=1,2$ let $\Z_i^\prime:=j^{-1}(\Z_i|_{j(GL_m)}).$ To finish the proof, we only need to find a $\QQ^\prime\in \AAnk[m]$ and a $\gamma\in GL_m$, for which Lemma \ref{radexpress} is valid for  $[\F\tlarrow\QQ^\prime,\Z\tlarrow\Z_1^\prime]$ and for $[\F\tlarrow\gamma\QQ^\prime,\Z\tlarrow\Z_2^\prime]$, because then by defining $\QQ:=j(\QQ^\prime)\circ\K$ and using (\ref{zconvx}), we obtain Lemma \ref{radexpress} for $[\F\tlarrow\QQ,\Z\tlarrow\Z_1]$ and for $[\F\tlarrow j(\gamma)\QQ,\Z\tlarrow\Z_2]$,  which in turn implies formula (\ref{zmaximz}) (by using again part C, and Lemma \ref{properties}). 

For $i=1,2,$ let $S_i$ be the element in $\Set{\TT{m}{\Z_i^\prime}}$ which contains 1. We have $|S_1|\neq|S_2|$ because these are the sizes of the jordan blocks. We choose $\QQ^\prime$ to be an AF in  $\Prink[m][|S_1|,|S_2|]$, such that for $S_{\QQ^\prime,1}$ and $S_{\QQ^\prime,2}$ being the elements in $\Set{\TT{m}{\QQ^\prime}}$ containing 1 and 2 respectively, we have $|S_{\QQ^\prime,i}|=|S_i|$ for $i=1,2$.  We choose $\gamma$ to be the element of $W_m$ permuting the first two rows. By applying Lemma \ref{radexpress} for $[\F\tlarrow\QQ^\prime]$ and then for $[\F\tlarrow\gamma\QQ^\prime]$ we see that the AF $\Z$ obtained is equal to a standard torus conjugate of $\Z_1^\prime$ and $\Z_2^\prime$ respectively, and hence we are done (by the previous paragraph). \end{proof}

\begin{defi}[First definition of $\Om(\F)$]\label{OFX}
Let $\F\in\AAnk[n].$ We define $\Om(\F)$ to be the set consisting of the minimal elements of \begin{equation*}\{a\in\UUU{n}:\X{\F}\cap a\neq\emptyset\}.\qedhere\end{equation*}\end{defi} 
\begin{defi}[Second definition of $\Om(\F)$, and $\mult{a}{\F}$]\label{AOF} Let $\F\in\AAnk[n]$ and $\Xi$ be an $(\F\rightarrow\Prink[n],GL_n)$-tree (known to exist by Theorem \ref{general}). We define $\Om(\F)$ to be the set consistsing of the minimal elements of $$\{a\in\UUU{n}:\text{An output AF of }\Xi\text{ belongs to }\Prink[n][a]\}.$$ 
Let $a\in\Om(\F)$. If the  output vertices of $\Xi$ having label in $\Prink[n][a] $ are finitely many, we denote by $\mult{a}{\F}$ the number of such vertices. If these vertices are infinitely many we write $\mult{a}{\F}=\infty.$
\end{defi}

\begin{prop}\label{dimofoccurence} $\F\in\AAnk[n]. $

\vsp 
\noindent\textbf{A.} The two definitions of $\Om(\F)$ we just gave, are equivalent (and hence the second one does not depend on $\Xi$).  

\vsp
\noindent{Let }$a\in\UUU{n}$ be such that $\X{\F}\cap a\neq\emptyset$, and for C and D below further assume that $a\in\Om(\F).$

\vsp\noindent\textbf{B. } $\Di{\X{\F}\cap a}\geq\frac{\Di{a}}{2}. $

\vsp\noindent\textbf{C. }$\Di{\X{\F}\cap a}=\frac{\Di{a}}{2}\iff \mult{a}{\F}<\infty.$

\vsp\noindent\textbf{D. }If $\mult{a}{\F}<\infty$, then $\mult{a}{\F}$ is equal to the number of connected components of $\X{\F}\cap a$. 

\vsp\noindent{In} particular (due to C and D) the definition of  $\mult{a}{\F}$ 
 does not depend in the choice of $\Xi$. 
\end{prop}

\begin{proof}We proceed inductively on the depth of $\Xi.$ Assume first that the depth is 0. Then $\Xi$ consists of just one vertex labeled with $\F$, and $\F\in\Prink[a].$ Hence the proposition follows by Parts \ref{zt} and \ref{ztt} of \ref{dimunip}.

Assume now that the depth of $\Xi$ is at least one. Consider the biggest $\AAnk$-subtrees of $\Xi$ with  input vertex a depth one vertex of $\Xi.$ The proposition is obtained by: applying the inductive hypothesis to these trees (in the place of $\Xi$),  and by  applying Lemma \ref{properties} to the first $\AAnk$-step of $\Xi$. 
\end{proof}
\begin{warning}\label{warn}Let $\F\in\AAnk[n].$ We will  not be mentioning---each time $\Om(\F)$ occurs--- the use of Theorem \ref{general} to obtain the existence of $\Xi$. Similarly we will be using without mention the independence of Definition \ref{AOF} on the choice of $\Xi$, and the equivalence between Definition  \ref{OFX} and the part of Definition \ref{AOF} addressing $\Om(\F)$. 
\end{warning}
\begin{maincor}\label{maincor}Let $\F\in\BBnk[n]$. Consider an orbit $a\in\UUU{n}$, such that there is an  $(\F\rightarrow \Prink[n],GL_n )$-tree with an output AF belonging in $\Prink[n][a]$.  Then:\begin{enumerate}[A.]
\item $\frac{\Di{a}}{2}\geq\Di{D_\F}; $
\item \begin{multline}\Di{D_\F}=\frac{\Di{a}}{2}\iff a\in\Om(\F)\text{ and } \mult{a}{\F}<\infty\\\iff a\in\Om(\F)\text{ and }\mult{a}{\F}=1. \end{multline}
\end{enumerate}
\end{maincor}
\begin{proof} Part A (resp. the first ``$\iff$'' in Part B) follows from Part B (resp. Part C) of Proposition \ref{dimofoccurence}, and from the dimension equality in Proposition \ref{basic}. Finally the second ``$\iff$" in Part B follows from Part D of  Proposition \ref{dimofoccurence} and Proposition \ref{basic} (due to which $\X{\F}\cap a$ is connected). Of course, we used  Theorem \ref{general} as in Warning \ref{warn}.  
\end{proof}

As a first example of using Main corollary \ref{maincor}  we discuss $\Prink[n]\cap\BBnk[n]$ in the rest of the current section.
\begin{defi}\label{capprime}
Let $\mathcal{X}$ be a subset of $\Prink[n]$. We define  \begin{equation*}\mathcal{X}\cap^\prime\BBnk[n]:=\{\F\in\mathcal{X}:J_\F\text{ is in the Richardson orbit openly intersecting }\Lie{D_\F}\}. \end{equation*}To avoid confusion, recall that $J_\F\in\Lie{D_\F}^t$. \end{defi}
\begin{remark}My motivation for this notation is: it turns out \begin{equation}\label{ptelos2}\Prink[n]\cap^\prime\BBnk[n]=\Prink[n]\cap\BBnk[n] \end{equation}(and hence for replacing $\Prink[n]$ with any of its subsets); and in Definition \ref{capprime} there is no mention of $\BBnk[n]$. We prove (\ref{ptelos2}) later in Corollary \ref{laststrong2}. However, the extent to which (\ref{ptelos2}) is needed for other results of the present paper is given in the lemma bellow.\end{remark}
\begin{remark}\label{pthp}In many of the cases the definition above is mentioned, we use without mention Fact \ref{ptelos}. 
\end{remark}
\begin{lem}\label{weakprime}$$\tPridbnk[n]\cap^\prime\BBnk[n]=\tPridbnk[n]\cap\BR{n}=\tPridbnk[n]\cap\BBnk[n]. $$
\end{lem}
\begin{proof} An AF in $\tPridbnk[n]\cap^\prime\BBnk[n]$  is nontrivial on every row and hence belongs to $\tPridbnk[n]\cap\BR{n}$. 

Let $\F\in\tPridbnk[n]\cap\BBnk[n].$ By using Main corollary \ref{maincor} for $[a\tlarrow \Om(J_\F)]$ and information for the Richardson orbit ((i), (ii), and (iii) in Fact \ref{rich3} ), we have $\F\in\tPridbnk[n]\cap^\prime\BBnk[n]$.
\end{proof}
\section{Examples of  AFs $\F$ satisfying $\Om(\F)\neq \{\Om(J_\F)\}$ (but still being a singleton), and more information on $\Prink[n]. $}\label{optionalsection}
The first result in this section (Proposition \ref{embedding}) gives simple examples of AFs $\F$, such that the set $\Om(\F)$ is a singleton, but frequently, $\Om(J_\F)$ is not its element. This is the only result in the current section which  is needed for the next section.  

The only uses of results of Section \ref{variety} in the current section are in\footnote{By uncovering Definition \ref{AOF} for an appropriate choice of  $\Xi$ in the statement of Proposition \ref{embedding} we obtain a weaker statement with the same proof without any use of Section \ref{variety} (that is not using the independence on choosing $\Xi$). However, this weaker statement can replace the uses of Proposition \ref{embedding} in the next section.} Proposition \ref{embedding} and in Corollary \ref{laststrong2}. Also, Definition \ref{capprime} is used in several statements.

Throughout the current section let $n$ and $k$ be any two positive integers, and $\kkk$ be any subfield of $\KKK$. Note that as part of Convention \ref{omitn}, for any $r,i,j$, we denote: $\JJ_r^{k+n}$ by $\JJ_r$ in Proposition \ref{embedding}, Proposition \ref{strong}, and Corollary \ref{corFJ}; $\JJ_r^n$ and $U_{n,(i,j)}$ respectively by $\JJ_r$ and $U_{(i,j)}$ in Proposition \ref{bigger}, Corollary \ref{corbigger}, and Theorem \ref{th3}.

\begin{prop}\label{embedding}Let $a=[a_1,...,a_m]$ be a partition of $n.$ Let $\F_a\in\Prink[n][a]$, and $\ja\F_a$ be its  lower right corner copy in $\AAnk[k+n]$.  Let $A_{a,k}$ be the subset of $\UUU{k+n} $ given by: 
$$A_{a,k}:=\left\{[k^\prime,a_1^\prime,a_2^\prime,...,a^\prime_m ]:k^\prime+\sum_{i}a_i^\prime =k+n,\hspace{1mm}k\leq k^\prime,\hspace{1mm}\text{and }\hspace{1mm} a_i^\prime\leq a_i\text{ for }1\leq i\leq m\right\}. $$Of course the condition $k\leq k^\prime$ in this set, is implied from the other conditions.

\vsp
\noindent\textit{Part }1. $\Om(\ja\F_a\rcirc\JJ_k)$ is a singleton, and its unique element is the minimal orbit in $A_{a,k}$.

\vsp
\noindent\textit{Part }2. Let $b$ be the unique element of $\Om(\ja\F_a\rcirc\JJ_k).$ The following  statements are equivalent. \begin{enumerate}
\item[a.]$\mult{b}{\ja\F_a\rcirc\JJ_k}=1$.  
\item[b.]$\mult{b}{\ja\F_a\rcirc\JJ_k}<\infty$.
\item[c.]$b=[k,a_1,...a_m].$
\item[d.]$\max\{a_1,...,a_m\}\leq k.$ 
\end{enumerate}
\end{prop}
\begin{tproof}\textit{\textbf{Proof of Part 1. }}Due to Lemma $\chh[lemma to leo]$ \ref{radexpress},  the first part is reduced to the case in which $\F_a\in\AAnk(U_{n})$. In turn  the proof of this reduced case, will easily follow from: the  recursive description of $\sXi(\ja\F_a\rcirc\JJ_k)$ given in the claim stated below, and Definition \ref{AOF} of $\Om(\F)$ for $[\F\tlarrow\ja\F_a\rcirc\JJ_k,\Xi\tlarrow\sXi(\F_a\rcirc\JJ_k)]$.

We need to use ordered partitions of $n$. We denote with  $\aaa,$ the ``ordered version of $a$", that is the $m$-tuple
$$\aaa:=(a_1,...,a_m). $$ We choose $\F_\aaa $ to be an AF in $\AAnk(U_{n})[a] $ such that the Jordan blocks of $J^t_{\F_\aaa}$, starting from the upper left corner and moving downwards, are of dimension $a_1\times a_1,...,a_m\times a_m$

 For $0\leq s\leq m$ let $\zAF[,s]$ be the trivial AF  with domain $D_{\zAF[,s]}=\prod_j U_{(k,j),}$ where the product is over the $j$ satisfying: $$k+\sum_{1\leq i\leq m-s} a_i<j\leq k+n. $$

\vsp 
\noindent\textbf{Claim. }\textit{The output AFs of  an initial $\AAnk$-subtree of  $\sXi(\ja\F_\aaa\circ\zAF[,s] \circ\JJ_k)$ are the following.  \begin{enumerate}\item[(1).] an AF in $\AAA(U_{k+n}) [k,a_1,...,a_m]$,  \item[(2).] The AFs $\ja\F_{\aaa^\prime}\circ\zAF[,s^\prime]^\prime\circ\JJ_{k^\prime}$ where:\begin{enumerate}\item[(i).] $\aaa^\prime,k^\prime,s^\prime,r$ vary over all the integers satisfying  
\begin{multline}\label{zx3}k^\prime=k+r,\quad \aaa^\prime=(a_1,...a_{m-s^\prime-1},a_{m-s^\prime}-r,a_{m-s^\prime+1},...),\\ s\leq s^\prime\leq m-1,\quad\text{ and }\quad 0< r\leq a_{m-s^\prime}.\end{multline}
\item[(ii).] For every choice of $\alpha^\prime$ as in (i), we choose one (among the conjugate ones) AF $\F_{\alpha^\prime}$ which is defined as $\F_{\aaa} $ for $a^\prime$ in the place of $a$, $\jj\F_{\alpha^\prime}$ is the lower right corner copy of $\F_{\alpha^\prime}$ in $\AAnk[k+n]$, $\zAF[,s^\prime]^\prime $ is defined as $\zAF[,s^\prime]$ for $k^\prime$ (resp. $\aaa^\prime$)in the place of $k$ (resp. $\aaa$). A minor clarification: in case $m=1,$ for $\aaa^\prime=(0)$ we define $\jj \F_{\aaa^\prime}=\F_{\emptyset,k+n}$.
\end{enumerate}  \end{enumerate}}

\vsp 
\noindent\textit{Proof of Claim. }This initial $\AAnk$-subtree of $\sXi(\ja\F_\aaa\circ\zAF[,s] \circ\JJ_k)$ is $$ \sI(\XX)\vee_\XX\xi$$ where: $\xi$ is the $\e$-quasipath along constant terms over all the positive root groups of the $k$-th row which are not contained in $D_{\ja\F_\aaa\circ\zAF[,s] \circ\JJ_k}$; and $\XX$ varies over all output AFs of $\xi$. 

   In case the reader prefers more details, they follow next (otherwise the proof of the claim is over).  The output constant term of $\xi$ (for which note that $\sI$ is trivial)   is equal to an AF as in (1) in the statement of the claim. The other output-AFs of $\xi$, take the form $\ja\F_\aaa\circ\YY_t\circ\JJ_k$; where $t$ and  $\YY_t$ satisfy: \begin{enumerate}
\item $D_{\YY_t}=\prod_{t\leq j\leq k+n}U_{(k,j)}$;
\item $\YY_t(U_{(k,j)})\neq \{0\}\iff j=t$;
\item $k+1\leq t\leq k+ \sum_{1\leq i\leq m-s}a_i$. 
\end{enumerate}
We define $s^\prime(t)$ to be the biggest integer for which $t\leq k+ \sum_{1\leq i\leq m-s^\prime(t)}a_i. $  Finally we define $r(t):=k+ \left(\sum_{1\leq i\leq m-s^\prime(t)}a_i\right)+1-t.$

 Observation \ref{unipc} gives  $$\siota(\F_\aaa\circ\YY_t\circ\JJ_k)=\ja\F_{\aaa^\prime}\circ\zAF[,s(t)]^\prime\circ\JJ_{k^\prime},$$  where $\aaa^\prime$, $k^\prime,$ and $\zAF[,s(t)]^\prime $ are as in (2) in the statement of the claim in the special case that $s^\prime=s^\prime(t) $ and $r=r(t)$.\hfill$\square$Claim
 
 By iteratively  using the claim we obtain that the set of output AFs of $\sXi(\ja\F_a\rcirc\JJ_k )$ is equal to  $A_{a,k}$ (we switch back to usual nonordered partitions ater these uses of the claim). Hence the $\AAnk$-tree definition of $\Om(\ja\F_a\rcirc\JJ_k)$, gives that $\Om(\ja\F_a\rcirc\JJ_k)$ consists of one element, which is the minimal element of $A_{a,k}$.\hfill$\square$Part 1

\vsp 
\noindent\textit{\textbf{Proof of Part 2.} }

\noindent\textit{Case $a\implies b$. }Trivial.

\noindent\textit{Case $b\implies c$. }We consider the initial $\e$-subquasipath  of $\sXi(\ja\F_a\rcirc\JJ_k)$ along constant terms,  consisting of $n-1$ $\e$-steps. We notice that in each such $\e$-step the nonconstant terms  form a single (and hence infinite) orbit for the action of a torus.

\noindent\textit{Case $c\implies d$. }Follows by using (again) Part 1.

\noindent\textit{Case $d\implies a$. }By applying Observation \ref{unipc} we obtain an AF $\XX\in\tPrink[k+n][b]$ such that: $$\siota(\XX)=\ja\F_a\rcirc\JJ_k$$ and hence $\mult{b}{\ja\F_a\rcirc\JJ_k}=\mult{b}{\XX}=1. $

\hfill$\square$Part 2

\end{tproof}

 The proposition above answers Version 1 of Weak questions \ref{zStrongc}.($\F_a\rcirc\JJ_{k}$) and \ref{zStrongcc}.($\F_a\rcirc\JJ_{k}$).  Therefore we also obtain the answers to Version 2 of the same questions. For example the answer to Version 2 of Weak question \ref{zStrongc}.($\F_a\rcirc\JJ_{k}$) is included in the corollary below.
\begin{cor}We adopt the notations of the proposition above. Let $\pi\in\Aut{\kkk,k+n,>}$ be such that $\Om^\prime(\pi)=b.$ Then $\F_a\rcirc\JJ_{k}(\pi) $ is nonzero, and it is factorizable if and only if any among a,b,c, or d in Part 2 of the proposition  above holds.\end{cor}
\begin{proof}
It directly follows from the proposition above and Corollary \ref{easyrefcor}.
\end{proof}
\noindent{In} the  two propositions below we obtain $\AAnk$-trees which together with Corollary \ref{easyrefcor} answer (in Corollaries \ref{corbigger} and \ref{corFJ}) Strong questions \ref{zStrongc} and \ref{zStrongcc} for certain AFs. 

In the two propositions below we make use of a later-obtained result (Corollary \ref{laststrong}) which is a Corollary of Theorem \ref{th3}; however, Corollary \ref{laststrong} could have been obtained from a small part of the proof of Theorem \ref{th3}.

\begin{prop}\label{strong}Let $a=[a_1,...,a_m]\in\UUU{n}$, and assume $a_1=\max\{a_1,...,a_m\}$. Let $\F_a\in \Pri[,n][a]\cap^\prime\BBnk[n]$, and $\ja\F_a$ being its lower right corner copy in $\AAA[,k+n].$ Also let $A_{a,k}$ be as in the previous proposition. There is a \begin{equation}\label{zfootz}\left(\ja\F_a\rcirc\JJ_{k}\rightarrow \bigcup_{c\in A_{a,k}}\tPri[,k+n][c]\cap^\prime\BBnk[k+n],\mirsl{k+n},\kkk\right)\text{-tree}.\end{equation}
\end{prop}
\begin{proof}By using a conjugate transpose version of Corollary \ref{laststrong} we are reduced to assuming $\F_a\in\tPriubnk[n]$. 

For $1\leq r\leq a_1-k$ we inductively define the root generated subgroups $N_r$ of $\prod_{k<j\leq k+n}U_{(k,j)}$ by requiring that: the centralizer of $\left(\prod_{1\leq y\leq r-1}N_y\right)D_{\ja\F_a\circ\JJ_{k}}$ in $\prod_{k<j\leq k+n}U_{(k,j)}$ is equal to $\prod_{1\leq y\leq r}N_y$. Let $0\leq s\leq \max\{a_1-k,0\}$ and $\zAF$ be the trivial AF with domain $\prod_{1\leq i\leq s}N_i$ (Note that if $s=0$, we have $\zAF=\F_{\emptyset,k+n} $). We state as a claim a refined version of the proposition,  which we then prove inductively.

\vsp
\noindent\textbf{Claim. }\textit{ There is an \begin{equation}\label{zfootzz}\left(\ja\F_a\rcirc\zAF\circ\JJ_{k}\rightarrow \bigcup_{c\in A_{a,k}}\tPri[,k+n][c]\cap^\prime\BBnk[k+n],\mirsl{k+n},\kkk\right)\text{-tree}.\end{equation}}

\vsp\noindent\textit{Proof of claim}. Let $\xi$ be the $(\ja\F_a\circ\zAF\circ\JJ_{k},\kkk)$-quasipath along constant  terms over the groups $N_r$ for $s<r\leq a_1-k$.

From Observation \ref{unipc}, we see that the output constant term  of $\xi,$ is of the form $\siota(\XX)$ for an $\XX$ in $\tPri[,k+n][k,a_1,...,a_m]\cap\BBnk[k+n]$. 
In particular, the claim is obtained in the special case $s=\max\{a_1-k,0\}.$ For the general case we proceed inductively in the variable $a_1-k-s$, by fixing a positive value for it, and assuming the claim is proven for all smaller values.

Since the  output constant term of  $\xi$ is already addressed, consider any other output AF of $\xi$. It admits the form \begin{equation}\label{zYjkz}\ja\F_{a}\circ\YY\circ\JJ_k\end{equation} where $\YY$ is a nontrivial AF on $\AAA(\prod_{1\leq y\leq t}N_y)$ for a choice  $s< t\leq a_1-k$. 
 By applying Observation \ref{diagonal} \chh[diagonal,obser] to $\F_a$, we obtain an element $u$ with the following properties: 
\begin{enumerate}
\item $u$ is in the lower right corner copy (in $GL_{k+n}$) of the Levi of the $GL_n$-parabolic subgroup with unipotent radical equal to $D_{\F_a}$ (therefore $u$ normalizes $D_{\ja\F_{a}\circ\YY\circ\JJ_k})$). Also the entries of $u$ belong to $\kkk$.
\item $u\YY$ is nontrivial on only one root group (which of course  belongs to $N_t$).
\end{enumerate}
By Observation \ref{unipc}, we obtain \begin{equation}\label{zfih}\siota(u(\ja\F_{a}\circ\YY\circ\JJ_k))=\ja\F_c\circ\zAF^\prime\circ\JJ_{k+t^\prime},\end{equation} where:
\begin{enumerate}
\item  $c=[a_1,a_2,...a_{l-1},a_l-t^\prime,a_{l+1},...]$ for some positive integers $l$, and $t^\prime\leq t$;
\item The AF $\F_c$ belongs to $\tPriub[,n-t^\prime][c]\cap\BBnk[n-t^\prime],$ and $\ja\F_c $ is its lower right corner copy in $\AAA[,k+n].$
\item $\zAF^\prime$ is the trivial AF with domain $\prod_{1\leq i\leq t-t^\prime}N^\prime_i$, where each $N_i^\prime$ is defined in the same way as $N_i$  by the replacements $[\F_a\tlarrow \F_c, k\tlarrow k+t^\prime ]$ (in the second paragraph of the proof of the proposition).
\end{enumerate}The claim for $[\F_a\tlarrow \F_c,k\tlarrow k+t^\prime,\zAF\tlarrow\zAF^\prime]$ is obtained from the inductive hypothesis.

By putting together all the $\AAnk$-trees we constructed we are done\footnote{If one prefers more detail, the $\AAnk$-tree in (\ref{zfootzz}) can be chosen to be $ \left(\Xi(f(\F^\ih))\vee\I(\F^\ih)\right)\vee_{\F^\ih}\xi,$ where:  $\F^\ih$ varies over all output AFs of $\xi$; $\F^\ih_0$ is the output constant term of $\xi$; $f(\F_0^\ih)=\XX$; $\I(\F_0^\ih):=\sI(f(\F_0^\ih))^{-1}$; for $\F^\ih$ being the AF in (\ref{zYjkz}),  $\I(\F^\ih)$ is the $\F^\ih$-path starting with conjugating by $u$ and continues with $\sI(u\F_1^\ih)$, and (again) $f(\F^\ih )$ is the output AF of $\I(\F ^\ih)$ (therefore the AF in (\ref{zfih})); $\Xi(f(\F_0^\ih))$  is trivial, and for $\F^\ih$ as in (\ref{zYjkz}), $\Xi(\F^\ih )$ is as in (\ref{zfootzz}) for $[\F_a\tlarrow \F_c,k\tlarrow k+t^\prime,\zAF\tlarrow\zAF^\prime]$.}.\hfill$\square$Claim

\end{proof}
\begin{prop}\label{bigger}Let $\F\in\Pri[,n]\cap^\prime\BBnk[n].$ For every orbit $c$ in $\UUU{n} $ which is bigger from $\Om(J_\F)$, there is an $(\F,GL_n,\kkk)$-tree with infinitely many  output vertices labeled with AFs in $\Pri[,n][c]\cap^\prime \BBnk[n] $.
\end{prop}
\begin{proof} By proceeding inductively we assume the proposition is correct for $n$ being replaced by any smaller number. 

Let  $\Om(J_\F)=[a_1,...,a_m],$ so that $a_1\geq a_2$.  It is sufficient to prove the proposition in the special case in which $c$ is obtained from $\Om(J_\F)$ by replacing $a_1,a_2$ with $a_1+1,a_2-1;$ because then, the general case follows by obtaining $c$ from $\Om(J_\F)$ by a finite number of such replacements.

\vsp
\noindent\textit{Case 1: $a_1<\max\{a_2,...,a_m\}$. }  Without loss of generality assume that $a_3$ is the biggest number in $\Om(J_\F).$ After conjugating $\F$ with an appropriate element in $W_n$ we are also reduced to assuming that  the  set  in $\Set{\TT{n}{\F}}$ containing  $1$, has $a_{3}$ elements.    By applying Observation \ref{unipc} we obtain $$\siota(\F)=j(\F^\ih)\rcirc\JJ_{a_3}, $$ where: $\F^\ih $ is an AF in $\Pri[,n-a_3]\cap^\prime\BBnk[n-a_3] $ with $\Om(J_{\F^\ih})$ obtained from $\Om(J_{\F})$ by removing one occurrence of $a_3$, and $j$ is  the lower right corner embedding of $GL_{n-a_3}$ in $GL_{n}.$ The inductive hypothesis contains the proposition for  $[n\tlarrow n-a_3,\F\tlarrow\F^\ih],$ and hence there is an $(\F^\ih, GL_{n-a_3},\kkk )  $-tree $\Xi(\F^\ih)$, such that for an infinite set $\VVV$ consisting of output vertices of $\Xi(\F^\ih)$, each $v\in \VVV$ is labeled with an AF, say $\Z_v^\ih,$ belonging in $\Pri[,n-a_3][c^\ih]\cap^\prime\BBnk[n-a_3], $ where $c^\ih$ is obtained from $c$ by removing one occurrence of $a_3. $ By using again  Observation \ref{unipc}, for each such $\Z_v^\ih$ we obtain a $\Z_v\in\Pri[,n][c]\cap^\prime\BBnk[n]$ such that $\siota(\Z_v)=j(\Z_v^\ih)\circ\JJ_{a_3} $. Hence we choose the tree in the statement of the proposition to be $${\sI(\Z_v)}^{-1} \vee_{u\in\VVV}j(\Xi(\F^\ih))\circ\JJ_{a_3}\vee\sI(\F).
 $$ \hfill$\square$Case 1

\vsp 
\noindent\textit{Case 2: $a_1\geq\max\{a_2,...,a_m\}$.}
 By using (a conjugate transpose version of) Corollary \ref{laststrong}, we are reduced to assuming  that $\F\in\tPriub[,n]\cap^\prime\BBnk[n].  $ Among topics considered later we also use the Definition (\ref{VpV}). After conjugating with an appropriate element in $W_n$ which preserves $D_\F,$ we are further reduced to assuming that for $S_1,S_2$ being the elements of  $\Set{\TT{n}{\F}}$ respectively containing 1 and 2, we have $|S_1|=a_2$ and $|S_2|=a_1.$ The subcase $m=2$ and $a_2=1$ is trivial, and hence we assume otherwise.  Define $n_1$ so that the first column on which $D_\F$ is nontrivial is the $n_1+1$-th one. For each nondiagonal  entry lying in the first row and in one of the first $n_1$-th columns consider the $\mathrm{Stab}_{GL_n}(\F)$-minimal $\ossa$-group which is nontrivial on this entry. We consider the $(\F,\e)$-quasipath along constant terms over these $\ossa$-groups. Let $\F_1$ be a nonconstant term of the last $\e$-step of this quasipath. We have\footnote{Formula (\ref{zdircheck}) follows from the proof of Property 3.1 in Theorem \ref{thD3} (it may take less time though for the reader to directly check (\ref{zdircheck})). Note that condition \ref{3lastp} in Definition \ref{cadhoc} is not currently satisfied. More precisely,  the unique pair $(\K,\FF)\in\mathcal{S}_{n_1,n}$ for which $j_{\FF}(\K)\circ\FF=\F_1,$ satisfies all the conditions in the definition of $\Skal{a_1}{(2,a_2),(1,a_3),...,(1,a_m)}$ except for condition \ref{3lastp}. However this condition is not needed in the proof of Property 3.1.}  \begin{equation}\label{zdircheck}\siota(\F_1)=j(\F_1^{\ih})\rcirc\JJ_{a_1+1} \end{equation} where $$\F_1^{\ih}\in\tPriub[,n-a_1-1][a_2-1,a_3,...,a_m] \cap^\prime\BBnk[n-a_1-1] $$and $j$ is the lower right corner embedding of $GL_{n-a_1-1}$ in $GL_n$. Therefore by using the inductive hypothesis for $[\F\tlarrow\F_1^{\ih}  ]$ and proceeding as in Case 1 we are done.
\end{proof}

We are now ready to address Strong question 1.$(\F)$ and 2.$(\F)$, for $\F$ as in the last two propositions above.
\begin{cor}\label{corbigger}Let $\F\in\Pri[,n]\cap^\prime\BBnk[n],$ and $\pi\in\Aut{\kkk,n,>}$  Then: 

\vspace{1mm}
\noindent{Part 1. }$\Om(J_\F)=\Om^\prime(\pi)\iff \F(\pi)\text{ is factorizable and nonzero};$

\vspace{1mm} 
\noindent{Part 2. }$\Om(J_\F)\leq \Om^\prime(\pi)\iff \F(\pi)\neq 0.$
\end{cor}  
\begin{proof}It directly follows Proposition \ref{bigger} and Corollary \ref{easyrefcor}.
\end{proof}
\begin{cor}\label{corFJ} Let $a,\F_a$  be defined as in the second last proposition (Proposition \ref{strong}), and $b$ be the unique element of $\Om(\F_a\rcirc\JJ_k)$ (explicitly described and proven to be unique in Proposition \ref{embedding}). Let $\pi\in\Aut{\kkk,n,>}$. Then

\vspace{1mm}
\noindent{Part 1. }$(k\geq a_1 \text{ and }b=\Om^\prime(\pi))\iff (\F_a\rcirc\JJ_k)(\pi)\text{ is factorizable and nonzero};$

\vspace{1mm} 
\noindent{Part 2. }$b\leq \Om^\prime(\pi)\iff (\F_a\rcirc\JJ_k)(\pi)\neq 0.$
\end{cor}
\begin{proof}It directly follows from Propositions \ref{strong} and \ref{bigger}, and Corollary \ref{easyrefcor}.
\end{proof}

A result on $\Prink[n] $ follows.

\begin{theorem}\label{th3}Let $\F\in\tPri[,n]$. Let $M=\prod^{\searrow}\jj GL_{n_i}$ be the  Levi of the $GL_n$-parabolic subgroup with unipotent radical $D_\F$. Let $j$ be the upper left corner embedding of $GL_{n-n_k}$ in $GL_n$. Let  $N_k$ be the unipotent radical of the $GL_n$-parabolic subgroup with Levi $j(GL_{n-n_k})\times^\searrow \jj GL_{n_k}$. Then: \begin{enumerate}\item There is an AF we call $\db{\F},$ for which we have:\begin{enumerate}\item$\db{\F}\in\tPridb[,n]\cap^\prime\BBnk[n] $ and $D_{\db{\F}}\subset j(SL_{n-n_k})N_k $;\item the matrices $J_{\F}$ and $J_{\db{\F}}$ are $j(SL_{n-n_k})(\kkk)$-conjugate. \end{enumerate}\item There is an $\left(\db{\F}\rightarrow\F,j(SL_{n-n_k})N_k,\kkk\right)$-path.    \end{enumerate}
\end{theorem}
\begin{tproof} By proceeding inductively, we  assume  the theorem is correct for $n$  being replaced by smaller values.

\vsp
\noindent\textbf{Claim.}\textit{ Assume  the theorem is correct in the special cases in which:  
\begin{equation}\label{zkoukis2}\max(\Set{T})>\sum_{1\leq i\leq k-2} n_i\qquad\text{ for all }\DDDD[n]\text{-components }T \in\TT{n}{\F}.\end{equation} 
Then the theorem is correct.}

\vsp
\noindent\textit{Proof of Claim.} The claim is trivial for $k=2$ and hence we assume $k\geq 3.$ We define  $$N_{k-1}:=\prod_{\substack{i<j\\\sum_{1\leq r\leq k-2} n_r<j\leq \sum_{1\leq r\leq k-1} n_r}} U_{(i,j)}\cap D_\F. $$

For any $\F_1,\F_2\in\AAnk[n]$   for which $\F_1\circ\F_2$ is defined, and for a group $H$ satisfying $D_\F\subseteq H\subseteq \mathrm{Stab}_{GL_n}(\F)$, we some times denote $\F_1\circ\F_2$ by $\F_1\circ_H\F_2 $.

We have $$\F=\F|_{j(SL_{n-n_k})}\circ_{ j^\prime(SL_{n-n_k-n_{k-1}})N_{k-1}}\F|_{N_k}$$ where $j^\prime$ is the upper left corner embedding of $GL_{(n-n_k-n_{k-1})}$ in $GL_{n}$. From the inductive hypothesis the theorem for $[\F\tlarrow j^{-1}(\F)]$ is true (note that $j^{-1}(\F)$=$j^{-1}(\F|_{SL_{n-n_k}})$). Therefore:
\begin{enumerate}\item There is an AF we call $\db{(\F|_{j(SL_{n-n_k})})}$ for which:\begin{enumerate}\item$j^{-1}(\db{(\F|_{j(SL_{n-n_k})})})\in\tPridb[,n-n_k]\cap^\prime\BBnk[n-n_k] $ and $D_{\db{(\F|_{j(SL_{n-n_k})})}}\subset j^\prime(SL_{n-n_k-n_{k-1}})N_{k-1} $;\item the matrices $J_{\F|_{j(SL_{n-n_k})}}$ and $J_{\db{(\F|_{j(SL_{n-n_k})})}}$ are $j^\prime(SL_{n-n_k-n_{k-1}})(\kkk)$-conjugate. \end{enumerate}\item There is an $\left(\db{(\F|_{j(SL_{n-n_k})})}\rightarrow\F|_{j(SL_{n-n_k})} ,j^\prime(SL_{n-n_k-n_{k-1}})N_{k-1},\kkk\right)$-path, which we call $\Xi\left(\db{(\F|_{j(SL_{n-n_k})})}\right). $    \end{enumerate}

Since $j^\prime(SL_{n-n_k-n_{k-1}})N_{k-1}$  is fixing $\F|_{N_k}$, we can define the AF of $\tPri[,n]$ given by
\begin{equation}\label{zfprif}\F^\prime:=\db{(\F|_{j(SL_{n-n_k})})}\circ \F|_{N_k},\end{equation}

and the $(\F^\prime\rightarrow\F, j^\prime(SL_{n-n_k-n_{k-1}})N_{k-1}N_{k},\kkk)$-path
$$\Xi(\F^\prime):=\Xi\left(\db{(\F|_{j(SL_{n-n_k})})}\right)\circ\F|_{N_k}. $$
Finally, the theorem for $[\F\tlarrow \F^\prime]$ is known, by being among the cases assumed in the claim. Therefore we obtain: an AF $\db{(\F^\prime)}$ such that the theorem is correct for $[\F\tlarrow\F^\prime,\db{\F}\tlarrow\db{(\F^\prime)}]$, and gives an  
$(\db{(\F^\prime)}\rightarrow \F^\prime,j(SL_{n-n_k})N_k,\kkk)$-path we call $\Xi(\db{(\F^\prime)})$. We see that  for $\db{\F}:=\db{(\F^\prime)},$ 1.(a) and 1.(b) in the statement of the theorem are satisfied, and that  $\Xi(\F^\prime)\vee\Xi\left(\db{(\F^\prime)}\right)$ is an $(\db{\F}\rightarrow\F,j(SL_{n-n_k})N_k,\kkk)$-path, hence completing the proof of the claim.\hfill$\square$Claim.
 
\vsp
\noindent{}By using the Claim, and after possibly conjugating $\aF$ with an element in $W_{\jj GL_{n_1}}$, we are reduced to the case in which the following two hold:\begin{enumerate}
\item Formula (\ref{zkoukis2})  is satisfied.
\item Let $T^1$ be the $\DDDD[n]$-component of $\TT{n}{\F}$ which satisfies $\min(T^1)=1$. If $T$ is any  $\DDDD[n]$-component of $\TT{n}{\F}$, we have   $|\Set{T^1}|\geq |\Set{T}|$.
\end{enumerate}

We have $|\Set{T^1}|\in\{k-1,k\}$. We assume $|\Set{T^1}|\geq 2$, because otherwise the proof is trivial. By Observation \ref{unipc} (for $T=T^1$ in case $|\Set{T^1}|=k-1$, and otherwise for the $T\in\DDDD[n]$ such that $\Set{T}=\Set{T^1}-\{\max(\Set{T^1})\}$) we have  
\begin{equation}\label{vv}\siota[,k-1](\F)=w_{k-1}(\F^{T^1})\circ\YY\circ\JJ_{k-1}, \end{equation}
where $w_{k-1}$ is defined as in that observation, and $\YY$ is defined next. In case $|\Set{T^1}|=k-1$,  the AF $\YY$ is the trivial AF with domain $\prod_{n-n_k< j\leq n}U_{(k-1,j)}$. In case $|\Set{T^1}|=k$, the AF $\YY$:\begin{enumerate}
\item has domain
$$D_\YY=\prod_{n-n_k< j\leq n}U_{(k-1,j)}\prod_{k\leq i\leq n-n_k}U_{(i,\max{(\Set{T^1})})};$$
\item is nontrivial only on one root group, which is $U_{(k-1,\max{(\Set{T^1})})}.$ 
\end{enumerate} 

Let $\db{(w_{k-1}(\F^{T^1}))}$ be an AF such that  the inductive hypothesis gives the theorem for $[\F\tlarrow {j^{\prime\prime}}^{-1}(w_{k-1}(\F^{T^1})),\db{\F}\tlarrow {j^{\prime\prime}}^{-1}(\db{(w_{k-1}(\F^{T^1}))})]$, where $j^{\prime\prime}$ is the standard embedding of $GL_{n-|\Set{T^1}|}$ onto $w_{k-1}(GL_{n})^{T^1}w_{k-1}^{-1}$. For $\Xi(\db{(w_{k-1}(\F^{T^1}))})$ being the $(\db{(w_{k-1}(\F^{T^1}))}\rightarrow w_{k-1}(\F^{T^1}))$-path obtained from this special case of the theorem, we have \begin{equation}\label{zroot}\Xi(\db{ (w_{k-1}(\F^{T^1}))})\circ \YY\circ\JJ_{k-1}\end{equation} is a well defined $$\left(\db{(w_{k-1}(\F^{T^1}))}\circ \YY\circ\JJ_{k-1}\rightarrow w_{k-1}(\F^{T^1})\circ \YY\circ\JJ_{k-1},j(SL_{n-n_k})N_k,\kkk\right)\text{-path}.$$ 

Hence we are only left with defining an AF $\db{\F} $ as in the statement of the theorem, and constructing an $(\db{\F}\rightarrow\db{(w_{k-1}(\F^{T^1}))}\circ \YY\circ\JJ_{k-1},j(SL_{n-n_k})N_k,\kkk)$-path. We will denote this path by $\Xi^\prime(\db{\F})$.

By $\db{\F}$ we choose any element of $\tPridb[,n]\cap^\prime\BBnk[n]$ such that, for $\db{{T}}^1$ being the $\DDDD[n] $-component of $\TT{n}{\db{\F}}$ satisfying $\min{(\Set{\db{T}^1})}=1,$ we have:. 
\begin{enumerate}
\item $|\Set{\db{T}^1}|=|\Set{T^1}|$;
\item If $|\Set{T^1}|=k$, then $\max{(\Set{\db{T}^1})}=\max{(\Set{T^1})}$.
\item $$w(\db{\F})^{\db{T}^1}=\db{(w_{k-1}(\F^{T^1}))} $$where $w$ is the minimal length Weyl group element satisfying $w\db{T}^1w^{-1}\supseteq\{1,...,k-1\}$.
\end{enumerate}

Finally we describe $\Xi^\prime(\db{\F}).$ It has $\sI[,k-1](\db{\F})$ as an initial subpath. By Observation \ref{unipc} we have $$\siota[,k-1](\db{\F})=\db{(w_{k-1}(\F^{T^1}))}\circ\YY^\prime\circ\JJ_{k-1}, $$where $\YY^\prime$ is defined as follows. If $|\Set{T^1}|=k-1$,  $\YY^\prime=\F_{\emptyset,n}$.  If $|\Set{T^1}|=k$, the AF $\YY^\prime$:\begin{enumerate}
\item has domain
$$D_{\YY^\prime}=\prod_{n-n^\prime_k< j\leq n}U_{(k-1,j)}\prod_{k\leq i\leq n-n^\prime_k}U_{(i,\max{(\Set{T^1})})}$$where $n_k^\prime$ is the number such that the lower right corner block of the Levi of the $GL_n$-parabolic subgroup with unipotent radical $D_{\db{\F}}$ is isomorphic to $GL_{n_k^\prime}$;
\item is nontrivial only on one root group, which is $U_{(k-1,\max{(\Set{T^1})})}.$ 
\end{enumerate} 
We define $\Xi^\prime(\db{\F})=\xi\vee\sI[,k-1](\db{\F})$ where $\xi$ is the $\AAnk$-path  explained as follows. If $|\Set{T^1}|=k-1$, it is the initial $\e$-subpath of $\sXi(\siota[,k-1](\db{\F}))$ along constant terms, having $n_k$ edges. If $|\Set{T^1}|=k$, it is obtained by  applying:
$$\exchange{U_{(n-n^\prime_k+i,\max{(\Set{T^1})})}}{U_{(k-1,n-n^\prime_k+i)}}\qquad\text{ for }1\leq i\leq n_k^\prime-n_k, $$\textbf{starting} from the left one and moving rightwards.
The proof is finished. We close with a summary of the path $\Xi(\F)$, given in the following picture:
\begin{multline}\underbrace{\db{\F}\overset{\sI[,k-1](\db{\F})}{\sim}\db{(w_{k-1}(\F^{T^1}))}\circ\YY^\prime\circ\JJ_{k-1}\rii{\xi}{} \db{(w_{k-1}(\F^{T^1}))}\circ\YY\circ\JJ_{k-1}}_{\Xi^\prime(\db{\F})}\\\rii{\Xi(\db{(w_{k-1}(\F^{T^1}))})\circ\YY\circ \JJ_{k-1}}{}\siota[,k-1](\F)\ssim{\sI[,k-1](\F)^{-1}}{}\F  \end{multline}where (for any data $\XX,\XX^\prime,\Xi$) by $\XX\overset{\Xi}{\rightarrow}\XX^\prime$ (resp. $\XX\overset{\Xi}{\sim}\XX^\prime$) we mean that $\Xi$ is an $(\XX\rightarrow\XX^\prime )$-path (resp. $(\XX\rightarrow\XX^\prime,\eu,\co)$-path).\end{tproof}
\begin{cor}\label{laststrong}Let $\F\in\Pri[,n]\cap^\prime\BBnk[n]. $ Then for an AF $\db{\F}\in\tPridb[,n][\Om(J_\F)]\cap^\prime\BBnk[n],$ there is an $(\db{\F}\rightarrow\F,\eu,\co,GL_n,\kkk)$-path.
\end{cor}\begin{proof}
By the theorem above and Lemma \ref{conjugatehat} we obtain an $(\db{\F}\rightarrow\F,GL_n,\kkk)$-path where $\db{\F}$ is an AF in $\in\tPridb[,n][\Om(J_\F)]\cap^\prime\BBnk[n].$ From this and from (ii) (in addition to (i)) in Fact \ref{rich3} we obtain
$$\Di{D_{\db{\F}}}\leq \Di{D_\F}=\frac{\Di{\Om(J_\F)}}{2}=\frac{\Di{\Om(J_{\db{\F}})}}{2}=\Di{D_{\db{\F}}}.   $$Therefore $\Di{D_{\db{\F}}}= \Di{D_\F},$ which implies that this path has no  $\e$-steps.\end{proof}
\begin{cor}\label{laststrong2}$\Pri[,n]\cap^\prime\BBnk[n]=\Pri[,n]\cap\BBnk[n]$.
\end{cor}
\begin{proof}The containment $\Pri[,n]\cap^\prime\BBnk[n]\subseteq\Pri[,n]\cap\BBnk[n]$ is obtained from Lemma \ref{weakprime} and the previous corollary. The opposite containment is obtained (identically to the last sentence in the proof of Lemma \ref{weakprime}) from Main corollary \ref{maincor} and information for the Richardson orbit.
\end{proof}

\section{Examples of AFs \textbf{$\F$} satisfying $|\Omm{\F}|>1$}\label{compositions}\subsection{Introduction}Throughout the current section, unless more restrictively specified, $N$ is assumed to be any positive integer. 

The main result of this section consists of the Theorems \ref{thD1},\ref{thD2}, and \ref{thD3}. In them, for certain AFs $\F\in\BBnk[N]$, we describe the subset $\Omm{\F}$ of $\Om(\F)$ which is defined as follows:\begin{sdefi}\label{Ofin}Let $\F\in\AAnk[N].$ We define \begin{equation*}\Omm{\F}:=\{a\in\Om(\F):\mult{a}{\F}<\infty\}.\qedhere \end{equation*}
\end{sdefi}
\noindent{In} many cases in these theorems, $\Omm{\F}$ has at least two elements. 

\begin{sremark}\label{multremark}In the present paper, we use the concept $\Omm{\F}$ only for choices of $\F$ in $\BBnk[N].$  Hence by using Main corollary \ref{maincor} we obtain that $\Omm{\F}$ is the set of orbits occuring in Version 1 of Weak question \ref{zStrongc}.($\F$).

Also these choices of $\F$ are (trivially) conjugates of AFs defined over number fields (even the rational numbers), and hence (after such a conjugation) the equivalent (by Corollary \ref{easyrefcor}) Version 2 of the same question is also addressed. For $\F$ being any AF in $\AAnk[N]$ defined over a number field, it may worth to find refinements of Fact \ref{nonzeroeulerian} through which $\Omm{\F}$ relates to finite sums of factorizable functions.

As an example of $a\in\UUU{N}$ and $\F\in\AAnk[N]$ such that $a\in\Omm{\F} $ and $\mult{a}{\F}>1$, choose: $N=4$; $D_\F$ to be generated by $U_{(1,3)}$, $U_{(2,4)}$, and a $\ossa$-group with nonzero nondiagonal entries $(1,4)$ and $(2,3)$; $\F$ to be nontrivial only in the third among these three groups; and finally $a=(2,1^2)$.\end{sremark}

The section is organized as follows. Subsection \ref{zpre} is optional\footnote{Subsection \ref{zpre} is not necessary for any result of the paper. A possible benefit of reading it to any extent is: some understanding  of the proofs of Theorems \ref{thD1}, \ref{thD2}, and \ref{thD3}, before reading most of the notations used in them.}. In this subsection: (i) we explain how we use pictures to describe $\AAnk$-trees; (ii) by using pictures we give a simple example (Example \ref{example}) of an AF $\F$ satisfying $|\Omm{\F}|=2;$ and (iii) we present pictures for many of the $\AAnk$-trees appearing in the proofs of  Theorems \ref{thD1}, \ref{thD2}, and \ref{thD3}, when restricted to some special cases satisfying $N\leq 20$. The proof of Example \ref{example} is chosen so that it has no any prerequisites (other than notations). In Subsection \ref{zpre2} we state and prove Theorems \ref{thD1}, \ref{thD2}, and \ref{thD3}. Finally, in Subsection \ref{zpre3} we consider a family of Rankin-Selberg integrals in each of which an AF addressed in Theorem \ref{thD1} or \ref{thD2} appears after some unfolding. 

\begin{convention}[$\{c,c\}=\{c\}$, and recalling some already stated conventions]For $c$ being any object, by $\{c,c\}$ we simply mean the singleton $\{c\}$ (no multiplicities involved). The motivation is the following: some times for two orbits $c^1,c^2$ which depend on the same parameters and for almost all the values of these parameters are different, we can use the notation $\{c^1,c^2\}$ even for the values that are left.

As part of Convention \ref{omitn}, throughout the section we denote $\JJ_r^N$ and $U_{N,(i,j)}$ respectively by $\JJ_r$ and $U_{(i,j)}$ (where $r,i,j$ are any numbers). Similarly in the definition of $\ossa$-groups (Definition \ref{ossagroup}) $n$ is replaced by $N$.

Finally recall the assumption in Definition \ref{ovtg}.\end{convention}
Observation \ref{diagonal} and similar ones are used without mention.
\subsection{Expressing $\AAnk$-trees with pictures}\label{zpre}
 Those who will chose to read the current (optional) subsection should not be discouraged from occasional mentions (for example definitions) of the next subsection. Colors are used in the pictures, but the information is retained in black-white copies.

\begin{sdefi}[$\mathrm{rg}(\F)$]\label{rgop}Let $\F\in\AAnk[N]. $ By $\mathrm{rg}(\F)$ we denote the restriction of $\F$ on the biggest algebraic subgroup of $D_\F$ which is generated by root groups. \end{sdefi}

 \begin{sremark}. A very related notation to ``$\mathrm{rg}(\F)$" which is used in the next subsection is ``$\FF$". . \end{sremark}
\begin{sdefi}[$H$-minimal]\label{VpV}Let $H$ be an algebraic subgroup of $GL_N.$ Let $V$ be a $\ossa$-group contained in $H.$ 

Let $R$ be a set of root groups so that each of them is nontrivial on a matrix entry on which $V$ is nontrivial. For $\rossa{V}$ being the smallest root generated group containing $V,$ let $p:\rossa{V}\rightarrow \rossa{V}$ be the group homomorphism which restricts to the identity on every root group in $R$ and is trivial on every other root group contained in $\rossa{V}.$

We say that $V$ is $H$-minimal, if and only if: for every choice of $R$ as in the previous paragraph, the $\ossa$-group $p(V)$ is not contained in $H.$ 
 \end{sdefi} 
 The rules below are always assumed to refer to the pictures of the current subsection;  for the other pictures of the paper (Observation \ref{unipc}), these rules are only partially used.
 
We fix (until the first occurence of ``$\triangle$" below) any picture. Then for a positive integer $N$, in the lower left corner of the picture we see several (some times only one)  AFs in $\AAnk[N]. $ We say that these AFs are the labels of the picture. The uppermost label, say $\F$, is the one most directly described by the picture. 

 Let $N^\prime$ be the number of entries in the first\footnote{We use the first row only because some times in the other rows we omit the first entry.} row of the picture. The uppermost label is of the form $\F=\F^{\{i:1\leq i\leq N-N^\prime\}}\circ\JJ_{N-N^\prime+1}$. Note that some times $N=N^\prime$ (recall that $\F^\emptyset:=\F$ and $\JJ_{1}:=\F_{\emptyset,N}$). 
 
 Let $i\geq 0$ and $j\geq 0$ be such that $D_{\F}$ is nontrivial on the $(N-N^\prime+i,N-N^\prime+j)$ matrix entry (which corresponds to the $(i,j)$ entry in the picture).  There is a unique $\mathrm{Stab}_{D_\F}(\mathrm{rg}(\F))$-minimal $\ossa$-group which is nontrivial on this matrix entry, and we call it $V_{(i,j)}$. The $(i,j)$ entry in the picture contains a disc (resp. ring, triangle) if and only if (i) (resp. (ii), (iii)) below holds:\begin{enumerate}
 \item[(i)] $V_{(i,j)}$ is a root group;
 \item [(ii)] $V_{(i,j)}$ is  $\mathrm{Stab}_{GL_N}(\mathrm{rg}(\F))$-minimal, and is not a root group;
 \item[(iii)] $V_{(i,j)}$ is not $\mathrm{Stab}_{GL_N}(\mathrm{rg}(\F))$-minimal.
 \end{enumerate}
The color of the  disc, ring, or triangle, inside the $(i,j)$ entry in the picture is black or gray. It is black if and only if  $\F(V_{(i,j)})\neq\{0\}$. \hfill$\triangle$fixing a picture

Each two successive pictures are seperated by nothing, by a dot, or by ``$\ssquare$".

For this paragraph fix any two successive pictures separated by nothing; then there is a $(\eu,\co)$-path  with input (resp. output) AF being the uppermost label of the first (resp. second) picture, which we describe next. This $(\eu,\co)$-path consists of at most one $\co$-step, which (if it exists) occurs at the end and is given by a conjugation by an element in $W_N$ (where $N$ is chosen as previously and is the same for both pictures). To discern more quickly the choice of conjugation, we number the entries of the main diagonal. As for the $\eu$-steps of this path they are given by operations of the form $\eu(X_i^j,Y_i^j)$ for $1\leq i\leq s$ and $1\leq j\leq t$ where:\begin{enumerate}
\item $s,t$ are two positive integers;
\item For all such $i$ and $j,$ $X_i^j$  is root group, and $Y_i^j$ is a $\ossa$-group (in many cases a root group as well);
\item $[X_i^j,Y_i^j]$ is a root group depending only on $j$.
\end{enumerate}
In the first picture:\begin{enumerate}
\item we fill with cyan (resp. orange lines) the entries on which $X_i^j$ (resp $Y_i^j$) are nontrivial;
\item if $t>1,$ we label with the number $j$ the  rectangles consisting of the  entries on which: $X_i^j$ and $Y_i^j$ are nontrivial for some $i$.
\end{enumerate}

The symbol ``$\ssquare$" means that we stop giving pictures which are used in the same proof. This proof, in the first case is the proof of Example \ref{example}, and in all the other cases is the proof of a special case (one choice of $\F$) of any among the Theorems \ref{thD1},\ref{thD2}, and \ref{thD3}. The proof of Example \ref{example} is fully given in the pictures; this is not the case for any other proofs.

A dot is used in the cases that are left.

  Let $\xi$ be any of the $\e$-trees  encountered in the pictures, and let $\F$ be it's input AF. Then $\xi$ is an $\e$-quasipath  along constant terms and over $\ossa$-groups (in many cases root groups). Also, each such $\ossa $-group is $\Stab{GL_N}{\XX}$-minimal for $\XX$ being the input AF of the $\e$-step (of $\xi$) over this $\ossa$-group. We describe $\xi$ by using   blue very thick  lines\footnote{For the black-white copy user, these are easily distinguishable by being by far the thickest lines.} for the edges of the rectangles consisting of the entries these groups are nontrivial. In some cases the lines are dashed (and a lighter blue is used); we do this when $\xi$ is not involved in a use of Exchange corollary \ref{exchange}. These lines (both continuous and dashed) appear in the picture with uppermost label $\F.$ The other labels of the same picture are output AFs of $\xi$ which form a set, say $\XXX$, such that\footnote{A choice of $\Xi$ in Definition \ref{lrif} (for the current choice of $\F$) is of the form $\I_{\XX}\vee_\XX\xi$, where $\XX$ varies over the output AFs of $\xi$ and each $\I_\XX$ is a $(\XX,\eu,\co)$-path; also $\I_{\XX}$ contains $\eu$-steps only in the proof of Property 2.1.} $\F\rif\XXX$ (this is explained in Definition \ref{lrif}). In contrast to $(\eu,\co)$-paths, there are no other pictures with uppermost label  an AF of $\xi.$ 

To avoid notation gaps one can skip the current and the next paragraph, and come back to them after reading from the proof of Theorem \ref{thD1}: the introduction (that is, until  Part 1 starts), and Subpart 1.2. The $\e$-quasipaths described by the rectangles with continuous lines  numbered by a number $r,$ are used to obtain Property 2 for the same value of $r$, which in turn is needed in the $r$-th use of Exchange corollary \ref{exchange}. The $\e$-quasipath with rectangles labeled with ``$\emptyset$" are ones not contributing to $\Omm{\F}$ for the reasons explained in Property 3.5.

 Finally, close to the end of the current subsection we give two pairs of succesive pictures  in which red horizontal lines appear (in the midle height of the entries they are crossing). With them, we give an alternative presentation, for the contribution in $\Omm{\F}$ in Theorems \ref{thD1},\ref{thD2}, and \ref{thD3}, from always choosing ``A" in each use of the Corollary \ref{exchange} exchange.  Conjugations in these two pairs of pictures are deferred completely for the end. The red lines are crossing the entries that are ``removed"\footnote{More precisely, let $T_n[a_1]$ be as in Theorem \ref{thD3} and in Definition \ref{VpV}, for $V$ be any $\ossa$-group contained in $D_\F$ intersecting nontrivially an entry in the $x$-th row for an $x\in\Set{T_n[a_1]}$. The replacement of $V$  by an Exchange corollary \ref{exchange} is a group $p(V)$ where $p$ is chosen as in Definition \ref{VpV} so that $p(V)$ is a root group corresponding to an entry which is not crossed  by any red line.}.  

The example below is the special case of Theorem \ref{thD1} for $[n\tlarrow 3,k\tlarrow 2,l\tlarrow 2]$. In the proof of this example we: 
\begin{itemize}
\item[(i)]  completelly follow $\sXi$; 
\item[(ii)] do not use any prerequisites (namely no use of Main corollary \ref{maincor}).
\end{itemize}Each of (i) and (ii) make the proof of this example different from being the restriction of the proof of  Theorem \ref{thD1} for $[n\tlarrow 3,k\tlarrow 2,l\tlarrow 2]$.
Also note that the successive pictures in this example explain fewer $\AAnk$-steps at a time.
 \begin{sexample}\label{example} Let $\F$ be an AF in $\AAnk[6]$ for which there are elements in $\KKK-\{0\}$ we call $a_1,a_2$, $a_3,b_1,b_2,b_3,c_1,c_2,$ such that
 \begin{equation}\label{z323}\F\left(\begin{pmatrix}1&x&z&x_{1,1}&x_{1,2}&x_{1,3}\\&1&y&x_{2,1}&x_{2,2}&x_{2,3}\\&&1&x_{3,1}&x_{3,2}&x_{3,3}\\&&&1&b_1x&b_3z\\&&&&1&b_2y\\&&&&&1
  \end{pmatrix}
  \right)=c_1 x+c_2 y+a_1 x_{1,1}+a_2 x_{2,2}+a_3x_{3,3},\end{equation}where $D_\F$ is given by the matrices in which $\F$ is evaluated in (\ref{z323}) for $x,y,z,x_{1,1},x_{1,2},...,x_{3,3}$ taking any value in $\KKK$. Based on the picture rules we see that $\F$ is described by the first picture inside the proof.  There is an $(\F\rightarrow\Prink[6] ,P_6)$-tree, for the output vertices of which, one is labeled with an AF in $\Prink[6][4,1^2] $, an other one is labeled with an AF in $\Prink[6][3,3], $  and all the other ones are labeled with AFs in $\bigcup_{b>\max\{[4,1^2],[3,3]\} } \Prink[6][b]$ (equivalently $\Om(\F)=\Omm{\F}=\{a\in\UUU{6}:\mult{a}{\F} =1\}=\{(4,1^2), (3,3)\}$).
 \end{sexample}
 \begin{proof}[\textbf{Proof} (of the sentence above)] Starting with $\F,$ we successively apply $$\eu(U_{(1,3)},U_{(3,4)})\qquad\text{ and }\qquad \eu(U_{(1,2)},U_{(2,4)}) $$ and let $\F^\prime$ be the output AF. Next we conjugate with the element $w\in W_{6}$ determined by the picture with label $w\F^\prime$. Next we successively apply $$\eu(wU_{(4,3)}w^{-1},wU_{(3,5)}w^{-1})\qquad\text{ and }\qquad\eu(wU_{(4,2)}w^{-1},wU_{(2,5)}w^{-1}),$$ and then  conjugate with the (minimal length again) element in $W_6$ giving the AF $\F^{\prime\prime} $ defined by  the picture labeled with $\F^{\prime\prime}$. Let $w^\prime$ be the product of the two elements with which we have conjugated so far. In $\F^{\prime\prime},$ we apply the $\e$-operation over $w^{\prime}U_{(5,6)}{w^\prime}^{-1}$. We obtain two terms $\Z_A$ and $\Z_B$ as in the pictures below; note that all the other terms share the same picture, and let $\XX$ be one of them.   By applying to $\Z_B$ the $\e$-operation over $w^\prime U_{(5,3)}{w^{\prime}}^{-1}$ we obtain AFs in $\Prink[6],$ the constant term among them belonging to $\Prink[6][4,1^2]$ and the other terms belonging in $ \Prink[6][5,1].$ By applying $$\exchange{w^\prime U_{(5,3)}{w^\prime}^{-1}}{w^\prime U_{(3,6)}{w^\prime}^{-1}}$$ to $\Z_B$ and to $\XX$, we respectively obtain AFs in $\Prink[6][3,3]$ and in $\Prink[6][4,2].$

\noindent\begin{ti}
 \rr{2}{4}{2} 
 \rl{1}{2}{2}
 \grid{6}
  \foreach \x in {1,2,...,6}
 {\dia{\x}{\x}}
 \foreach \x/\y in {2/d,3/c,4/b,5/a,6/a}
 {\ci{\x}{1}{\y}}
 \foreach \x/\y in {3/d,4/a,5/b,6/a}
 {\ci{\x}{2}{\y}}
 \foreach \x/\y in {4/a,5/a,6/b}
 {\ci{\x}{3}{\y}}
 \foreach \x/\y in {5/d,6/c}
 {\ci{\x}{4}{\y}}
 \foreach \x/\y in {6/d}
 {\ci{\x}{5}{\y}}
 \fir{\F}
 \end{ti}\begin{ti}
 \grid{6}
 \foreach \x in {1,2,...,6}
 {\dia{\x}{\x}}
 \foreach \x/\y in {2/a,3/a,4/b,5/a,6/a}
 {\ci{\x}{1}{\y}}
 \foreach \x/\y in {3/d,5/b,6/a}
 {\ci{\x}{2}{\y}}
 \foreach \x/\y in {5/a,6/b}
 {\ci{\x}{3}{\y}}
 \foreach \x/\y in {5/b,6/a}
 {\ci{\x}{4}{\y}}
 \foreach \x/\y in {6/d}
 {\ci{\x}{5}{\y}}
 \fir{\F^\prime}
 \end{ti}\begin{ti}
 \rl{2}{3}{2}
 \rr{3}{5}{2}
 \grid{6}
 \foreach \x/\y in {1/1,2/4,3/2,4/3,5/5,6/6}
 {\dia{\x}{\y}}
 \foreach \x/\y in {2/b,3/a,4/a,5/a,6/a} 
 {\ci{\x}{1}{\y}}
 \foreach \x/\y in {5/b,6/a} 
 {\ci{\x}{2}{\y}}
 \foreach \x/\y in {4/d,5/b,6/a} 
 {\ci{\x}{3}{\y}}
 \foreach \x/\y in {5/a,6/b} 
 {\ci{\x}{4}{\y}}
 \foreach \x/\y in {6/d} 
 {\ci{\x}{5}{\y}}
 \fir{w\F^{\prime}}
 \end{ti}\begin{ti}

 \grid{6}
 \rbr{3}{6}{1}
 \foreach \x/\y in {1/1,2/4,3/5,4/2,5/3,6/6}
 {\dia{\x}{\y}}
 \foreach \x/\y in {2/b,3/a,4/a,5/a,6/a} 
 {\ci{\x}{1}{\y}}
 \foreach \x/\y in {3/b,4/a,5/a,6/a} 
 {\ci{\x}{2}{\y}}
 \foreach \x/\y in {6/f} 
 {\ci{\x}{3}{\y}}
 \foreach \x/\y in {5/f,6/a} 
 {\ci{\x}{4}{\y}}
 \foreach \x/\y in {6/b} 
 {\ci{\x}{5}{\y}}
 \fir{\F^{\prime\prime}}
 \end{ti}.  
 \begin{ti}
 \grid{6}
 \rbr{3}{5}{1}
 \foreach \x/\y in {1/1,2/4,3/5,4/2,5/3,6/6}
 {\dia{\x}{\y}}
 \foreach \x/\y in {2/b,3/a,4/a,5/a,6/a} 
 {\ci{\x}{1}{\y}}
 \foreach \x/\y in {3/b,4/a,5/a,6/a} 
 {\ci{\x}{2}{\y}}
 \foreach \x/\y in {6/a} 
 {\ci{\x}{3}{\y}}
 \foreach \x/\y in {5/b,6/a} 
 {\ci{\x}{4}{\y}}
 \foreach \x/\y in {6/b} 
 {\ci{\x}{5}{\y}}
 \fir{\Z_B}
 \end{ti}. 
 \begin{ti}
  \grid{6}
  \rr{5}{6}{1}\rl{3}{5}{1}
  \foreach \x/\y in {1/1,2/4,3/5,4/2,5/3,6/6}
  {\dia{\x}{\y}}
  \f{1}{2/b,3/a,4/a,5/a,6/a} 
   \f{4}{3/b,4/a,5/a,6/a} 
  \f{5}{6/b} 
   \f{2}{5/a,6/a} 
   \f{3}{6/b} 
  \fir{\Z_A}
  \end{ti} \begin{ti}
   \grid{6}
    \foreach \x/\y in {1/1,2/4,3/5,4/2,5/3,6/6}
   {\dia{\x}{\y}}
   \f{1}{2/b,3/a,4/a,5/a,6/a} 
    \f{4}{3/b,4/a,5/a,6/a} 
   \f{5}{5/a,6/b} 
    \f{2}{5/a,6/a} 
   
   \end{ti}.
 \begin{ti}
  \grid{6}
  \rr{5}{6}{1}\rl{3}{5}{1}
  \foreach \x/\y in {1/1,2/4,3/5,4/2,5/3,6/6}
  {\dia{\x}{\y}}
  \f{1}{2/b,3/a,4/a,5/a,6/a} 
   \f{4}{3/b,4/a,5/a,6/a} 
  \f{5}{6/b} 
   \f{2}{5/b,6/a} 
   \f{3}{6/b} 
  \fir{\XX}
  \end{ti}
  
 \begin{ti}
   \grid{6}
      \foreach \x/\y in {1/1,2/4,3/5,4/2,5/3,6/6}
   {\dia{\x}{\y}}
   \f{1}{2/b,3/a,4/a,5/a,6/a} 
    \f{4}{3/b,4/a,5/a,6/a} 
   \f{5}{5/a,6/b} 
    \f{2}{5/b,6/a} 
   
   \end{ti}$\ssquare$\end{proof}
\begin{sremark}If instead we had fully followed the special case $[n\tlarrow 3,k\tlarrow 2,l\tlarrow 2]$ of the proof we give for Theorem \ref{thD1}, and had  reserved succesions of pictures seperated by nothing, for bigger $(\eu,\co)$-paths (this is what we usually do), the two first pictures would have been:

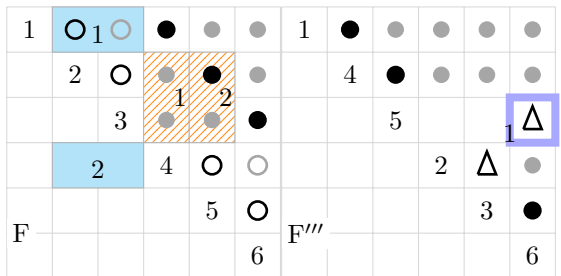
\vspace{1mm}
  \begin{ti}\grid{6}
 \rr[1]{2}{4}{2}\rr[2]{2}{5}{2}
 \rl[1]{1}{2}{2}\rl[2]{4}{2}{2}
 
 \foreach \x in {1,2,...,6}
 {\dia{\x}{\x}}
 \foreach \x/\y in {2/d,3/c,4/b,5/a,6/a}
 {\ci{\x}{1}{\y}}
 \foreach \x/\y in {3/d,4/a,5/b,6/a}
 {\ci{\x}{2}{\y}}
 \foreach \x/\y in {4/a,5/a,6/b}
 {\ci{\x}{3}{\y}}
 \foreach \x/\y in {5/d,6/c}
 {\ci{\x}{4}{\y}}
 \foreach \x/\y in {6/d}
 {\ci{\x}{5}{\y}}
 \fir{\F}
 \end{ti}\begin{ti}\grid{6}
 
  \rn[1]{3}{6}{1}
  \foreach \x/\y in {1/1,2/4,3/5,4/2,5/3,6/6}
  {\dia{\x}{\y}}
  \foreach \x/\y in {2/b,3/a,4/a,5/a,6/a} 
  {\ci{\x}{1}{\y}}
  \foreach \x/\y in {3/b,4/a,5/a,6/a} 
  {\ci{\x}{2}{\y}}
  \foreach \x/\y in {6/f} 
  {\ci{\x}{3}{\y}}
  \foreach \x/\y in {5/f,6/a} 
  {\ci{\x}{4}{\y}}
  \foreach \x/\y in {6/b} 
  {\ci{\x}{5}{\y}}
  \fir{\F^{\prime\prime\prime}}
  \end{ti}.\end{sremark}

Next we give pictures describing $\AAnk$-trees appearing in the proof of Theorems \ref{thD1},\ref{thD2} and \ref{thD3}. The names of the labels are defined later; more precisely, $\F_{n,k,l}$ and $\F_{a,k}$ are respectively defined in Definitions \ref{zfnkl} and \ref{ladefi}, and the other labels are defined inside the proofs of the theorems.

\renewcommand{\di}{15}

\noindent\begin{ti}[0.4]
\grid{15}
\rl[1]{1}{2}{3}
\rr[1]{2}{5}{3}
\rr[2]{6}{9}{3}\rl[2]{5}{6}{3}\rr[2]{2}{9}{3}\rl[2]{5}{2}{3}
\rr[3]{6}{10}{3}\rl[3]{9}{6}{3}\rr[3]{2}{10}{3}\rl[3]{9}{2}{3}
\rr[4]{6}{13}{3}\rl[4]{10}{6}{3}\rr[4]{2}{13}{3}\rl[4]{10}{2}{3}
\rr[4]{11}{13}{2}\rl[4]{10}{11}{2}
\f{1}{2/d,3/c,4/c,5/b,6/a,7/a,8/a,9/a,10/a,11/a,12/a,13/a,14/a,15/a}
\f{2}{3/d,4/c,5/a,6/b,7/a,8/a,9/a,10/a,11/a,12/a,13/a,14/a,15/a} 
\f{3}{4/d,5/a,6/a,7/b,8/a,9/a,10/a,11/a,12/a,13/a,14/a,15/a}
\f{4}{5/a,6/a,7/a,8/b,9/a,10/a,11/a,12/a,13/a,14/a,15/a}
\f{5}{6/d,7/c,8/c,9/b,10/a,11/a,12/a,13/a,14/a,15/a}
\f{6}{7/d,8/c,9/a,10/b,11/a,12/a,13/a,14/a,15/a}
\f{7}{8/d,9/a,10/a,11/b,12/a,13/a,14/a,15/a}
\f{8}{9/a,10/a,11/a,12/b,13/a,14/a,15/a}
\f{9}{10/d,11/c,12/c,13/a,14/a,15/a}
\f{10}{11/d,12/c,13/b,14/a,15/a}
\f{11}{12/d,13/a,14/b,15/a}
\f{12}{13/a,14/a,15/b}
\f{13}{14/d,15/c}
\f{14}{15/d}
\f{15}{}
\draw[line width=1 pt]\mm{0}{4}--\mm{4}{4}--\mm{4}{8}--\mm{8}{8}--\mm{8}{12}--\mm{12}{12}--\mm{12}{15};
\fir{\F_{4,4,3}}
\end{ti}\begin{ti}[0.4]\grid{15}  
\rn[1]{5}{14}{2}
\f{1}{2/b,3/a,4/a,5/a,6/a,7/a,8/a,9/a,10/a,11/a,12/a,13/a,14/a,15/a}
\f{5}{3/b,4/a,5/a,6/a,7/a,8/a,9/a,10/a,11/a,12/a,13/a,14/a,15/a}
\f{9}{4/b,5/a,6/a,7/a,8/a,9/a,10/a,11/a,12/a,13/a,14/a,15/a}
\f{10}{5/b,6/a,7/a,8/a,9/a,10/a,11/a,12/a,13/a,14/a,15/a}
\f{13}{14/f,15/e}
\f{2}{7/f,8/e,9/b,10/a,11/a,12/a,13/a,14/a,15/a}
\f{3}{8/d,9/a,10/b,11/a,12/a,13/a,14/a,15/a}
\f{4}{9/a,10/a,11/b,12/a,13/a,14/a,15/a}
\f{6}{10/f,11/e,12/a,13/a,14/a,15/a}
\f{7}{11/d,12/b,13/a,14/a,15/a}
\f{8}{12/a,13/b,14/a,15/a}
\f{11}{13/d,14/b,15/a}
\f{12}{14/a,15/b}
\f{14}{15/d}
\f{15}{}
\draw[line width= 0.7pt]\mm{5}{8}--\mm{8}{8}--\mm{8}{11}--\mm{11}{11}--\mm{11}{13}--\mm{13}{13}--\mm{13}{15};
\fir[2.5]{\siota(\F_{4,4,3})=\Z_0\circ\JJ_{5},}
\fir{\Z_{0,A}\circ\JJ_{5},\hsp\Z_{0,B}\circ\JJ_5}
\end{ti}.
\renewcommand{\di}{11}

\noindent\begin{ti}[0.4]
\draw[step=1,\gray](0,11)grid(1,10);
\draw[step=1,\gray](1,0)grid(11,11);
\rn[1]{1}{10}{2}
\f{13}{}
\f{2}{3/d,4/c,5/b,6/a,7/a,8/a,9/a,10/a,11/a}
\f{3}{4/d,5/a,6/b,7/a,8/a,9/a,10/a,11/a}
\f{4}{5/a,6/a,7/b,8/a,9/a,10/a,11/a}
\f{6}{6/d,7/c,8/a,9/a,10/a,11/a}
\f{7}{7/d,8/b,9/a,10/a,11/a}
\f{8}{8/a,9/b,10/a,11/a}
\f{11}{9/d,10/b,11/a}
\f{12}{10/a,11/b}
\f{14}{11/d}
\f{15}{}
\draw[line width=0.7 pt]\mm{1}{4}--\mm{4}{4}--\mm{4}{7}--\mm{7}{7}--\mm{7}{9}--\mm{9}{9}--\mm{9}{11};
\fir[2.5]{\jj\QQ_1\circ\JJ_5,}\fir{\Z_{0,B}\circ\JJ_5}
\end{ti}.
\begin{ti}[0.4]
\draw[step=1,\gray](0,11)grid(1,10);
\draw[step=1,\gray](1,0)grid(11,11);
\rn[1]{2}{3}{2}
\rn[1]{5}{6}{2}
\rl{1}{2}{8}
\rr{2}{10}{8} 
\f{13}{10/b,11/a}
\f{2}{5/b,6/a,7/a,8/a,9/a,10/a,11/a}
\f{3}{4/d,5/a,6/b,7/a,8/a,9/a,10/a,11/a}
\f{4}{5/a,6/a,7/b,8/a,9/a,10/a,11/a}
\f{6}{8/a,9/a,10/a,11/a}
\f{7}{7/d,8/b,9/a,10/a,11/a}
\f{8}{8/a,9/b,10/a,11/a}
\f{11}{9/d,10/b,11/a}
\f{12}{10/a,11/b}
\f{14}{11/d}
\f{15}{}
\draw[line width=1 pt]\mm{1}{4}--\mm{4}{4}--\mm{4}{7}--\mm{7}{7}--\mm{7}{9}--\mm{9}{9}--\mm{9}{11};
\fir[2.5]{\Z_1^{-}\circ\JJ_5,}
\fir{\Z_{0,A}\circ\JJ_5}
\end{ti}
\begin{ti}[0.4]
\draw[step=1,\gray](0,11)grid(1,10);
\draw[step=1,\gray](1,0)grid(11,11);
\rn[2]{2}{11}{1}
\f{13}{2/b,3/a,4/a,5/a,6/a,7/a,8/a,9/a,10/a,11/a}
\f{14}{11/f}
\f{2}{6/b,7/a,8/a,9/a,10/a,11/a}
\f{3}{5/f,6/a,7/b,8/a,9/a,10/a,11/a}
\f{4}{6/a,7/a,8/b,9/a,10/a,11/a}
\f{6}{9/a,10/a,11/a}
\f{7}{8/f,9/b,10/a,11/a}
\f{8}{9/a,10/b,11/a}
\f{11}{10/f,11/a}
\f{12}{11/b}
\f{15}{}
\draw[line width=1 pt] (11,1)--(10,1)--(10,3)--(8,3)--(8,6)--(5,6)--(5,9);
\fir[2.5]{\siota(\Z_1^{-}\circ\JJ_5)=\Z_1\circ\JJ_6,} 
\fir{\Z_{1,A}\circ\JJ_6,\hsp\Z_{1,B}\circ\JJ_6}
\end{ti}.
\renewcommand{\di}{10}
\noindent\begin{ti}[0.4]
\draw[step=1,\gray](0,10)grid(1,9);
\draw[step=1,\gray](1,0)grid(10,10);
\rn[2]{1}{10}{1}
\rl{2}{3}{2}
\rr{3}{5}{2}
\f{14}{}
\f{2}{5/b,6/a,7/a,8/a,9/a,10/a}
\f{3}{4/d,5/a,6/b,7/a,8/a,9/a,10/a}
\f{4}{5/a,6/a,7/b,8/a,9/a,10/a}
\f{6}{8/a,9/a,10/a}
\f{7}{7/d,8/b,9/a,10/a}
\f{8}{8/a,9/b,10/a}
\f{11}{9/d,10/a}
\f{12}{10/b}
\f{15}{}
\draw[line width=1 pt] (10,1)--(9,1)--(9,3)--(7,3)--(7,6)--(4,6)--(4,9);
\fir[2.5]{\jj\QQ_2\circ\JJ_6,} 
\fir{\Z_{1,B}\circ\JJ_6}
\end{ti}
\begin{ti}[0.4]
\draw[step=1,\gray](0,10)grid(1,9);
\draw[step=1,\gray](1,0)grid(10,10);
\rbr{3}{4}{4}
\f{14}{}
\f{2}{3/b,4/a,5/a,6/a,7/a,8/a,9/a,10/a}
\f{6}{8/a,9/a,10/a}
\f{3}{5/d,6/b,7/a,8/a,9/a,10/a}
\f{4}{6/a,7/b,8/a,9/a,10/a}
\f{7}{7/d,8/b,9/a,10/a}
\f{8}{8/a,9/b,10/a}
\f{11}{9/d,10/a}
\f{12}{10/b}
\f{15}{}
\draw[line width=1 pt] (10,1)--(9,1)--(9,3)--(7,3)--(7,5)--(5,5)--(5,7);
\fir[2.5]{j_2(\siota(\QQ_2))\rcirc\JJ_6,} 
\fir{\text{{\tiny $($right hand side of (\ref{rhs})$)\rcirc \JJ_6$} }}
\end{ti}.\begin{ti}[0.4]
\draw[step=1,\gray](0,10)grid(1,9);
\draw[step=1,\gray](1,0)grid(10,10);
\rn[2]{3}{4}{1}
\rn[3]{6}{7}{1} 
\rn[3]{8}{9}{1} 
\rr{2}{10}{8}
\rl{1}{2}{8}

\f{14}{10/b}
\f{2}{5/b,6/a,7/a,8/a,9/a,10/a}
\f{3}{5/a,6/b,7/a,8/a,9/a,10/a}
\f{4}{5/a,6/a,7/b,8/a,9/a,10/a}
\f{6}{8/a,9/a,10/a}
\f{7}{8/b,9/a,10/a}
\f{8}{8/a,9/b,10/a}
\f{11}{10/a}
\f{12}{10/b}
\f{15}{}
\draw[line width=1 pt] (10,1)--(9,1)--(9,3)--(7,3)--(7,6)--(4,6)--(4,9);
\fir[2.5]{\Z_2^{-}\circ\JJ_6,}
\fir{\Z_{1,A}\circ\JJ_6}
\end{ti}

\noindent\begin{ti}[0.4]
\draw[step=1,\gray](0,10)grid(1,9);
\draw[step=1,\gray](1,0)grid(10,10);

\f{14}{2/b,3/a,4/a,5/a,6/a,7/a,8/a,9/a,10/a}
\f{15}{}
\f{2}{6/b,7/a,8/a,9/a,10/a}
\f{3}{6/a,7/b,8/a,9/a,10/a}
\f{4}{6/a,7/a,8/b,9/a,10/a}
\f{6}{9/a,10/a}
\f{7}{9/b,10/a}
\f{8}{9/a,10/b}
\f{11}{}
\f{12}{}
\draw(10,2)--(8,2)--(8,5)--(5,5)--(5,8);
\fir[0.6]{\siota(\Z_2^{-}\circ\JJ_6)=\Z_2\circ\JJ_7}
\end{ti}$\ssquare$\renewcommand{\di}{20}

\noindent\begin{ti}[0.35]\grid{20}
\rl[1]{1}{2}{6}
\rr[1]{2}{8}{6}
\rr[2]{2}{9}{6}\rl[2]{8}{2}{6}
\rr[3]{2}{15}{6}\rl[3]{9}{2}{6}\rr[3]{10}{15}{5}\rl[3]{9}{10}{5}
\f{1}{2/d,3/c,4/c,5/c,6/c,7/c,8/b,9/a,10/a,11/a,12/a,13/a,14/a,15/a,16/a,17/a,18/a,19/a,20/a}
\f{2}{3/d,4/c,5/c,6/c,7/c,8/a,9/b,10/a,11/a,12/a,13/a,14/a,15/a,16/a,17/a,18/a,19/a,20/a}
\f{3}{4/d,5/c,6/c,7/c,8/a,9/a,10/b,11/a,12/a,13/a,14/a,15/a,16/a,17/a,18/a,19/a,20/a}
\f{4}{5/d,6/c,7/c,8/a,9/a,10/a,11/b,12/a,13/a,14/a,15/a,16/a,17/a,18/a,19/a,20/a}
\f{5}{6/d,7/c,8/a,9/a,10/a,11/a,12/b,13/a,14/a,15/a,16/a,17/a,18/a,19/a,20/a}
\f{6}{7/d,8/a,9/a,10/a,11/a,12/a,13/b,14/a,15/a,16/a,17/a,18/a,19/a,20/a}
\f{7}{8/a,9/a,10/a,11/a,12/a,13/a,14/b,15/a,16/a,17/a,18/a,19/a,20/a}
\f{8}{9/d,10/c,11/c,12/c,13/c,14/c,15/a,16/a,17/a,18/a,19/a,20/a}
\f{9}{10/d,11/c,12/c,13/c,14/c,15/b,16/a,17/a,18/a,19/a,20/a}
\f{10}{11/d,12/c,13/c,14/c,15/a,16/b,17/a,18/a,19/a,20/a}
\f{11}{12/d,13/c,14/c,15/a,16/a,17/b,18/a,19/a,20/a}
\f{12}{13/d,14/c,15/a,16/a,17/a,18/b,19/a,20/a}
\f{13}{14/d,15/a,16/a,17/a,18/a,19/b,20/a}
\f{14}{15/a,16/a,17/a,18/a,19/a,20/b}
\f{15}{16/d,17/c,18/c,19/c,20/c}
\f{16}{17/d,18/c,19/c,20/c}
\f{17}{18/d,19/c,20/c}
\f{18}{19/d,20/c}
\f{19}{20/d}
\f{20}{}
\draw[line width=1 pt]\mm{0}{7}--\mm{7}{7}--\mm{7}{14}--\mm{14}{14}--\mm{14}{20};
\fir{\F_{7,3,2}}
\end{ti}\begin{ti}[0.35]\grid{20}
\rn[1]{4}{16}{5}
\f{1}{2/b,3/a,4/a,5/a,6/a,7/a,8/a,9/a,10/a,11/a,12/a,13/a,14/a,15/a,16/a,17/a,18/a,19/a,20/a}
\f{8}{3/b,4/a,5/a,6/a,7/a,8/a,9/a,10/a,11/a,12/a,13/a,14/a,15/a,16/a,17/a,18/a,19/a,20/a}
\f{9}{4/b,5/a,6/a,7/a,8/a,9/a,10/a,11/a,12/a,13/a,14/a,15/a,16/a,17/a,18/a,19/a,20/a}
\f{15}{16/f,17/e,18/e,19/e,20/e}
\f{2}{6/f,7/e,8/e,9/e,10/e,11/a,12/a,13/a,14/a,15/a,16/a,17/a,18/a,19/a,20/a}
\f{3}{7/d,8/c,9/c,10/c,11/b,12/a,13/a,14/a,15/a,16/a,17/a,18/a,19/a,20/a}
\f{4}{8/d,9/c,10/c,11/a,12/b,13/a,14/a,15/a,16/a,17/a,18/a,19/a,20/a}
\f{5}{9/d,10/c,11/a,12/a,13/b,14/a,15/a,16/a,17/a,18/a,19/a,20/a}
\f{6}{10/d,11/a,12/a,13/a,14/b,15/a,16/a,17/a,18/a,19/a,20/a}
\f{7}{11/a,12/a,13/a,14/a,15/b,16/a,17/a,18/a,19/a,20/a}
\f{10}{12/d,13/c,14/c,15/c,16/b,17/a,18/a,19/a,20/a}
\f{11}{13/d,14/c,15/c,16/a,17/b,18/a,19/a,20/a}
\f{12}{14/d,15/c,16/a,17/a,18/b,19/a,20/a}
\f{13}{15/d,16/a,17/a,18/a,19/b,20/a}
\f{14}{16/a,17/a,18/a,19/a,20/b}
\f{16}{17/d,18/c,19/c,20/c}
\f{17}{18/d,19/c,20/c}
\f{18}{19/d,20/c}
\f{19}{20/d}
\f{20}{}
\draw\mm{5}{10}--\mm{10}{10}--\mm{10}{15}--\mm{15}{15}--\mm{15}{20};
\fir[2.5]{\siota(\F_{7,3,2})=\Z_0\circ\JJ_4,}
\fir{\Z_{0,A}\circ\JJ_5,\hsp\Z_{0,B}\circ\JJ_5}
\end{ti}.\renewcommand{\di}{17}  
\begin{ti}[0.35] 
\draw[step=1,\gray](0,17)grid(1,16);
\draw[step=1,\gray](1,0)grid(17,17);
\rn[1]{1}{13}{5}
\f{15}{}
\f{2}{3/b,4/a,5/a,6/a,7/a,8/a,9/a,10/a,11/a,12/a,13/a,14/a,15/a,16/a,17/a}
\f{3}{4/d,5/c,6/c,7/c,8/b,9/a,10/a,11/a,12/a,13/a,14/a,15/a,16/a,17/a}
\f{4}{5/d,6/c,7/c,8/a,9/b,10/a,11/a,12/a,13/a,14/a,15/a,16/a,17/a}
\f{5}{6/d,7/c,8/a,9/a,10/b,11/a,12/a,13/a,14/a,15/a,16/a,17/a}
\f{6}{7/d,8/a,9/a,10/a,11/b,12/a,13/a,14/a,15/a,16/a,17/a}
\f{7}{8/a,9/a,10/a,11/a,12/b,13/a,14/a,15/a,16/a,17/a}
\f{10}{9/d,10/c,11/c,12/c,13/b,14/a,15/a,16/a,17/a}
\f{11}{10/d,11/c,12/c,13/a,14/b,15/a,16/a,17/a}
\f{12}{11/d,12/c,13/a,14/a,15/b,16/a,17/a}
\f{13}{12/d,13/a,14/a,15/a,16/b,17/a}
\f{14}{13/a,14/a,15/a,16/a,17/b}
\f{16}{14/d,15/c,16/c,17/c}
\f{17}{15/d,16/c,17/c}
\f{18}{16/d,17/c}
\f{19}{17/d}
\f{20}{}
\draw\mm{2}{7}--\mm{7}{7}--\mm{7}{12}--\mm{12}{12}--\mm{12}{17};
\fir[2.5]{\jj\QQ_1\circ\JJ_4,}
\fir{\Z_{0,B}\circ\JJ_4}
\end{ti}.\begin{ti}[0.35] 
\draw[step=1,\gray](0,17)grid(1,16);
\draw[step=1,\gray](1,0)grid(17,17);
\rn[1]{2}{3}{5}
\rl{1}{2}{11}
\rr{2}{13}{11}
\f{15}{13/b,14/a,15/a,16/a,17/a}
\f{2}{8/a,9/a,10/a,11/a,12/a,13/a,14/a,15/a,16/a,17/a}
\f{3}{4/d,5/c,6/c,7/c,8/b,9/a,10/a,11/a,12/a,13/a,14/a,15/a,16/a,17/a}
\f{4}{5/d,6/c,7/c,8/a,9/b,10/a,11/a,12/a,13/a,14/a,15/a,16/a,17/a}
\f{5}{6/d,7/c,8/a,9/a,10/b,11/a,12/a,13/a,14/a,15/a,16/a,17/a}
\f{6}{7/d,8/a,9/a,10/a,11/b,12/a,13/a,14/a,15/a,16/a,17/a}
\f{7}{8/a,9/a,10/a,11/a,12/b,13/a,14/a,15/a,16/a,17/a}
\f{10}{9/d,10/c,11/c,12/c,13/b,14/a,15/a,16/a,17/a}
\f{11}{10/d,11/c,12/c,13/a,14/b,15/a,16/a,17/a}
\f{12}{11/d,12/c,13/a,14/a,15/b,16/a,17/a}
\f{13}{12/d,13/a,14/a,15/a,16/b,17/a}
\f{14}{13/a,14/a,15/a,16/a,17/b}
\f{16}{14/d,15/c,16/c,17/c}
\f{17}{15/d,16/c,17/c}
\f{18}{16/d,17/c}
\f{19}{17/d}
\f{20}{}
\draw\mm{2}{7}--\mm{7}{7}--\mm{7}{12}--\mm{12}{12}--\mm{12}{17};
\fir[2.5]{\Z_1^-\circ\JJ_4,}

\fir{\Z_{0,A}\circ\JJ_4}
\end{ti} 
 
\begin{ti}[0.35] 
\draw[step=1,\gray](0,17)grid(1,16);
\draw[step=1,\gray](1,0)grid(17,17);
\rn[2]{2}{14}{4} 
\f{15}{2/b,3/a,4/a,5/a,6/a,7/a,8/a,9/a,10/a,11/a,12/a,13/a,14/a,15/a,16/a,17/a}
\f{16}{14/f,15/e,16/e,17/e}
\f{2}{9/a,10/a,11/a,12/a,13/a,14/a,15/a,16/a,17/a}
\f{3}{5/f,6/e,7/e,8/e,9/b,10/a,11/a,12/a,13/a,14/a,15/a,16/a,17/a}
\f{4}{6/d,7/c,8/c,9/a,10/b,11/a,12/a,13/a,14/a,15/a,16/a,17/a}
\f{5}{7/d,8/c,9/a,10/a,11/b,12/a,13/a,14/a,15/a,16/a,17/a}
\f{6}{8/d,9/a,10/a,11/a,12/b,13/a,14/a,15/a,16/a,17/a}
\f{7}{9/a,10/a,11/a,12/a,13/b,14/a,15/a,16/a,17/a}
\f{10}{10/f,11/e,12/e,13/e,14/a,15/a,16/a,17/a}
\f{11}{11/d,12/c,13/c,14/b,15/a,16/a,17/a}
\f{12}{12/d,13/c,14/a,15/b,16/a,17/a}
\f{13}{13/d,14/a,15/a,16/b,17/a}
\f{14}{14/a,15/a,16/a,17/b}
\f{17}{15/d,16/c,17/c}
\f{18}{16/d,17/c}
\f{19}{17/d}
\f{20}{}
\fir[2.5]{\siota(\Z_1^-\circ\JJ_4)=\Z_1\circ\JJ_5,}
\fir{\Z_{1,A}\circ\JJ_5,\hsp\Z_{1,B}\circ\JJ_5}
\end{ti}.\renewcommand{\di}{16}  \begin{ti}[0.35] 
\draw[step=1,\gray](0,16)grid(1,15);
\draw[step=1,\gray](1,0)grid(16,16);
\rn[2]{1}{13}{4} 
\f{16}{}
\f{2}{8/a,9/a,10/a,11/a,12/a,13/a,14/a,15/a,16/a}
\f{3}{4/d,5/c,6/c,7/c,8/b,9/a,10/a,11/a,12/a,13/a,14/a,15/a,16/a}
\f{4}{5/d,6/c,7/c,8/a,9/b,10/a,11/a,12/a,13/a,14/a,15/a,16/a}
\f{5}{6/d,7/c,8/a,9/a,10/b,11/a,12/a,13/a,14/a,15/a,16/a}
\f{6}{7/d,8/a,9/a,10/a,11/b,12/a,13/a,14/a,15/a,16/a}
\f{7}{8/a,9/a,10/a,11/a,12/b,13/a,14/a,15/a,16/a}
\f{10}{9/d,10/c,11/c,12/c,13/a,14/a,15/a,16/a}
\f{11}{10/d,11/c,12/c,13/b,14/a,15/a,16/a}
\f{12}{11/d,12/c,13/a,14/b,15/a,16/a}
\f{13}{12/d,13/a,14/a,15/b,16/a}
\f{14}{13/a,14/a,15/a,16/b}
\f{17}{14/d,15/c,16/c}
\f{18}{15/d,16/c}
\f{19}{16/d}
\f{20}{}
\fir[2.5]{\jj\QQ_2\circ\JJ_5,}
\fir{\Z_{1,B}\circ\JJ_5}
\end{ti}.  
 
\begin{ti}[0.35] 
\draw[step=1,\gray](0,16)grid(1,15);
\draw[step=1,\gray](1,0)grid(16,16);
\rl{1}{2}{11}
\rr{2}{13}{11}
\rn[2]{3}{4}{4} 
\rn[2]{8}{9}{4}
\f{16}{13/b,14/a,15/a,16/a}
\f{2}{8/a,9/a,10/a,11/a,12/a,13/a,14/a,15/a,16/a}
\f{3}{8/b,9/a,10/a,11/a,12/a,13/a,14/a,15/a,16/a}
\f{4}{5/d,6/c,7/c,8/a,9/b,10/a,11/a,12/a,13/a,14/a,15/a,16/a}
\f{5}{6/d,7/c,8/a,9/a,10/b,11/a,12/a,13/a,14/a,15/a,16/a}
\f{6}{7/d,8/a,9/a,10/a,11/b,12/a,13/a,14/a,15/a,16/a}
\f{7}{8/a,9/a,10/a,11/a,12/b,13/a,14/a,15/a,16/a}
\f{10}{13/a,14/a,15/a,16/a}
\f{11}{10/d,11/c,12/c,13/b,14/a,15/a,16/a}
\f{12}{11/d,12/c,13/a,14/b,15/a,16/a}
\f{13}{12/d,13/a,14/a,15/b,16/a}
\f{14}{13/a,14/a,15/a,16/b}
\f{17}{14/d,15/c,16/c}
\f{18}{15/d,16/c}
\f{19}{16/d}
\f{20}{}
\fir[2.5]{\Z_2^-\circ\JJ_5,}
\fir{\Z_{1,A}\circ\JJ_5}
\end{ti}\begin{ti}[0.35] 
\draw[step=1,\gray](0,16)grid(1,15);
\draw[step=1,\gray](1,0)grid(16,16);
\rn[3]{2}{14}{3}
\f{16}{2/b,3/a,4/a,5/a,6/a,7/a,8/a,9/a,10/a,11/a,12/a,13/a,14/a,15/a,16/a}
\f{17}{14/f,15/e,16/e}
\f{2}{9/a,10/a,11/a,12/a,13/a,14/a,15/a,16/a}
\f{3}{9/b,10/a,11/a,12/a,13/a,14/a,15/a,16/a}
\f{4}{6/f,7/e,8/e,9/a,10/b,11/a,12/a,13/a,14/a,15/a,16/a}
\f{5}{7/d,8/c,9/a,10/a,11/b,12/a,13/a,14/a,15/a,16/a}
\f{6}{8/d,9/a,10/a,11/a,12/b,13/a,14/a,15/a,16/a}
\f{7}{9/a,10/a,11/a,12/a,13/b,14/a,15/a,16/a}
\f{10}{14/a,15/a,16/a}
\f{11}{11/f,12/e,13/e,14/a,15/a,16/a}
\f{12}{12/d,13/c,14/b,15/a,16/a}
\f{13}{13/d,14/a,15/b,16/a}
\f{14}{14/a,15/a,16/b}
\f{18}{15/d,16/c}
\f{19}{16/d}
\f{20}{}
\fir[2.5]{\siota(\Z_2^-\circ\JJ_5)=\Z_2\circ\JJ_6,}
\fir{\Z_{2,A}\circ\JJ_6,\hsp\Z_{2,B}\circ\JJ_6}
\end{ti}. \renewcommand{\di}{15} 

\begin{ti}[0.35] 
\draw[step=1,\gray](0,15)grid(1,14);
\draw[step=1,\gray](1,0)grid(15,15);
\rn[3]{1}{13}{3} 
\rbr[$\emptyset$]{8}{10}{3}
\rbr[$\emptyset$]{3}{5}{3}

\f{17}{}
\f{2}{8/a,9/a,10/a,11/a,12/a,13/a,14/a,15/a}
\f{3}{8/b,9/a,10/a,11/a,12/a,13/a,14/a,15/a}
\f{4}{5/d,6/c,7/c,8/a,9/b,10/a,11/a,12/a,13/a,14/a,15/a}
\f{5}{6/d,7/c,8/a,9/a,10/b,11/a,12/a,13/a,14/a,15/a}
\f{6}{7/d,8/a,9/a,10/a,11/b,12/a,13/a,14/a,15/a}
\f{7}{8/a,9/a,10/a,11/a,12/b,13/a,14/a,15/a}
\f{10}{13/a,14/a,15/a}
\f{11}{10/d,11/c,12/c,13/a,14/a,15/a}
\f{12}{11/d,12/c,13/b,14/a,15/a}
\f{13}{12/d,13/a,14/b,15/a}
\f{14}{13/a,14/a,15/b}
\f{18}{14/d,15/c}
\f{19}{15/d}
\f{20}{}
\fir{\jj\QQ_3\circ\JJ_6}
\end{ti}$\ssquare$  \renewcommand{\di}{16}

\begin{ti}[0.4]
\grid{16}
\rl[1]{1}{2}{7}
\rr[1]{2}{9}{7}
\rl[2]{9}{2}{7}\rl[2]{9}{10}{1}
\rr[2]{2}{11}{7}\rr[2]{10}{11}{1}\rr[2]{2}{3}{1}
\f{1}{3/d,4/c,5/c,6/c,7/c,8/c,9/b,10/a,11/a,12/a,13/a,14/a,15/a,16/a}
\f{2}{3/c,4/d,5/c,6/c,7/c,8/c,9/a,10/b,11/a,12/a,13/a,14/a,15/a,16/a}
\f{3}{5/d,6/c,7/c,8/c,9/a,10/a,11/b,12/a,13/a,14/a,15/a,16/a}
\f{4}{5/c,6/d,7/c,8/c,9/a,10/a,11/a,12/b,13/a,14/a,15/a,16/a}
\f{5}{7/d,8/c,9/a,10/a,11/a,12/a,13/b,14/a,15/a,16/a}
\f{6}{7/c,8/d,9/a,10/a,11/a,12/a,13/a,14/b,15/a,16/a}
\f{7}{9/a,10/a,11/a,12/a,13/a,14/a,15/b,16/a}
\f{8}{9/a,10/a,11/a,12/a,13/a,14/a,15/a,16/b}
\f{9}{11/d,12/c,13/c,14/c,15/c,16/c}
\f{10}{11/c,12/d,13/c,14/c,15/c,16/c}
\f{11}{13/d,14/c,15/c,16/c}
\f{12}{13/c,14/d,15/c,16/c}
\f{13}{15/d,16/c}
\f{14}{15/c,16/d}
\f{15}{}
\f{16}{}
\draw(16,8)--(8,8)--(8,16);
\draw(16,2)--(14,2)--(14,4)--(12,4)--(12,6)--(10,6)--(10,8);
\draw(2,16)--(2,14)--(4,14)--(4,12)--(6,12)--(6,10)--(8,10);
\fir{\F_{(4,4),2} } 
\end{ti}\begin{ti}[0.4]
\grid{16}
\rn[1]{3}{13}{4}
\f{1}{2/b,3/a,4/a,5/a,6/a,7/a,8/a,9/a,10/a,11/a,12/a,13/a,14/a,15/a,16/a}
\f{9}{3/b,4/a,5/a,6/a,7/a,8/a,9/a,10/a,11/a,12/a,13/a,14/a,15/a,16/a}
\f{11}{13/f,14/e,15/e,16/e}
\f{2}{6/d,7/c,8/c,9/c,10/c,11/b,12/a,13/a,14/a,15/a,16/a}
\f{3}{7/f,8/e,9/e,10/e,11/a,12/a,13/a,14/a,15/a,16/a}
\f{4}{7/c,8/d,9/c,10/c,11/a,12/b,13/a,14/a,15/a,16/a}
\f{5}{9/d,10/c,11/a,12/a,13/b,14/a,15/a,16/a}
\f{6}{9/c,10/d,11/a,12/a,13/a,14/b,15/a,16/a}
\f{7}{11/a,12/a,13/a,14/a,15/b,16/a}
\f{8}{11/a,12/a,13/a,14/a,15/a,16/b}
\f{10}{12/d,13/c,14/c,15/c,16/c}
\f{12}{13/c,14/d,15/c,16/c}
\f{13}{15/d,16/c}
\f{14}{15/c,16/d}
\f{15}{}
\f{16}{}
\draw(16,6)--(10,6)--(10,13);
\fir[2.5]{\siota(\F_{(4,4),2})=\Z_0(\F_{(4,4),2})\circ\JJ_{3}, }
\fir{\Z_{0,A}(\F_{(4,4),2}),\hsp\Z_{0,B}(\F_{(4,4),2})} 
\end{ti}$\ssquare$

\renewcommand{\di}{15}

\noindent\begin{ti}[0.4]
\grid{15}
\rl[1]{1}{2}{3}
\rr[1]{2}{5}{3}
\rr[2]{6}{9}{3}\rl[2]{5}{6}{3}\rr[2]{2}{9}{3}\rl[2]{5}{2}{3}
\rr[3]{6}{10}{3}\rl[3]{9}{6}{3}\rr[3]{2}{10}{3}\rl[3]{9}{2}{3}
\rr[4]{6}{13}{3}\rl[4]{10}{6}{3}\rr[4]{2}{13}{3}
\rr[4]{11}{13}{2}\rl[4]{10}{11}{2}\rl[4]{10}{2}{3}
\rr[6]{6}{14}{3}\rl[6]{13}{6}{3}\rr[6]{2}{14}{3}
\rr[6]{11}{14}{2}\rl[6]{13}{11}{2}\rl[6]{13}{2}{3}
\rr[8]{6}{15}{3}\rl[8]{14}{6}{3}\rr[8]{2}{15}{3}
\rr[8]{11}{15}{2}\rl[8]{14}{11}{2}\rl[8]{14}{2}{3}
\lex[5]{2}{3}{2}\lex[5]{6}{7}{2}
\lex[7]{3}{4}{1}\lex[7]{7}{8}{1}\lex[7]{11}{12}{1}
\f{1}{2/d,3/c,4/c,5/b,6/a,7/a,8/a,9/a,10/a,11/a,12/a,13/a,14/a,15/a}
\f{2}{3/d,4/c,5/a,6/b,7/a,8/a,9/a,10/a,11/a,12/a,13/a,14/a,15/a} 
\f{3}{4/d,5/a,6/a,7/b,8/a,9/a,10/a,11/a,12/a,13/a,14/a,15/a}
\f{4}{5/a,6/a,7/a,8/b,9/a,10/a,11/a,12/a,13/a,14/a,15/a}
\f{5}{6/d,7/c,8/c,9/b,10/a,11/a,12/a,13/a,14/a,15/a}
\f{6}{7/d,8/c,9/a,10/b,11/a,12/a,13/a,14/a,15/a}
\f{7}{8/d,9/a,10/a,11/b,12/a,13/a,14/a,15/a}
\f{8}{9/a,10/a,11/a,12/b,13/a,14/a,15/a}
\f{9}{10/d,11/c,12/c,13/a,14/a,15/a}
\f{10}{11/d,12/c,13/b,14/a,15/a}
\f{11}{12/d,13/a,14/b,15/a}
\f{12}{13/a,14/a,15/b}
\f{13}{14/d,15/c}
\f{14}{15/d}
\f{15}{}
\draw[line width=1pt]\mm{0}{4}--\mm{4}{4}--\mm{4}{8}--\mm{8}{8}--\mm{8}{12}--\mm{12}{12}--\mm{12}{15};
\fir{\F_{4,4,3}}
\end{ti}.\begin{ti}[0.4]
\grid{15}
\f{1}{2/b,3/a,4/a,5/a,6/a,7/a,8/a,9/a,10/a,11/a,12/a,13/a,14/a,15/a}  
\f{5}{3/b,4/a,5/a,6/a,7/a,8/a,9/a,10/a,11/a,12/a,13/a,14/a,15/a} 
\f{9}{4/b,5/a,6/a,7/a,8/a,9/a,10/a,11/a,12/a,13/a,14/a,15/a} 
\f{10}{5/b,6/a,7/a,8/a,9/a,10/a,11/a,12/a,13/a,14/a,15/a} 
\f{13}{6/b,7/a,8/a,9/a,10/a,11/a,12/a,13/a,14/a,15/a} 
\f{14}{7/b,8/a,9/a,10/a,11/a,12/a,13/a,14/a,15/a} 
\f{15}{} 
\f{2}{11/b,12/a,13/a,14/a,15/a} 
\f{3}{11/a,12/b,13/a,14/a,15/a}
\f{4}{11/a,12/a,13/b,14/a,15/a}
\f{6}{14/a,15/a}
\f{7}{14/b,15/a}
\f{8}{14/a,15/b}
\f{11}{} 
\f{12}{}
\draw (10,8)--(10,5)--(13,5)--(13,2)--(15,2);
\fir{\Z_{2}\circ\JJ_7}
\end{ti}$\ssquare$

\renewcommand{\di}{15}

\begin{ti}[0.4]
\grid{15}
\rl[1]{1}{2}{4}
\rr[1]{2}{6}{4}
\rl[2]{6}{7}{4}\rl[2]{6}{2}{4}
\rr[2]{7}{11}{4}\rr[2]{2}{11}{4}
\rl[3]{11}{7}{4}\rl[3]{11}{2}{4}
\rr[3]{7}{12}{4}\rr[3]{2}{12}{4}
\rl[5]{12}{7}{4}\rl[5]{12}{2}{4}\rl[5]{12}{13}{1}
\rr[5]{7}{14}{4}\rr[5]{2}{14}{4}\rr[5]{13}{14}{1}\rr[5]{3}{4}{1}\rr[5]{8}{9}{1}
\lex[4]{2}{4}{2}\lex[4]{7}{9}{2}  
\f{1}{2/d,3/c,4/c,5/c,6/b,7/a,8/a,9/a,10/a,11/a,12/a,13/a,14/a,15/a}
\f{2}{4/d,5/c,6/a,7/b,8/a,9/a,10/a,11/a,12/a,13/a,14/a,15/a}
\f{3}{4/c,5/d,6/a,7/a,8/b,9/a,10/a,11/a,12/a,13/a,14/a,15/a}
\f{4}{6/a,7/a,8/a,9/b,10/a,11/a,12/a,13/a,14/a,15/a}
\f{5}{6/a,7/a,8/a,9/a,10/b,11/a,12/a,13/a,14/a,15/a}
\f{6}{7/d,8/c,9/c,10/c,11/b,12/a,13/a,14/a,15/a}
\f{7}{9/d,10/c,11/a,12/b,13/a,14/a,15/a}
\f{8}{9/c,10/d,11/a,12/a,13/b,14/a,15/a}
\f{9}{11/a,12/a,13/a,14/b,15/a}
\f{10}{11/a,12/a,13/a,14/a,15/b}
\f{11}{12/d,13/c,14/c,15/c}
\f{12}{14/d,15/c}
\f{13}{14/c,15/d}
\f{14}{}
\f{15}{}
\draw[line width=1 pt](15,5)--(10,5)--(10,10)--(5,10)--(5,15);
\draw[line width=1 pt](15,2)--(13,2)--(13,4)--(11,4)--(11,5);
\draw[line width=1 pt](10,7)--(8,7)--(8,9)--(6,9)--(6,10);
\draw[line width=1 pt] (5,12)--(3,12)--(3,14)--(1,14)--(1,15);
\fir{\F_{(3,2),3}}
\end{ti}.\begin{ti}[0.4]
\grid{15}
\f{1}{2/b,3/a,4/a,5/a,6/a,7/a,8/a,9/a,10/a,11/a,12/a,13/a,14/a,15/a} 
\f{6}{3/b,4/a,5/a,6/a,7/a,8/a,9/a,10/a,11/a,12/a,13/a,14/a,15/a} 
\f{11}{4/b,5/a,6/a,7/a,8/a,9/a,10/a,11/a,12/a,13/a,14/a,15/a} 
\f{12}{5/b,6/a,7/a,8/a,9/a,10/a,11/a,12/a,13/a,14/a,15/a} 
\f{14}{} 
\f{2}{10/b,11/a,12/a,13/a,14/a,15/a}
\f{3}{9/d,10/a,11/b,12/a,13/a,14/a,15/a}
\f{4}{10/a,11/a,12/b,13/a,14/a,15/a}
\f{5}{10/a,11/a,12/a,13/b,14/a,15/a}
\f{7}{14/a,15/a}
\f{8}{13/d,14/b,15/a}
\f{9}{14/a,15/a}
\f{10}{14/a,15/b}
\f{13}{15/d} 
\f {15}{}
\draw[line width=1 pt](15,2)--(13,2)--(13,6)--(9,6)--(9,10);
\fir{\Z_3(\F_{(3,2),3})\circ\JJ_5}
\end{ti}$\ssquare$
\subsection{Main results: Theorems \ref{thD1}, \ref{thD2} and \ref{thD3}}\label{zpre2}Recall the use of ``$\nless$" in subscripts is explained in Definition \ref{hatnless}. 
\begin{sdefi}[$j_{\FF}$, $\tj$, $P(\FF)$, $p_C$, $p_i$, $\cdot$, $\jj^{-1}\cdot p_i$]\label{dem}
Let $\FF\in \Priubnk[N,\nless]\cap^\prime\BBnk[N]$. Recall that by Lemma \ref{weakprime} we have $\Priubnk[N,\nless]\cap^\prime\BBnk[N]=\Priubnk[N,\nless]\cap\BBnk[N]$. . 

Let $M=\prod_{1\leq i\leq k}^{\searrow}j_{M,i}(GL_{n_i})$ be the Levi  of the $GL_N$-parabolic subgroup with unipotent radical $D_{\FF}$ (where each $j_{M,i}$ is chosen to be a standard embedding). Let $C$ be a subset of $\{1,...,k\}.$ We define $p_C:M\rightarrow M$ to be the group homomorphism which: restricts to the identity on $\prod_{i\in C}j_{M,i}(GL_{n_i})$, and is trivial on $\prod_{i\in \{1,...,k\}-C}j_{M,i}(GL_{n_i})$. When $C$ is chosen to be a singleton we identify it with its element. The composition of functions is denoted by ``$\cdot$". For $1\leq i\leq k$, we define $\jm\cdot p_i:=j_{M,i}^{-1}\cdot p_i$. Notice that $p_{1}|_{\text{Stab}_{M}(\FF)}$ is injective. We define $P(\FF)={\jj}^{-1}\cdot p_{1}(\text{Stab}_{M}(\FF))$, and $$\tj:=j_{\FF}:=\text{ the inverse of }({\jj}^{-1}\cdot p_1)|_{\text{Stab}_{M}(\FF)}.$$Therefore $\jj_{\FF}$ is an isomorphism from $P(\FF)$ onto $ \text{Stab}_{M}(\FF).$ The notations $p_C$, $p_i$, $\jj^{-1}\cdot p_i$, and $\tj$ (instead of ``$\jj_{\FF}$"), occur only inside proofs (except for $\tj$ in the next subsection), and in each such proof we clarify the choice of $\FF$.  \end{sdefi}

We define next a set $\mathcal{S}_{n,N}$ of pairs of AFs, such that the AFs of the form $j_{\FF}(\K)\circ\FF$ for $(\K,\FF)\in \mathcal{S}_{n,N}$: cover the AFs in the statements of the  theorems of this section, and  allow as to state Lemma \ref{Sgen} below in the necessary generality to use it in the proofs of these theorems. 

\begin{sdefi}[$\mathcal{S}_{n,N}$]\label{pts}Let $0<n<N$. We define the set $\mathcal{S}_{n,N}$ to consist of the pairs $(\K,\FF)$ of AFs 
satisfying the following conditions.\begin{enumerate}
\item  $\FF\in\Priubnk[N,\nless]\cap^\prime\BBnk[N]. $
\item   $\K\in\AnkT{n}_{\nless}\cap\BR{n}$.
\item $D_\K\subset P(\FF).$ 
\item Let $T^\prime$ be any $\DDDD[N]$-component of $\TT{N}{\FF}$ for which $\min(\Set{T^\prime})$ is not in the set 
$$\{\min(\Set{\tj(T)})|T\text{ is a $\DDDD[n]$-component of }\TT{n}{\K}\}. $$  We have
\begin{equation}\label{zofcz}|\mathrm{Set}(T^\prime)|=\mathrm{max}\{|\mathrm{Set}(T)|:T\text{ is a }\DDDD[N]\text{-component of } \TT{N}{\FF}\}. \end{equation} Of course (by the first condition) formula (\ref{zofcz}) is the same as saying that $|\Set{T^\prime}|$ is equal to the number of blocks of the Levi of the $GL_N$-parabolic subgroup with unipotent radical $D_{\FF}.$
\end{enumerate}
\end{sdefi}
In the proofs of the next three theorems we will need to know that various among the AFs we encounter, belong to $\BBnk[x] $ for a positive integer $x$ (see Property 1 within the proof of each Theorem). These containments will quickly follow from the next lemma.   

\begin{slem}\label{Sgen}Let $0<n<N$, and $(\K,\FF)\in S_{n,N}$. Let $\Xi_1$ be an $(\F_{\emptyset,N}\rightarrow \FF,\BBnk[N])$-path (known to exist by applying Lemma \ref{weakprime} to a conjugate transpose of $\FF$). Finally let $\Xi_2$ be the $(\F_{\emptyset,n}\rightarrow\K,\BBnk[n])$-path appearing in Definition \ref{brow} in (\ref{zBrs}) (for the current choice of $\K$). Then 
$$j_{\FF}(\Xi_2)\vee\Xi_1 $$ is an $(\F_{\emptyset,N}\rightarrow j_{\FF}(\K)\circ\FF,\BBnk[N])$-path. In particular $j_{\FF}(\K)\circ\FF\in\BBnk[N]$.
\end{slem}

\begin{proof}The fourth condition\chh[eine tetarti?] of Definition \ref{pts} implies that the groups acting on $\Xi_2$ (the ones appearing in Definition \ref{brow} in the sentence after (\ref{zBrs})) are contained in $P(\FF).$
\end{proof}
\begin{sdefi}[$\rif$]\label{lrif}Let $\F\in\AAnk[N]. $ Consider  an $(\F,GL_N)$-tree $\Xi.$ Let $\VVV$ be a set consisting of  output vertices of $\Xi$, and for each $u\in \VVV$ let $\F_{u} $ be the label of $u$. For every  AF $\Z$ appearing as the label of an output vertex of $\Xi$ which is not a vertex in $\VVV$,  we assume  there are infinitely many output vertices of $\Xi$ with label equal to $\Z.$ We then write \begin{equation*}\F\rif\{\F_{u}:u\in \VVV\}.\qedhere \end{equation*}
\end{sdefi}
\noindent{It turns out} we use the definition above only for $\VVV$ having at most two elements. 

In many cases we use Main corollary \ref{maincor} through  its following corollary:
\begin{Ecor}\label{exchange}Consider AFs $\F,\F_1,\F_2\in\BBnk[N]$ and $\F_1^\prime,\F^\prime_2\in\AAnk[N]$, which satisfy:\begin{enumerate}
\item $\Di{D_\F}=\Di{D_{\F_1}}=\Di{D_{\F_2}}$,
\item $\F\rif\{\F_1^\prime,\F_2^\prime\}$,
\item $\F_1\rif\F_1^\prime$ and $\F_2\rif\F_2^\prime$.
\end{enumerate}Then $\Omm{\F}=\Omm{\F_1}\cup\Omm{\F_2} $.
\end{Ecor}
\begin{proof}It follows directly from Main corollary \ref{maincor}.
\end{proof}
\noindent{The} other uses of Main corollary \ref{maincor} in the current subsection include the use of the following lemma: 
\begin{slem}[Known]\label{dimensioncounting} Consider positive integers $m$ and  $0<r<N$, and let $a=[a_1,a_2,...,a_m]$ be an  orbit in $\UUU{N-r}$. Then 
$$\frac{\Di{[r,a]}}{2}=\frac{\Di{a}}{2}+\Di{D_{\JJ_{r}}}+\sum_{1\leq i\leq m}\max\{a_i-r,0\}.  $$
\end{slem}\begin{proof}It is easily obtained from Fact \ref{rich3} ((i) and (ii)) and Fact \ref{ptelos}. \end{proof}
\begin{sdefi}[$\Om_k$]\label{okato}Let $\Om$ be a subset of $\UUU{N},$ and $k\leq N$  a positive integer. We define
$\Om_k $ to be the set consisting of the partitions in $\Om$ in which each term is smaller or equal to $k,$ that is:

\begin{equation*}\Om_k=\{[a_1,...a_m]\in \Om : m\geq 1, \text{ for }1\leq i\leq m\text{ we have }a_i\leq k \}.\qedhere\end{equation*}\end{sdefi}
We will use several times and without mention the following observation:
\begin{sob}\label{bound}Let $\F$ be an AF in  $\AAnk[N]$ such that for a positive integer $k$ we have $\F|_{D_{\JJ_{k}}}=\JJ_k$. Then  $\Om(\F)_{k+1}=\emptyset.$
\end{sob}
\begin{proof}It directly follows form the  definition of $\Om(\F)$ based on $\X{\F}$ (Definition \ref{OFX}). 
\end{proof}
\begin{sdefi}[$\F_{n,k,l},$ $\FF_{n,k,l}$, $\WW_n$]\label{zfnkl}Consider integers: $n\geq 1$,   $2\leq k\geq l\geq 1. $ For $N=(n-1)k+l$, we choose $\FF_{n,k,l}$ to be an AF in  $\Priubnk[N,\nless][k^{n-1},l]\cap^\prime\BBnk[N] ; $ which is trivial  on the $n(l-1)+1$-th row. We define $$\F_{n,k,l}:=j_{\FF_{n,k,l}}(\WW_n)\circ\FF_{n,k,l}  $$
where $\WW_n\in\AAnk(U_n)[n]$. 

Of course the choice of these AFs is unique up to conjugation with a subtorus of $T_N.$ This choice does not affect the statements of the theorems, and within the proofs, as the data $n,k,l$ varies, it is assumed to be made as needed. 
\end{sdefi}
\begin{sdefi}[Notations about partitions]\label{orbcomp}

\vsp
\noindent{(i)}. Consider nonegative integers $m>0,t_1,...t_m,a_1,...,a_m.$ We write
$$[a_1^{t_1},...,a_m^{t_m}]=[\underbrace{a_1,...,a_1}_{t_1\text{ times}},\underbrace{a_2,...,a_2}_{t_2\text{ times}},....\underbrace{a_m,...a_m}_{t_m,\text{ times}}]. $$ 

\vsp 
\noindent{(ii).} Consider nonegative integers $m_3>m_2>m_1>0,t,a_1,...,a_{m_3}.$ We write $$[a_1,...a_{m_1},[a_{m_+1},...,a_{m_2}]^ta_{m_2+1},...,a_{m_3}]:=[a_1,...,a_{m_1},a_{m_1+1}^t,...,a_{m_2}^t,a_{m_2+1},...,a_{m_3}]. $$  

\vsp\noindent{(iii)}. Consider a positive integer $x<N$, let $\Om$ be a subset of $\UUU{N-x}$, and $[a_1,...,a_m]$ be a partition in $\UUU{x}$. We define  $$[a_1,...,a_m,\Om]:=\{[a_1,...,a_m,a]:a\in\Om\}.$$

\vsp
\noindent{(iv) (minor usefulness).} We extent (ii) in the obvious way in some cases that $t=-1.$ Let everything be as in (ii) except for $t.$ For every positive integer $a,$ let $x(a)$ and $y(a) $ be the number of occurrences of $a$ respectively in the sequences $a_1,..a_{m_1},a_{m_2+1},...a_{m_3}$ and $a_{m_1+1},...,a_{m_2}.$ Assume that for every choice of  $a$ we have $x(a)-y(a)\geq 0.$ Then by $$[a_1,...a_{m_1},[a_{l+1},...,a_{m_2}]^{-1}a_{m_2+1},...,a_{m_3}] $$ we denote the partition in which each positive integer $a$ occurs $x(a)-y(a)$ times.
\end{sdefi}
In stating the theorem below, the motivation for the names of the orbits is  that we can roughly view:  $A$ and $B$  as operators corresponding to the two terms of the recursion of orbits inside the proof in Property 2, and the product as function composition\footnote{ Unlike the other notations we adopt, each operator is right to its operand.} (and the superscripts as exponents). We do not give (and do not need) any  precise standalone meaning for ``A" and ``B".
\begin{stheorem}\label{thD1} Consider integers $n\geq 3,$ $k\geq2,$ $l\geq 2$.  Let $N$ and $\F_{n,k,l}$ be as in Definition \ref{zfnkl} above. We define the following orbits in $\UUU{N}$:
$$\aA{n}{k}{l}{n-2}:=[k+n-1,l-1,[k-1]^{n-2}],\qquad  \bB{n}{k}{l}{n-2}:=[[k+1]^{n-1},l-n+1],$$
$$\label{zABeq}\AB{n}{k}{l}{n}:=\left\{\begin{array}{cc}[k+2,l-1,[k+2,k-2]^{\frac{n-3}{2}},k-1]&n\text{ odd}\\\hspace{1mm}[k+2,l-1,[k+2,k-2]^{\frac{n}{2}-2},k+1,k-2]&n\text{ even}\end{array}\right.,
 $$
\begin{equation}\label{zABBeq}\ABB{n}{k}{l}{n-3}:=[k+2,l-1,\bB{n-2}{k}{k-1}{n-4}]=[k+2,l-1,[k+1]^{n-3},k-n+2]. \end{equation}
Then $\Omm{\F_{n,k,l}}$ is the set with any of  its elements being among the   orbits of the previous sentence, which satisfies the next properties:$$\aA{n}{k}{l}{n-2} \in\Omm{\F_{n,k,l}},\quad \bB{n}{k}{l}{n-2}\in \Omm{\F_{n,k,l}}\iff l-n+1\geq 0,$$ \begin{equation}\label{z210}\AB{n}{k}{l}{n}\in \Omm{\F_{n,k,l}}\iff l\neq k,\end{equation}\begin{equation}\label{z2101}\ABB{n}{k}{l}{n-3}\in \Omm{\F_{n,k,l}}\iff k-n-l+3\geq 0.\end{equation} 

If instead $n=2$, we still define all the orbits above except for the one in (\ref{zABBeq}) (note that they are equal), and we have $\Omm{\F_{2,k,l}}=\{\aA{2}{k}{l}{0}\}$.
 
\end{stheorem} \begin{tproof}By proceeding inductively we assume the theorem is correct for $N$ being replaced by any smaller number. The proof is given in two Parts. 

The results about $\AAnk$-trees which are obtained by only using the definition of $\AAnk$-trees, is the topic of Part 1. Part 1 is separated in Subparts 1.1, 1.2, 1.3. In Subpart 1.1,  $\siota(\F_{n,k,l})$ is calculated (see also (optionally) pictures with labels: $\F_{4,4,3}$ and $\F_{7,3,2}$).  Next in Subpart 1.2 we state Properties 0,1,2,3, which is all one needs to remember (from the content of Part 1) to proceed with Part 2. Property 2 is a ``recursion relation of AFs" starting with $\siota(\F),$ and stopping with AFs  for which information is obtained in Property 3. We prove Properties 0,1,2,3   in Subpart 1.3.

In Part 2,  we make use of Section \ref{variety}, namely, by using Main corollary \ref{maincor}. Each use of Main corollary \ref{maincor} is either through Exchange corollary \ref{exchange} or it is a direct use. Part 2 starts by using Exchange corollary \ref{exchange}, to convert the ``recursion relation of AFs" we saw in Property 2, into a recursion relation of orbits of the same AFs. Exchange corollary \ref{exchange} is used for each value of the variable of the recursion (which is the one denoted by $r$); these uses are possible due to Properties 1 and 2. Part 2 then continues by using Property 3, Proposition \ref{embedding},  (directly) Main corollary \ref{maincor} (in Claim 2), and the inductive hypothesis.

\noindent\textit{\textbf{Part 1.} The study of $\AAnk$-trees based only on direct use of $\AAnk$-tree definitions.}

\noindent\textit{\textbf{Subpart 1.1}: The calculation of $\siota(\F_{n,k,l})$.} We obtain $\siota(\F_{n,k,l})$, by choosing\footnote{We made the same choice in the pictures (in order to have fewer)} to explicitly describe the conjugate of  $\sI(\F_{n,k,l})$, in which there is only one $\co$-step and appears in the end. Consider the torus $T_{N,2}$ of $\DDDD[N]$, defined by 
$$\Set{T_{N,2}}:=\{1+in|0\leq i\leq l-1\}\cup\{2+(l-1)n\}\cup\{1+ln+i(n-1)|0\leq i\leq k-l-1\}. $$
Notice that if $k=l$,  the rightmost set in the union is empty. Let $C_1, C_2,...,C_{k}$ be the root groups of $\SRV{T_{N,2}}$ in the order that starts with $C_1:=U_{(1,2)}$ and moves downwards. The conjugate of $\sI(\F_{n,k,l})$ which we describe is equal to $$\co\vee\I_{k}\vee...\vee\I_1$$ where: \begin{itemize}
\item For $1\leq i\leq k,$ $\I_i$ is obtained by successively applying the operations $\exchange{X}{Y}$ for all choices of root groups $X,Y,$ for which: $Y\subset D_\F,$ $[X,Y]=C_i$, $Y$ lies in the same column as $C_i,$ and $ Y$ does not lie in the same row as any $C\in\SRV{T_{N,2}}\}. $ As a reminder of the order, we apply these $\eu$-operations by starting with  $Y$ lying in the biggest possible row and moving upwards.
\item $\co$ is the $\co$-step obtained by conjugating with the minimal length element $w_2$ of $W_N$ with the property: $$\Set{w_2T_{N,2}w_2^{-1}}=\{1,2,...,k+1\}.$$
\end{itemize}
We see that there is a $\Z_0\in \AAnk[N] $ such that: $\siota(\F_{n,k,l})=\Z_0\circ\JJ_{k+1}$.\hfill$\square$Subpart 1.1

\vsp
\noindent\textit{\textbf{Subpart 1.2:} The results of Part 1 (Properties 0,1,2,3).} We  consider two families of AFs, denoted by: $\Z_r$ for $0\leq r\leq n-2$, and $\QQ_r$ for $1\leq i\leq n-1.$ Their definition is given later in Subpart 1.3; in the current subpart we state properties that is sufficient to remember (even without definitions) to proceed with Part 2.

 For $1\leq r\leq n-1$ we define $N_r:=N-k-r.$ 

\vsp
\noindent\textit{Property 0. }The AFs $\Z_r$ (resp. $\QQ_r$) are defined for $0\leq r\leq n-2$ (resp. $1\leq r\leq n-1$), and belong to $\AAnk[N]$  (resp. $\AAnk[N_r] $). For $1\leq r\leq n-1,$ let $j_r$ be the lower right corner embedding of $GL_{N_r}$ inside $GL_N.$ Most of the time we write:
$$\ja\QQ_r:=j_r(\QQ_r). $$We have $\jj\QQ_{n-1}=\Z_{n-2}$. 

\vspace{1mm}

\noindent{We use Property 0} without mention.

\vspace{1mm}
\noindent\textit{Property 1. } For all $r$ for which the definitions are valid (that is $0\leq r\leq n-2$ for the left hand side, and $1\leq r\leq n-1$ for the right hand side) we have: $$\Z_r\circ\JJ_{k+r+1}\in\BBnk[N]\qquad\text{and}\qquad\QQ_r\in\BBnk[N_r].$$

\vspace{1mm}
\noindent\textit{Property 2 (recursion relation). }Assume that $0\leq r<n-2$. There are two  AFs $\Z_{r,A}$ and $\Z_{r,B}$ in $\AAnk[N],$ for which: \\2.1: 
$\Z_r\circ\JJ_{k+r+1}\rif\{\Z_{r,A}\circ\JJ_{k+r+1},\Z_{r,B}\circ\JJ_{k+r+1}\}; $\\
2.2: $\Z_{r+1}\circ \JJ_{k+r+2}\rif \Z_{r,A}\circ\JJ_{k+r+1}; $\\2.3: $\ja\QQ_{r+1}\rcirc\JJ_{k+r+1}\rif \Z_{r,B}\circ\JJ_{k+r+1};$\\
 2.4: $\Di{D_{\Z_r\circ\JJ_{k+r+1}}}=\Di{D_{\Z_{r+1}\circ \JJ_{k+r+2}}}=\Di{D_{\ja\QQ_{r+1}\rcirc\JJ_{k+r+1}}}.$

\vspace{1mm} 
\noindent\textit{Property 3.} 3.1: $\siota(\F_{n,k,l})=\Z_0\circ\JJ_{k+1} $\hspace{3mm} (this is already obtained in Subpart 1.1, except for defining $\Z_0$). \\3.2: Assume $n\geq 3.$ Then $\QQ_1=\F_{n-1,k,l-1}$.\\3.3: Assume $n\geq 4$.  Let $\zAF^-$ and $\zAF$ be the trivial functions with domains\\ $\prod_{N_2-(k-l)(n-2)-(n-3)\leq j\leq N_2}U_{N_2,(l-1,j)}$ and $\prod_{l-1\leq j\leq N_2}U_{N_2,(l-1,j)}$ respectively. Then \begin{equation}\label{rhs}\siota(\QQ_2)=\jj\F_{n-2,k,k-1} \circ \zAF^-\circ \JJJJ{N_2}{l-1}\rif \jj\F_{n-2,k,k-1}\rcirc\left(\zAF\circ\JJJJ{N_2}{l-1}\right)\end{equation}where $\jj\F_{n-2,k,k-1}$ is the lower right corner copy of $\F_{n-2,k,k-1}$ in $\AAnk[N_2] $.\\3.4: $\QQ_{n-1}\in\Prink[N_{n-1}][l-1,(k-1)^{n-2}].$  \\3.5: $\big((r>2\text{ or }(l=k\text{ and }r\geq 2))\text{ and }r\neq n-1\big)\implies\QQ_r\rif \emptyset$.

\hfill$\square$Subpart 1.2

\noindent\textit{\textbf{Subpart 1.3: }Proof of Properties 0,1,2,3. }We start with some definitions. 

We adopt Definition \ref{dem} for $[\FF\tlarrow\FF_{n,k,l}]$ (for example $\tj:=\jj_{\FF_{n,k,l}}$)

For  $1\leq r\leq n, $ let $T_n\{r\}$ and $T_n[r]$ be the tori in $\DDDD[n]$ satisfying

 $$\Set{T_n\{r\}}=\{r\}\qquad\text{and}\qquad\Set{T_n[r]}=\{1,2,...,r\},$$ For $2\leq r\leq n, $  let $T_{N,r}$ be the torus in $\DDDD[N]$ satisfying:
 $$\Set{T_{N,r}}:=\Set{\tj(T_n\{1\})}\cup\Set{p_{\{i:l\leq i\leq k\}}(\tj(T_n\{2\}))}\cup\Set{p_k(\tj(T_n[r]))}, $$ and  let $w_{r}$ be the minimal length element in $W_{N}$  satisfying $$\Set{w_{r}T_{N,r}w^{-1}_{r}}=\{1,2,...,k+r-1\}.$$ For $1\leq r\leq n,$ let $\row{r}$ be the subgroup of $GL_n$ consisting of the matrices in which any entry  not lying in the $r$-th row and not lying (resp. lying) in the main diagonal equals to 0 (resp. 1). 

We define a group $H$ as follows:\begin{itemize}
\item If $l=k$, we define $H=GL_n^{\{1\}}$ (which note it is also equal to $GL_n^{T\{1\}}$).
\item If $l<k$, we define $H$ to be the standard copy of $P_{n-1}$ in $GL_n^{\{1\}}$.
\end{itemize}

 We define three homomorphisms: $\tjA: P(\FF_{n,k,l})\rightarrow GL_N$ and     $\tjp,\tjB:H\rightarrow GL_N$, as follows: $$\tjA:=p_k\cdot\tj;$$for $l=k$
 $$\tjp=\tj|_H\qquad\text{and}\qquad\tjB=p_{\{i:1\leq i\leq k-1\}}\cdot\tj|_H; $$for $l<k$
  $$\tjp|_{GL_n^{T[2]}}=\tj|_{GL_n^{T[2]}},\qquad\qquad\tjB|_{GL_n^{T[2]}}=p_{\{i:1\leq i\leq k-1\}}\cdot\tj|_{GL_n^{T[2]}},$$ $$\tjp|_{\row{2}}=p_{\{i:1\leq i\leq l-1\}\cup\{k\}}\cdot\tj|_{\row{2}},\qquad\tjB|_{\row{2} }:=p_{\{i:1\leq i\leq l-1\}}\cdot\tj|_{\row{2}}.$$ 
 For $0\leq r\leq n-2$ we define
\begin{equation}\label{zdztt}\Z_r:=w_{r+2}\left(\tjp(\WW_n^{T_n[r+1]}))\circ \FF_{n,k,l}^{T_{N,r+2}}\right)=w_{r+2}\left(\tj(\WW_n^{T_n[r+2]})\circ\tjp(\WW_n|_{\mathrm{row}(r+2)})\circ \FF_{n,k,l}^{T_{N,r+2}}\right);\end{equation}
and for $1\leq r+1\leq n-1$ we define $\QQ_{r+1}$ so that
$$\ja\QQ_{r+1}=w_{r+2}\left(\tj(\WW_n^{T_n[r+2]})\circ\tjB(\WW_n|_{\mathrm{row}(r+2)})\circ \FF_{n,k,l}^{T_{N,r+2}}\right). $$
Of course if $r>0$ or $k=l$, we have $\Z_r=w_{r+2}\left(\tj(\WW_n^{T_n[r+1]})\circ \FF_{n,k,l}^{T_{N,r+2}}\right)$. 

Every $\AAnk$-tree that is mentioned is assumed to be an $(\AAnk,\KKK)$-tree.

\vspace{1mm} 
\noindent\textit{Proof of Property 0. }Directly from the definitions.\hfill$\square$Property 0

\vspace{1mm}
\noindent\textit{Proof of Property 1. }Let $r$ satisfy $0\leq r\leq n-2.$ Since  $$\left(j_{r+1}^{-1}\left(w_{r+2}(\tj(\WW_n^{T_n[r+2]})\circ\tjB(\Wrow{r+2}))\right),j_{r+1}^{-1}\left(w_{r+2}(\FF_{n,k,l}^{T_{N,r+2}})\right)\right)\in \mathcal{S}_{n-1,N_{r+1}},$$ we obtain from Lemma \ref{Sgen} that $$\QQ_{r+1}\in\BBnk[N_{r+1}].$$

 Let $\Xi$ be the $(\F_\emptyset\rightarrow \QQ_{r+1}, \BBnk[N_{r+1}])$-path mentioned in this lemma. For every root group $V$ of the domain of  $\Wrow{r+2}$; we replace in the path $j_{r+1}(\Xi),$ the $\e$-step over $w_{r+2}\tjB(V)w_{r+2}^{-1}$,  with an $\e$-step over $w_{r+2}\tjp(V)w_{r+2}^{-1}$. The choice of $\e$-step to which ``an" refers is made so that after these replacements we still have an $\AAnk$-path. By applying $\circ\JJ_{k+r+1}$ to the path obtained, and using the same groups as in $j_{r+1}(\Xi)$ for the free transitive actions, we obtain a $$(\JJ_{k+r+1}\rightarrow\Z_r\circ\JJ_{k+r+1},\BBnk[N])\text{-path.}$$Hence $\Z_r\circ\JJ_{k+r+1}\in\BBnk[N].$\hfill$\square$Property 1

\vspace{1mm}
\noindent\textit{Proof of Property 2.1. }Let $\xi$ be the initial subtree of $\sXi(\Z_r\circ\JJ_{k+r+1})$ given by $$\xi=\I[](\XX)\vee_\XX\xi_1 $$ where:\begin{enumerate}
\item $\xi_1$ is the $(\Z_r\circ\JJ_{k+r+1},\e)$-quasipath along constant terms  over the  root groups of the domain of  $w_{r+2}\tjA(\WW|_{\mathrm{row}(r+2)})w_{r+2}^{-1}.$ 
\item  $\XX$ varies over the output AF of $\xi_1$ except for the terms of the last $\e$-step  of $\xi_1$;
\item $\I(\XX)$ is the first $\eu$-step of $\sI(\XX).$\end{enumerate}

\noindent{For}  $1\leq i<n-r-2$ let $$\mathcal{Y}_i:=\{\YY:\text{for a nonconstant term $\XX$ of the $i$-th $\e$-step of $\xi_1$,  }\YY\text{ is the output AF of }\I(\XX)\}, $$and
 $T_n(i)$ be the torus in $\DDDD[n]$ defined by
$$\Set{T_n(i)}:=\{j:n+1-i\leq j\leq n\}. $$ We see that $w_{r+2}\tj(T_n(i))w_{r+2}^{-1}$, acts transitively in $\mathcal{Y}_i.$

Finally consider the set of output AFs of the last $\e$-step. In this set, the torus $$w_{r+2}T_{N,r+2}\tjB(T_n\{r+2\})w_{r+2}^{-1} $$is acting with three orbits, two of them being singletons. The two elements in the singleton orbits are of the form $\Z_{r,A}\circ\JJ_{k+r+1}$ and $\Z_{r,B}\circ\JJ_{k+r+1}$ where $\Z_{r,A}$ and $\Z_{r,B}$ are AFs satisfying \begin{equation}\label{zmionena}w_{r+2}\Z_{r,B}\left(\tjA (U_{n,(r+2,r+3)})\right)=\{0\}\qquad\text{and}\qquad w_{r+2}\Z_{r,A}\left(\tjB(U_{n,(r+2,r+3)})\right)=\{0\}.\end{equation}\hfill$\square$Property 2.1 and definitions of $\Z_{r,A}$ and $\Z_{r,B}$
   
\vspace{1mm}
\noindent\textit{Properties 2.2 and 2.3.} We consider the  $(\ja\QQ_{r+1}\circ\JJ_{k+r+1},\e)$-quasipath along constant terms which is over the same groups as $\xi_1$ in the previous property  (which again is an initial subquasipath of $\sXi(\ja\QQ_{r+1}\circ\JJ_{k+r+1})$). In the set of output AFs of any among the  $\e$-steps of this quasipath, the torus $w_{r+2}T_{N,r+2}w_{r+2}^{-1}$ acts with two orbits, one consisting of the constant term. Hence Property 2.3 is obtained, and we continue with Property 2.2.

Consider the AF $\Z_{r+1}^{-}$ which is defined the same way as $\Z_{r}$ except that $\tjp$  is replaced with $\tjA$, that is
$$\Z_{r+1}^-:=w_{r+2}\left(\tj(\WW_n^{T_n[r+2]})\circ\tjA(\WW_n|_{\mathrm{row}(r+2)})\circ \FF_{n,k,l}^{T_{N,r+2}}\right).$$
We consider the  $(\Z_{r+1}^-\circ\JJ_{k+r+1},\e) $-quasipath  along constant terms  over the $\ossa$-groups $w_{r+2}\tjB(V)w_{r+2}^{-1},$ for $V$ being any $GL_n$-root group contained in $D_{\WW_n|_{\mathrm{row}(r+2)}}$. We see that its output constant term is $\Z_{r,A}\circ\JJ_{k+r+1}$. In each $\e$-step obtained, the torus $$w_{r+2}\tjB( T_n\{r+2\})w_{r+2}^{-1}$$ acts with two orbits, one consisting of the constant term, and hence
\begin{equation}\label{zbiri}
\Z_{r+1}^-\circ\JJ_{k+r+1}\rif\Z_{r,A}\circ\JJ_{k+r+1}.
\end{equation} 
Finally we have \begin{equation}\label{mikro}\siota(\Z_{r+1}^-\circ\JJ_{k+r+1})=\Z_{r+1}\circ\JJ_{k+r+2}.\end{equation} From this and from (\ref{zbiri}), Property 2.2 is obtained.\hfill$\square$Properties 2.2 and 2.3

\vsp
\noindent{Proof of Property 2.4}. It is a trivial counting (and is already discerned in the proof of Properties 2.1,2.2,2.3).\hfill$\square$Property 2.4  

\vsp
\noindent\textit{Proof of Property 3. }Property 3.1 is already proved in Subpart 1.1. Properties  3.2 and 3.4 directly follow from the definitions of $\Z_r$ and $\QQ_r$. 

We continue with  Property 3.3. The equality for  $\siota(\QQ_2)$ is obtained from Observation \ref{unipc}. 

Then to obtain the ``$\rif$": we consider the $(\jj\F_{n-2,k,k-1} \circ \zAF^-\circ \JJJJ{N_2}{l-1},\e)$-quasipath along constant terms  over the root groups belonging in the domain of $\zAF$ and not in the domain  of $\zAF^-$; and notice that in  each $\e$-step in this quasipath, the torus $w\tjB(T_n\{2\})w^{-1}$ for some\footnote{We choose it so that  $\Set{w\tjB(T_n\{2\})w^{-1}}=\{k+3,...,k+l+1\}$.}  $w\in W_N$ (after it is embedded on $GL_{N_2}$) is acting with two orbits, one consisting of the constant term.

We are left with Property 3.5, which we prove for $\QQ_{r+1} $ instead of $\QQ_r$. We consider the  $(\QQ_{r+1},\e)$-quasipath along constant terms over the $\ossa$-groups of the form  $j_{r+1}^{-1}(w_{r+2}\tjB(V)w_{r+2}^{-1})$, for $V$ being a  $GL_n$-root group contained in $U_n|_{\row{r+1}}$ other than the simple one. In all the $\e$-steps of this quasipath except the last one, we consider the action by $j_{r+1}^{-1}(w_{r+2}\tjB(T_n\{r+1\})w_{r+2}^{-1})$, which consists of two orbits, one consisting of the constant term. In the last $\e$-step---which is  the one over $j_{r+1}^{-1}( w_{r+2}\tjB(U_{n,(r+1,r+3)})w_{r+2}^{-1})$--- we consider the action  by $j_{r+1}^{-1}(w_{r+2}\tjB(U_{n,(r+2,r+1)})w_{r+2}^{-1}).$ This action is transitive, and hence  we are done.\hfill$\square$Property 3$\square$Part 1  

\vsp
\noindent\textit{\textbf{Part 2:} The use of Main Corollary \ref{maincor}.} By using Properties 1,2, we apply Exchange  corollary \ref{exchange} for each $r$ satisfying  $0\leq r<n-2,$ and obtain the recursion relation of orbits:
\begin{equation}\label{recursionorbit}\Omm{\Z_{r}\circ\JJ_{k+r+1}}=\Omm{\ja\QQ_{r+1}\rcirc\JJ_{k+r+1}}\cup\Omm{\Z_{r+1}\circ\JJ_{k+r+2}}. \end{equation} This formula for all values of $r$ that is proven, and Property 3.1, give:
  
\begin{equation}\label{zpola}\Omm{\F_{n,k,l}}=\bigcup_{1\leq r\leq n-1}\Omm{\ja\QQ_r\rcirc\JJ_{k+r}}. \end{equation}
From Property 3.5, we see that most of the terms in this union are empty. In detail:
\begin{equation}\label{zABCD}\Omm{\F_{n,k,l}}=\bigcup_{r\in S}\Omm{\ja\QQ_r\rcirc\JJ_{k+r}} \end{equation}where $S=\left\{\begin{array}{cc}\{1,2\}\cup\{n-1\}&\text{if }l<k\text{ and }n>2\\\{1\}\cup\{n-1\}&\text{otherwise} \end{array}\right..$

By Proposition \ref{embedding} we have
\begin{equation}\label{zB}\Omm{\ja\QQ_{r}\rcirc\JJ_{k+r}}=[k+r,\Omm{\QQ_r}_{k+r}]\qquad\text{for }1\leq r\leq n-1. \end{equation}
By (\ref{zB}) and Property 3.4 we have 
\begin{equation}\label{zAA}\Omm{\jj\QQ_{n-1}\rcirc\JJ_{k+n-1}}=\{\aA{n}{k}{l}{n-2}\}.\end{equation}

 From  (\ref{zpola}) and (\ref{zAA}) we see that the subset of $\Omm{\F_{n,k,l}}$ which equals $\Omm{\jj\QQ_{n-1}\rcirc\JJ_{k+n-1}},$ consists of  $\aA{n}{k}{l}{n-2}$.  The rest of $\Omm{\F_{n,k,l}}$ is obtained by: (\ref{zABCD}); (\ref{zB}); stating and proving the two claims below; and from the following three trivial equalities, which hold as long as the orbits involved are defined, $$\bB{n}{k}{l}{n-2}=\left[k+1,\bB{n-1}{k}{l-1}{n-3}\right], $$ $$\AB{n}{k}{l}{n}=\left[k+2,l-1,\AB{n-2}{k}{k-1}{n-2}\right], $$
 and
 $$\ABB{n}{k}{l}{n-3}=\left[k+2,l-1,\bB{n-2}{k}{k-1}{n-4}\right]. $$

\vsp
\noindent\textbf{Claim 1.}\textit{ Assume $n\geq 3$. Then
\begin{equation}\label{zOmm1}\Omm{\QQ_1}_{k+1}=\left\{\begin{array}{cc}\{\bB{n-1}{k}{l-1}{n-3}\}&\text{if }l-n+1\geq 0\\\emptyset&\text{otherwise}\end{array}\right.. \end{equation}}

\vsp 
\noindent\textbf{Claim 2. }\textit{Assume $n\geq 4$, and that $l<k$. If $k-n-l+3\geq 0 $, we have  $$\Omm{\QQ_2}_{k+2}=\left\{\left[l-1,\AB{n-2}{k}{k-1}{n-2}\right],\left[l-1,\bB{n-2}{k}{k-1}{n-4}\right]\right\}.$$ If $k-n-l+3< 0 $, only the first among these two orbits belongs to this set, that is
 $$\Omm{\QQ_2}_{k+2}=\left\{\left[l-1,\AB{n-2}{k}{k-1}{n-2}\right]\right\}. $$}

\vspace{1mm}
\noindent\textit{Proof of Claim 1. }Recall (Property 3.2) that $\QQ_1=\F_{n-1,k,l-1}$. Hence if $l>2$ the claim is already proven as part of the inductive hypothesis. Let $l=2$. For $n=3$ we see that $\F_{n-1,k,1}\in\AAnk(U_{k+1})[k+1]  $ and hence the claim follows. Let $n>3$. For $j$ being the lower right corner embedding of $GL_{N_1-1}$ in $GL_{N_1}$, we have $\sI(\F_{n-1,k,1})=j(\sI(\F_{n-2,k,k}))\circ\JJ_2^{N_1}$. From this and Property 3.1 for $[n\tlarrow n-2,l\tlarrow k]$ we have  $\siota(\F_{n-1,k,1})|_{D_{\JJ_{k+2}}}=\JJ_{k+2}$, and hence every orbit of  $\Om(\F_{n-1,k,1})$ is bigger or equal to $(k+2,1^{N_1-k-2})$.\hfill$\square$Claim 1

\vspace{1mm}
\noindent\textit{Proof of Claim 2.}  From Property 3.3 we have $$\QQ_2\rif\jj\F_{n-2,k,k-1}\rcirc\left(\zAF\circ\JJJJ{N_2}{l-1}\right). $$ 
This together with Property 1 and Main Corollary \ref{maincor} implies:
\begin{equation}\label{zclaim21}\Omm{\QQ_2}_{k+2}=[l-1,\Omm{\F_{n-2,k,k-1} }_{k+2}]\cap\left\{c\in\UUU{N_2}:\frac{\Di{c}}{2}=\Di{D_{\QQ_2}}\right\}. \end{equation} 
The inductive hypothesis gives $\Omm{\F_{n-2,k,k-1}}_{k+2}=$ \begin{equation}\label{zclaim22}\left\{\begin{array}{cc}\{\AB{n-2}{k}{k-1}{n-2},\bB{n-2}{k}{k-1}{n-4}\}&\text{if }k-n+2\geq 0\\\{\AB{n-2}{k}{k-1}{n-2}\}&\text{if }k-n+2<0 \end{array}\right.,\end{equation}and therefore, by using (\ref{zclaim21}) and (\ref{zclaim22}) we are left with comparing dimensions. By using (for example) the equality in Property 3.3 we obtain
\begin{equation}\label{zdimipol}\Di{D_{\QQ_2}}=\Di{D_{\F_{n-2,k,k-1} }}+\Di{D_{\JJJJ{N_2}{l-1}}}+(n-2)(k-l)+n-3. \end{equation}Let $c=(c_i)$ be one among the (one or two) elements of $[l-1,\Omm{\F_{n-2,k,k-1} }_{k+2}]$. By using Lemma \ref{dimensioncounting}, and then (to obtain the first term of the right hand side below) using Main Corollary \ref{maincor}, we have:
\begin{equation}\label{zalimia}
\frac{\Di{c}}{2}=\Di{D_{\F_{n-2,k,k-1} }}+\Di{D_{\JJJJ{N_2}{l-1}}}+\sum_i\text{max}\{c_i-l+1,0\}.\end{equation} Since the right hand sides of (\ref{zdimipol}) and (\ref{zalimia}) differ only in the last term, by using (\ref{zclaim21}), (\ref{zdimipol}), and  (\ref{zalimia}), we are left with checking whether or not the following equality is correct: 
$$\sum_i\text{max}\{c_i-l+1,0\}=(n-2)(k-l)+n-3.$$By (\ref{zclaim22}), it turns out to be correct if and only if:
 $$c=[l-1,\AB{n-2}{k}{k-1}{n-2}]$$or$$\hspace{2mm} c=[l-1,\bB{n-2}{k}{k-1}{n-4}]\qquad\hspace{2mm}\text{and}\qquad\hspace{2mm} k-n-l+3\geq 0\hspace{2mm}.$$\hfill$\square$Claim 2$\square$Part 2

\end{tproof}The way we break the proof of the next two theorems into two parts is analogous to the same breaking in the proof above; however subparts (which appear again only in Theorem \ref{thD3}) are frequently not analogous.
\begin{stheorem}\label{thD2} Consider integers $n\geq 3$ and  $k\geq 2.$  Let $N$ and $\F_{n,k,1}$ be as in Definition \ref{zfnkl}. We define the following orbits in $\UUU{N}$: $$\aA{n}{k}{1}{n-3}:=\left[n+k-1,[k-1]^{n-2}\right],\quad \bB{n}{k}{1}{n-3}:=[k+2,[k+1]^{n-3},k-n+2], $$
and$$\BAB{n}{k}{1}{n-2}:=\left\{\begin{array}{cc}\left[k+2,[k+2,k-2]^{\frac{n-3}{2}},k-1\right]&n\text{ odd}\\
\hspace{1mm}\left[k+2,[k+2,k-2]^{\frac{n}{2}-2},k+1,k-2\right]&n\text{ even}\end{array}\right..
 $$For $k\geq 3,$ $\Omm{\F_{n,k,1}}$ is given by: \begin{enumerate}\item[(a)] $\{\aA{n}{k}{1}{n-3},\BAB{n}{k}{1}{n-2}\}\subseteq\Omm{\F_{n,k,1}}$;\item[(b)] $\bB{n}{k}{1}{n-3}\in\Omm{F_{n,k,1}}\iff k-n+2\geq 0$;
 \item[(c)]  $\Omm{F_{n,k,1}}$ does not contain any element other than the ones implied from (a) and (b).\end{enumerate}   For $k=2,$ by also defining the orbit in $\UUU{N}$ given by:
 $$\bB{n}{2}{1}{\lfloor\frac{n-2}{2}\rfloor}:=\left\{\begin{array}{cc}\left[4^{\frac{n-1}{2}},1\right]&\text{ if }n \text{ is odd}\\\hspace{1mm}\left[4^{\frac{n-2}{2}},3\right]&\text{ if }n \text{ is even}
 \end{array}\right., $$ we have $$\Omm{\F_{n,2,1}}=\{\aA{n}{2}{1}{n-3},\bB{n}{2}{1}{\lfloor\frac{n-2}{2}\rfloor}\}.$$

\end{stheorem} 
\begin{proof}By proceeding inductively we assume the theorem is correct for $N$ being replaced by any smaller number.   

\vsp 
\noindent\textit{\textbf{Part 1. }The study of $\AAnk$-trees based only on direct use of $\AAnk$-tree definitions.} For $j$ being the lower right corner embedding of $GL_{N-1}$ in  $GL_N,$ we have: $$\F_{n,k,1}= j(\F_{n-1,k,k})\circ\JJ_2.$$ One notices that all the $\AAnk$-trees in proof of Theorem \ref{thD1} are embedded by $j$ in $\Stab{GL_N}{\JJ_2}$. Therefore by\begin{itemize}
\item redefining  the AFs $\Z_r$ for $0\leq r\leq n-3$, and $\QQ_r$ for $1\leq r\leq n-2$, by embedding the old definitions to the lower right corner and applying $\circ\JJ_2$;
\item redefining for $1\leq r\leq n-2$ the number $N_r$ to equal to $N-k-r-1,$ retaining the meaning of $j_r$ (of course with the renewed meaning of $N_r$ in its definition), and adopting again the convention $\ja\QQ_r=j_r(\QQ_r)$; 
\end{itemize}we obtain the redefined Properties 0,1,2,3 below.

\vsp
\noindent\textit{Property 0. }The AFs $\Z_r$ (resp. $\QQ_r$) are defined for $0\leq r\leq n-3$ (resp. $1\leq r\leq n-2$), and belong to $\AAnk[N]$  (resp. $\AAnk[N_r] $). For $0\leq r\leq n-2,$ let $j_r$ be the lower right corner embedding of $GL_{N_r}$ inside $GL_N.$ We have $\jj\QQ_{n-2}=\Z_{n-3}$.

\vspace{1mm}

\noindent{We use Property 0} without mention.

\vspace{1mm}
\noindent\textit{Property 1. }For all $r$ that the definitions are valid (that is $0\leq r\leq n-3$ for the left hand side, and $1\leq r\leq n-2$ for the right hand side) we have: $$\Z_r\circ\JJ_{k+r+2}\in\BBnk[N]\qquad\text{and}\qquad\QQ_r\in\BBnk[N_r].$$

\vspace{1mm}
\noindent\textit{Property 2 (recursion relation). }Assume that $0\leq r<n-3$. There are two AFs  $\Z_{r,A}$ and $\Z_{r,B}$ in $\AAnk[N]$,  such that:\\2.1: $\Z_r\circ\JJ_{k+r+2}\rif\{\Z_{r,A}\circ\JJ_{k+r+2},\Z_{r,B}\circ\JJ_{k+r+2}\};$\\
2.2: $\Z_{r+1}\circ \JJ_{k+r+3}\rif \Z_{r,A}\circ\JJ_{k+r+2};$\\2.3: $\ja\QQ_{r+1}\rcirc\JJ_{k+r+2}\rif \Z_{r,B}\circ\JJ_{k+r+2};$\\2.4: 
$\Di{D_{\Z_r\circ\JJ_{k+r+2}}}=\Di{D_{\Z_{r+1}\circ \JJ_{k+r+3}}}=\Di{D_{\ja\QQ_{r+1}\rcirc\JJ_{k+r+2}}}. $

\vspace{1mm} 
\noindent\textit{Property 3.} 3.1: $\siota(\F_{n,k,1})=\Z_0\circ\JJ_{k+2}.$ \\3.2: Assume $n\geq 4.$ Then $\QQ_1=\F_{n-2,k,k-1}$.\\3.3: We do not need (and hence omit) any analogue of Property 3.3 in the proof of Theorem \ref{thD1}.\\3.4: $\QQ_{n-2}\in\Prink[(n-2)(k-1)][[k-1]^{n-2}]$.\\3.5: $n-2\neq r\geq 2\implies\QQ_r\rif \emptyset$.\hfill$\square$Part 1

\noindent\textit{\textbf{Part 2:} The use of Main Corollary \ref{maincor}.} By using Properties 1,2, we apply Exchange corollary \ref{exchange} for each $r$ satisfying  $0\leq r<n-3,$ and obtain the recursion relation of orbits:
\begin{equation}\label{recursionorbit2}\Omm{\Z_{r}\circ\JJ_{k+r+2}}=\Omm{\ja\QQ_{r+1}\rcirc\JJ_{k+r+2}}\cup\Omm{\Z_{r+1}\circ\JJ_{k+r+3}}. \end{equation} This formula for all values of $r$ that is proven, and Property 3.1, give:

\begin{equation}\label{zpolap}\Omm{\F_{n,k,1}}=\bigcup_{1\leq r\leq n-2}\Omm{\ja\QQ_r\rcirc\JJ_{k+r+1}}. \end{equation}
From property 3.5, we see that most of the terms in the union are empty. In detail:
\begin{equation}\label{zABCDp}\Omm{\F_{n,k,1}}=\Omm{\jj\QQ_1\rcirc\JJ_{k+2}}\cup\Omm{\jj\QQ_{n-2}\rcirc\JJ_{k+n-1}} \end{equation}
 By Proposition \ref{embedding} we have
 \begin{equation}\label{zBp}\Omm{\ja\QQ_{r}\rcirc\JJ_{k+r+1}}=[k+r+1,\Omm{\QQ_r}_{k+r+1}]\qquad\text{for}\qquad 1\leq r\leq n-2. \end{equation}
  From (\ref{zBp}) and Property 3.4 we obtain:
 \begin{equation}\label{zAAp}\Omm{\jj\QQ_{n-2}\rcirc\JJ_{k+n-1}}=\{\aA{n}{k}{1}{n-3}\}.\end{equation}

 From (\ref{zpolap}) and (\ref{zAAp}), we see that the subset of $\Omm{\F_{n,k,1}}$ which is equal to $\Omm{\jj\QQ_{n-2}\rcirc\JJ_{k+n-1}}$, consists of one element which is $\aA{n}{k}{1}{n-3}$. If  $k>2$, the rest of $\Omm{\F_{n,k,1}}$, is obtained from: (\ref{zABCDp});(\ref{zBp}); Property 3.2; Theorem \ref{thD1}; and if $n\geq 4$, from the two next trivial equalities,
 $$\bB{n}{k}{1}{n-3}=[k+2,\bB{n-2}{k}{k-1}{n-4}],$$ $$\BAB{n}{k}{1}{n-2}=\left[k+2,\AB{n-2}{k}{k-1}{n-2}\right]. $$If k=2, the rest of $\Omm{\F_{n,k,1}}$ is obtained from: (\ref{zABCDp}); (\ref{zBp}); the claim stated and proven below; and  from the trivial equality 
 $$\bB{n}{2}{1}{\lfloor\frac{n-2}{2}\rfloor}=\left[4,\bB{n-2}{2}{1}{\lfloor\frac{n-4}{2}\rfloor}\right]\qquad\text{for }n\geq 5. $$

\vsp
\noindent\textbf{Claim. }\textit{Assume $k=2$. $$\Omm{\QQ_1}_{4}=\left\{\begin{array}{cc}\{\bB{n-2}{2}{1}{\lfloor\frac{n-4}{2}\rfloor}\}&\text{if }n>4\\\hspace{1mm}[3]&\text{if }n=4\end{array}\right..$$}

\vspace{1mm}
\noindent\textit{Proof of Claim. }Recall (Property 3.2) that $\QQ_1=\F_{n-2,2,1}$. If $n\geq 5$ we are done by the inductive hypothesis. Finally If  $n=4$, $\QQ_1$ belongs in $\AAnk(U_3)[3] $, and hence the claim is obtained.  \hfill$\square$Claim$\square$Part 2 

\vspace{1mm}

\end{proof}

\begin{sdefi}[$\Skal{k}{(a_1,l_1),...,(a_m,l_m)}$, $\Skac{k}{(a_1,l_1),...,(a_m,l_m)} $]\label{cadhoc}Consider positive integers: $m\geq 1,$ $k\geq 2,$ $a_1,...,a_m, $ and $l_1,...,l_m$. Let $n=\sum_{1\leq i\leq m}a_i$ and $N=\sum_{1\leq i\leq m}l_i+k(n-m).$ We denote by $\Skal{k}{(a_1,l_1),...,(a_m,l_m)}$ the set consisting of the pairs $(\K,\FF)$ of AFs with the following properties:\begin{enumerate}
\item $(\K,\FF)\in \mathcal{S}_{n,N}$.
\item Throughout the current definition let $T^1,...,T^m$ be the $\DDDD[n]$-components of $\TT{n}{\K},$ ordered so that $$\min(\Set{T^1})<...<\min(\Set{T^m}).$$  We have $a_i=|\Set{T^i}|$. 
\item Fix an $i$ with $1\leq i\leq m$. For $T$ being the $\DDDD[N]$-component of $\TT{N}{\FF}$ satisfying $$\min(\Set{j_{\FF}(T^i)})=\min(\Set{T}),$$ we have  $l_i=|\Set{T}|$.
\item\label{3lastp} If $|\Set{T^1}|\geq 2,$ we require that:  for every $i$ for which $\min(\Set{T^i})$ is bigger or equal from the second smallest element of $\Set{T^1},$ we have $l_i=k.$ 
\item\label{2lastp} Let  $0\leq t< n.$ Let $T^{t,1},...$ be the $\DDDD[n]$-components of $\TT{n}{\K^{\{i:1\leq i\leq t\}}}$, ordered so that $\min(\Set{T^{t,1}})<\min(\Set{T^{t,2}}),...$ . Then $|\Set{T^{t,1}}|\geq |\Set{T^{t,2}}|\geq...$ . For example, for $t=0$ we have: $a_1\geq...\geq a_m$.

\item\label{lastp} Let $L$ be a $GL_n$-root group contained in $D_\K$. Let $l$ be the biggest number for which $p_{l}(j_{\FF}(L))$ is nontrivial, $t$ be the number for which $GL_t$ is the image of $\jm\cdot p_{l}$, and $\alpha$ be the root of $GL_t$ for which ${\jj}^{-1}\cdot p_{l}(j_{\FF}(L))=U_{t,\alpha}$. Then: \begin{enumerate}\item For all roots $\beta$ of $GL_t$ for which $\beta-\alpha$ is a positive root of $GL_t$, we have $U_{t,\beta}\in{\jj}^{-1}\cdot p_{l}(j_{\FF}(D_\K))$.\item For half among the positive roots $\beta$ of $GL_t$ for which $\alpha-\beta$ is a positive root of $GL_t$, we again have $U_{t,\beta}\in{\jj}^{-1}\cdot p_{l}(j_{\FF}(D_\K))$.\end{enumerate}
\end{enumerate}We also define \begin{equation*}\Skac{k}{(a_1,l_1),...,(a_m,l_m)}:=\{j_{\FF}(\K)\circ\FF:(\K,\FF)\in\Skal{k}{(a_1,l_1),...,(a_m,l_m)}\}.\qedhere\end{equation*}
\end{sdefi}Below we give a lemma which one can easily skip by replacing each of its uses with a few routine arguments. We use it in the proof of Properties 3.1 and 2.3 inside the proof of Theorem \ref{thD3}, where it implies in a not very ad-hoc way that the domain of an AF as we move along the path $\sI(\F)$  changes ``as usual". 
\begin{sdefi}[Only used in the lemma below]\label{generategroup}
 Let $H$ be a connected algebraic group and $S$ a set of connected algebraic subgroups of $H$ that generate $H.$ We then write $\langle S\rangle =H.$
 
 Also (as in Definition \ref{VpV}) for any $\ossa$-group $V$ we denote by $\rossa{V}$ the smallest group containing $V$ which is generated by root groups.\end{sdefi}

\begin{slem}\label{2aeu}Let $S$ be a set of $\ossa$-groups such that (i),(ii), (iii), and (iv) below hold:

\vsp 
\noindent{(i)}For any two groups $V_1,V_2$ in $S$, $[V_1,V_2]$ is a $\ossa$-group and $\rossa{[V_1,V_2]}=[\rossa{V_1},\rossa{V_2}]$.

\vsp
\noindent{(ii)}  For any two groups $V_1,V_2$ in $S$, there is a subset $S(V_1,V_2)$ of $S$ such that: any two groups in $S(V_1,V_2)$ commute; $[V_1,V_2]\subseteq\prod_{V_3\in S(V_1,V_2)} V_3$; and $\rossa{[V_1,V_2]}=\prod_{V_3\in S(V_1,V_2)} \rossa{V_3}$. 

\vsp
\noindent{(iii)} For every entry there is at most one group in S which is nontrivial on this entry.

\vsp 
\noindent{(iv)} $\langle S\rangle\subseteq U_n.$

\vsp
Let $\F\in\AAnk[N] $ be such that: $D_\F=\langle S\rangle $, and $\sI(\F)$ is a nontrivial path starting with a $\eu$-step. Let $\F^\prime$ be the output AF of this $\eu$-step. Also, let the $(k,l_1)$ entry be the first one ``removed" by $\sI(\F)$. That is, we first choose $l_1$ as big as possible so that $U_{(1,l_1)}\subset D_\F$ and $\F(U_{(1,l_1)})\neq \{0\}, $ and then choose $k$ as big as possible so that $D_\F$ is nontrivial on the $(k,l_1)$ entry. Then (I), (II), and (III), below hold:

\vsp
\noindent{(I). }On every entry on which $\langle S\rangle$ is nontrivial, a (unique due to (iii)) group in $S$ is also nontrivial. 

\vsp 
\noindent{Before stating (II) and (III)} we need to define $Y$ and $S^\prime.$ We define $Y$ to be the group in $S$ which is nontrivial on the $(k,l_1)$ entry. If $\langle S\rangle$ is trivial on the $(1,k)$ entry, we define $S^\prime:=S\cup\{U_{(1,k)}\}-\{Y\}.$ Otherwise, for $V_k$ being the group in $S$ which is nontrivial on the $(1,k)$ entry, and $p$ being as in Definition \ref{VpV} for $[V\tlarrow V_k]$ and for $R$ making the biggest (with respect to set containment) choice not containing $U_{(1,k)}$, we define $S^\prime:=S\cup\{U_{(1,k)},p(V_k)\}-\{Y,V_k\}.$

\vsp 
\noindent{(II). }$D_{\F^\prime}=\langle S^\prime\rangle$.

\vsp
\noindent{(III). } Conditions (i),(ii), (iii), and (iv) remain correct after the replacement $[S\tlarrow S^\prime].$
\end{slem}

\begin{proof}We use (iv) without mention.

From (i) and (ii) we inductively obtain that for $m>0$ the following holds: on every entry on which the set \begin{equation}\label{zsetm}\{u_1...u_m:\text{each }u_i\text{ belongs to a group in }S\}\end{equation} is nontrivial, a group in $S$ is nontrivial as well. For big enough $m$, the set in (\ref{zsetm}) is equal to $\langle S\rangle$ and hence (I) is obtained.

We continue with (II). As in the proof of Theorem \ref{general} let $L_{(k-1,l_1)}$ be the subgroup of $U_n$ generated by all positive root groups except for $U_{(i,l_1)}$ for $k-1<i<l_1$. From (ii) we obtain that
\begin{equation}\label{zvnotinz}Y\not\in\bigcup_{V_1,V_2\in S}S(V_1,V_2);\end{equation} and hence that $\langle S-\{Y\}\rangle$ is a normal subgroup of $\langle S\rangle$, and  $\langle S-\{Y\}\rangle Y=\langle S\rangle$. Hence, by also using the containments (the second being proper)
$\langle S-\{Y\}\rangle\subseteq\langle S \rangle\cap L_{(k-1,l_1)}\subset \langle S \rangle,  $ and Fact \ref{areconnected}, we obtain \begin{equation}\label{zcxv}\langle S-\{Y\}\rangle =\langle S\rangle\cap L_{(k-1,l_1)}.\end{equation}  
From the definition of $\sI$ we obtain the first equality below, and by using (\ref{zcxv}) (and $D_{\F}=\langle S\rangle$) we obtain the second one: $$D_{\F^\prime}=(D_{\F}\cap L_{(k-1,l_1)})U_{(1,k)}=\langle S-\{Y\}\rangle U_{(1,k)}.$$ The equality in this formula of the left hand side with the right hand side imply (II).

Let $k<l<l_1.$ From the definition of $\sI$ we have $U_{1,l}\subset D_\F.$ From this containment, from a basic property on the correspondence between algebraic groups and Lie algebras\footnote{ Let $G$ be an algebraic group, and $R$ be a set of closed connected subgroups of $G$ that generate $G.$ Then $\Lie{G}$ is generated as a Lie algebra from the Lie algebras $\Lie{H} $ for $H\in R.$},  and from (ii)  we have that $\Lie{U_{1,l}}$ is contained in the \textbf{vector space} direct sum $$\bigoplus_{X\in S}\Lie{X}.   $$This and (iii) imply that $U_{(1,l)}\in S.$ Hence (III) is obtained (by also using (i), (ii), (iii) and (iv)).
\end{proof}
\begin{sremark}Of course, by Fact \ref{areconnected} we obtain  $\
\F^\prime=\exchange{U_{(1,k)}}{Y}\F. $
\end{sremark}
\begin{sdefi}[$\F_{a,k}$]\label{ladefi}Let $a=[a_1,...,a_m]\in\UUU{n}$ and $\K_a\in\Pridbnk[n][a]_{\nless}\cap^\prime\BBnk[n]$. We define $\FF_{n,k,k}$ as previously (Definition \ref{zfnkl}). We define   $$\F_{a,k}:=j_{\FF_{n,k,k}}(\K_a)\circ\FF_{n,k,k}. $$ Note that $\F_{a,k}\in\Skac{k}{(a_1,k),...,(a_m,k)}.$   
 \end{sdefi}
\begin{remark}Note that by using Corollary \ref{laststrong} the theorem below implies its extended form in which for $\K_a$ one only assumes $\K_a\in\Prink[n][a]\cap^\prime\BBnk[n]$.
\end{remark}
\begin{stheorem}\label{thD3}Consider positive integers $n\geq 2,k\geq 2,m,a_1\geq...\geq a_m,$ such that $n=\sum_{1\leq i\leq m}a_i.$

\vsp
\noindent\textit{Part $(\alpha)$. }Let $a:=[a_1,...,a_m],$ and   $x$ be the number of occurrences of  1 in this partition. Let $\F_{a,k}$ be defined as previously. Let 
$$\aAbig{a}{k}{n-2m+x}:=[k+a_1-1,k+a_2-1,...k+a_m-1,[k-1]^{n-m}], $$  and 
$$\bBbig{a}{k}{n-2m+x}:=[[k+1]^{n-m},k-a_1+1,k-a_2+1,...,k-a_m+1]. $$ Then $\Omm{\F_{a,k} }=\left\{\begin{array}{cc}\{\aAbig{a}{k}{n-2m+x},\bBbig{a}{k}{n-2m+x}\}&\text{if }k-a_1+1\geq 0\\ \{\aAbig{a}{k}{n-2m+x}\}&\text{if }k-a_1+1<0 \end{array}\right..$

\vsp
\noindent\textit{Part $(\beta)$. }Consider positive integres $l_1,...,l_m,$ and let $\F\in\Skac{k}{(a_1,l_1),...,(a_m,l_m)}$. The set $$\Omm{\F}_{k+1}, $$contains at most one element. It is nonempty if and only if
$l_i-a_i+1\geq 0$ for all i. If it is nonempty, its unique element is
\begin{equation}\label{zBBBBb}[[k+1]^{n-m},l_1-a_1+1,l_2-a_2+1,...,l_m-a_m+1]. \end{equation}
\end{stheorem}
\begin{tproof}Let $N$ be the number for which $\F\in\AAnk[N].$ Within every statement about $\F_{a,k}$ we assume that $\F=\F_{a,k}$, and hence the meaning of notations defined in terms of $\F$ (e.g, $N$) is adjusted. 

By proceeding inductively we assume the theorem (both $(\alpha)$ and $(\beta)$) is correct for $N$ being replaced by any smaller number. In the case $a_1=1$, the theorem follows directly from condition \ref{2lastp} in Definition \ref{cadhoc}, and the definition of $\Om(\ast). $ We therefore assume $a_1>1.$

 We keep track of the uses of Conditions \ref{3lastp}, \ref{2lastp}, \ref{lastp} in Definition \ref{cadhoc} (and usually use the first three conditions without mentioning). 

\vsp
\noindent\textit{\textbf{Part 1.} The study of $\AAnk$-trees based only on direct use of $\AAnk$-tree definitions.}

\vsp 
\noindent\textit{\textbf{Subpart 1.1.} Definitions.} Let $(\K,\FF)\in S(k,(a_1,l_1),...,(a_m,l_m))$ be chosen so that $\F=j_{\FF}(\K)\circ\FF$. Definition \ref{dem} is adopted for the current choice of $\FF$ (for example  $\tj:=\jj_{\FF}$).

Let $T_n[a_1]$ be the  $\DDDD[n]$-component of $\TT{n}{\K}$ satisfying \\$\min(\Set{T_n[a_1]})=1$. Let $f(1),f(2),...,f(a_1)$, be the increasing sequence for which$$\Set{T_n[a_1]}=\{f(1),...,f(a_1)\} .$$  Of course we  have $f(1)=1.$ 

For $1\leq r\leq a_1,$ let $T_n\{r\}$ and $T_n[r]$ be the tori in $\DDDD[n]$ satisfying:

 $$\Set{T_n\{r\}}=\{f(r)\}\qquad\text{and}\qquad\Set{T_n[r]}=\{f(1),f(2),...,f(r)\}.$$  For $2\leq r\leq a_1, $ let $T_{N,r}$ be the torus in $\DDDD[N]$ satisfying
 $$\Set{T_{N,r}}:=\Set{\tj(T_n\{1\})}\cup\Set{p_{\{i:l\leq i\leq k-1\}}(\tj(T_n\{2\}))}\cup\Set{p_k(\tj(T_n[r]))}, $$and let $w_r$ be the minimal length element in $W_{N}$  satisfying $$w_rT_{N,r}w_r^{-1}=\{1,...,k+r-1\}.$$
 
We define $H$ to be the subgroup of $GL_n$ which is described as follows. If $l_1=k$, $H=GL_n^{\{1,2,...,f(2)-1\}}$. If $l_1\neq k  $, $H$ is the mirabolic subgroup of $GL_n^{\{1,...f(2)-1\}}$, in which the  $GL_1$ copy of the Levi appears in the upper left corner. Note that by condition \ref{3lastp} in Definition \ref{cadhoc} we have $H\subseteq P(\FF).$

 We define three homomorphisms: $\tjA: P(\FF)\rightarrow GL_N$ and     $\tjp,\tjB:H\rightarrow GL_N$, as follows: $$\tjA:=p_k\cdot\tj;$$for $l_1=k  $
 $$\tjp=\tj|_H\qquad\text{and}\qquad\tjB|=p_{\{i:1\leq i\leq k-1\}}\cdot\tj|_H; $$for $l_1\neq k $
  $$\tjp|_{GL_n^{\{1,2,...,f(2)\}}}=\tj|_{GL_n^{\{1,2,...,f(2)\}}},\qquad\qquad\tjB|_{GL_n^{\{1,2,...,f(2)\}}}=p_{\{i:1\leq i\leq k-1\}}\cdot\tj|_{GL_n^{\{1,2,...,f(2)\}}},$$ $$\tjp|_{\mathrm{row}f(2)}=p_{\{i:1\leq i\leq  l_1-1\}\cup\{k\}}\cdot\tj|_{\mathrm{rowf}(2)},\qquad\tjB|_{\mathrm{rowf}(2)}:=p_{\{i:1\leq i\leq l_1-1\}}\cdot\tj|_{\mathrm{rowf}(2)}.$$ 
  
  We define $$\tj(\K)^{T_n[r]}:=\tj(\K)|_{({\jj}^{-1}\cdot p_{l_1})^{-1}\left(\left({\jj}^{-1}\cdot p_{l_1}^{}(\tj(\K))\right)^{\{f(1),...f(r)\}}\right)}.$$ Some times we have $\tj(\K)^{T_n[r]}=\tj(\K^{T_n[r]})$ (e.g for $\F=\F_{a,k}$).
Finally we define $\Z_r(\F)$ and $\QQ_r(\F)$. It is routine\footnote{One can also obtain this from information occurring within the proofs of properties 1,2,3 below. For example, from the proof of Property 3.1 we obtain that $\Z_0(\F)$ is well defined.} to check that their  definitions are valid. Assume (for the rest of the current subpart) that: $l_1>1$ or $a_1=2.$ For $0\leq r\leq a_1-2$ we define
\begin{equation}\label{zdz}\Z_r(\F):=w_{r+2}\left(\tj(\K)^{T_n[r+2]}\circ\tjp(\K|_{\mathrm{row}f(r+2)})\circ \FF^{T_{N,r+2}}\right).\end{equation}
For $1\leq r+1\leq a_1-1$ let $N_{r+1}:=N-k-r-1$, and $\QQ_{r+1}(\F)$ be the AF in $\AAnk[N_{r+1}]$, such that for $\ja\QQ_{r+1}(\F)$ being its lower right corner copy in $\AAnk[N]$ we have 
$$\ja\QQ_{r+1}(\F)=w_{r+2}\left(\tj(\K)^{T_n[r+2]}\circ\tjB(\Krow{r+2})\circ \FF^{T_{N,r+2}}\right). $$\hfill$\square$Subpart 1.1

\noindent\textit{\textbf{Subpart 1.2}: The  results of Part 1.} 
In stating properties 0,1,2,3 below, it is assumed unless otherwise specified that: \begin{equation}\label{zzla}l_1>1 \qquad\text{ or }\qquad a_1=2.\end{equation}

\vsp
\noindent\textit{Property 0} . The AFs $\Z_r(\F)$ (resp. $\QQ_r(\F)$) are defined for $0\leq r\leq a_1-2$ (resp. $1\leq r\leq a_1-1$), and belong to $\AAnk[N]$  (resp. $\AAnk[N_r] $). We have $\jj\QQ_{a_1-1}(\F)=\Z_{a_1-2}(\F)$.  

\vspace{1mm}

\noindent{We use Property 0} without mention.

\vspace{1mm}
\noindent\textit{Property 1. }For all $r$ for which the definitions are valid (that is $0\leq r\leq a_1-2$ for the left hand side, and $1\leq r\leq a_1-1$ for the right hand side) we have: $$\Z_r(\F)\circ\JJ_{k+r+1}\in\BBnk[N]\qquad\text{and}\qquad\QQ_r(\F)\in\BBnk[N_r].$$

\vspace{1mm}
\noindent\textit{Property 2 (recursion relation). }Let $0\leq r<a_1-2$. There are AFs  $\Z_{r,A}(\F)$ and $\Z_{r,B}(\F)$ in $\AAnk[N] $ such that:\\2.1: $\Z_r(\F)\circ\JJ_{k+r+1}\rif\{\Z_{r,A}(\F)\circ\JJ_{k+r+1},\Z_{r,B}(\F)\circ\JJ_{k+r+1}\}; $\\
2.2: $\Z_{r+1}(\F)\circ \JJ_{k+r+2}\rif \Z_{r,A}(\F)\circ\JJ_{k+r+1}; $\\2.3: $\ja\QQ_{r+1}(\F)\rcirc\JJ_{k+r+1}\rif \Z_{r,B}(\F)\circ\JJ_{k+r+1};$\\2.4:$\Di{D_{\Z_r(\F)\circ\JJ_{k+r+1}}}=\Di{D_{\Z_{r+1}(\F)\circ \JJ_{k+r+2}}}=\Di{D_{\ja\QQ_{r+1}(\F)\rcirc\JJ_{k+r+1}}}. $

\vspace{1mm} 
\noindent\textit{Property 3.} 3.1: We remove the requirement (\ref{zzla}) (only for 3.1). If $l_1>1\text{ or }a_1=2$, then $\siota(\F)=\Z_0(\F)\circ\JJ_{k+1} $. If $l_1=1\text{ and }a_1>2 $, then $\siota(\F)|_{D_{\JJ_{k+2}}}=\JJ_{k+2}.$   \\3.2:We have $$\QQ_1(\F)\in\Skac{k}{(a^\prime_1,l_1^\prime),...,(a_{m^\prime}^\prime,l_{m^\prime}^\prime)} $$ where $m^\prime$ and the pairs $(a^\prime_1,l_1^\prime),...,(a_{m^\prime}^\prime,l_{m^\prime}^\prime)$ are defined as follows: If $l_1>1$, then $m^\prime=m$ and the pairs  are the same as the pairs $(a_1-1,l_1-1),(a_2,l_2),...,(a_m,l_m),$ but in the order after we move $(a_1-1,l_1-1)$ to the right of every $(a_i,l_i)$ for each $2\leq i<f(2)$; and if $l_1=1$, then $m^\prime=m-1$ and the pairs are the same as $(a_2,l_2),...,(a_m,l_m)$ and in the same order.\\3.3: We do not need (and hence omit) any analogue of Property 3.3 of Theorem \ref{thD1}.\\3.4: Throughout property 3.4 let $\F=\F_{a,k}.$ We have 

  $$\QQ_{a_1-1}(\F_{a,k} )\rif \ja\F_{a^\ih,k}\rcirc\Jdeg[N_{a_1-1}]$$ where $a^\ih:=[a_2,...,a_m]$, $\ja\F_{a^\ih,k}$ is the lower right corner copy of $\F_{a^\ih,k}$ in $\AAnk[N_{a_1-1}] $, and $\Jdeg[N_{a_1-1}]$ is an AF satisfying: \begin{itemize}
\item $J_{\Jdeg[N_{a_1-1}]}\in\Prink[a_1-1][[k-1]^{a_1-1}] $,
\item the set $D_{\Jdeg[N_{a_1-1}]}$ is the unipotent\chh[j] radical of the $GL_{N_{a_1-1}}$-parabolic subgroup with Levi $$GL_1^{(a_1-1)(k-1)}\times^{\searrow} GL_{N_{a_1-1}-(a_1-1)(k-1)}.$$
\end{itemize}Of course a choice (among conjugates) of $\F_{a^\ih,k}$ is made as needed.\\
3.5: $n-1> r>1\implies\QQ_r(\F_{a,k} )\rif \emptyset$. 

\hfill$\square$Subpart 1.2

\vsp 
\noindent\textit{\textbf{Subpart 1.3: }The proof of Properties 0,1,2,3.} Every $\AAnk$-tree that is mentioned is assumed to be an $(\AAnk,\KKK)$-tree.

\noindent\textit{Proof of Property 0. } It follows directly from the definitions in Subpart 1.1.\hfill$\square$Property 0

\noindent\textit{Proof of Property 3.1.} By Lemma \ref{2aeu}, we explicitly see how in an $\AAnk$-step of $\sI(\F) $ the output AF differs from the input AF, for as long as---starting with the first $\AAnk$-step of $\sI(\F) $--- these $\AAnk$-steps are either $\eu$-steps or $\co$-steps obtained by a conjugation by an element in $W_N;$ and due to condition \ref{lastp} in Definition \ref{cadhoc}, we see that this is the case for all the $\AAnk$-steps in $\sI(\F)$.

More details follow. Let
 $$T^{\prime\prime},w^{\prime\prime},k^{\prime\prime}:=\left\{\begin{array}{cl}T_{N,2},w_2,k&\text{if }l_1>1\text{ or }a_1=2 \\T_{N,3},w_3,k+1&\text{if }l_1=1\text{ and }a_1>2\end{array}\right..$$ Let $C_1,...,C_{k^{\prime\prime}}$ be the elements of $\SRV{T^{\prime\prime}}$ in the order that: starts by choosing $C_1$ to be in the first row, and moves downwards. Then the conjugate of $\sI(\F)$ in which there is only one $\co$-step and appears in the end is equal to 
$$\co\vee\I_{k^{\prime\prime}}\vee...\vee\I_1  $$ where:\begin{enumerate}
\item For $1\leq i\leq k^{\prime\prime},$ the path $\I_i$ is obtained by successive applications of Lemma \ref{2aeu} for the following choices of $Y$: $Y$ is as in the statement of this Lemma after conjugating with the minimal length element in $W_N$ which maps $U_{(1,2)},...U_{(i-1,i)}$ respectively to $C_1,...,C_{i-1}$, $Y$ is either a root group or a group of the form $\tj(V)$ for $V$ being a $GL_n$-root group,   $Y\subset D_\F$, $Y$ is nontrivial on the column on which $C_i$ lies, and $Y$ is trivial on any row on which a root group in $\SRV{T^{\prime\prime}}$ lies. 
\item $\co$ is the $\co$-step obtained from the conjugation by $w^{\prime\prime}$. 
\end{enumerate}\hfill$\square$Property \textit{3.1} 

\vsp 
\noindent\textit{Proof of Property 3.2. }It directly follows  from the definition of $\QQ_1(\F).$ We give more details. For $1\leq x\leq \ref{lastp}$ we  denote by $x^\prime$ the condition $x$ in Definition \ref{cadhoc} for $$\left[\tj_{\FF}(\K)\tlarrow j_{1}^{-1}\left(w_{2}(\tj(\K)^{T_n[2]}\circ\tjB(\K|_{\rowf{2}}))\right),\FF\tlarrow j_{1}^{-1}\left(w_{2}(\FF^{T_{N,2}})\right) \right]$$(where the second ``$\tlarrow$" precedes the first one). Up to several uses of condition 1 which we do not mention, we have: $x\implies x^\prime$ for $x=1,2,3,\ref{3lastp},\ref{2lastp}$; and  $(\ref{3lastp}\text{ and }\ref{lastp})\implies \ref{lastp}^\prime.$\hfill $\square$Property \textit{3.2}

\vsp
\noindent\textit{Proof of the Property 2.1. } Let  \begin{equation}\label{zvma}\xi=\I[](\XX)\vee_\XX\xi_1 \end{equation} where:\begin{enumerate}
\item[(i)] $\xi_1$ is the $(\Z_r(\F)\circ\JJ_{k+r+1},\e)$-quasipath along constant terms over the  root groups of the domain of  $w_{r+2}\tjA(\K|_{\mathrm{rowf}(r+2)})w_{r+2}^{-1}$ in the order described next. For $1\leq i\leq a_1-r-2,$ let $X_i$ be the set of  $w_{r+2}\tjp(U_{n,(f(r+2),x)})w_{r+2}^{-1}$ for $x$ being the $i$-th biggest number of a set in $\Set{\TT{n}{\K}}.$ Then for $j>i$ we finish with $X_i$ before starting with $X_j,$ and in each $X_i$ we start with its rightmost element  and moving left except for $w_{r+2}\tjp(U_{n,(f(r+2),f(r+2+i)})w_{r+2}^{-1}$ which is chosen last (among the elements of $X_i$).
\item[(ii)] For each $1\leq i< a_1-r-2$ let $\XXX_i$ be the set of nonconstant terms of  the $\e$-step of $\xi_1$ which is over $w_{r+2}\tjp(U_{n,(f(r+2),f(a_1+1-i)})w_{r+2}^{-1}$. Then, in (\ref{zvma}) the AF $\XX$ varies over $\cup_{1\leq i< a_1-r-2}\XXX_i$
\item[(iii)]  For  $\XX\in \XXX_i$  we define $\I(\XX) $ to be the $(\XX,\eu)$-step obtained from \begin{equation}\label{zmeg}\exchange{w_{r+2}\tjA(U_{n,(f(r+2),f(a_1-i)})w_{r+2}^{-1}}{w_{r+2}\tj(U_{n,(f(a_1-i),f(a_1+1-i)})w_{r+2}^{-1}}.\end{equation}    
 \end{enumerate}

\noindent{For}  $1\leq i<a_1-r-2$, let $$\mathcal{Y}_i:=\{\YY:\text{for an }\XX\in\XXX_i\text{, the output AF of }I(\XX)\text{ is }\YY\}$$ and  $T_n(i)$ be the torus in $\DDDD[n]$ defined by
$\Set{T(i)}=\{f(a_1+1-j):1\leq  j\leq i\}. $ We see that $w_{r+2}\tj(T_n(i))w_{r+2}^{-1}$, acts transitively in $\mathcal{Y}_i.$

Consider an $\e$-step in $\xi_1$ which is over $w_{r+2}\tjp(U_{n,(f(r+2),x)})w_{r+2}^{-1}$ for an $x$ not in $\Set{T_n[a_1]}$.  Then for $T$ being the $\DDDD[n]$-component in $\TT{n}{\K}$ such that $x\in\Set{\K}$,  the torus  $w_{r+2}\tj(T)w_{r+2}^{-1}$ acts transitively on the set of nonconstant terms of this $\e$-step.

The only output AFs of $\xi$ that are left are the terms of the last $\e$-step. We proceed identically as in Theorem \ref{thD1}. In the set consisting of these AFs, the torus $$w_{r+2}T_{N,r+2}\tjB(T_n\{r+2\})w_{r+2}^{-1} $$ acts with three orbits, two of them being singletons. The two elements in the singleton orbits are of the form $\Z_{r,A}(\F)\circ\JJ_{k+r+1}$ and $\Z_{r,B}(\F)\circ\JJ_{k+r+1}$ where $\Z_{r,A}(\F)$ and $\Z_{r,B}(\F)$ are AFs satisfying \begin{equation}\label{zmionena3}w_{r+2}\Z_{r,B}(\F)\left(\tjA (U_{n,(f(r+2),f(r+3))})\right)=\{0\}\quad\text{and}\quad w_{r+2}\Z_{r,A}(\F)\left(\tjB(U_{n,(f(r+2),f(r+3))})\right)=\{0\}.\end{equation}\hfill$\square$Property 2.1 and definitions of $\Z_{r,A}(\F)$ and $\Z_{r,B}(\F)$

\vsp
\noindent\textit{Proof of  Properties $1,2.2,2.3,2.4,$. } The proof is directly obtained  from the proof of the properties with the same names in Theorem \ref{thD1} by replacing the symbols $n,$ $\WW_n,$ $\FF_{n,k,l},$ $\row{r+2},$ $\Z_r,$ $\QQ_r,$ $\Z_{r,A},$ and $\Z_{r,B}$ with $a_1,$ $\K,$ $\FF,$ $\rowf{r+2},$ $\Z_r(\F)$,  $\QQ_r(\F), $ $\Z_{r,A}(\F),$ and $\Z_{r,B}(\F)$ respectively; with the  exceptions of replacing $\tj(\WW_n^{T_{n}[r]})$ with $\tj(\K)^{{T_{n}[r]}}$ and not replacing $\mathcal{S}_{n-1,N_{r+1}}$ with anything. Of course any redefined symbols (e.g. $T_{N,r},T_n[i],N$) admit their new meanings. 

Also to increase precision: we can use Lemma \ref{2aeu} for the $(\eu,\co)$-path in the proof of 2.3, and we replace the first sentence in Property 2.2 with ``We consider the initial $\e$-subquasipath of $\sXi(\ja\QQ_{r+1}(\F)\circ\JJ_{k+r+1})$ along constant terms  consisting of  $\Di{\K|_{\rowf{r+2}}}$ $\e$-steps".\hfill$\square$Properties 1,2.2,2.3,2.4

\vsp
\noindent\textit{Proof of Property 3.5.}  Consider the $(\QQ_{r+1}(\F_{a,k}),\e) $-quasipath along constant terms over the groups of the form  $j_{r+1}^{-1}(w_{r+2}\tjB(U_{n,(f(r+1),t)})w_{r+2}^{-1})$, where: $t$ is the $y$-th biggest number in a set in $\Set{\TT{n}{\K}}$ for a $y\leq a_1-r-3$, or   (in the last $\e$-step) $t=f(r+3)$. In all these $\e$-steps except the last one, we consider the action by $j_{r+1}^{-1}(w_{r+2}\tjB(T_n\{r+1\})w_{r+2}^{-1})$, which consists of two orbits, one consisting of the constant term. In the last $\e$-step we consider the action  by $j_{r+1}^{-1}(w_{r+2}\tjB(U_{n,(f(r+2),f(r+1))})w_{r+2}^{-1}).$ This action is transitive, and hence we are done.\hspace*{50mm}\hfill$\square$Property 3.5

\vsp 
\noindent\textit{Proof of Property 3.4.} Let $w$ be the minimal length element in $W_{N_{a_1-1}}$ such that: for $T$ being the torus in $\DDDD[N_{a_1-1}]$ satisfying $\Set{T}=\{f(i)-1:2\leq i\leq a_1\}$, we have $\Set{wTw^{-1}}=\{1,...,a_1-1\}.$ 

We consider the $\QQ_{a_1-1}(\F_{a,k} ) $-quasipath starting with conjugation by $w$ and then continuing (and finishing) with the biggest initial subquasipath of $\sXi$ of the form 
$$ \xi_{a_1-1}\vee\I_{a_1-1}...\vee\xi_1\vee\I_1$$where \begin{itemize}
\item Each $\I_{i}$ is a $(\eu,\co)$-path and each $\xi_i$ is an $\e$-quasipath along constant terms;
\item Each $\vee$ is over a single AF, and in case $\vee$ is followed by $\I_{i+1}$ on the left, this AF is the constant term of $\xi_i.$ 
\end{itemize} Note that the output AF of each such $\I_i$ is obtained from Observation \ref{unipc}. Also note---to make the proof more self contained---that each $\xi_i$ is over  $U_{N_{a_1-1},(i(k-1),j)}$ where $ i(k-1)<j\leq N_{a_1-1}-n+a_1.$\hfill$\square$Property 3.4$\square$Subpart 1.3$\square$Part 1

\vsp
\noindent\textit{\textbf{Part 2:} The use of Main Corollary \ref{maincor}.} If $l_1>1$, by using Properties 1,2, we apply Exchange corollary \ref{exchange} for each $r$ satisfying  $0\leq r<a_1-2,$ and obtain the recursion relation of orbits:
\begin{equation}\label{recursionorbit3}\Omm{\Z_{r}(\F)\circ\JJ_{k+r+1}}=\Omm{\ja\QQ_{r+1}(\F)\rcirc\JJ_{k+r+1}}\cup\Omm{\Z_{r+1}(\F)\circ\JJ_{k+r+2}}. \end{equation} Next we consider the case $\F=\F_{a,k}. $ Formula (\ref{recursionorbit3}), for all values of $r$ that is proven, and Property 3.1, give:
\begin{equation}\label{zzzbo}\Omm{\F_{a,k}}=\bigcup_{1\leq r\leq a_1-1}\Omm{\ja\QQ_r(\F_{a,k} )\rcirc\JJ_{k+r}}, \end{equation}and then by applying Property 3.5 we have:

\begin{equation}\label{zzzABCD}\Omm{\F_{a,k}}=\Omm{\jj\QQ_1(\F_{a,k} )\rcirc\JJ_{k+1}}\cup\Omm{\jj\QQ_{a_1-1}(\F_{a,k})\rcirc\JJ_{k+a_1-1}}. \end{equation}
By using Proposition \ref{embedding},  we obtain:
\begin{equation}\label{zzzB}\Omm{\ja\QQ_{r}(\F_{a,k})\rcirc\JJ_{k+r}}=[k+r,\Omm{\QQ_1(\F_{a,k})}_{k+r}]\qquad\text{for}\qquad 1\leq r\leq a_1-1. \end{equation}

Let $n^\ih=n-a_1,m^{\ih}=m-1,$ and $x^\ih=x$ (this is not a typo). From (\ref{zzzbo}), (\ref{zzzB}),  the claim stated and proven below, and the trivial equality $$\aAbig{a}{k}{n-2m+x}=[k+a_1-1,[k-1]^{a_1-1},\aAbig{a^\ih}{k}{n^\ih-2m^\ih+x^\ih}], $$ we  see that the subset of $\Omm{\F_{a,k}}$ which equals $\Omm{\jj\QQ_{a_1-1}(\F_{a,k})\rcirc\JJ_{k+a_1-1}}$, consists of a unique element which is $\aAbig{a}{k}{n-2m+x} $. We see that the rest of $\Omm{\F_{a,k}}$ is obtained from: (\ref{zzzABCD}),  (\ref{zzzB}), Property 3.2, and the inductive hypothesis on $(\beta)$ (recall $(\beta)$ is part of the statement of the theorem). To be more concise, we omit mentioning the usual trivial equality (and also omit it in the proof of $(\beta)$ below).

\vsp 
\noindent\textbf{Claim. }\textit{$\Omm{\QQ_{a_1-1}(\F_{a,k})}=\{\big[[k-1]^{a_1-1},\aAbig{a^\ih}{k}{n^\ih-2m^\ih+x^{\ih}}\big]\}$.}

\vsp

To finish the proof of the theorem  we are left with proving this claim, and $(\beta)$.  

\vsp 
\noindent\textit{Proof of  $(\beta)$. }We proceed very similarly to some of the thoughts above. 

In case $l_1=1$ and $a_1>2,$ from Property 3.1 we obtain $\Omm{\F}_{k+1}=\emptyset,$  and hence we are done. 

We are left with the case: $l_1>1$ or  $a_1=2$. It is sufficient to prove \begin{equation}\label{zzztial}\Omm{\F}_{k+1}=\Omm{\ja\QQ_1(\F)\circ\JJ_{k+1}}_{k+1}=[k+1,\QQ_1(\F)_{k+1}], \end{equation} because then we are done by Property 3.2 and the inductive hypothesis. The second equality in (\ref{zzztial}) is obtained from Proposition \ref{embedding}. Finally, the first equality in (\ref{zzztial}) is obtained: from Property 3.1; and in case $a_1>2$, by also  using (\ref{recursionorbit3}) for $r=0$ and noting that $\Omm{\Z_1(\F)\circ\JJ_{k+2}}_{k+1}=\emptyset.$\hfill$\square$($\beta$)

\vspace{1mm}
\noindent\textit{Proof of Claim.}  
We define
$$c^1:=[[k-1]^{a_1-1},\aAbig{a^\ih}{k}{n^\ih-2m^\ih+x^\ih}]\qquad\text{and}\qquad c^2:=[[k-1]^{a_1-1},\bBbig{a^\ih}{k}{n^\ih-2m^\ih+x^\ih}]. $$
From Property 3.4 and Main corollary \ref{maincor} we have
$$\Omm{\QQ_{a_1-1}(\F_{a,k})}=[[k-1]^{a_1-1},\Omm{\F_{a^\ih,k}  }]\cap\left\{c\in\UUU{N_{a_1-1}}:\frac{\Di{c}}{2}=\Di{D_{\QQ_{a_1-1}(\F_{a,k})}}\right\}, $$ and therefore, by also using the inductive hypothesis, we have \begin{equation}\label{zzzal}
\Omm{\QQ_{a_1-1}(\F_{a,k})}=\{c^1,c^2\}\cap\left\{c\in\UUU{N_{a_1-1}}:\frac{\Di{c}}{2}=\Di{D_{\QQ_{a_1-1}(\F_{a,k})}}\right\}.\end{equation} Hence we are  left with proving that the right hand side of (\ref{zzzal}) is equal to $\{c^1\}$.

By recalling (only for convenience) the proof of Property 3.4 (and ``ignoring" in it the  $\e$-steps), we obtain: $\Di{D_{\QQ_{a_1-1}(\F_{a,k})}}=$
\begin{equation}\label{zzzkialimia}\Di{D_{\F_{a^\ih,k} }}+\sum_{1\leq i\leq a_1-1}\Di{D_{\JJJJ{N_{a_1-1}-(i-1)(k-1)}{k-1}}}+(a_1-1)(n-a_1). \end{equation}  Let $c:=[c_1,c_2,...]$ be an orbit in $\{c^1,c^2\}.$ By using Lemma \ref{dimensioncounting}, and then (to obtain the first term of the right hand side below) using Main corollary \ref{maincor}, we have:  $\frac{\Di{c}}{2}=$
\begin{equation}\label{zzzalimia}\Di{D_{\F_{a^\ih,k}  }}+\sum_{1\leq i\leq a_1-1}\Di{D_{\JJJJ{N_{a_1-1}-(i-1)(k-1)}{k-1}}}\\+(a_1-1)\sum_{i}\text{max}\{c_i-k+1,0\}.\end{equation} Since  (\ref{zzzkialimia}) and (\ref{zzzalimia}) differ only in the last term, by using (\ref{zzzal}), (\ref{zzzkialimia}) and (\ref{zzzalimia}), we are left with checking that: 
$$\sum_{i}\text{max}\{c_i-k+1,0\}=n-a_1$$ is correct if and only if $c=c^1$.\hspace{62mm}\hfill$\square$Claim $\square$Part 2\end{tproof}

\subsection{A family of Rankin-Selberg integrals}\label{zpre3}We consider a set $\mathcal{X}$ of Rankin-Selberg integrals, in which after some unfolding (same for all $I\in\XXX$), an AF adressed in one among Theorems \ref{thD1} and \ref{thD2} appears; more precisely, for the integral $I\in\XXX$ defined in (\ref{intI}) below, this AF is a conjugate transpose of $\F_{n,k,k-1}$. From this unfolding (Proposition \ref{ziden}), Main corollary \ref{maincor}, and Corollary \ref{easyrefcor}, we directly conclude that the integrals in a certain subset $\mathcal{Y}$ of $\mathcal{X}$ are factorizable (and expect them to be nonzero) for appropriate choices of data (for each choice of  representations that make the integral belong to $\mathcal{Y}$). By adding to the previous sentence a use of Theorems \ref{thD1} and \ref{thD2} we directly describe $\mathcal{Y}$ in terms of Speh representations and (parabolic) induction. 

Beyond proving the integrals in $\mathcal{Y}$ are factorizable, I am at an early stage in studying them. I checked in a few cases, the points where Eisenstein series possibly contribute a pole in the integral, and the information seemed to be compatible with: expressing the integral with $L$-functions with known analytic continuation.   
 
 In addressing integrals in $\mathcal{Y}$ in the future, a calculation that may worth doing early is: see how they relate to familiar integrals, as a result of unfolding them in ways different from the one we choose in the present paper. We later give one such example (Example \ref{zktwo}). 
 \begin{sdefi}[$\Pac{x}$, $Z_x$]\label{accociate}Let $x$ be a positive integer. We denote by $\Pac{x}$ the
  $GL_x$-parabolic subgroup with Levi component $GL_{x-1}\times^{\searrow} GL_1$ (here ``ac" stands for ``associate"). Also we denote by $Z_x$ the center of $GL_x.$ 
 \end{sdefi}
 
 The notations $\FF_{n,k,k},\F_{n,k,k-1},j_{\FF},\tj$, were explained in Definitions \ref{dem} and \ref{zfnkl}, and their meaning is retained in the current subsection.  
 \begin{sdefi}[$I,\mathcal{X}$]\label{fldef}  Let $N,n,k$ be three positive integers satisfying $N=nk$, and $\kkk$ be a number field. Let   $\pi$ be a $GL_n(\A)$-cuspidal automorphic representation, and $\varphi$ a decomposable cusp form in it. Let  $\rho\in\Aut{\kkk,N-1,>}$ and $\chi\in\Aut{\kkk,1},$ be chosen so that the Eisenstein series $E$ obtained from a section $f\in\Ind_{\Pac{N}(\A)}^{GL_N(\A)}(\rho\otimes\chi\otimes\trivr{U_{\Pac{N}}})$ is absolutely convergent. Let $\varphi$ be a $GL_n(\A)$-cusp form. Let $\FF:=\FF_{n,k,k} $. We write $\tj$ instead of $j_{\FF}$ (recall  Definition \ref{dem}). We define the integral 
\begin{multline}\label{intI}I:=\int_{GL_n(\kkk)Z_N(\A)\s GL_n(\A)}\varphi(g) \FF(E)(\tj(g))dg\\=\int_{(\tj( GL_n)D_{\FF})(\kkk)Z_N(\A)\s (\tj( GL_n)D_{\FF})(\A)}\varphi(g)E(n\tj(g))\psi_{\FF}^{-1}(n)dnd\tj(g). \end{multline}We denote by $\XXX$ the set of the integrals $I$ for all possible choices of data within the current definition.
\end{sdefi} \noindent{All} symbols from the definition above (not just $I$ and $\mathcal{X}$) retain their meaning throughout the subsection.
\begin{sdefi}[$\mathcal{Y}$]\label{fldef2}By $\mathcal{Y}$ we denote the subset of $\mathcal{X}$ consisting of the integrals $I$ in which $\rho$ is further required to satisfy $\Om^\prime(\rho)\in\Omm{\F_{n,k,k-1}}.  $
\end{sdefi}
The unfolding of an $I\in\XXX$ until a conjugate transpose of $\F_{n,k,k-1}$ is reached, is given in two steps: first by using the Eisenstein series expansion of $E,$ and then by using the Fourier expansion over $U_n(\kkk)\s U_n(\A)$ of $\varphi.$ In both steps, there is a group action (by conjugation) with an open orbit, and every other orbit gives zero contribution. The outcome  is summarized in Proposition \ref{ziden}. We now start this unfolding.

  By using the Eisenstein series expansion of $E$ to unfold $I$ we have: $I=$
\begin{equation}\label{zIdouble}\sum_{w\in \Pac{N}(\kkk)\s GL_N(\kkk)/(\tj( GL_n)D_{\FF})(\kkk)}\int_{(w^{-1}\Pac{N}w\cap \tj(GL_n)D_{\FF})(\kkk)Z_N(\A) \s (\tj( GL_n)D_{\FF})(\A)}\varphi(g)f(wn\tj(g))\psi_{\FF}^{-1}(n)dndg. \end{equation}
By using the lemma below, we factor out from all the summands above except one,  the integral $\int_{\kkk\s \A}\psi_\kkk^{-1}(x)dx$ (and hence conclude that each such summand is zero). We continue the unfolding of $I$ after proving this lemma.
\begin{sdefi}[$\mathrm{diag}_{\ast}(\ast)$]\label{fldef3} For  $t_1,...t_N\in\KKK$, we define $$\mathrm{diag}_{GL_N}(t_1,...,t_N):=\begin{pmatrix}t_1&&\\&\ddots&\\&&t_N
\end{pmatrix}.$$ Let $j$ be a standard embedding of $GL_m$ in $GL_N.$ We define \begin{equation*}\mathrm{diag}_{j(GL_m)}(t_1,...,t_m):=j(\mathrm{diag}_{GL_m}(t_1,...,t_m)).\qedhere \end{equation*}\end{sdefi} 
Unless otherwise specified, identities including as variables $t_1,t_2,...$ or $g$ are assumed to be stated for all values of these variables that make all the definitions involved valid (e.g. each $t_i$ will always take as values all elements in $\KKK-\{0\}$).

\begin{slem}\label{cosets} There is a set of representatives for the double cosets of $$\Pac{N}(\kkk)\s GL_N(\kkk)/(\tj( GL_n)D_{\FF})(\kkk),$$which consists only of elements of $W_N$. There is a unique double coset, so that for $w$ being a representative we have
\begin{equation}\label{oneterm}\FF(w^{-1}U_{\Pac{N}}w\cap D_{\FF})=\{0\}. \end{equation} 
\end{slem}
\begin{proof}Let $P_{\FF}$ be the $GL_N$-parabolic subgroup with unipotent radical equal to $D_{\FF}$, and  $M_{\FF}={\prod_{1\leq i\leq k}}^{\pse}G_i$ be the Levi of $P_{\FF}.$ 

The first sentence of the lemma will follow by using the Bruhat decomposition to prove:
\begin{equation}\label{zzcosets}\Pac{N}(\kkk)\s GL_N(\kkk)/(\tj (GL_n)D_{\FF})(\kkk)=\Pac{N}(\kkk)\s GL_N(\kkk)/P_{\FF}(\kkk). \end{equation}

Consider a double coset $X$ of $\Pac{N}(\kkk)\s GL_N(\kkk)/P_{\FF}(\kkk)$. Due to the Bruhat decomposition there is a  $w\in W_N$ representing it. By choosing $M^1_{\FF}$ to be the product of all the $G_i,$ except an appropriate one, we have:   $$(i)\hspace{1mm}M_{\FF}(\kkk)=(M_{\FF}^1\tj(GL_n))(\kkk)\qquad\text{and}\qquad (ii)\hspace{1mm} wM_{\FF}^1(\kkk)w^{-1}\subset \Pac{N}(\kkk). $$  

Let $x\in X$. Then $x=pwp_{\FF}$ for some $p\in \Pac{N}(\kkk)$ and $p_{\FF}\in P_{\FF}(\kkk)$. By $(i)$ we decompose: $p_{\FF}=mh$ where $m\in M^1_{\FF}(\kkk)$ and $h\in (\tj(GL_n)D_{\FF})(\kkk).$  We have \begin{equation}x=(pwmw^{-1})wh.\end{equation} We see from $(ii)$ that $pwmw^{-1}\in \Pac{N}(\kkk)$. Hence $x\in \Pac{N}(\kkk)w(\tj(GL_n)D_{\FF})(\kkk)$. By letting $x$ vary over $X$, we obtain that $X$ is a subset of (and hence equal to) $\Pac{N}(\kkk)w(\tj(GL_n)D_{\FF})(\kkk)$. Therefore we proved (\ref{zzcosets}). We are left with proving that there is exactly one double coset with its representatives satisfying (\ref{oneterm}).

Let $w_1$ be defined by $$w_1\text{diag}_{GL_N}(t_1,t_2,...t_N)w_1^{-1}=\text{diag}_{GL_N}(t_2,t_3,...t_{N},t_1).$$ We see 
that (\ref{oneterm}) is satisfied for $w=w_1$. In any double coset except the one  represented by $w_1,$ there is  a  number $i>n$ and a representative $w\in W_N$ such that
$$w\text{diag}_{GL_N}(t_1,t_2,...t_N)w^{-1}=\text{diag}_{GL_N}(t_1,t_2,...t_{i-1},t_{i+1},...t_{N},t_i). $$ Then $$U_{(i-n,i)}\subseteq  w^{-1}U_Pw\cap D_{\FF},$$ and since $\FF(U_{(i-n,i)})\neq \{0\},$ we are done.
\end{proof}
Returning to the calculation of $I$, we see that in (\ref{zIdouble}) only the summand corresponding to $\Pac{N}(\kkk)w_1(\tj(GL_n)D_{\FF})(\kkk)$ is possibly nonzero.  

The representative we choose for the double coset  $\Pac{N}(\kkk)w_1(\tj (GL_n)D_{\FF})(\kkk)$ will not be $w_1$. We choose instead the element $w_a$ defined by:
$$w_ag{w_a}^{-1}=w_1gw_1^{-1}\quad\forall g\in G_1,\quad\text{ and }\quad w_a\text{diag}_{G_i}(t_1,t_2,...t_n)w_a^{-1}=\text{diag}_{w_1G_iw_1^{-1}}(t_2,t_3,....,t_n,t_1). $$  
We also define $\FF^\prime $, $\tj^\prime$, $w_b$ and $L$ by:  $$\FF^\prime:=(w_a\FF)|_{\Pac{N}},\qquad w_b\text{diag}_{GL_n}(t_1,t_2,...t_n)w_b^{-1}=\text{diag}_{GL_n}(t_2,t_3,...t_n,t_1),$$ $$\tj^\prime(g):= w_a\tj(w_b^{-1}gw_b)w_a^{-1}, \qquad X:=(w_a\Pac{N}w_a^{-1}\cap\tj(GL_n)D_{\FF})(\A)\s (\tj(GL_n)D_{\FF})(\A).$$
We have
\begin{equation}\label{secint}I=\int_{X}\int_{\Pac{n}(\kkk)Z(\A)\s \Pac{n}(\A)}\varphi(pw_bg)\FF^\prime(f)(\tj^\prime(p)w_an\tj(g))\psi_{\FF}^{-1}(n) dpdnd\tj(g). \end{equation}where $n$ and $\tj(g)$ are the variables integrated over $X$. In the next unfolding step we proceed identically as in the familiar $GL_n\times GL_{n-1}$ Rankin-Selberg integrals. That is, we choose an AF $\WW_n$ in $\AAA(U_n)[n],$ we consider the Fourier expansion  

\begin{equation}\label{zps}\varphi(pw_bg)=\sum_{\gamma\in U_n(\kkk)\s \Pac{n}(\kkk)}\WW_n(\varphi)(\gamma pw_bg), \end{equation} and then we unfold as follows:
$$I=\int_{X}\int_{U_n(\kkk)Z_n(\A)\s \Pac{n}(\A)}\WW_n(\varphi)(pw_bg)\FF^\prime(f)(\tj^\prime(p)w_an\tj(g))\psi_{\FF}^{-1}(n)dpdnd\tj(g) $$ 
\begin{equation}\label{zlastintegral}=\int_{X}\int_{U_n(\A)Z_n(\A)\s \Pac{n}(\A)}\WW_n(\varphi)(pw_bg)\left(\tj^\prime(-\WW_n\right)\circ\FF^\prime)(f)(\tj^\prime(p)w_an\tj(g))\psi_{\FF}^{-1}(n)dnd\tj(g). \end{equation}
Therefore we proved:
\begin{sprop}\label{ziden}Every integral $I\in\mathcal{X}$ is equal to the integral in (\ref{zlastintegral}).
\end{sprop}
The AF $\tj^\prime(-\WW_n)\circ\FF^\prime$, is a conjugate transpose to  $\F_{n,k,k-1},$ and hence the conclusions in the first paragraph of the current subsection are reached (Theorem \ref{thD1} (resp. Theorem \ref{thD2} is used if and only if $k>2$ (resp. $k=2$)).  Note (by reasoning as in the second sentence of that paragraph) we also obtain the vanishing of the integrals $I$ not contained in $\mathcal{Y}$ for which $\frac{1}{2}\Di{\Om^\prime(\rho)}=\Di{D_{\tj^\prime(-\WW_n)\circ\FF^\prime}}.$ 

As an example of relating integrals in $\mathcal{Y}$ to known ones, we have:

\begin{sexample}\label{zktwo}Let $k=2$.  Let $\rho$ be chosen so that (i) $\Om^\prime(\rho)=\bB{n}{2}{1}{\lfloor\frac{n-2}{2}\rfloor}$, and (ii) one of the $GL_n$-parabolic subgroups from which $E$ is induced, is the one with Levi $GL_n\times GL_n$. Then by unfolding $I$ using the Eisenstein series expansion corresponding to this parabolic (instead of $P_N^{\mathrm{ac}}$), we obtain as inner integrals known Rankin-Selberg integrals (\cite{Ginzeul}, Section 5).\hfill$\square$\end{sexample}

Next we show that the integrals in $\mathcal{Y}$ satisfy a dimension equation which has been formulated by  D. Ginzburg. This equation is satisfied by many Rankin-Selberg integrals, and among the papers D. Ginzburg  addresses it are  \cite{GinzAn},\cite{Ginzeul}, \cite{Ginzsmall}, and with J. Hundley in \cite{GiH} (it is stated in these papers, respectively in equation (4), Experimental Fact (page 204), Definition 3, in the formula following formula (1)). In these papers, in relation to this equation special emphasis is given: in integrals unfolding to Fourier coefficients, and in integrals in which after unfolding the Whittaker Fourier coefficient appears at least once. The two remarks below together with the previous sentence explain a relation between these papers of D. Ginzburg and J. Hundley with the present paper. 

The dimension equation of D. Ginzburg for the set of integrals $\mathcal{Y}$, after using Corollary \ref{attachingfc}, becomes equation (\ref{GinD}) below, which in turn is obtained (in the proof below) as a corollary of Main corollary \ref{maincor}.

\begin{scor}Any $I\in\mathcal{Y}$  satisfies   the equation:
\begin{multline}\label{GinD}\frac{1}{2}\Di{\Om^\prime(\pi)}+\frac{1}{2}\Di{\Om^\prime\left(\Ind_{\Pac{N}(\A)}^{GL_n(\A)}(\rho\otimes\chi\otimes\trivr{U_P})\right)}=\\\Di{D_{\FF}}+\Di{GL_n}-\Di{Z_n}.\end{multline} 
\end{scor}

\begin{proof} From the definition of $\Om^\prime(\ast) $ and information on the Richardson orbit ((ii) in Fact \ref{rich3}) we have
\begin{equation}\label{zinduct}\frac{1}{2}\Di{\Om^\prime\left( \Ind_{P(\A)}^{GL_n(\A)}(\rho\otimes\chi\otimes\trivr{U_{\Pac{N}}})\right)}=\frac{1}{2}\Di{\Om^\prime(\rho)}+\Di{U_P}. \end{equation}  Main corollary \ref{maincor} implies
\begin{equation}\label{zrmc}\frac{1}{2}\Di{\Om^\prime(\rho)}=\Di{D_{\tj^\prime(-\WW_n)\circ\FF^\prime}}.\end{equation} These two formulas and the genericity of $\pi$ directly imply (\ref{GinD}). \end{proof}The two non-precisely stated remarks below, discuss why the corollary above is not an accident. \begin{sremark}[Similar to thoughts in \cite{GiH} in Section 4]By considering a direct\footnote{For example the right hand side is (again) chosen to equal  $\Di{H}-1$ for $H$ being the biggest group for which an integral with domain $H(\kkk)\s H(\A)$ is directly discerned.} analogue of (\ref{GinD}) for each\footnote{That is the three integrals: (\ref{intI}), (\ref{secint}), and (\ref{zlastintegral}).} integral representing $I$ and occurring during its unfolding, we see that the left  sides  of two such successive analogues differ by the same number as the right sides, partly as a consequence of the possibly nontrivial contribution being from the open orbit. That is, in the first unfolding step (which corresponds to the Eisenstein series expansion of $E$) the difference is  $\Di{U_{P_N^{\mathrm{ac}}}}$, and in the second one (which corresponds to the Fourier expansions of $\varphi$) there is no difference.\end{sremark}
\begin{sremark}\label{BGD}Let $\F$ be an AF in $\BBnk[N]$ which is defined over $\kkk$, and $\sigma\in\Aut{\kkk,N,>}.$ Assume that $\F(\sigma)$ is nonzero. Then Main Corollary \ref{maincor}, Corollary \ref{easyrefcor}, and Corollary \ref{attachingfc}, imply that:   $\F(\sigma)$ satisfies the dimension equation of D. Ginzburg if and only if it is factorizable.\end{sremark}

\section{Appendix.}\label{appendixs}We fix: a positive integer $n$ throughout the appendix, and a number field $\kkk$ in Subsections \ref{preaut} and \ref{Eisexp}.
\subsection{Prerequisities on $\UUU{n}$ (the nilpotent orbits of $\Lie{GL_n}$)}\label{nila}. We defined $\UUU{n}$ in Subsection \ref{inil} .

 Recall from linear algebra that: matrices with all the eigenvalues over a given field, are conjugate over the same field to a Jordan matrix. Therefore each orbit in $\UUU{n}$ is defined over the rational numbers, and the set of its rational points is a single $GL_n(\kkk)$-orbit.

\begin{slem}[Known]\label{ranks}
    Let $X$ and $Y$ be two nilpotent elements of $\Lie{GL_n}$. Then
    $$\Om(X)\geq \Om(Y)\iff\mathrm{rank}(X^k)\geq\mathrm{rank}(Y^k)\text{ for all }k>0. $$ 
    \end{slem}
    \begin{proof}The proof of this directly follows by  observing that for any nilpotent element $Z\in \Lie{GL_n}$ with $\Om(Z)$ equal to an orbit $[a_1,a_2,...]$, we have
    $$\text{rank}(Z^k)=\sum_{i|a_i\geq k}(a_i-k).$$
     It appears for example in \cite{Col} as Lemma 6.2.2.\end{proof}
   \begin{sdefi}[$\mathcal{A}\text{$[B_n]$}$]\label{boreln}Let\chh[xoris command]  $B_n$ be the Borel subgroup of $GL_n$. We define $$\AAnk[][B_n]:=\{\F\in\AnkT{n}:D_\F\triangleleft U_n\}=\{\F\in\AAnk[n]:D_\F\triangleleft B_n\}. $$ The second equality above is known and we never use it.\end{sdefi}
    
   \begin{slem}\label{nless}
   Let $\F\in\AAnk[][B_n]_{\nless}$. Let $a\in\UUU{n}$ be such that  $\X{\F}\cap a\neq\emptyset.$ Then $a\geq \Om(J_\F)$. 
    \end{slem}
    \noindent{Before} proceeding with the proof we remark that: the last two sentences can be replaced by saying $\Om(\F)=\{\Om(J_\F)\}$, but since we use this lemma before obtaining the equivalence of the definitions of $\Om(\F)$, we avoided mentioning $\Om(\F)$.
    
    \begin{proof} Let $X:=\{a: \F(U_a)\neq \{0\} \}$, $J\in \X{\F}$, and $k$ be a natural number. By doing the matrix multiplication $J^k$ (in this proof superscripts are alwasy powers), we obtain that the maximal elements of the set
    $$A(J^k):=\{\alpha\text{ is a root of }GL_n:J^k\text{ is nonzero in the entry corresponding to }\alpha\} $$
    are exactly the ones that take the form:  \begin{equation}\label{zsetb}\sum_{b\in Y}b\end{equation} for $Y$ being a subset of $X$ containing $k$ elements. Notice that
    
     \begin{equation}\label{zkoukisp}\text{rank}(J^k)\geq \text{number of maximal elements of }A(J^k). \end{equation}Due to (\ref{zsetb}), the right hand side of (\ref{zkoukisp}) is independent of the choice of $J$ inside $\X{\F}$. Notice that  in $A(J_\F^k)$ all the elements are maximal elements of it, and  by choosing $J$ to be $J_\F$ in (\ref{zkoukisp}) we obtain an equality. Then by letting again $J$ be any element in $\X{\F}$ we have 
     $$\text{rank}(J^k)\geq\text{rank}(J^k_\F). $$ From this and Lemma \ref{ranks} we obtain the proof.\end{proof}
    Throughout  the rest of the current subsection we fix a  $GL_n$-parabolic subgroup $P$ and  review some basic information on the Richardson orbit (which is defined below) openly intersecting $\Lie{U_P}$ (recall $U_P$ and $M_P$ respectively denote the unipotent radical and the Levi of $P$).  
    
    Some times that the Richardson orbit is mentioned, (i) in the Fact below is used without mention. Also certain uses of Fact \ref{ptelos} in the two previous sections are not mentioned (more on this in Remark \ref{pthp}).
\begin{sfact}[Known]\label{rich3}

\hspace{1mm}
\begin{enumerate}
\item[(i)] The action by conjugation of $P$ on $\Lie{U_P}$ has an open orbit (which is the unique open orbit). Let throughout the rest of the current fact $u$ be a matrix in this orbit. The orbit $\Om(u)$ in $\UUU{n}$ is called the \textbf{Richardson orbit openly intersecting} $\Lie{U_P}$ (of course it does not depend on $u$).
\item[(ii)]  $\Di{\Om(u)}=2\Di{U_P}.$
\item[(iii)] $\Om(u)$ is the biggest orbit in $\UUU{n} $ that intersects $\Lie{U_P}$.
\item[(iv)] $\Om(u)\cap\Lie{U_P}$ is a single orbit for the action of $P$ by conjugation on $\Lie{U_P}$ (hence the orbit in (i)).
\end{enumerate}
\end{sfact}
\begin{proof}(i) is part of Richardsons dense orbit theorem. (ii) follows from (i) and the topological description of the order in $\UUU{n}$ (see for example Theorem 6.2.5 in \cite{Col}). By using (iii), we obtain (iv) as a corollary of a result of Lusztig and Spaltenstein in \cite{Lusztig} (in \cite{Col} it appears in Theorem 7.1.1).

\end{proof}

\begin{sfact}[Known, see for example \cite{Col} 7.2.3]\label{ptelos}For $ GL_{n_1}\times...\times GL_{n_m}$ being isomorphic to the  Levi of $P$, the Richardson orbit in $\Lie{U_P}$ is equal to $[n_1,...,n_m]^t\cap \Lie{U_P}$ (here the superscript ``$t$" denotes the transpose partition).
\end{sfact}
\begin{proof}The reference above gives a proof; here we partly recall it.
 
An element $J$ of $[n_1,...,n_m]^t\cap \Lie{U_P}$ can be chosen as follows. We inductively define a partition of the first $n$ natural numbers into sets $X_i$
$$\{1,2,...,n\}=\bigcup_i X_i, $$ with the following steps: 

\vsp 
\noindent\textit{Step  $i.$} Consider the blocks of $M$ in which there are diagonal entries not chosen in a previous step (e.g in step 1 we consider all the blocks). We choose the set $X_i$ so that: 1) it is disjoint from $\cup_{1\leq j\leq i-1}{X_j}$, and 2) it consists of one diagonal entry from each of the blocks considered. 

\vsp\noindent{We } define $J$ to be an element  in $\Lie{GL_n}$ in which the nonzero entries are exactly the ones corresponding to the root vector spaces of $\cup_{1\leq i\leq m}\SRV{X_i}$. One directly discerns  the definition of transpose partition in these steps. From that we obtain that $|X_i|$ is equal to the $i$-th biggest element of $[n_1,...n_m]^t$, and thus that $J\in [n_1,...n_m]^t.$ 

For every $k>0$, one can establish  an upper bound for $\{\text{rank}(n^k):n\in \Lie{U_P}\}$ which is attained by rank$(J^k)$, and hence by using (iii) and (iv) in Fact \ref{rich3} we are done. 
\end{proof}
\begin{scor}[Known, used only in the proof of Theorem \ref{Eukerianf}]\label{constab}Let $u$ be in the Richardson orbit openly intersecting $U_P$. Then $$\{g\in GL_n:gug^{-1}\in \Lie{U_P}\}\subseteq P.$$ 

\end{scor}
\begin{proof}
Let $M_n$ be the space of  $n\times n $ matrices with entries in $\KKK$. For every subset $X$ of $M_n$ we denote by $C_X(u)$ the centralizer of $u$ in $X.$ Since $C_{M_n}(u)$ is a vector space, it is also an irreducible variety. Hence its open subvariety $C_{GL_n}(u)$ is also  irreducible, therefore connected, and since by Richardson's dense orbit theorem the identity connected component of $C_{GL_n}(u)$ is equal to $C_{P}(u)$, we have
\begin{equation}\label{zcentral}
C_{GL_n}(u)=C_{P}(u).
\end{equation}

Let $g$ be as in the statement of the corollary. Then by (iv) in Fact \ref{rich3}, we obtain a $p\in P$ such that $gug^{-1}=pup^{-1},$ and hence $p^{-1}g\in C_{GL_n}(u)$, which in turn---due to (\ref{zcentral})---  gives $g\in P.$
\end{proof}
\subsection{Prerequisities on $GL_n(\A)$-automorphic representations.}\label{preaut}

For automorphic forms we adopt the definition that can be found for example in \cite{MW3}. This means---by adopting the other definitions there\footnote{I mostly don't review them, because in the present paper we rarely need to mention any of them. The maximal compact subgroup $K=\prod_{\upsilon}K_\upsilon $ is chosen so that in each finite, real and complex place $\upsilon$ we respectively have $K_\upsilon=GL_n(\ooo_\upsilon), \{g\in GL_n(R):gg^t=1\},\{g\in GL_n: g\bar{g}^t=1\}$.}--- that a $GL_n(\A)$-automorphic form is a  function over $GL_n(\kkk)\s GL_n(\A)$ which is:  of moderate growth, smooth, $K$-finite, and $\mathfrak{z}$-finite (where $\mathfrak{z}$ is the center of the enveloping algebra of $\Lie{GL_n(\Ai_{\kkk,\infty} )}.$ By a theorem of Harish-Chandra, an automorphic form is of uniform moderate growth. By a $GL_n(\A)$-automorphic function we mean a function over $GL_n(\kkk)\s GL_n(\A)$ which is:  of uniform moderate growth, smooth, and $K$-finite.  A $GL_n(\A)$-automorphic representation is an irreducible subquotient of the $(\Lie{GL_n(\Ai_{\kkk,\infty})},K_\infty)\times GL_n(\Ai_{\kkk,\mathrm{f}}))$-module consisting of all $GL_n(\A)$-automorphic forms.

By the work of R.P. Langlands, in each $GL_n$-automorphic representation its elements are described as Eisenstein series induced from discrete data. The discrete spectrum in turn, turns out to consist only of Speh representations (\cite{MW2}). Next, we recall some definitions and facts, and fix notation,  about Eisenstein series and Speh representations.

\subsubsection*{Eisenstein series. } We denote by $\Ind$ the normalized  induction. 

Let $n=\sum_{1\leq i\leq m}n_i$ be a partition of $n$ on positive integers. Let $\sigma_1,...\sigma_m$ be $GL_{n_i}(\A)$-automorphic representations.  Let $P$ be the $GL_n$-parabolic subgroup with Levi $${\prod_{1\leq i\leq m}}^{\pse}GL_{n_i}.$$ Consider a function $f$ in the representation $\Ind_{P(\A)}^{GL_n(\A)}(\pi_1\otimes...\otimes\pi_m\otimes \trivr{U_P})$. Let $s=(s_1,...,s_b)\in C^b$. \chh[complex] We define $f_s$ to be the element of $\Ind_{P(\A)}^{GL_n(\A)} (|\det|^{s_1}\sigma_1\otimes...\otimes|\det|^{s_m}\sigma_m\otimes\trivr{U_P})$ for which:  $$f_s(g)=\prod_i|\det(h_i)|^{s_i}f(g),$$where $h_1,...,h_m$ are any matrices in $GL_{n_1},...,GL_{n_m}$ respectively, for which we have a decomposition
$$g=u\begin{pmatrix}h_1&\\&\ddots&&\\&&h_m\end{pmatrix}k $$ for  $u$ belonging to $U_n(\A)$ and $k$ belonging to the maximal  subgroup subgroup $K$ of  $GL_n(\A).$ If the real parts of $s_i-s_{i+1}$ for $1\leq i\leq m-1$ are big enough, the  Eisenstein series $E(f_s)$ is defined by the absolutely convergent\footnote{See for example \cite{MW3} II.1.5.} series:
$$E(f_s)(g):=\sum_{\gamma\in P(\kkk)\s GL_n(\kkk)}f_s(\gamma g),  $$and by the work of R.P. Langlands it turns out to have meromorphic continuation to $C^b$.  
\subsubsection*{Speh representations}
Let $a$ be a positive divisor of $n$, and $b:=\frac{n}{a}. $ Let $\tau$ be a $GL_a(\A)$-cuspidal automorphic representation, and $P$ be the $GL_n$-parabolic subgroup with Levi isomorphic to $GL_a^b.$ Then for certain choices of $f\in\Ind_{P(\A)}^{GL_n(\A)}(\tau\otimes...\otimes\tau\otimes\trivr{U_P})$ and $g\in GL_n(\A)$, it turns out that $E(g,f_s)$ has a multiresidue at $$\Lambda=\left(\frac{b-1}{2},\frac{b-3}{2},...,\frac{1-b}{2}\right).$$ The multiresidues obtained in this way for a fixed choice of $\tau$ (and fixed choice of $n$), form  a $GL_n(\A)$-automorphic representation (for irreducibility see \cite{MW2}) which is called a Speh representation. 
 
\subsection{Continuation of Section \ref{forms}}\label{Eisexp}

In the proof of  Theorem \ref{Eukerianf} we use the following known lemma:

\begin{slem}[Known]\label{GJL}
Let $\pi$ be a Speh representation of $GL_n(\A)$, defined by a $GL_a$-cuspidal automorphic representation. Let $\XX\in\AAA{(U_n)}$. Then 
$$\XX(\pi)\neq 0\implies\Om(J_{\XX})= \Om^\prime(\pi). $$

\end{slem}
\begin{proof} We consider separately the two next cases:

\vsp
\noindent\textit{Case $1$. } Assume $\Om(J_\XX)<\Om^\prime(\pi)$. Then $\XX(\pi)=0.$

\vsp 
\noindent\textit{Case $2$. } Assume $\Om(J_\XX)>\Om^\prime(\pi)$. Then $\XX(\pi)=0.$

\vsp
\noindent\textit{Proof of Case $1$.} Let $E$ be an automorphic form in $\pi.$ Let $\tau $ be the $GL_a(\A)$-cuspidal automorphic representation defining $\pi.$  We choose a function $f\in\Ind_{\PPPP(\A)}^{GL_n(\A)}(\tau\otimes...\otimes\tau\otimes \trivr{U_{\PPPP}}),$ such that the Eisenstein series $E(f_s),$ has a multiresidue at $s=\Lambda,$ equal to $E.$

 There is an AF $\YY\in\AAA[,n]$, such that  $\XX=\YY\circ\zAF$, for  $\zAF$ being  a trivial AF with domain the unipotent radical of  a $GL_n$-parabolic such that $\zAF(E(f_s))=0$, which in turn implies $\XX(E)=0 $ \hfill$\square$Case 1
   
   \vsp 
   \noindent\textit{Proof of Case $2.$ } Let $y$ be the biggest number in the partition  $\Om(J_\XX).$ We therefore have 
$y> a.$ Let $T$ be an element in $\TT{n}{\XX}$ satisfying $|\Set{T}|=y $.  Let $\XX_1$ be the restriction of $\XX$ at the unipotent radical of the $GL_n$-parabolic with Levi: \begin{equation}\label{zlevidec}GL_{\min(\Set{T})}\times^{\searrow}GL_1^{y-2}\times^{\searrow}GL_{\max(\Set{T})}. \end{equation} For an AF $\YY\in\AAA[,n]$  we have $\XX=\YY\circ\XX_1,$ and hence it is enough to prove that $\XX_1(\pi)=0.$

Starting with $\XX_1$, for $r=1,...,y-1$ we successively apply   $$\exchange{\prod_{1\leq j\leq \min(\Set{T})-1}U_{(\min(\Set{T})-1+r,j)}}{\prod_{1\leq i\leq \min(\Set{T})-1}U_{(i,\min(\Set{T})+r)}};$$ and then we conjugate with the minimal length element $w$ of $W_n$, for which  $\Set{wTw^{-1}}=\{1,...,y\}$. Let $\iota(\XX_1)$ be the output AF of this $\AAnk$-path. There is a $\YY^\prime$ such that $\iota(\XX_1)=\YY^\prime\circ\JJ_{y}$. Fact \ref{nonzeroeulerian} implies $\XX_1(\pi)=0\iff\iota(\XX_1)(\pi) =0$, and hence it is enough to prove that $\JJ_y(\pi)=0.$  This is proved in the second paragraph in the proof of Proposition 5.3 in \cite{Ginzbconj}. Alternatively to this use of \cite{Ginzbconj} ---if the reader prefers a global argument--- one can use Lemma 3.2 in \cite{JiangLiu}.\hspace{91mm}\hfill$\square$Case 2 
\end{proof}
\begin{sremark}Corollary \ref{partlysplit} slightly simplifies the proof above by implying that: given any $\F\in\AAA[,n]$ and any $\kkk$-subgroup $V$ of $D_\F$, we can find a $\Z\in\AAA[,n]$ such that $\F=\Z\circ\F|_V.$
\end{sremark} 
 \begin{stheorem}\label{Eukerianf}Let $\F\in\AAA(U_n),$ and $\pi\in\Aut{\kkk,n,>}$. 
 \begin{description}
 \item[Part A] Assume $\Om(J_\F)=\Om^\prime(\pi)$. Then $\F(\pi)$ is  nonzero factorizable.
 \item[Part B] Assume the orbit $\Om(J_\F)$ is bigger or unrelated to $\Om^\prime(\pi)$. Then $\F(\pi)=0$.
 \end{description}
 \end{stheorem} 
 \begin{tproof}We start with some notation and considering main formula (\ref{richardson}) below, we continue with  Part B, and finish with Part A (in the first step of which, the proof of Part B will be needed again). 
 
 Recall the Levi and the unipotent radical of any $GL_n$-parabolic subgroup $R$ are denoted by $M_R$ and $U_R$ respectively.
 
  Let $E$ be an element of  $\pi$, considered as an absolutely convergent Eisenstein series induced from discrete data over a $GL_n$-parabolic $P$, and let $f$ be the section of $E$ over this data.
 
 Let $m,n_1,...,n_m,a_1,...,a_m,b_1,...,b_m$ and $P_\pi$ be as in Definition \ref{Ppi}. We further assume for the order of blocks that  $M_P=\prod_{1\leq i\leq  m}^{\searrow}GL_{n_i}$ and $P_\pi=\prod_{1\leq i\leq m}^{\searrow}GL_{b_i}^{a_i}$. We call $\sigma_i$ the Speh representation on $GL_{n_i}(\A)$ appearing in the inducing data of $E$ over $P$.

      From the Eisenstein series expansion of $E$ over $P$ we obtain: 
 \begin{equation}\F(E)(g)= \sum_{w\in P(\kkk)\s GL_n(\kkk)/U_n(\kkk)}\int_{n\in (w^{-1}Pw\cap U_n)(\kkk)\s U_n(\A)}f(wng)\psi^{-1}_{\F}(n)dn. \end{equation}Due to the Bruhat decomposition, the  sum is replaced by the same sum  over the $w\in  W_{M_P}\s W_n.$ In this and every other identity that the variable $g$ appears, except for Part B, it is assumed that it is valid for all $g\in GL_n(\A).$ Representatives of cosets will frequently be identified with their cosets. The summand corresponding to a choice of $w$  is denoted by $\F(E,w)$. Therefore the previous expansion takes the following form, which will be referred in the rest of the proof as main formula (\ref{richardson}):
 \begin{equation}\label{richardson}\F(E)=\sum_{w\in  W_{M_P}\s W_n}\F(E,w). \end{equation} 
The  form of $\F(E,w), $ we will be using (usually without mention) is the one obtained by changing variables $n\rightarrow w^{-1}nw$, that is:
  \begin{equation}\label{conjoug}\F(E,w)(g)=\int_{(P\cap wU_n w^{-1})  (\kkk)\s {wU_nw^{-1} (\A)}}f(nwg)\psi_{w\F}(n)^{-1}dn. \end{equation}

 \vsp 
 \noindent\textit{\textbf{Proof of Part B.}} Assume that $\F(E)$ is nontrivial, and hence by main formula (\ref{richardson}) we fix a choice of $w$ and of $g$ (until the end of Part B) such that $\F(E,w)(g)\neq 0$. Notice that $J_{w\F}=wJ_\F w^{-1} $ (which we also use without mention).

 By knowing the left hand side below, we have
  \begin{multline}\label{mult}\F(E,w)(g)\neq 0\implies\text{ the integral $\int_{\kkk\s \A}\psi_{\kkk}^{-1}(x)dx$ is not}\\\text{factored out from  $\F(E,w)(g)$}\overset{(\ref{conjoug})}{\implies} J_{w\F}\in \Lie{P}\overset{  }{\iff} wJ_{\F}w^{-1}\in \Lie{P}.\end{multline}

 For every vector subspace $V$ of $\Lie{GL_n}$ we define $\Proj{V}$ to be the projection of $\Lie{GL_n}$ onto $V$ with respect to the inner product that maps each $X,Y\in\Lie{GL_n}$ to $\mathrm{tr}(\bar{X}^tY).$ We have
  
 \begin{multline}\label{zprsum}wJ_\F w^{-1}\overset{(\ref{mult})}{=}\Proj{\Lie{U_{P}}}(wJ_\F w^{-1})+\sum_{1\leq i\leq m}\Proj{\Lie{GL_{n_i}}}(wJ_\F w^{-1})\\=\Proj{\Lie{U_{P}}}(wJ_\F w^{-1})+\sum_{1\leq i\leq m}J_{(w\F)|_{GL_{n_i}}}. \end{multline}Since $\F(E,w)$ is nontrivial, by interpreting it as the integral in (\ref{conjoug}), we infer that its inner integral over $$(M_P\cap wU_nw^{-1})(\kkk)\s (M_P\cap wU_nw^{-1})(\A)$$ is also nontrivial, which means that $(w\F)|_{GL_{n_i}}(\sigma_i)$ is nontrivial for $1\leq i\leq m.$ Hence---by also noticing that $GL_{n_i}\cap wU_nw^{-1}$ is a unipotent radical of a possibly nonstandard Borel of $GL_{n_i}$--- we apply  Lemma \ref{GJL} for $[\pi\tlarrow\sigma_i]$ and obtain  $$\Om(\sigma_i)=\Om\left(J_{(w\F)|_{GL_{n_i}}}\right)\qquad\qquad\text{ for }1\leq i\leq m. $$
       Hence, for $1\leq i\leq m$, we find an element $\gamma_i\in GL_{n_i},$ such that $\gamma_i J_{(w\F)|_{GL_{n_i}}}\gamma_i^{-1}\in \Lie{U_{P_{\sigma_i}}}$. By defining $\gamma:=(\prod_{1\leq i\leq m}\gamma_i)w$, and using formula (\ref{zprsum}), we obtain 
     \begin{equation}\label{zbigconjj} \gamma J_\F {\gamma}^{-1}\in \Lie{U_{\PPPP}},\end{equation} and hence by (iii) in Fact \ref{rich3} we obtain $\Om(J_\F)\leq \Om^\prime(\pi)$. \hfill$\square$Part B
     
     \vsp 
     \noindent\textit{\textbf{Proof of Part A. Step 1.} We prove that at most one term in main formula (\ref{richardson}) is possibly  nontrivial.} Assume that for $w^1,w^2\in W_{n}$, the functions  $\F(E,w^1)$ and $\F(E,w^2)$ are both nontrivial. Let $\gamma^{1},$ $\gamma^{2}$ be such that the proof of Part B is valid for $[(w,\gamma)\tlarrow (w^1,\gamma^1),(w^2,\gamma^2)].$ Then by using formula (\ref{zbigconjj}) for $[\gamma\tlarrow\gamma^1,\gamma^2]$ and Corollary \ref{constab}  we have $\gamma^1(\gamma^2)^{-1}\in\PPPP$; and since (by definition) $w^1(\gamma^1)^{-1}\in P$ and $w^2(\gamma^2)^{-1}\in P$, we obtain $w^1(w^2)^{-1}\in P$.\hfill$\square$Step 1  
     
   \noindent\textit{\textbf{Step 2.} Expressing $\F(E)$ as a factorizable integral.}  By using Fact \ref{ptelos} and its proof\footnote{Let $J$ be as in this fact for $[P\tlarrow\PPPP]$. Let $w_1$ be the element in $W_n$ determined by $w_1 J_\F w_1^{-1}=J.$ Since the blocks of same size are filled by the partition in the same step,  we obtain an element $w_2\in W_{M_P}$ such that for $w_A:=w_2w_1$ we obtain the three properties.} we obtain an element $w_A\in W_n$ such  that: \begin{enumerate}
       \item $w_AJ_\F w_A^{-1}\in \Lie{P}$;
       \item $P^t\cap w_AU_nw_A^{-1}=\prod{\jj U_{n_i}}N^-$, where $N^-$ is a unipotent subgroup of $U^t_{P}$, which is normalized by $\prod{\jj U_{n_i}}$.
       \item $\Pr_{\Lie{GL_{n_i}}}(J^t_{w_A\F})$ is the Jordan matrix with all its Jordan blocks being of dimension $a_i\times a_i$ (recall, we defined $\Pr$ exactly before the formula (\ref{zprsum})).
       \end{enumerate}
       This together with main formula (\ref{richardson}) and step 1, gives:
        \begin{equation}\label{zeubrad}\F(E)(g)=\int_{N^-(\A)}\int_{\prod U_{n_i}(\kkk)\s \prod U_{n_i}(\A)}f(xn^-w_Ag)(w_A\psi_\F)^{-1}(xn^-) dxdn^-. \end{equation}

       \vsp 
              \noindent\textbf{Claim. }(Known).\textit{  Consider integers $k=ab$, and let $\rho$ be a $GL_k(\A)$-Speh automorphic representation, defined by a $GL_a(\A)$-cuspidal automorphic representation $\tau$. Let $Q$ be the  $GL_k$-parabolic subgroup with Levi isomorphic to $GL_a^b$. Then applying the constant term on $U_Q$ to the elements in $\rho$, gives a nontrivial intertwining operator: 
              $$\rho\rightarrow \Ind_{Q(\A)}^{GL_k(\A)}\left(|\det|^{-\frac{b-1}{2}}\tau\otimes|\det|^{-\frac{b-3}{2}}\tau\otimes...\otimes |\det|^{\frac{b-1}{2}}\tau\otimes\trivr{U_{Q}}\right). $$ }
              
              \vspace{1mm}
              \noindent\textit{Proof of Claim.} Let $\ffff$ be a decomposable vector in $\Ind_{Q(\A)}^{GL_k(\A)}(\tau\otimes ... \tau\otimes\trivr{U_{Q}}).$ The constant term on $U_Q$ of $E(\ffff_s)$ is the  finite sum of intertwining operators applied to $\ffff_s$ given by:
              \begin{equation}\label{zsumi}\sum_{w}M(w)(\ffff_s), \end{equation}where $$M(w)(\ffff_s)(g):=\int_{(w^{-1}U_{Q}w\cap U_{Q})(\A)\s U_{Q}(\A)}\ffff_s(wng)dn,$$ and $w$ takes the values in the set $$\{w:wM_Qw^{-1}=M_Q,\text{ and }w\text{ is a minimal length representative of a coset in }Q\cap W_k\s W_k\}. $$Let $w_0$ be the element of this set with the maximal length. The proof is reduced to proving that:\begin{enumerate}
              \item[(a).] by applying the multiresidue at $s=\Lambda$ at (\ref{zsumi}), all terms  except possibly the term corresponding to $w_0$, vanish.
              \item[(b).] for appropriate choice of $\ffff$ the term (of the multiresidue) corresponding to $w_0$ doesn't vanish.
              \end{enumerate}
              
              Each term $M(w)(\ffff_s)$ is expressed as a tensor product of local intertwining operators evaluated at the local components of $\ffff_s.$ By using the Gindikin-Karpelevich formula, the part of this product over the  places in which all the data is unramified,  is expressed  as a product of a spherical vector times a fraction of certain (incomplete) global $L$-functions such that:\begin{enumerate}
                  \item the  $L$-function in the  numerator\footnote{The numerator and denominator refer to the fraction of $L$-functions and not the fraction of polynomials equal to it} admits the multiresidue at $s=\Lambda$, if and only if $w=w_0$;
                  \item for $s=\Lambda,$ the  $L$-function in the denominator, is evaluated  inside the region of absolute convergence.
                  \end{enumerate} 
                  For a full explanation of the previous two sentences see \cite{Shahidi} Theorem 6.3.1.  
                  
                  Therefore, to finish the proof of statements $(a)$ and $(b),$ we need to check that for the finitely many places in which part of the data is ramified, the corresponding component of the tensor product:  1)does not have a pole for $s=\Lambda,$  
                   and 2) does not vanish for all $\ffff$ with an appropriate local component at $\upsilon.$ These follow from \cite{MW2} I.10.\hfill$\square$Claim

       \vspace{1mm}
       \noindent{Let $\sigma_i^\prime$ } be the irreducible subrepresentation of $$\Ind_{Q_{i}(\A)}^{GL_{n_i}(\A)}\left(|\det|^{-\frac{b-1}{2}}\tau_i\otimes|\det|^{-\frac{b-3}{2}}\tau_i\otimes|\det|^{\frac{b-1}{2}}\tau_i\otimes\trivr{U_{Q_{i}}}\right)$$where $Q_i$ is the $GL_{n_i}$-parabolic subgroup in the place of $Q$ when the claim above is stated for $[k\tlarrow n_i]$. This means that $\sigma_i^\prime$ is isomorphic to $\sigma_i.$ Let  
       $$\sigma^\prime:=\sigma_1^\prime\otimes...\otimes\sigma_m^\prime,\qquad\text{ and }\qquad\pi^\prime:=\Ind_{P(\A)}^{GL_{n}(\A)}(\sigma^\prime\otimes\trivr{U_P}). $$ Therefore $\sigma^\prime$ and $\pi^\prime$ are isomorphic respectively to $\sigma $ and $\pi.$ From this and from (\ref{zeubrad}) we obtain that: for any choice of $f^\prime\in\pi^\prime$, the Eisenstein series $E$ can be chosen so that it further satisfies 
       \begin{equation}\label{zeuleriannn}\F(E)(g)=\int_{N^-(\A)}\int_{\prod_{i,j} U_{a_i}^j(\kkk)\s \prod_{i,j} U_{a_i}^j(\A)}f^\prime(xn^-w_Ag)(w_A\psi_\F)^{-1}(xn^-)dxdn^-, \end{equation} where $U_{a_i}^j$ is the group of the unipotent upper triangular matrices of the $j$-th copy (starting from the upper left) of $GL_{a_i}$ inside $GL_{n_i}$. Since the inner integration equals Whittaker functions of cusp forms, by using the uniqueness of Whittaker models, we conclude that for a decomposable vector $f^\prime$ the integral $\F(E)$ is factorizable.\hfill$\square \text{Step 2} $
       
       \vspace{1mm}
       \noindent\textit{\textbf{Step 3.} Proving that $\F(E)$ is nonzero. }   Let $f^\prime_A=w_Af^\prime$, and $\psi_{\F,A}:=w_A\psi_\F$. We require that $f^\prime$ is decomposable. Hence $f^\prime_A$ is also decomposable, and let
       $$f^\prime_A=\bigotimes_{\upsilon}f^\prime_{A,\upsilon} $$ be a decomposition it admits.  The Whittaker function we encountered is denoted by: $$\lambda(f_A^\prime)(g):=\int_{\prod_{i,j} U_{a_i}^j(\kkk)\s \prod_{i,j} U_{a_i}^j(\A)}f_A^\prime(xg)\psi^{-1}_A(x)dx. $$ It admits a factorization:
       $$\lambda(f_A^\prime)(g)=\prod_{\upsilon}\lambda_\upsilon(f^\prime_{A,\upsilon})(g_\upsilon), $$
       for $\lambda_\upsilon$ being a model on $\sigma^\prime_\upsilon$ satisfying $\lambda_\upsilon(f^\prime_{A,\upsilon}(ng))=\psi_A(n)\lambda_\upsilon(f^\prime_{A,\upsilon}(g)) $ for all $n\in\prod_{i,j} U_{a_i}^j(\kkk_\upsilon)$, and for almost all places $\upsilon$ normalized by $\lambda_\upsilon(f^\prime_{A,\upsilon})(1)=1$.

       Formula (\ref{zeuleriannn}) for $g=1$ is now rewritten as:
       \begin{equation}\label{zproda}\F(E)(1)=\int_{N^-(\A)}\lambda(f^\prime_A)(n^-)\psi^{-1}_{\F,A}(n^-)dn^-=\prod_{\upsilon}\F(E)(1)_\upsilon, \end{equation} where $$\F(E)(1)_\upsilon:=\int_{N^-(\kkk_\upsilon)}\lambda_\upsilon(f^\prime_{A,\upsilon})(n^-)\psi^{-1}_{\F,A,\upsilon}(n^-)dn^-. $$
       
      For any place $\upsilon$ let $\ooo_\upsilon$ be the ring of integers of $\kkk_\upsilon$. By replacing in (\ref{zproda}) the integrant with its absolute value, we see that the convergence and nonvanishing of $\int_{N^-(\A)}|\lambda(f^\prime_A)(n^-)|dn^-$  means that the product
      $$\prod_{\upsilon}\int_{N^-(\kkk_\upsilon)}|\lambda_\upsilon(f^\prime_{A,\upsilon})(n_\upsilon^-)|dn_\upsilon^-=\prod_{\upsilon}\left(1+\int_{N^-(\kkk_\upsilon)-N^-(\ooo_\upsilon)}|\lambda_\upsilon(f^\prime_{A,\upsilon})(n_\upsilon^-)|dn_\upsilon^-\right) $$
      is absolutely convergent, and hence, since $|\psi^{-1}_{\F,A,\upsilon}(n^-)|=1$ for $n^-\in\ooo_\upsilon$ and almost all places $\upsilon,$ there is a finite set of places $S$, such that the product  
       $$\prod_{\upsilon\not\in S}\left(1+\int_{N^-(\kkk_\upsilon)-N^-(\ooo_\upsilon)}\lambda_\upsilon(f^\prime_{A,\upsilon})(n_\upsilon^-)\psi^{-1}_{\F,A,\upsilon}(n^-)dn_\upsilon^-\right)=\prod_{\upsilon\not\in S}\int_{N^-(\kkk_\upsilon)}\lambda_\upsilon(f^\prime_{A,\upsilon})(n_\upsilon^-)\psi^{-1}_{\F,A,\upsilon}(n^-)dn_\upsilon^- $$is absolutely convergent as well (and hence nonzero). 
       Therefore to finish the proof we only need to prove the following claim. 
        
      \vspace{1mm}   
       \noindent\textbf{Claim} (Known).\textit{  Let  $\upsilon$ be any place. There is an element of $\Ind_{P(\kkk_\upsilon)}^{GL_n(\kkk_\upsilon)}(\sigma_\upsilon^\prime\otimes\trivr{U_P}),$ so that for all choices of decomposable $f^\prime$ for which $f^\prime_{A,\upsilon}$ equals to that element, the term $\F(E)(1)_\upsilon$ is nonzero. }

       \vspace{1mm}
       \noindent\textit{Proof of Claim}. The proof that follows is the same as the  proof in \cite{Shahidi} (pages 51 and 52). Consider functions $\phi_1^\prime\in C_c^\infty(M_{P}(\kkk_\upsilon)),$ $\phi_1^{\prime\prime}\in C_c^\infty(U_{P}(\kkk_\upsilon)),$ and $\phi_2\in C_c^\infty(U_P^t(\kkk_\upsilon)),$ let $u$ be a vector in  $\sigma_\upsilon^\prime$ and let $\phi:C_c^\infty(GL_n(\kkk_\upsilon))$ given by
       
        $$\phi(g)=\left\{\begin{array}{cc}\phi^{\prime}_1(m)\phi_1^{\prime\prime}(n)\phi_2(n^-)u&g=mnn^- \text{ where }m\in M_P(\kkk_\upsilon),n\in U_P(\kkk_\upsilon),n^-\in U_P^t(\kkk_\upsilon)\\0&\text{ otherwise } \end{array}\right..$$   
        We can choose $f_{A,\upsilon}$ so that 
       $$f_{A,\upsilon}(g)=\int_{P(\kkk_\upsilon)}\sigma_\upsilon^\prime(m^{-1})\delta_P^{\frac{1}{2}}(m^{-1})\phi(mng)dmdn,$$ where $\delta_P^{\frac{1}{2}} $ is the normalization factor of the  induction.
       We have 
       
       $$\F(E)(1)_\upsilon=\int_{N^-(\kkk_\upsilon)}\phi_2(n^-)\psi^{-1}_{\F,A,\upsilon}(n^-)dn^-\int_{U_{P}(\kkk_\upsilon)}\phi_1^{\prime\prime}(n)dn\lambda_\upsilon((\sigma^\prime\otimes \delta_P ^{\frac{1}{2}})(\phi_1^\prime)u). $$     
       The three functions $\phi_2$, $\phi_1^{\prime\prime}$ and $\phi_1^\prime$, and the vector $u$, can easily be chosen to make each among the previous three terms  nonzero.\hfill$\square$Claim$\square$Step 3$\square$Part A

   \end{tproof}
   
    \begin{sdefi}[$\Om(\pi)$] For attaching Fourier coefficients to orbits in $\UUU{n} $  and defining $\Om(\pi),$ we follow the definitions exactly as stated in \cite{Ginzsmall} Section 3, except that there, unipotent (instead of nilpotent) orbits are used.\end{sdefi}    
   \begin{scor}\label{attachingfc}
   Let $\pi\in\Aut{\kkk,n,>}$. Then $\Om(\pi)=\{\Om^\prime(\pi)\}. $  
   \end{scor}
   \begin{tproof}The proof directly follows from the Claim below, Lemma \ref{nless}, and  Corollary \ref{easyrefcor} (used in this order). The claim below is certainly known, but I do not know a reference and hence I prove it. 
   
   \vsp 
   \noindent\textbf{Claim }(Known). \textit{For an $\Om\in \UUU{n}$ let $\F\in\AAA[,n]$ be such that $\psi_{\F}=\psi_{U_2(\Om)}$ (as part of the definition above,  $\psi_{U_2(\Om)}$ is chosen in any way as in \cite{Ginzsmall} Section 3).  Then for a matrix $\gamma\in GL_n(\kkk)$ we have: $\gamma\F\in\AAA[][B_n]_{\nless} $ and $\Om(J_{\gamma\F})=\Om$.} 
   
   \vsp 
   \noindent\textit{Proof of Claim}. We fully adopt the notation in \cite{Ginzsmall} in Section 3, except that $\Om\in\UUU{n}$ (instead of being a unipotent orbit). 
      
      First we explain that there is a $\gamma^\prime\in M(\Om)$ such that $\gamma^\prime u_2(\Om){\gamma^\prime}^{-1}$ belongs to a $\ossa$-group.  Let $\F^\prime $ be the AF satisfying $D_{F^\prime}=U^t_2(\Om)$ and $J_{\F^\prime}=u_2(\Om).$ Let $A$ be the subset of $\{1,...,n\}$ such that: $x\in A$ if and only if the $x$-th row of $GL_n$ intersects $h_a(t)$ at an even power of $t.$ Also let $B=\{1,...,n\}-A.$ Then $\gamma^\prime$ is obtained by applying the transpose version of  Lemma \ref{conjugatehat} for $[\F\tlarrow \jm{\F^\prime}^A,\jm{\F^\prime}^B].$ 
      
      By replacing (only in this sentence) $u_2(\Om)$ with $\gamma^\prime u_2(\Om){\gamma^\prime}^{-1}$, one easily describes explicitly the set  $\mathrm{Stab}_\Om\cap T_n,$ and notices that it is connected and acts nontrivially on every entry of $U_3(\Om)$. Therefore $\gamma^\prime\F(U_3(\Om))=\{0\}$ and hence $\F(U_3(\Om))=\{0\}.$  By applying  Lemma \ref{conjugatehat} for $[\F\tlarrow \jm{\F}^A,\jm{\F}^B]$ (although we could have avoided this second use of it), we obtain a matrix $\gamma\in M(\Om)(\kkk)$ such that $\gamma\F^{\prime\prime}\gamma^{-1}\in \AAA[][B_n]_{\nless} .$ Since $\gamma D^t_\F\gamma^{-1}$=$D^t_\F$ we have $J_{\gamma\F\gamma^{-1}}=\gamma J_{\F}\gamma^{-1},$ and hence we also have $\Om(J_{\gamma\F})=\Om$.\hfill$\square$Claim
      
      \end{tproof}
   \section{Index of notations}
   AF, $\AAnk[n,]$ $\AAnk$, $D_\F$, $\AAnk(N) $, $\F_{\emptyset,n}$, ``$\kkk$" in the subscripts, $\psi_\kkk,$ $\psi_\F$, $\F(\phi)$, $\Aut{\kkk,n}$, $\Aut{\kkk,n,>}$, factorizable\hin  \ref{FCAut} and \ref{clar}\\$\gamma\F$\hin\ref{treesroughly}\\$\Prink[n]$ \hin\ref{treesroughly} (second last paragraph)\\$\KKK$\hin\ref{vara}  $P$, $M_P$, $U_P$, $U_n$, $P_n$, $\mirsl{n},$ $U_{n,(i,j)}$ or $U_{(i,j)}$, $T_n$, $W_n$, $W_H$\hin\ref{parroot}\\$\mathrm{tr}$, $(...)^t$, $\Lie{...}$\hin\ref{trtr}\\$\kkk$\hin\ref{clar}\\$|...|$\hin\ref{sa1}\\$[H_1,H_2]$\hin\ref{sa2}\\ $\JJ_r^n$ or $\JJ_r$\hin\ref{zjr}\\ $\F|_H$\hin\ref{nrest}\\ $\Stab{H}{\F}$, centralizer\hin\ref{scent}\\$\X{\F}$\hin\ref{vardef} \\$\AnkT{n}$, $\hat{(...)}$,  $(...)_\nless$\hin\ref{hatnless}\\ $J_\F$\hin \ref{uniqJ}\\ $j(\F)$\hin\ref{embjf}\\ Standard embedding, standard copy\hin \ref{sesc}\\ $\DDDD[n]$, $\Set{T}$, $\SRV{T}$, $\DDDD[n]$-components of $\TT{n}{\F}$, $\Set{\TT{n}{\F}}$\hin \ref{zdnf} and \ref{dnsetdn}\\ $\jj H$, $\ja\F$, $\jm\F$, $GL_n^T$, $GL_n^S$, $\F^T$, $\F^S$\hin\ref{jjj}\\ $\times,$ $\prod$,  Blocks of a Levi, $\times^\searrow,$ $\prod^\searrow$\hin \ref{tpts} and \ref{searrow}\\ $\UUU{n}$, $[a_1,...,a_m]$, $>$ (for orbits), $\Om(X)$\hin \ref{inil}\\ $[...\tlarrow...]$\hin \ref{tlarrown}\\$\circ$\hin\ref{circd}\\ $\rcirc$\hin\ref{rcircd}\\$\AAnk$-operation, $\e$-operation, $\e(V)$, $\eu$-operation, $\exchange{X}{Y}$, conjugation\hin\ref{Aoper}, \ref{eoper}, \ref{coper}, and \ref{euoperation}\\$\AAnk$-labeled out-tree\hin\ref{Alab1} and \ref{Alab2}\\Subtree\hin\ref{subtree}\\Input vertex, output vertex, input AF (resp AF, output AF) (all refering to an $\AAnk$-labeled out-tree)\hin\ref{outAl}\\Depth \hin\ref{outinput}\\$\AAnk$-steps, $\e$-steps, $\eu$-steps, $\co$-steps \ref{Astepdef}, \ref{estepdef}, \ref{eusteps}, and  \ref{costepsdef} \\action on an $\e$-step \hin\ref{actionestep}\\$\AAnk$-trees\hin\ref{trees}\\$(\F\rightarrow S,\ast,H,\kkk)\text{-tree}$\hin\ref{moretrees}\\$\AAnk$-paths, $\AAnk$-quasipaths, $\quasi{\Xi}$ \hin\ref{quasidef}\\$j(\Xi)$, $\Xi\circ...$\hin\ref{ctree}\\$\vee$\hin\ref{veenot1} and \ref{veenot2}\\$\mathrm{I}^{-1}$\hin\ref{inversep}\\Terms of an $\e$-step, constant and nonconstant terms\hin\ref{termse}\\Along constant terms\hin\ref{alongconstant}\\Over the groups (refering to a tree)\hin\ref{ovtg}\\Conjugate of an $\AAnk$-tree\hin\ref{conjt}\\$\sXi(\F)$ \hin proof of Theorem \ref{general}\\ $\sI(\F) $\hin\ref{sIdef} \\$\sXi$, $\sI$\hin\ref{sXom}\\$\ossa$-group\hin\ref{ossagroup}\\$...[a]$ (e.g. $\Prink[n][a] $)\hin\ref{afise}\\$\sI[,k](\F)$\hin\ref{sIk}\\$\Om^\prime(\pi)$\hin\ref{Ppi}\\$ \itr{\sett{X}}{N}$\hin\ref{zarlem}\\$\BBnk[n]$, $\BBnk[n]$-paths\hin\ref{defBBB}\\$\BR{n}$\hin\ref{brow}\\$\Pridbnk[n]$, $\Priubnk[n]$  \hin\ref{sbsquares}\\$\Om(\F),\mult{a}{\F}$\hin\ref{OFX} and \ref{AOF}\\$...\cap^\prime\BBnk[n] $\hin\ref{capprime}\\$\Omm{\F}$\hin\ref{Ofin}\\$\mathrm{rg}(\F)$\hin\ref{rgop}\\$H$-minimal $\ossa$-group, $\rossa{V}$, \ref{VpV}\\$j_{\FF}$, $\tj$, $P(\FF)$, $p_C$, $p_i$, $\jj^{-1}\cdot p_i$\hin\ref{dem}\\$\mathcal{S}_{n,N}$\hin\ref{pts}\\$\rif$\hin\ref{lrif}\\$\Om_k$\hin\ref{okato}\\$\F_{n,k,l},$ $\FF_{n,k,l}$\hin\ref{zfnkl}\\$[...[...]...],$ $[...,\Om]$ \hin\ref{orbcomp}\\$\Skal{k}{(a_1,l_1),...,(a_m,l_m)}$, $\Skac{k}{(a_1,l_1),...,(a_m,l_m)} $\hin\ref{cadhoc}\\$\langle S\rangle$\hin\ref{generategroup}\\$\F_{a,k}$\hin\ref{ladefi}\\$\Pac{x}$\hin\ref{accociate}\\$I$, $\XXX$\hin\ref{fldef}  \\$\mathcal{Y}$\hin\ref{fldef2}\\$\mathrm{diag}_{\ast}(\ast)$\hin\ref{fldef3}\\Richardson orbit\hin\ref{rich3}\\$\mathrm{Ind}$, $f_s$, $E(f_s)$, Speh representations\hin\ref{preaut}\\$\Om(\pi)$\hin above \ref{attachingfc}

\end{document}